\documentclass[10pt]{amsart}

\usepackage{graphicx}
\usepackage{times}
\usepackage[colorlinks=true,linkcolor=blue,citecolor=blue]{hyperref}%

\usepackage{xcolor}
\usepackage{stmaryrd}
\usepackage{pdfrender,xcolor}

\newtheorem{theorem}{Theorem}[section]
\newtheorem{lemma}[theorem]{Lemma}
\newtheorem{corollary}[theorem]{Corollary}
\newtheorem{prop}[theorem]{Proposition}

\newtheorem{ass}[theorem]{Assumption}

\newtheorem{pprop}[theorem]{Heuristic}
\newtheorem{plemma}[theorem]{Heuristic}
\theoremstyle{definition}
\newtheorem{definition}[theorem]{Definition}

\theoremstyle{remark}
\newtheorem{remark}[theorem]{Remark}
\numberwithin{equation}{section}

\usepackage[top=0.75in, bottom=0.75in, left=0.75in, right=0.75in]{geometry}

\usepackage[scr]{rsfso}
\usepackage[english]{babel}
\usepackage{fancyhdr}
\usepackage{amsmath}
\usepackage{amsthm}
\usepackage{amssymb}
\usepackage{eucal}
\usepackage{comment}
\usepackage{enumitem}
\setlist{leftmargin=*}
\usepackage[integrals]{wasysym}

\newcommand\nc{\newcommand}
\nc{\on}{\operatorname}
\nc{\E}{\mathbf{E}}
\nc{\R}{\mathbb R}
\nc{\C}{\mathbb C}
\nc{\Q}{\mathbb Q}
\nc{\Z}{\mathbb Z}
\nc{\N}{\mathbb N}
\nc{\F}{\mathbb F}
\nc{\wt}{\widetilde}
\nc{\ol}{\overline}
\nc{\short}[3]{0 \longrightarrow #1 \longrightarrow #2 \longrightarrow #3 \longrightarrow 0}
\nc{\pd}[2]{\frac{\partial #1}{\partial #2}}
\nc{\rnc}{\renewcommand}
\nc{\e}{\varepsilon}
\nc{\DMO}{\DeclareMathOperator}
\nc{\grad}{\nabla}
\nc{\Exp}{\mathbf{Exp}}
\nc{\fsp}{\fontdimen2\font=2.17pt}

\rnc{\leq}{\leqslant}
\rnc{\geq}{\geqslant}
\rnc{\d}{\mathrm{d}}
\rnc{\O}{\mathrm{O}}
\rnc{\exp}{\mathbf{Exp}}
\newenvironment{nouppercase}{%
  \renewcommand{\uppercasenonmath}[1]{}}{}
\title{\fsp\Large {Fluctuations for some non-stationary interacting particle systems via Boltzmann-Gibbs Principle}}

\author{ \large Kevin Yang}

\usepackage{setspace}
\begin{document}
\setstretch{0.99}
\fsp
\raggedbottom
\begin{nouppercase}
\maketitle
\end{nouppercase}
\begin{center}
\today
\end{center}

\begin{abstract}
{\fontdimen2\font=1.8pt Conjecture II.3.6 of Spohn in \cite{FSPDE1} and Lecture 7 of Jensen-Yau in \cite{FSPDE2} ask for a general derivation of universal fluctuations of hydrodynamic limits in large-scale stochastic interacting particle systems. However, the past few decades have witnessed only minimal progress according to \cite{FSPDE3}. In this paper, we develop a general method for deriving the so-called Boltzmann-Gibbs principle for a general family of non-integrable and non-stationary interacting particle systems, thereby responding to Spohn and Jensen-Yau. Most importantly, our method depends mostly on \emph{local} and \emph{dynamical}, and thus more general/universal, features of the model. This contrasts with previous work \cite{BR,CYau,GJ15,JM}, all of which rely on global and non-universal assumptions on invariant measures or initial measures of the model. As a concrete application of the method, we derive the KPZ equation as a large-scale limit of the height functions for a family of non-stationary and non-integrable exclusion processes with an environment-dependent asymmetry. This establishes a first result to Big Picture Question 1.6 in \cite{KPZAIM} for non-stationary and non-integrable ``speed-change" models that have also been of interest beyond KPZ \cite{DPSW,F,FHU,K}.} 
\end{abstract}

{\hypersetup{linkcolor=blue}
\setcounter{tocdepth}{1}
\tableofcontents}

\section{Introduction}\label{section:intro}
Deriving rigorously continuum equations of classical fluid mechanics as large-scale descriptions of locally conserved quantities in Newtonian particle systems is a famous open problem in mathematical physics. However, it has seen little progress \cite{BLRB}. Morrey \cite{Mor} gave a formal derivation based on \emph{local equilibrium} and \emph{local Gibbs states}, but rigorous proof of necessary local ergodicity of Hamiltonian systems has remained elusive. Considering instead statistical mechanical systems, which may be viewed as Hamiltonian systems with additional randomness,  largely resolves this difficulty. Indeed, there has been remarkable progress on deriving many continuum fluid equations, known as \emph{hydrodynamic limits}, from stochastic interacting particle systems, largely based on the works of Guo-Papanicolaou-Varadhan \cite{GPV}, Varadhan \cite{V}, and Yau \cite{YauRE} that make precise the two notions of local equilibrium and local Gibbs states for stochastic systems; see \cite{DLPP,KL,FSPDE1} for thorough reviews in this direction. 

However, a complete picture of hydrodynamic equations via statistical mechanics requires understanding \emph{conjecturally universal fluctuations} of locally conserved quantities in the stochastic model about hydrodynamic limits. {(By ``universal", we mean a scaling limit for fluctuations that does not depend on the precise microscopic structures of the system at hand but only the choice of the scaling and a few numbers or moments.)} To this end, much less is known. We discuss the history of this universality problem shortly. To highlight its significance, a general derivation of universal local fluctuations was asked for by Spohn \cite{FSPDE1}, in the form of Conjecture II.3.6, and by Jensen-Yau \cite{FSPDE2}, in the form of an open problem in Lecture 7; almost no progress has been made in the past few decades according to \cite{FSPDE3}. {Let us expand on this more precisely.
\begin{itemize}
\item Conjecture II.3.6 in \cite{FSPDE1} asks the question of how to use \emph{local} statistics to derive scaling limits for fluctuations of hydrodynamic limits in non-stationary interacting particle systems. In particular, the physical reasoning given therein supposes that the non-stationary particle system is sufficiently close to stationary at local scales. (This is the ``extended local equilibrium hypothesis" therein.) Using this information, one can then deduce formally what the scaling limit for fluctuations should be. The question, which is what Conjecture II.3.6 asks, is how to prove (any of) this rigorously. (Technically, Conjecture II.3.6 in \cite{FSPDE1} asks about non-stationary particle systems whose hydrodynamic limits are also non-constant. We do not address this case simply because the scaling for the models that we study in this paper does not seem to allow for it. We clarify this later in the introduction. In any case, the heart of Conjecture II.3.6 in \cite{FSPDE1} is a method of using local statistics and local stationarity to derive scaling limits for fluctuations. As smooth is approximately constant on local scales, a thorough investigation for non-stationary systems with constant hydrodynamic limits should, in principle, shed light on the case of non-constant but smooth hydrodynamic limits.)
\item Problem 3 in Section 7 of \cite{FSPDE2} asks the same as Conjecture II.3.6 in \cite{FSPDE1} with the same interest in non-stationary models. (\cite{FSPDE2} also emphasizes interest in ``non-equilibrium" case, which includes the case of non-constant hydrodynamic limit discussed in the previous bullet point.) \cite{FSPDE2}, however, notes that for non-stationary and/or non-equilibrium models, only one result due to Chang-Yau \cite{CYau} had been available at the time. In particular, \cite{FSPDE2} asks for progress beyond this work, which already addressed a large class of models. We further discuss \cite{CYau} shortly. (\cite{FSPDE2} also asks for scaling limits for fluctuations in non-stationary and non-equilibrium systems in dimension $\d\geq2$. We do not address this case in this paper, and we leave it for future work.)
\end{itemize}
}Additionally, since Spohn \cite{FSPDE1} and Jensen-Yau \cite{FSPDE2}, there has been a surge of activity and interest in \emph{nonlinear KPZ statistics} (where ``KPZ" means Kardar-Parisi-Zhang) as large-scale limits of fluctuations \cite{KPZAIM}. To this end, even less is known.

{\bf We respond to these conjectures and open problems with a general derivation of the so-called \emph{Boltzmann-Gibbs principle}} based on local dynamic properties of the stochastic model as asked for by Spohn \cite{FSPDE1}. It fuses well enough with stochastic analytic methods to rigorously derive KPZ fluctuations from a large class of stochastic particle systems that are beyond perturbations of stochastically reversible models and therefore in some version of non-equilibrium, in spirit of Problem 3 in Section 7 of \cite{FSPDE2}. To start, we now discuss relevant prior work and questions of Spohn \cite{FSPDE1} and Jensen-Yau \cite{FSPDE2}; see also Chapter 11 of \cite{KL}.
\begin{itemize}
\item The Boltzmann-Gibbs principle was originally developed by Brox-Rost \cite{BR} to derive hydrodynamic limit fluctuations. Their method succeeds only for statistically stationary/equilibrium systems. It has since been streamlined \cite{KL,KLO} and derived for many equilibrium models \cite{C,DPSW,DG,GJ15,GJS15,GP,LV,S1,S2,S3,Z}. However, assuming, or even explicitly knowing, statistical equilibrium is certainly {a restrictive global constraint}. For example, interactions with stochastic reservoirs, or so-called ``open boundaries", immediately breaks any understanding of invariant measures \cite{CK} except for in special situations. Moreover, it is not even believed that the equilibrium method should succeed in a general non-equilibrium setting; see \cite{CYau}.
\item To avoid a need for understanding global invariant measures explicitly, Jara-Menezes \cite{JM} adapted the relative entropy method of Yau \cite{YauRE}, which was originally introduced for deriving hydrodynamic limits, to rigorously implement the strategy of local equilibrium/Gibbs states due to Morrey \cite{Mor}. However, as we work at the delicate fluctuation scale, in \cite{JM} the authors require a strong initial closeness to local Gibbs states in a global sense; initially, the model is close to a Gibbs state at local scales, but this must be true everywhere in order to solve a global many-body eigenvalue problem. So, this method {also} depends on strong global assumptions. In any case, by this method, Jara-Menezes derive fluctuations for a smoothly inhomogeneous exclusion process \cite{JM} whose variants were studied in \cite{CTIn,CR,JM}. In \cite{JL}, Jara-Landim do this for a class of exclusion processes with additional Glauber-type disturbances/perturbations. 
\item In a groundbreaking work of Chang-Yau \cite{CYau}, hydrodynamic limit fluctuations were rigorously derived, with continuum limit given by a linear Gaussian stochastic PDE, basically without any conditions on the initial data beyond being reasonable initial data for the limit stochastic PDE. Chang-Yau \cite{CYau} specialize to a system of diffusions; their work is similarly based upon solving a many-body eigenvalue problem by means of large-deviations estimates and close-to-optimal log-Sobolev inequalities for the global invariant measure. Therefore, although the results of Chang-Yau \cite{CYau} are for non-equilibrium systems, analysis of global Gibbs states/invariant measures is essential. Moreover, it is unclear if the work of \cite{CYau} can be used to access KPZ fluctuations in non-equilibrium models. This is because the KPZ equation requires analytic considerations to solve when outside the invariant measure, and the analysis in \cite{CYau} seems difficult to upgrade at the level of appropriate norms; see Remark \ref{remark:BGI2}.
\end{itemize}
In this paper, the Boltzmann-Gibbs principle is derived with \emph{local}, and thus more general, considerations involving only system dynamics, not by directly exploiting global invariant measures.  Modulo details, the ingredients for our method are listed below; for a more detailed illustration of the method, see Section \ref{subsection:s}, which we have set up in a fairly general fashion.
\begin{itemize}
\item On \emph{local} mesoscopic scales, the dynamics admit an almost-optimal and ``elliptic" log-Sobolev inequality; this implies strong local relaxation of dynamics as assumed by Morrey \cite{Mor}. This assumption is very different than global assumptions in \cite{CYau}. {For example, in many models containing interactions with reservoirs at localized ``boundary" sites, the invariant measure is poorly understood except for a small set of special cases \cite{CK}. For this class of ``open boundary models", the invariant measures are by no means perturbations of their ``boundary-free" versions, even on macroscopic scales; see \cite{CK}. This obstructs the method of \cite{CYau}, which is based on calculations of local marginals of an explicit global invariant measure. However, except in $\mathrm{O}(1)$-many small sets near the reservoirs, the \emph{local} dynamics, and thus \emph{their} invariant measures, are unaffected. Also, near any reservoir, the system looks like a half-space model, which have better understood invariant measures. Thus, the method in this paper has potential applications to open boundary models, which are also of a non-equilibrium flavor and, again, behave quite differently than models without boundary and also garnered a high amount of interest recently; see \cite{CK,CS,FSPDE3,Y20}. See \cite{Y20}, in particular, for an application of the first steps of the method developed herein to some open boundary models (whose invariant measures are unknown and never used in \cite{Y20}); we discuss further in Section \ref{subsection:context}. 

In a similar spirit, the models in \cite{JM} do not admit explicit invariant measures because of the smooth inhomogeneity; this is one motivation for \cite{JM}. But locally, smooth is basically constant, so the inhomogeneity does not obstruct our method (modulo a few perturbations).} By the same token, our method also holds for fluctuations of smooth non-constant hydrodynamic limits as in \cite{CYau,JM} after some perturbative adjustments. {However, for the scaling that gives nonlinear KPZ statistics, making sense of non-constant hydrodynamic limits itself seems to pose an issue. Indeed, these should formally be ``infinite-time" (or ``infinite-speed") viscosity solutions to hyperbolic Hamilton-Jacobi equations, whose meaning is only clear for constant solutions. For this reason only, we do not discuss fluctuations about non-constant hydrodynamic limits.}
\item On local mesoscopic scales, regularity of fluctuations is roughly that of a white noise, which is what we expect for their SPDE limits; this is not an assumption and usually falls out of the analysis, and it controls which local Gibbs states are relevant.
\item We emphasize these ingredients concern only \emph{local dynamics} of the model. Properties of the global invariant measure may be helpful at a technical level, but they should not be essential to {deriving SPDEs from fluctuations}. {For example, the methods in this paper use an explicit product measure that happens to be an invariant measure for the entire process (as opposed to just the dynamics in a local set). However, we do not necessarily require that this measure is invariant. All we need are entropy production bounds (see Lemma \ref{lemma:le3}). (Intuitively, these bounds are a convenient quantification of \emph{local} equilibration; see the first bullet point in this list. The aforementioned invariance just makes this calculation much shorter than those in \cite{Y20,YauRE}; see Lemma 7.4 in \cite{Y20}, for instance.)} 

As for initial data, we only require that it can be made sense of by the macroscopic SPDE. This is, again, a basic requirement.
\end{itemize}

{In the case where large-scale limits of fluctuations are given by the KPZ equation, one of the main benefits of our local strategy is a Boltzmann-Gibbs principle which holds in a much stronger topology than in \cite{CYau}. Beyond being possibly of interest in its own right, this seems to be important for deriving KPZ equation fluctuations, whose solution theory currently requires either a strong stationary assumption \cite{GJ15} that we aim to avoid or analysis in relatively strong norms \cite{BG,Hai13}.}

Instead of developing a general theory of deriving Boltzmann-Gibbs principles, we specialize to KPZ fluctuations in a class of non-integrable and non-stationary interacting particle systems. The main result of this paper, namely Theorem \ref{theorem:KPZ}, is convergence of height function (or ``current") fluctuations to the KPZ equation for a class of exclusion processes with environment-dependent dynamics. These are of high interest both in KPZ \cite{GJ15} and beyond KPZ \cite{DPSW,F,FHU,K,LV}. This adds to the almost empty set of non-integrable, non-stationary interacting particle systems for which universality of KPZ equation fluctuations is proven. 

Let us introduce the KPZ equation more precisely below, in which $\xi$ is Gaussian space-time white noise on $\R_{\geq0}\times\mathbb{T}^{1}$ with $\mathbb{T}^{1}=\R/\Z$ the torus, in which $\bar{\mathfrak{d}}$ is constant, and in which $\infty$ is meant to suggest \eqref{eq:KPZ} as a scaling limit:
\begin{align}
\partial_{T}\mathbf{h}^{\infty} \ = \ 2^{-1}\Delta\mathbf{h}^{\infty} - \bar{\mathfrak{d}}\grad\mathbf{h}^{\infty} - 2^{-1}|\grad\mathbf{h}^{\infty}|^{2} + \xi. \label{eq:KPZ}
\end{align}
The KPZ equation \eqref{eq:KPZ} was originally derived in \cite{KPZ} to be a universal model for dynamical interface fluctuations describing the statistics of propagating fires, bacterial colonies, epidemic spread, tumor growth, and crack formations. However, it was already apparent in \cite{KPZ} the important observation that $\mathbf{u}^{\infty}=\grad\mathbf{h}^{\infty}$ describes hydrodynamic fluctuations. As for a brief history, in \cite{BG}, Bertini-Giacomin show that height function fluctuations in the asymmetric simple exclusion process (ASEP) converge to KPZ with $\bar{\mathfrak{d}}=0$. In \cite{BG}, the integrability of ASEP is leveraged crucially. Related works \cite{CGST,CST,CT} employed the same integrability method to show convergence to KPZ for height function fluctuations in a limited number of special systems. For non-integrable models, there has only been a successful general approach for \emph{stationary} systems \cite{GJ15,GP}. Progress for non-integrable, non-stationary particle systems is minimal beyond a few works that we discuss after presenting Theorem \ref{theorem:KPZ}. Environment-dependent speed-change dynamics are of particular interest for KPZ (see Big Picture Question 1.6 of \cite{KPZAIM}), which is why we study it here.

In a nutshell, the difficulty in universality of \eqref{eq:KPZ} and the Boltzmann-Gibbs principle is as follows. Suppose we let $\mathbf{h}'$ denote the solution to \eqref{eq:KPZ}, but instead of $|\grad\mathbf{h}'|^{2}$ we have $\mathbf{F}(\grad\mathbf{h}')$ for a general $\mathbf{F}$. In \cite{KPZ}, a formal Taylor series implies $\mathbf{h}'$ converges to \eqref{eq:KPZ} under a ``critical scaling" with an explicit $\mathbf{F}$-dependent coefficient in front of the quadratic and explicit $\bar{\mathfrak{d}}=\bar{\mathfrak{d}}(\mathbf{F})$; see \cite{HQ}. Such coefficients are \emph{wrong}, however, unless $\mathbf{F}$ is a degree-two polynomial, in which case the calculation is trivial because one already starts with KPZ. The picture for particle systems is similar. General environment-dependence roughly corresponds to general nonlinearities $\mathbf{F}$ whose effective limits we must compute. Moreover, the integrable ASEP model that was studied in \cite{BG} is associated to degree-2 $\mathbf{F}$ for which homogenization is formally trivial. Making precise the asymptotics for general $\mathbf{F}$ is the heart of proving universality, and it is one of our main motivations.

One explanation for why the Taylor series heuristic in \cite{KPZ} is incorrect is that KPZ is a \emph{singular} SPDE; the roughness of the $\xi$-noise makes the equation classically ill-posed. A way of solving \eqref{eq:KPZ} (see \cite{BG}) is to instead define $\mathbf{h}^{\infty}=-\log\mathbf{Z}^{\infty}$ where $\mathbf{Z}^{\infty}$ solves the \emph{stochastic heat equation} (SHE) below, which can be solved with Ito-Walsh calculus; this is the \emph{Cole-Hopf transform}:
\begin{align}
\partial_{T}\mathbf{Z}^{\infty} \ = \ 2^{-1}\Delta\mathbf{Z}^{\infty} - \bar{\mathfrak{d}}\grad\mathbf{Z}^{\infty}- \mathbf{Z}^{\infty}\xi. \label{eq:SHE}
\end{align}
We conclude by tying the Boltzmann-Gibbs principle and KPZ into the same story: the Boltzmann-Gibbs principle is the mechanism by which the correct coefficients in the limit $\mathbf{h}'\to\mathbf{h}^{\infty}$, which we discussed in a previous paragraph, are computed.
\subsection{The Model}
We start by introducing the interacting particle systems of interest  as Markov processes on a finite state space. In words, the process below is the ASEP system in \cite{BG} with additional environment-dependent asymmetry of speed $N$ that affects the nonlinearity in the dynamics of its height function in non-integrable fashion; the height function is constructed in Definition \ref{definition:height}. The parameter $N\in\Z_{>0}$ is a scaling parameter we take infinitely large to recover limit SPDEs; this is the ``large-$N$ limit".
\begin{itemize}
\item Define $\mathbb{T}_{N}=\Z/N\Z$ to be the microscopic $N$-point torus that we embed into the one-dimensional lattice $\Z$ upon identifying every element in $\mathbb{T}_{N}$ by an integer between 0 and $N-1$. Arithmetic on the torus $\mathbb{T}_{N}$ is taken mod $N$ unless said otherwise.
\item Provided any set $\mathbb{K}_{N}\subseteq\mathbb{T}_{N}$, define the corresponding state space $\Omega_{\mathbb{K}_{N}}=\{\pm1\}^{\mathbb{K}_{N}}$. For convenience, we define $\Omega=\Omega_{\mathbb{T}_{N}}$. Elements of $\Omega_{\mathbb{K}_{N}}$ sets are denoted by $\eta=(\eta_{x})_{x\in\mathbb{K}_{N}}$. The interpretation of $\eta$-variables in terms of particles is the following. Adopting spin notation of \cite{DT}, if $\eta_{x}=1$, there is a particle at $x\in\mathbb{T}_{N}$. Otherwise, if $\eta_{x}=-1$, there is no particle there.
\item Consider a functional $\mathfrak{d}:\Omega\to\R$ and define $\mathfrak{d}_{x}(\eta)=\mathfrak{d}(\tau_{x}\eta)$, in which $\tau_{x}\eta$ shifts the $\eta$-configuration by ${x}\in\mathbb{T}_{N}$ to recenter at $x$, so that {$(\tau_{x}\eta)_{z}=\eta_{z+x}$} for all $z\in\mathbb{T}_{N}$. {We now let $\mathsf{L}_{x}$ denote the infinitesimal generator for a symmetric simple exclusion process with speed 1 on $\{x,x+1\}\subseteq\mathbb{T}_{N}$. To specify it, for any $\eta\in\Omega$, let $\eta^{z,w}\in\Omega$ be the configuration defined by $\eta^{z,w}_{x}=\eta_{x}$ if $x\neq z,w$ and $\eta^{z,w}_{z}=\eta_{w}$ and $\eta^{z,w}_{w}=\eta_{z}$. (In words, $\eta^{z,w}$ swaps occupation numbers at $z,w$.) For any $\mathfrak{f}:\Omega\to\R$, we define
\begin{align}
\mathsf{L}_{x}\mathfrak{f}(\eta) \ = \ \mathfrak{f}(\eta^{x,x+1})-\mathfrak{f}(\eta). \nonumber
\end{align}
}We now define the generator of the Markov process of interest here via $\mathsf{L}_{N}=\mathsf{L}_{N,\mathrm{S}}+\mathsf{L}_{N,\mathrm{A}}$:
\begin{align}
\mathsf{L}_{N,\mathrm{S}} \ = \ 2^{-1}N^{2}{\sum}_{x\in\mathbb{T}_{N}}\mathsf{L}_{x} \quad \mathrm{and} \quad \mathsf{L}_{N,\mathrm{A}} \ = \ 2^{-1}N^{\frac32}{\sum}_{x\in\mathbb{T}_{N}}\left(\mathbf{1}_{\substack{\eta_{x}=-1\\ \eta_{x+1}=1}}-\mathbf{1}_{\substack{\eta_{x}=1\\ \eta_{x+1}=-1}}\right)\left(1+N^{-\frac12}\mathfrak{d}_{x}\right)\mathsf{L}_{x}. \label{eq:gen}
\end{align}
\item Denote by $\eta_{\mathrm{t}}$ the particle configuration at time $\mathrm{t}\geq0$ under the Markov process with generator $\mathsf{L}_{N}$. More generally, given any $\mathrm{t}\geq0$, any $x\in\mathbb{T}_{N}$, and any functional $\mathfrak{f}:\Omega\to\R^{d}$, we define $\mathfrak{f}_{\mathrm{t},x}=\mathfrak{f}(\tau_{x}\eta_{\mathrm{t}})$; recall the spatial shift $\tau_{x}$ from above.
\end{itemize}
\begin{definition}\label{definition:height}
\fsp Define the following \emph{height function}, in which $\mathbf{h}^{N}_{T,0}$ is equal to $2N^{-1/2}$ times the net flux of particles crossing $0$, with the convention that leftward traveling particles count as positive flux; this is the same height function as in \cite{BG} but now on the torus $\mathbb{T}_{N}$. {Also, we assume $\mathbf{h}^{N}_{0,0}=0$.}

We now define the \emph{Gartner transform}, for which we introduce the renormalization term $\mathrm{R}=\mathrm{R}_{1}+\mathrm{R}_{2}$ with $\mathrm{R}_{1}=2^{-1}N-24^{-1}$ and $\mathrm{R}_{2}=N^{1/2}\mathrm{R}_{2,1}+\mathrm{R}_{2,2}+\mathrm{R}_{2,3}=N^{1/2}\mathrm{R}_{2,1}{+2^{-1}\bar{\mathfrak{d}}}+\mathrm{R}_{2,3}$, in which $\bar{\mathfrak{d}}$ is from Definition \ref{definition:mSHE+1}. The constant $\bar{\mathfrak{d}}$ is the same constant appearing in the $\mathrm{SHE}(\bar{\mathfrak{d}})$ scaling limit in our main result of Theorem \ref{theorem:KPZ}. We define the remaining two terms $\mathrm{R}_{2,1}$ and $\mathrm{R}_{2,3}$ shortly; roughly, they come from the \emph{hydrodynamic flux} of the $\mathfrak{d}$-asymmetry. First, we have
\begin{align}
\mathbf{h}_{T,x}^{N} \ = \ \mathbf{h}_{T,0}^{N} + N^{-\frac12}{{\sum}_{y=1}^{x}}\eta_{T,y} \quad \mathrm{and} \quad \mathbf{Z}_{T,x}^{N} \ = \ \exp\left(-\mathbf{h}_{T,x}^{N}+\mathrm{R}T\right) \quad \mathrm{on} \quad \R_{\geq0}\times\mathbb{T}_{N}.
\end{align}
To define the renormalization counterterm $\mathrm{R}_{2,1}$, define {$\E_{0}$} as the expectation with respect to the product Bernoulli measure on $\Omega$ whose one-dimensional marginals have expectation equal to the hydrodynamic limit 0. Define $\mathrm{R}_{2,1}={-}2^{-1}{\E_{0}}(\mathfrak{d}-\mathfrak{d}\eta_{0}\eta_{1})$ as the hydrodynamic limit of the flux of the environment-dependent asymmetry. In particular, in the exponential defining $\mathbf{Z}^{N}$, we look at height function \emph{fluctuations} after subtracting the leading-order hydrodynamic limit/flux. Indeed, hydrodynamic limits tell us the normalized height function (not its fluctuations in $\mathbf{h}^{N}$) is roughly $\mathrm{R}_{2,1}$ in expectation when close to a constant density profile, and when multiplying by $N^{1/2}$ to study fluctuations, we must renormalize by $N^{1/2}\mathrm{R}_{2,1}$. This provides an interpretation, from interacting particle systems and hydrodynamic limits, of renormalizations needed in \cite{Hai13} to make sense of KPZ.

To wrap up this construction, let us define the order 1 counter-term $\mathrm{R}_{2,3}={\E_{0}}\wt{\mathfrak{s}}$ where $\wt{\mathfrak{s}}$ is a functional defined in Definition \ref{definition:mSHE+1}. Roughly speaking, it captures, at a level of hydrodynamic limits, a transport-induced-growth of the height function coming from the $\mathfrak{d}$ asymmetry; this is the parallel, for the $\mathfrak{d}$ asymmetry, of the $24^{-1}$-term in the renormalization constant $\mathrm{R}_{1}$ which comes from the leading-order asymmetry and that was also present in \cite{BG}; see Remark \ref{remark:mSHE+2} for more detailed explanation of $\mathrm{R}_{2,3}$.
\end{definition}
\begin{remark}\label{remark:height}
\fsp We linearly interpolate values of functions on $\mathbb{T}_{N}$ for all times to get a piecewise linear function on $N\mathbb{T}^{1}=\R/N\Z$.
\end{remark}
\subsection{The Theorem}
Our main result is showing that $\mathbf{Z}^{N}\to\mathrm{SHE}(\bar{\mathfrak{d}})$ for a particular value of $\bar{\mathfrak{d}}\in\R$ that is determined by a few statistics of the $\mathfrak{d}$-asymmetry; we shortly specify this value. First we require a structural assumption for the $\mathfrak{d}$-asymmetry, which is also necessary in the approach to universality of KPZ by what are known as \emph{energy solutions} in \cite{GJ15,GJS15}, for example. Such an assumption is often called a \emph{gradient condition}. It implies (see \cite{GJ15,GJS15}) a family of \emph{explicit} product invariant measures.
\begin{ass}\label{ass:grad}
\fsp The support of $\mathfrak{d}:\Omega\to\R$, defined as the smallest subset of $\mathbb{T}_{N}$ such that $\mathfrak{d}$ depends only on $\eta$-variables in $\mathbb{T}_{N}$, is contained in a neighborhood of $0\in\mathbb{T}_{N}$ with radius at most the uniformly bounded constant $\mathfrak{l}_{\mathfrak{d}}\in\Z_{>0}$. There is a uniformly bounded functional $\mathfrak{w}$ whose support is contained in the same neighborhood so that $\mathfrak{d}\grad_{1}^{\mathbf{X}}\eta=\mathfrak{d}(\eta_{1}-\eta_{0})=\grad_{1}^{\mathbf{X}}\mathfrak{w}=\tau_{1}\mathfrak{w}-\mathfrak{w}$.
\end{ass}
\begin{remark}
\fsp We can perturb Assumption \ref{ass:grad} to make invariant measures globally intractable. Little would change if perturbations are not too large so log-Sobolev inequalities on mesoscopic scales drastically change. For example, perturbations that affect the system on global time-scales but take too long for mesoscopic dynamics to detect are certainly allowable.
\end{remark}
We turn to scaling limits. This starts with the following rescaling operators that give ``macroscopic" coordinates.
\begin{definition}\label{definition:RG}
\fsp Given $\psi:\R_{\geq0}\times\mathbb{T}_{N}\to\R$, define $\Gamma^{N,\mathbf{X}}\psi:\mathbb{T}^{1}\to\R$ via linearly interpolating values of $\Gamma^{N,\mathbf{X}}\psi_{x}=\psi_{0,Nx}$ for $x\in N^{-1}\mathbb{T}_{N}\subseteq\mathbb{T}^{1}$. Define $\Gamma^{N}\psi:\R_{\geq0}\times\mathbb{T}^{1}\to\R$ by interpolating values of $\Gamma^{N}\psi_{\mathrm{t},x}=\psi_{\mathrm{t},Nx}$ for $x\in N^{-1}\mathbb{T}_{N}\subseteq\mathbb{T}^{1}$.
\end{definition}
We now present a class of initial conditions of the particle system/the height function for which the KPZ equation limit will be established. We are almost \emph{forced} to take some assumption of the following kind, because the limit SPDEs themselves need reasonable initial data to be well-defined.
\begin{definition}\label{definition:id}
\fsp {A probability measure on $\Omega$ is \emph{stable} if the following conditions are satisfied. First, with probability 1 under said measure, the total number of particles on $\mathbb{T}_{N}$ is $N/2$. Equivalently, the sum of $\eta_{x}$ over $x\in\mathbb{T}_{N}$ under said measure is zero. Next,} provided any $p\geq1$ and any $\mathfrak{l}\in\Z$ and any $\mathfrak{u}\in[0,2^{-1})$, we have the following estimate uniformly over $\mathbb{T}_{N}$, in which $\grad_{\mathfrak{l}}^{\mathbf{X}}$ is a spatial gradient $\grad_{\mathfrak{l}}^{\mathbf{X}}\phi_{x}=\phi_{x+\mathfrak{l}}-\phi_{x}$ for any $\phi:\mathbb{T}_{N}\to\R$:
\begin{align}
\E|\mathbf{h}_{0,x}^{N}|^{2p} \ \lesssim_{p} \ 1 \quad \mathrm{and} \quad \E|\grad_{\mathfrak{l}}^{\mathbf{X}}\mathbf{h}_{0,x}^{N}|^{2p} \ \lesssim_{p,\mathfrak{u}} \  N^{-2p\mathfrak{u}}|\mathfrak{l}|^{2p\mathfrak{u}}.
\end{align}
Also, $\Gamma^{N,\mathbf{X}}\mathbf{h}^{N}$ converges in law as $N\to\infty$ with respect to the uniform norm on the space $\mathscr{C}(\mathbb{T}^{1})$ of continuous functions.
\end{definition}
\begin{remark}\label{remark:id}
\fsp {We make a few clarifications about Definition \ref{definition:id}. The assumption that $\eta_{x}$ sums to zero with probability 1 guarantees that $\eta_{T,x}$ sums (over $x\in\mathbb{T}_{N}$) to zero with probability 1 for all $T\geq0$. Indeed, the total particle number is conserved. From this, we deduce the gradient relation $\eta_{T,x}=N^{1/2}(\mathbf{h}_{T,x}^{N}-\mathbf{h}_{T,x-1}^{N})$.

Let us give examples of stable initial data. Take product (mean-zero) Bernoulli measure on $\{\pm1\}^{\mathbb{T}_{N}}$. Condition on the subset of $\eta\in\{\pm1\}^{\mathbb{T}_{N}}$ where $\eta_{x}$ sums to zero over $x\in\mathbb{T}_{N}$. A Brownian bridge version of Donsker's invariance principle implies that this is stable initial data, and the limit of $\Gamma^{N,\mathbf{X}}\mathbf{h}^{N}$ is Brownian bridge on $\mathbb{T}^{1}$. This stable initial data gives the stationary measure for the $\eta$ process. An example of deterministic (and thus non-stationary) stable measure is given by the flat data $\eta_{x}=(-1)^{x}$. In this case, the limit of $\Gamma^{N,\mathbf{X}}\mathbf{h}^{N}$ is the zero function on $\mathbb{T}^{1}$. In general, given any continuous function $\mathsf{F}$ on $\mathbb{T}^{1}$, one can construct stable initial data such that $\Gamma^{N,\mathbf{X}}\mathbf{h}^{N}$ has limit $\mathsf{F}$. (This is a random walk bridge version of the fact that Brownian bridge has dense support in $\mathscr{C}(\mathbb{T}^{1})$.)}

Finally, if $\mathbf{Z}^{N}$ is uniformly bounded above and below, then Definition \ref{definition:id} is basically equivalent to the same but $\mathbf{h}^{N}$ replaced by $\mathbf{Z}^{N}$. Actually, we can take $\mathbf{Z}^{N}$ singular as SHE is smoothing; this would only change the topology in which we study $\mathbf{Z}^{N}$.
\end{remark}
Let $\mathscr{D}_{1}$ be the Skorokhod space until time 1 with values in $\mathscr{C}(\mathbb{T}^{1})$; see \cite{Bil} for the Skorokhod topology. The final time 1 is not important and can be replaced by any fixed time independent of $N$. We will not explicitly give here the transport coefficient $\bar{\mathfrak{d}}$, which appears in the limit $\mathrm{SHE}(\bar{\mathfrak{d}})$, until Definition \ref{definition:mSHE+1}, because it requires nontrivial set-up. The key feature is that it agrees with the equilibrium linear transport coefficient in \cite{GJ15} given by a correction to a hydrodynamic limit contribution of the $\mathfrak{d}$-asymmetry.
\begin{theorem}\label{theorem:KPZ}
\fsp Suppose we take the sequence of Gartner transforms $\mathbf{Z}^{N}$ with stable initial data for $\mathbf{h}^{N}$. If we adopt \emph{Assumption \ref{ass:grad}}, the renormalized sequence $\Gamma^{N}\mathbf{Z}^{N}$ converges to $\mathrm{SHE}(\bar{\mathfrak{d}})$ with respect to the Skorokhod topology on $\mathscr{D}_{1}$ with the initial data $\lim_{N\to\infty}\exp(-\Gamma^{N,\mathbf{X}}\mathbf{h}^{N})$. The transport coefficient $\bar{\mathfrak{d}}\in\R$ that determines the limit $\mathrm{SHE}(\bar{\mathfrak{d}})$ is a derivative of an equilibrium expectation of the flux corresponding to the $\mathfrak{d}$-asymmetry; see \emph{Definition \ref{definition:mSHE+1}}.
\end{theorem}
\begin{remark}
Observe the $\mathfrak{d}$-asymmetry is biased to the left. Moreover, any jump $x+1\to x$ in the system increases the value of $\mathbf{h}^{N}$ at $x$. Thus, the average growth speed of $\mathbf{h}^{N}$ increases as $\mathfrak{d}$ increases; this is why the leading-order $N^{1/2}\mathrm{R}_{2,1}$-renormalization from the $\mathfrak{d}$-asymmetry is proportional to $\mathfrak{d}$. This implies the nonlinearity in the hydrodynamic limit, which resembles the role of general $\mathbf{F}$ in our Taylor expansion discussion prior to \eqref{eq:SHE}, is proportional to and therefore ``positive" in $\mathfrak{d}$. Said Taylor expansion heuristic ultimately deduces from this that the KPZ/SHE limits for fluctuations have $+\bar{\mathfrak{d}}\grad$ instead of $-\bar{\mathfrak{d}}\grad$; although the exact coefficients predicted by Taylor expansion are possibly incorrect if not done carefully, its qualitative prediction for direction of transport/growth is correct, as substantiated by Theorem \ref{theorem:KPZ}.
\end{remark}
To the author's knowledge, Theorem \ref{theorem:KPZ} is a \emph{first} result on non-integrable and non-stationary interacting particle systems in which a homogenized linear transport term $\bar{\mathfrak{d}}\grad$ in $\mathrm{SHE}(\bar{\mathfrak{d}})$ is obtained in the KPZ limit. The proof estimates what-will-be-error-terms quantitatively, so we can let the speed of the $\mathfrak{d}$-asymmetry to be a slightly larger power of $N$ to obtain $\mathrm{SHE}(``\infty")$, where $``\infty"$ means follow a constant diverging-speed characteristic. Also, as in \cite{CYau}, the Boltzmann-Gibbs principle is sometimes applied to linearize environment-dependence of symmetric dynamics, recovering a Laplacian in the limit from a nonlinear second-order operator. The method herein can do this after few refinements. In a nutshell, this is homogenization in the top-order differential; Theorem \ref{theorem:KPZ} is homogenization of lower-order terms, while perturbations in the top-order are more delicate. To give a complete discussion of our method, we discuss the refinements in relation to \cite{CYau}; see Remarks \ref{remark:regremark1} and \ref{remark:regremark2}. But we defer details to future work; they are not complicated once we give Remarks \ref{remark:regremark1} and \ref{remark:regremark2} and apply calculations already in \cite{CYau} that combine naturally and generally with the ideas in this paper. These details are also separate from singular KPZ fluctuations of interest here.
\subsection{Additional Context}\label{subsection:context}
Theorem \ref{theorem:KPZ} can be interpreted in the following fashion. We establish in this paper a general method of deriving the Boltzmann-Gibbs principle for interacting particle systems, and to illustrate its utility and ``strength", we obtain nonlinear and singular KPZ fluctuations for a general set of particle system height functions. By ``strength", we refer to the fact that earlier work, including but certainly not limited to \cite{BR,CYau,GJ15,JM}, establishes Boltzmann-Gibbs principles that only hold in a sense that is too weak to address the singular behavior of KPZ. Indeed, the KPZ equation and SHE are PDEs that must be interpreted in a sufficiently strong topology to establish convergence in the context of particle systems \cite{ACQ,BG,DT}; while previous Boltzmann-Gibbs principles do not play well with the analytic procedure needed to solve SHE, our method gives a Boltzmann-Gibbs principle that \emph{does}, allowing us to rigorously derive singular KPZ fluctuations. By ``strength", we also refer to local nature of our Boltzmann-Gibbs principle and its derivation.

Towards universality of KPZ, Theorem \ref{theorem:KPZ} contributes to an almost empty set of non-stationary and non-integrable interacting particle systems for which convergence to KPZ is rigorously shown. In \cite{DT}, non-simple exclusion processes of maximal-particle-jump-length at most three are studied successfully. These are basically integrable if one is able to analyze hydrodynamic limits \cite{DT}. In \cite{Y}, the jump-length condition in \cite{DT} was removed; the necessary ingredient was a very weak form of the Boltzmann-Gibbs principle to show effective vanishing for a one-dimensional subspace of ``fluctuating observables" or ``pseudo-gradients" as defined in \cite{Y}. The key technical development here is to upgrade the weak principle of \cite{Y} to a full Boltzmann-Gibbs principle that not only applies to fluctuating observables but computes generally nonzero effective limits of general local observables. {To this end, we develop a non-stationary version of the multiscale renormalization of \cite{GJ15,SX}}. This necessary multiscale step is part of what makes the full Boltzmann-Gibbs principle more difficult compared to \cite{Y}; we use only the more robust one-block step of \cite{GPV} to analyze locally fluctuating terms, but to compute the macroscopic effects of general local observables, we require the more difficult two-blocks step of \cite{GPV}.  Finally, with some rather technical multiscale refinements, \cite{Y} extends to KPZ fluctuations in open boundary systems \cite{Y20}, for which little is known, by locality of its method; again, the same holds for our Boltzmann-Gibbs principle, letting us add to the universality of the so-called open KPZ equation \cite{CS}, for which minimal progress has been made. Extensions of earlier work on hydrodynamic fluctuations of linear Gaussian type, such as \cite{CYau}, to open boundary versions are also possible by using the method herein and similar technical refinements.  These are all currently being carried out by the author.

In \cite{Hai13,Hai14,HQ,HX}, regularity structures were used to study both the KPZ equation and its universality for generalizations of KPZ for non-quadratic nonlinearities that we discussed earlier. However, regularity structures depend on strong assumptions on the $\xi$-noise. To the author's knowledge, it is not known how to apply regularity structures to tackle either universality of KPZ or Boltzmann-Gibbs principles for interacting particle systems. It would certainly be interesting to see these developments. 
\subsection{The infinite volume case}
{This paper treats particle systems on a discrete torus (with limit SPDE on a compact torus). The use of the torus is purely for technical convenience as it gives a priori spatial compactness. (It is a frequent assumption in studies of large-scale asymptotics of interacting particle systems; see \cite{CYau}, for instance.) However, all our methods are spatially local, and the limit SPDE \eqref{eq:KPZ} is well-posed on the full line $\R$, so the infinite volume case for systems on $\Z$ should be doable, for example via the method in \cite{Y}.}
\subsection{Outline}
As this paper has many technically involved moving pieces, we make an effort to explain many points.
\begin{itemize}
\item In Section \ref{section:mSHE}, we derive an \emph{approximate} microscopic version of $\mathrm{SHE}(\bar{\mathfrak{d}})$ for $\mathbf{Z}^{N}$. This is standard for proving KPZ fluctuations.
\item Section \ref{section:proofKPZ} proves Theorem \ref{theorem:KPZ} \emph{assuming} three key ingredients, the last of which we show in Section \ref{section:proofKPZ} and the first two of which we outline. In particular, we introduce and discuss the Boltzmann-Gibbs principle. We then outline the rest of the paper.
\end{itemize}
\subsection{Conventions}
We write here a list of conventions, including notation, that are used frequently in the paper.
\begin{itemize}
\item {We use Landau big-Oh notation. Also, $\mathrm{a}\lesssim\mathrm{b}$ is equivalent to $\mathrm{a}=\mathrm{O}(\mathrm{b})$, and $\mathrm{a}\gtrsim\mathrm{b}$ is equivalent to $\mathrm{b}\lesssim\mathrm{a}$.}
\item The notation $\mathrm{SHE}(\bar{\mathfrak{d}})$ stands for the solution of SHE \eqref{eq:SHE} with linear transport coefficient $\bar{\mathfrak{d}}$.
\item We let $\mathscr{D}_{1}$/$\mathscr{C}_{1}$ be the Skorokhod space of cadlag paths/space of continuous paths until time 1 valued in $\mathscr{C}(\mathbb{T}^{1})$.
\item The microscopic length-scale is order 1. The macroscopic length-scale is order $N$. Mesoscopic length-scales are in between. The microscopic time-scale is order $N^{-2}$. The macroscopic time-scale is order $1$. Mesoscopic time-scales are in between.
\item Set $\mathbb{T}_{N}=\Z/N\Z$ and $\mathbb{T}^{1}=\R/\Z$. Recall we chose an embedding $\mathbb{T}_{N}\subseteq\Z$. For $\alpha>0$, define $\alpha\mathbb{T}^{1}=\R/\alpha\Z$.
\item Whenever we say $\mathbb{I}\subseteq\mathbb{T}_{N}$ for some subset $\mathbb{I}\subseteq\Z$, we mean its image under the mod-$|\mathbb{T}_{N}|$ map $\Z\to\Z/N\Z=\mathbb{T}_{N}\subseteq\Z$.
\item Provided any $z\in\mathbb{T}_{N}$, we let $|z|$ denote the absolute value after the embedding $\mathbb{T}_{N}\subseteq\Z$.
\item For stable initial data, see Definition \ref{definition:id} and Remark \ref{remark:id}. For rescaling operators $\Gamma^{N,\mathbf{X}}$ and $\Gamma^{N}$, see Definition \ref{definition:RG}. 
\item For any $\eta\in\Omega$ and $x\in\mathbb{T}_{N}$, define $\tau_{x}\eta$ to be the configuration defined by $(\tau_{x}\eta)_{z}=\eta_{-x+z}$ for all $z\in\mathbb{T}_{N}$ For any functional $\mathfrak{f}:\Omega\to\R$ and $x\in\mathbb{T}_{N}$ and $S\geq0$, we define $\tau_{x}\mathfrak{f}=\mathfrak{f}\circ\tau_{x}:\Omega\to\R$ to recenter $\mathfrak{f}$ at $x$ and $\mathfrak{f}_{S,y}=\tau_{y}\mathfrak{f}(\eta_{S})$.
\item Given any functional $\mathfrak{f}:\Omega\to\R$, we define its \emph{support} to be the smallest subset of $\mathbb{T}_{N}$ for which $\mathfrak{f}$ depends only on $\eta$-variables in that subset. For example, if $\mathfrak{f}(\eta)=\eta_{0}$, then the support of $\mathfrak{f}$ is the single point $\{0\}\subseteq\mathbb{T}_{N}$.
\item For the $\mathfrak{l}_{\mathfrak{d}}$ length-scale and the support of $\mathfrak{d}:\Omega\to\R$, see Assumption \ref{ass:grad}.
\item For $\mathrm{t}_{\mathrm{st}}$ or $\e_{\mathrm{ap}}$ and $\e_{\mathrm{RN}}$, see Definition \ref{definition:KPZ1}. For $\e_{1}$ and $\e_{\mathrm{RN},1}$, see Propositions \ref{prop:BGI1} and \ref{prop:BGI2}. For $\mathbf{Y}^{N}$, see Definition \ref{definition:KPZ5}. 
\item Provided any finite, not necessarily uniformly bounded, set $\mathrm{I}$, define the averaged summation $\wt{\sum}_{x\in\mathrm{I}}=|\mathrm{I}|^{-1}{\sum}_{x\in\mathrm{I}}$.
\item For any $\phi:\mathbb{T}_{N}\to\R$ and $\mathfrak{l}\in\Z$, define the spatial gradient on length-scale $\mathfrak{l}$ by $\grad_{\mathfrak{l}}^{\mathbf{X}}\phi_{x}=\phi_{x+\mathfrak{l}}-\phi_{x}$. We also define the discrete Laplacian via the composition {$\Delta=-\grad_{1}^{\mathbf{X}}\grad_{-1}^{\mathbf{X}}$}. Lastly define $\Delta^{!!}=N^{2}\Delta$ and $\grad_{\mathfrak{l}}^{!}=N\grad_{\mathfrak{l}}^{\mathbf{X}}$.
\item For any $\psi:[0,1]\to\R$ and $\mathrm{t}\in\R$, define the time-gradient on time-scale $\mathrm{t}$ by $\grad_{\mathrm{t}}^{\mathbf{T}}\psi_{\mathrm{s}}=\psi_{\mathrm{s}+\mathrm{t}}-\psi_{\mathrm{s}}$; if $\mathrm{s}+\mathrm{t}\not\in[0,1]$, then replace $\mathrm{s}+\mathrm{t}$ in the definition of $\grad^{\mathbf{T}}_{\mathrm{t}}\psi_{\mathrm{s}}$ by the boundary point $\{0,1\}$ closest to $\mathrm{s}+\mathrm{t}$.
\item For any $a,b\in\R$, we define the discretized interval $\llbracket a,b \rrbracket = [a,b]\cap\Z$.
\item For any $p\geq1$, we let $\|\|_{\omega;p}$ denote the $p$-norm with respect to \emph{all} the randomness in the particle system. Provided any $\mathrm{t}\geq0$ and spatial set $\mathbb{K}$ and function $\phi$, we define $\|\phi\|_{\mathrm{t};\mathbb{K}}=\sup_{(s,y)\in[0,\mathrm{t}]\times\mathbb{K}}|\phi_{s,y}|$.
\item For any $S,T\geq0$, we define $\mathbf{O}_{S,T}=|T-S|$ usually as an on-diagonal heat kernel factor; see Proposition \ref{prop:heat}.
\end{itemize}
\subsection{Acknowledgements}
The author thanks Amir Dembo for useful discussion, and support from the ARCS Foundation. The author would also like to thank the editor(s) and anonymous referees for detailed feedback, which greatly improved this paper.
%
%
%
\section{Approximate Microscopic SHE}\label{section:mSHE}
\begin{definition}
\fsp For $\sigma \in \R$, let $\E_{\sigma}$ be expectation with respect to product Bernoulli measure on $\Omega$ with $\E_{\sigma}\eta_{x} = \sigma$ for $x\in\mathbb{T}_{N}$.
\end{definition}
\begin{definition}\label{definition:mSHE+1}
\fsp Define $\mathfrak{q} \overset{}= \frac12\mathfrak{d}-\frac12\mathfrak{d}\cdot\eta_{0}\eta_{1}$. Its support is contained in $\llbracket-\mathfrak{l}_{\mathfrak{d}},\mathfrak{l}_{\mathfrak{d}}\rrbracket\subseteq\mathbb{T}_{N}$ for $\mathfrak{l}_{\mathfrak{d}}\in\Z_{\geq1}$ uniformly bounded.
\begin{itemize}[leftmargin=*]
\item Define $\wt{\mathfrak{q}}\overset{}=\tau_{-2\mathfrak{l}_{\mathfrak{d}}}\mathfrak{q}$ to shift the support of $\mathfrak{q}$ strictly to the left of $0\in\mathbb{T}_{N}$. Define $\bar{\mathfrak{q}}=\wt{\mathfrak{q}}-\E_{0}\wt{\mathfrak{q}}-\bar{\mathfrak{d}}\eta_{0}$ with $\bar{\mathfrak{d}}\overset{}=\partial_{\sigma}\E_{\sigma}\wt{\mathfrak{q}}|_{\sigma=0}$.
\item Define {$\wt{\mathfrak{s}}(\eta)\overset{}=-\wt{\mathfrak{q}}(\eta)\cdot\sum_{y=0}^{2\mathfrak{l}_{\mathfrak{d}}-1}\eta_{-y}$} and the $\E_{0}$-fluctuation $\mathfrak{s}\overset{}=\wt{\mathfrak{s}}-\E_{0}\wt{\mathfrak{s}}$.
\end{itemize}
\end{definition}
\begin{remark}\label{remark:mSHE+2}
\fsp Recall $\E_{0}\wt{\mathfrak{s}}$ is a part of the renormalization constant in the exponential $\mathbf{Z}^{N}$. To understand this renormalization, since $\wt{\mathfrak{q}}$ is local, we can write it as a polynomial in $\eta_{x}$-variables for $x$ in a fixed neighborhood of the origin. When we multiply its degree $\neq1$ monomials by a linear term to get $\wt{\mathfrak{s}}$, we get a polynomial with no constant term and therefore zero $\E_{0}$-expectation. Thus, degree $\neq1$ terms in $\wt{\mathfrak{q}}$, and therefore of $\mathfrak{q}$, do not produce constants that need to be renormalized. However, a linear term in $\mathfrak{q}$ can be cancelled into a constant after multiplication by a linear statistic since $\eta_{x}^{2}=1$, and non-zero constants have non-zero $\E_{0}$-expectation, so these terms yield constants that then need to be part of the renormalization of the height function and $\mathbf{Z}^{N}$. {On the other hand, if $\mathfrak{q}$ replaced by the linear functional $\eta\mapsto\eta_{0}$, then $\eta_{T,x}\mathbf{Z}_{T,x}^{N}\approx c_{1}N^{1/2}\grad_{-1}^{\mathbf{X}}\mathbf{Z}^{N}_{T,x}+c_{2}\mathbf{Z}^{N}_{T,x}$ with constants $c_{i}=c_{i}$ ultimately follows by Taylor expansion as in Section 2 of \cite{DT}. One can readily check that $c_{2}$ is obtained by replacing $\wt{\mathfrak{q}}(\eta)$ by $\eta\mapsto\eta_{0}$ in $\wt{\mathfrak{s}}$ and then taking $\E_{0}$.} Therefore, the renormalization $\E_{0}\wt{\mathfrak{s}}$ in $\mathbf{Z}^{N}$ can be equivalently computed by first linearizing the $\wt{\mathfrak{q}}$-environment-dependence to get ASEP without environment dependence as in \cite{BG} and then computing the renormalization for this homogenized/linearized ASEP by following the classical calculation in \cite{BG} of Bertini-Giacomin.
\end{remark}
\begin{prop}\label{prop:mSHE+}
\fsp We have the following with notation defined afterwards, in which $|\mathfrak{b}_{i;}|\lesssim1$ are possibly random:
\begin{align}
\d\mathbf{Z}_{T,x}^{N} \ = \ \mathscr{L}_{N}\mathbf{Z}_{T,x}^{N}\d T + \mathbf{Z}_{T,x}^{N}\d\xi_{T,x}^{N} - N^{\frac12}\bar{\mathfrak{q}}_{T,x}\mathbf{Z}_{T,x}^{N}\d T - \mathfrak{s}_{T,x}\mathbf{Z}_{T,x}^{N}\d T + N^{-\frac12}\mathfrak{b}_{1;T,x}\mathbf{Z}_{T,x}^{N}\d T + N^{-\frac12}\grad_{\star}^{!}\left(\mathfrak{b}_{2;T,x}\mathbf{Z}_{T,x}^{N}\right)\d T. \nonumber
\end{align}
%
\begin{itemize}[leftmargin=*]
\item Let us first define the discrete first-order gradient $\grad_{\mathfrak{l}}^{\mathbf{X}}\varphi_{x}=\varphi_{x+\mathfrak{l}}-\varphi_{x}$ provided any $\mathfrak{l} \in \Z$ and $\varphi:\mathbb{T}_{N}\to\R$. We proceed to define $\Delta^{!!} = N^{2}\grad_{1}^{\mathbf{X}}\grad_{-1}^{\mathbf{X}}$ and $\grad_{\mathfrak{l}}^{!} = N\grad_{\mathfrak{l}}^{\mathbf{X}}$. The first term in the equation above is defined by $\mathscr{L}_{N} \overset{}= 2^{-1}\Delta^{!!} {+ \bar{\mathfrak{d}}\grad_{-1}^{!}}$.
\item The $\d\xi_{\bullet,x}^{N}$-term is a martingale differential/compensated Poisson process corresponding to jumps over $\{x,x+1\}\subseteq\mathbb{T}_{N}$. {Put precisely, it is the following measure in $T$ (given any $x$) that describes the change in $\mathbf{Z}_{T,x}^{N}$ according to clocks in the $\eta$ process:
\begin{align*}
\d\xi_{T,x}^{N} \ = \ &(\mathrm{e}^{2N^{-\frac12}}-1)\mathbf{1}_{\eta_{T,x}=1}\mathbf{1}_{\eta_{T,x+1}=-1}[\d\mathcal{Q}_{T,x}^{N,\mathrm{S},\to}-\tfrac12N^{2}\d T]\ + \ (\mathrm{e}^{-2N^{-\frac12}}-1)\mathbf{1}_{\eta_{T,x}=-1}\mathbf{1}_{\eta_{T,x+1}=1}[\d\mathcal{Q}_{T,x}^{N,\mathrm{S},\leftarrow}-\tfrac12N^{2}\d T] \\
- \ &(\mathrm{e}^{2N^{-\frac12}}-1)\mathbf{1}_{\eta_{T,x}=1}\mathbf{1}_{\eta_{T,x+1}=-1}[\d\mathcal{Q}_{T,x}^{N,\mathrm{A},\to}-(\tfrac12N^{\frac32}+\tfrac12N\mathfrak{d}_{x}(\eta_{T}))\d T]\\
+ \ &(\mathrm{e}^{-2N^{-\frac12}}-1)\mathbf{1}_{\eta_{T,x}=-1}\mathbf{1}_{\eta_{T,x+1}=1}[\d\mathcal{Q}_{T,x}^{N,\mathrm{A},\leftarrow}-(\tfrac12N^{\frac32}+\tfrac12N\mathfrak{d}_{x}(\eta_{T}))\d T].
\end{align*}
The clocks $\mathcal{Q}^{N,\mathrm{S},\to}$ and $\mathcal{Q}^{N,\mathrm{S},\leftarrow}$ are Poisson processes in $T$ of speed $2^{-1}N^{2}$. The clocks $\mathcal{Q}^{N,\mathrm{A},\to}$ and $\mathcal{Q}^{N,\mathrm{A},\leftarrow}$ are Poisson processes in $T$ of speed $2^{-1}N^{3/2}+2^{-1}N\mathfrak{d}_{x}(\eta_{T})$, which is positive for sufficiently large $N$ as $|\mathfrak{d}|\lesssim1$. Lastly, the predictable quadratic covariation between any two distinct (compensated) Poisson clocks is zero.}
\item When we write $\grad_{\star}^{!}$, we sum over the choices $\star=1,-2\mathfrak{l}_{\mathfrak{d}}$ with $\mathfrak{b}_{2;}$ depending possibly on $\star$ but still uniformly bounded.
\end{itemize}
\end{prop}
We provide a proof of Proposition \ref{prop:mSHE+} at the end of the subsection to avoid obstructing important takeaways in this section. Roughly speaking, the particle system at hand is ASEP from \cite{DT} but with only simple jumps and additional asymmetry $N^{-1}\mathfrak{d}$, and the Gartner transform $\mathbf{Z}^{N}$ is also that of \cite{DT} but with additional deterministic $\mathrm{R}_{2}t$ drift in the exponential. In view of these two observations, we follow the derivation of the microscopic SHE for the Gartner transform in Section 2 of \cite{DT}. Roughly, the only difference is the $N^{-1}\mathfrak{d}$ asymmetry. As jumps in $\mathbf{Z}^{N}$ are order {$N^{-1/2}\mathbf{Z}^{N}$}, the effect of $N^{-1}\mathfrak{d}$ asymmetry is order {$N^{1/2}\mathbf{Z}^{N}$} after time-scaling. We linearize the flux $\mathfrak{q}$ of this $N^{-1}\mathfrak{d}$ asymmetry to get $\bar{\mathfrak{q}}$ in Definition \ref{definition:mSHE+1}, and Taylor expansions/summation-by-parts give us the last three terms in the $\mathbf{Z}^{N}$-equation after cancelling with the additional $\mathrm{R}_{2}$-drift in $\mathbf{Z}^{N}$. The last three terms in the $\mathbf{Z}^{N}$ equation ultimately vanish in the large-$N$ limit. Now, to make Proposition \ref{prop:mSHE+} useful, we consider its mild form.
\begin{definition}\label{definition:heat}
\fsp We let $\mathbf{H}_{S,T,x,y}^{N}$ on $\R_{\geq0}^{2}\times\mathbb{T}_{N}^{2}$ be the heat kernel defined by $\mathbf{H}_{S,S,x,y}^{N}=\mathbf{1}_{x=y}$ and $\partial_{T}\mathbf{H}_{S,T,x,y}^{N}=\mathscr{L}_{N}\mathbf{H}_{S,T,x,y}^{N}$, where $\mathscr{L}_{N}$ acts on the backwards spatial variable $x\in\mathbb{T}_{N}$. Provided any test function $\varphi:\R\times\mathbb{T}_{N}\to\R$, we additionally define a pair of space-time and spatial heat operators, for which we give three ways that each operator may be written in this paper:
\begin{subequations}
\begin{align}
\mathbf{H}_{T,x}^{N}(\varphi) \ &= \ \mathbf{H}_{T,x}^{N}(\varphi_{S,y}) \ = \ \mathbf{H}_{T,x}^{N}(\varphi_{\bullet,\bullet}) \ = \ \int_{0}^{T}{\sum}_{y\in\mathbb{T}_{N}}\mathbf{H}_{S,T,x,y}^{N}\cdot\varphi_{S,y} \ \d S \\
\mathbf{H}_{T,x}^{N,\mathbf{X}}(\varphi) \ &= \ \mathbf{H}_{T,x}^{N,\mathbf{X}}(\varphi_{0,y}) \ = \ \mathbf{H}_{T,x}^{N,\mathbf{X}}(\varphi_{0,\bullet}) \ = \ {\sum}_{y\in\mathbb{T}_{N}}\mathbf{H}_{0,T,x,y}^{N}\cdot\varphi_{0,y}.
\end{align}
\end{subequations}
\end{definition}
\begin{corollary}\label{corollary:mSHE+}
\fsp Admit the setting and notation of \emph{Proposition \ref{prop:mSHE+}}. We have the stochastic integral equation
\begin{align}
\mathbf{Z}_{T,x}^{N} \ = \ \mathbf{H}_{T,x}^{N,\mathbf{X}}(\mathbf{Z}_{0,\bullet}^{N}) + \mathbf{H}_{T,x}^{N}(\mathbf{Z}^{N}\d\xi^{N}) - \mathbf{H}_{T,x}^{N}(N^{\frac12}\bar{\mathfrak{q}}\mathbf{Z}^{N}) - \mathbf{H}_{T,x}^{N}(\mathfrak{s}\mathbf{Z}^{N}) + N^{-\frac12}\mathbf{H}_{T,x}^{N}(\mathfrak{b}_{1;}\mathbf{Z}^{N}) + N^{-\frac12}\mathbf{H}_{T,x}^{N}\left(\grad_{\star}^{!}\left(\mathfrak{b}_{2;}\mathbf{Z}^{N}\right)\right). \nonumber
\end{align}
\end{corollary}
\begin{proof}[Proof of \emph{Proposition \ref{prop:mSHE+}}]
We follow the derivation of the microscopic SHE in Section 2 of \cite{DT}. Following their first steps at the beginning of Section 2, we derive the following for the time-differential of $\mathbf{Z}^{N}$, which we discuss below:
\begin{align}
\d\mathbf{Z}_{T,x}^{N} \ &= \ N^{2}\Phi_{T,x}^{\mathrm{S}}\mathbf{Z}_{T,x}^{N}\d T + N^{2}\Phi_{T,x}^{\mathrm{A},1}\mathbf{Z}_{T,x}^{N}\d T + N^{2}\Phi_{T,x}^{\mathrm{A},2}\mathbf{Z}_{T,x}^{N}\d T + \mathrm{R}_{1}\mathbf{Z}_{T,x}^{N}\d T + \mathrm{R}_{2}\mathbf{Z}_{T,x}^{N}\d T + \mathbf{Z}_{T,x}^{N}\d\xi_{T,x}^{N}. \label{eq:mSHE+1}
\end{align}
We clarify $\Phi^{\mathrm{S}}$ and $\Phi^{\mathrm{A},1}$ and $\Phi^{\mathrm{A},2}$ coefficients shortly. We briefly note, however, that \eqref{eq:mSHE+1} is a martingale/Dynkin decomposition for $\mathbf{Z}^{N}$, where the martingale is explicitly recorded in terms of Poisson clocks in the particle system. {In particular, to derive \eqref{eq:mSHE+1}, one starts with the following integral equation (that comes from the Dynkin formula), in which the stochastic integral on the LHS should be interpreted as integration of $\mathbf{Z}^{N}_{S,x}$ against the measure $\d\xi_{S,x}^{N}$:
\begin{align}
\int_{0}^{T}\mathbf{Z}_{S,x}^{N}\d\xi_{S,x}^{N} \ = \ \mathbf{Z}_{T,x}^{N}-\mathbf{Z}_{0,x}^{N}-\int_{0}^{T}(\mathrm{R}+\mathsf{L}_{N})\mathbf{Z}_{S,x}^{N}\d S, \nonumber
\end{align}
where $\mathrm{R}$ is the renormalization constant in Definition \ref{definition:height}, and $\mathsf{L}_{N}$ is the generator of the particle system. Indeed, as in Section 2 of \cite{DT}, integrating all of the clock terms in $\d\xi^{N}_{S,x}$ accounts for the total change $\mathbf{Z}_{T,x}^{N}-\mathbf{Z}_{0,x}^{N}$. The drift terms in $\d\xi^{N}_{S,x}$ account for the generator term $\mathsf{L}_{N}\mathbf{Z}^{N}_{S,x}$. The $\mathsf{L}_{N,\mathrm{S}}$ part of $\mathsf{L}_{N}$ yields $N^{2}\Phi^{S}\mathbf{Z}^{N}$ in \eqref{eq:mSHE+1}, and the $\mathsf{L}_{N,\mathrm{A}}$ parts yields $N^{2}\Phi^{\mathrm{A},1}\mathbf{Z}^{N}+N^{2}\Phi^{\mathrm{A},2}\mathbf{Z}^{N}$. (One can set $\Phi^{\mathrm{S}},\Phi^{\mathrm{A},1}, \Phi^{\mathrm{A},2}$ for this to be true; we are then left to compute these terms.) In particular, the claim about vanishing quadratic covariations follows since each (compensated) Poisson clock in the statement of Proposition \ref{prop:mSHE+} comes from a different $\mathsf{L}_{x}$ operator $\mathsf{L}_{N}$. The $\mathrm{R}\mathbf{Z}^{N}$ term comes from the fact that $\mathbf{Z}^{N}_{T,x}$ exponentiates $\mathrm{R}T$. It gives $\mathrm{R}_{1}\mathbf{Z}^{N}+\mathrm{R}_{2}\mathbf{Z}^{N}$ in \eqref{eq:mSHE+1}.}

The exact formulas for $\Phi^{\mathrm{S}}$ and $\Phi^{\mathrm{A},1}$ are not important to this proof as we deal with them via citing the calculations in Section 2 of \cite{DT}; the same applies to $\mathrm{R}_{1}$. The emphasis of this calculation will be computing $\Phi^{\mathrm{A},2}$ and $\mathrm{R}_{2}$, the former of which is equal to the following ``instantaneous" growth/change in $\mathbf{Z}^{N}$ that results from motion of the particle system through the $\mathfrak{d}$-asymmetry:
\begin{align*}
\Phi_{T,x}^{\mathrm{A},2} \ &\overset{\bullet}= \ 2^{-1}N^{-1}\mathfrak{d}_{T,x}\mathbf{1}_{\eta_{T,x}=-1}\mathbf{1}_{\eta_{T,x+1}=1}\left(\Exp(-2N^{-\frac12})-1\right) -  2^{-1}N^{-1}\mathfrak{d}_{T,x}\mathbf{1}_{\eta_{T,x}=1}\mathbf{1}_{\eta_{T,x+1}=-1}\left(\Exp(2N^{-\frac12})-1\right).
\end{align*}
In Section 2 of \cite{DT}, through Taylor expansion and lengthy though elementary calculations, the authors identify the contribution in \eqref{eq:mSHE+1} of $\Phi^{\mathrm{S}}$ and $\Phi^{\mathrm{A},1}$ and $\mathrm{R}_{1}$ with a discrete approximation of the continuum Laplacian. Precisely, they establish the identity
\begin{align}
{N^{2}\Phi_{T,x}^{\mathrm{S}}\mathbf{Z}_{T,x}^{N} + N^{2}\Phi_{T,x}^{\mathrm{A},1}\mathbf{Z}_{T,x}^{N} + \mathrm{R}_{1}\mathbf{Z}_{T,x}^{N} \ = \ 2^{-1}\Delta^{!!}\mathbf{Z}_{T,x}^{N}.} \label{eq:mSHE+2}
\end{align}
Provided \eqref{eq:mSHE+1} and \eqref{eq:mSHE+2}, we are left with computing $\Phi^{\mathrm{A},2}$ and $\mathrm{R}_{2}$ contributions. To this end, it will {be} convenient to first define $\mathrm{E}_{\pm,N}=\Exp(\pm 2N^{-1/2})-1$ along with two ``trigonometric-type" functions $\mathrm{T}^{\pm,N}=\mathrm{E}_{-,N}\pm\mathrm{E}_{+,N}$. Let us also observe the identity $2\mathbf{1}(\eta=\pm1)=1\pm\eta$ for $\eta\in\{\pm1\}$, which can be checked immediately. This allows us to rewrite the indicator functions in $\Phi^{\mathrm{A},2}$ explicitly as local functionals of the particle system and thus lets us compute as follows:
\begin{align}
\Phi^{\mathrm{A},2}_{T,x} \ &= \ 8^{-1}N^{-1}(1-\eta_{T,x})(1+\eta_{T,x+1})\mathfrak{d}_{T,x}\mathrm{E}_{-,N} - 8^{-1}N^{-1}(1+\eta_{T,x})(1-\eta_{T,x+1})\mathfrak{d}_{T,x}\mathrm{E}_{+,N} \\
&= \ 8^{-1}N^{-1}\mathrm{T}^{-,N}\left(\mathfrak{d}_{T,x} - \mathfrak{d}_{T,x}\eta_{T,x}\eta_{T,x+1}\right)+8^{-1}N^{-1}\mathrm{T}^{+,N}\mathfrak{d}_{T,x}\grad_{1}^{\mathbf{X}}\eta_{T,x}. \label{eq:mSHE+3}
\end{align}
As {$\wt{\mathfrak{q}}_{T,x}-\mathfrak{q}_{T,x}=\grad_{-2\mathfrak{l}_{\mathfrak{d}}}^{\mathbf{X}}\mathfrak{q}_{T,x}$, where $\grad_{-2\mathfrak{l}_{\mathfrak{d}}}^{\mathbf{X}}$ acts on $x$} (see Definition \ref{definition:mSHE+1}), for the first term in \eqref{eq:mSHE+3}, we have
\begin{align}
8^{-1}\mathfrak{d}_{T,x}-8^{-1}\mathfrak{d}_{T,x}\eta_{T,x}\eta_{T,x+1} \ &= \ 4^{-1}\wt{\mathfrak{q}}_{T,x} - 4^{-1}\grad^{\mathbf{X}}_{-2\mathfrak{l}_{\mathfrak{d}}}\mathfrak{q}_{T,x}\\
&= \ 4^{-1}\bar{\mathfrak{q}}_{T,x} +  4^{-1}\E_{0}\wt{\mathfrak{q}} + 4^{-1}\bar{\mathfrak{d}}\eta_{T,x} - 4^{-1}\grad^{\mathbf{X}}_{-2\mathfrak{l}_{\mathfrak{d}}}\mathfrak{q}_{T,x}. \label{eq:mSHE+4}
\end{align}
We will now multiply the calculation \eqref{eq:mSHE+3} by $\mathbf{Z}^{N}$, use the identity \eqref{eq:mSHE+4},  and then add the additional drift $N^{-2}\mathrm{R}_{2}\mathbf{Z}^{N}$. We will match the resulting sum and identities to the non-$\Delta$ and non-$\xi^{N}$ terms in the proposed SDE for $\mathbf{Z}^{N}$. For the purposes of clearer organization, we write these calculations in the following bullet points. We address each term in \eqref{eq:mSHE+4} in written order. Let us clarify that throughout the following list, we may change $\mathfrak{b}_{1;}$ from line to line, but it is always a sum of an $N$-independent number of order $N^{-1/2}$ error terms that come from Taylor expansion. Lastly, recall $\mathrm{R}_{2}=N^{1/2}\mathrm{R}_{2,1}+\mathrm{R}_{2,2}+\mathrm{R}_{2,3}$. 
\begin{itemize}
\item Let us first match $4^{-1}N\mathrm{T}^{-,N}\bar{\mathfrak{q}}$ from \eqref{eq:mSHE+4} plugged into \eqref{eq:mSHE+3} to $-N^{1/2}\bar{\mathfrak{q}}$ in the proposed SDE up to error $\mathrm{O}(N^{-1/2})$. This follows, by definition of $\mathrm{T}^{-,N}$ from immediately before \eqref{eq:mSHE+3}, via $\mathrm{T}^{-,N}\sim-4N^{-1/2}+\mathrm{O}(N^{-3/2})$.
\item Let us now match $4^{-1}N\mathrm{T}^{-,N}\E_{0}\wt{\mathfrak{q}}$, obtained by plugging \eqref{eq:mSHE+4} in \eqref{eq:mSHE+3}, with $-N^{1/2}\mathrm{R}_{2,1}$ so that these terms cancel each other in the $\mathbf{Z}^{N}$ SDE, again up to $\mathrm{O}(N^{-1/2})$ that adds to $\mathfrak{b}_{1;}$. By definition $\mathrm{R}_{2,1}=-2^{-1}\E_{0}(\mathfrak{d}-\mathfrak{d}\cdot\eta_{0}\eta_{1})=-\E_{0}\mathfrak{q}=-\E_{0}\wt{\mathfrak{q}}$ since product Bernoulli measure in $\E_{0}$ is invariant under spatial shifts. It now suffices to again use $\mathrm{T}^{-,N}\sim-4N^{-1/2}+\mathrm{O}(N^{-3/2})$. 
\item We match $\mathrm{R}_{2,2}+4^{-1}N\mathrm{T}^{-,N}\bar{\mathfrak{d}}\eta$ again obtained by plugging \eqref{eq:mSHE+4} in \eqref{eq:mSHE+3} to the first-order operator $-\bar{\mathfrak{d}}\grad_{-1}^{!}=-N\bar{\mathfrak{d}}\grad_{-1}^{\mathbf{X}}$ in $\mathscr{L}_{N}$ up to $\mathrm{O}(N^{-1/2})$ to be absorbed into $\mathfrak{b}_{1;}$:
\begin{align}
4^{-1}N\mathrm{T}^{-,N}\bar{\mathfrak{d}}\eta \mathbf{Z}^{N} + \mathrm{R}_{2,2}\mathbf{Z}^{N} \ = \ -\bar{\mathfrak{d}}\grad_{-1}^{!}\mathbf{Z}^{N} \ = \ -N\bar{\mathfrak{d}}\grad_{-1}^{\mathbf{X}}\mathbf{Z}^{N}. \label{eq:finaledit2.8}
\end{align}
We compute $\grad_{-1}^{\mathbf{X}}\mathbf{Z}^{N}$ with Taylor expansion via its definition (see Section 2 of \cite{DT}):
{
\begin{align}
\grad_{-1}^{\mathbf{X}}\mathbf{Z}^{N}_{T,x} \ &= \ \mathrm{e}^{-\mathbf{h}_{T,x-1}^{N}+\mathrm{R}T}-\mathrm{e}^{-\mathbf{h}_{T,x}^{N}+\mathrm{R}T} \ = \ (\mathrm{e}^{\mathbf{h}_{T,x}^{N}-\mathbf{h}_{T,x-1}^{N}}-1)\mathbf{Z}_{T,x}^{N} \ = \ (N^{-\frac12}\eta_{T,x}+2^{-1}N^{-1}+\mathrm{O}(N^{-\frac32}))\mathbf{Z}_{T,x}^{N}. \nonumber
\end{align}
}Thus $-N\bar{\mathfrak{d}}\grad^{\mathbf{X}}_{-1}\mathbf{Z}^{N}\sim (-N^{1/2}\bar{\mathfrak{d}}\eta-2^{-1}\bar{\mathfrak{d}})\mathbf{Z}^{N}+\mathrm{O}(N^{-1/2})\mathbf{Z}^{N}$. On the other hand, Taylor expansion gives $4^{-1}N\mathrm{T}^{-,N}\bar{\mathfrak{d}}\eta\sim -N^{1/2}\bar{\mathfrak{d}}\eta+\mathrm{O}(N^{-1/2})$ that can again be absorbed by $\mathfrak{b}_{1;}$. Recalling {$\mathrm{R}_{2,2}=2^{-1}\bar{\mathfrak{d}}$}, we get the desired matching \eqref{eq:finaledit2.8}.
\item We move to {$-4^{-1}N\mathrm{T}^{-,N}\grad^{\mathbf{X}}_{-2\mathfrak{l}_{\mathfrak{d}}}\mathfrak{q}_{T,x}\cdot\mathbf{Z}^{N}_{T,x}$} again obtained by plugging \eqref{eq:mSHE+4} in \eqref{eq:mSHE+3}. We compute/match it as follows:
\begin{align}
-4^{-1}N\mathrm{T}^{-,N}{\grad^{\mathbf{X}}_{-2\mathfrak{l}_{\mathfrak{d}}}\mathfrak{q}_{T,x}\cdot\mathbf{Z}^{N}_{T,x}} + \mathrm{R}_{2,3}\mathbf{Z}^{N}_{T,x} \ = \ -\mathfrak{s}_{T,x}\mathbf{Z}^{N}_{T,x}+N^{1/2}\grad^{\mathbf{X}}_{-2\mathfrak{l}_{\mathfrak{d}}}(\mathfrak{b}_{2;T,x}\mathbf{Z}^{N}_{T,x}). \label{eq:mSHE+last}
\end{align}
We clarify $\mathfrak{b}_{2;}$ shortly. We start with calculation below to be explained after; recall $\wt{\mathfrak{q}}=\tau_{-2\mathfrak{l}_{\mathfrak{d}}}\mathfrak{q}$:
\begin{align}
&-4^{-1}N\mathrm{T}^{-,N}{\grad^{\mathbf{X}}_{-2\mathfrak{l}_{\mathfrak{d}}}\mathfrak{q}_{T,x}\cdot\mathbf{Z}^{N}_{T,x}} \ = \ -4^{-1}N\mathrm{T}^{-,N}{\grad^{\mathbf{X}}_{-2\mathfrak{l}_{\mathfrak{d}}}\left(\mathfrak{q}_{T,x}\mathbf{Z}^{N}_{T,x}\right)} + 4^{-1}N\mathrm{T}^{-,N}\wt{\mathfrak{q}}_{T,x}\grad^{\mathbf{X}}_{-2\mathfrak{l}_{\mathfrak{d}}}\mathbf{Z}^{N}_{T,x} \nonumber \\
&= \ -4^{-1}N\mathrm{T}^{-,N}{\grad^{\mathbf{X}}_{-2\mathfrak{l}_{\mathfrak{d}}}\left(\mathfrak{q}_{T,x}\mathbf{Z}^{N}_{T,x}\right)} + 4^{-1}N\mathrm{T}^{-,N}\wt{\mathfrak{q}}_{T,x}N^{-1/2}\left({\sum}_{\mathfrak{j}=1}^{2\mathfrak{l}_{\mathfrak{d}}}\tau_{x-\mathfrak{j}}\eta_{T} + \mathrm{O}(N^{-1})\right)\mathbf{Z}^{N}_{T,x}. \label{eq:mSHE+5}
\end{align}
The first line follows by a discrete version of the Leibniz rule that can be verified by unfolding discrete gradients and cancelling terms. The second line \eqref{eq:mSHE+5} follows by Taylor expanding $\grad^{\mathbf{X}}_{-2\mathfrak{l}_{\mathfrak{d}}}\mathbf{Z}^{N}$ as in Section 2 of \cite{DT}. Because $\mathrm{T}^{-,N}=-4N^{-1/2}+\mathrm{O}(N^{-3/2})$, we can absorb $\mathrm{O}(N^{-1})$ in \eqref{eq:mSHE+5} to $N^{-1/2}\mathfrak{b}_{1;}$ and drop it from \eqref{eq:mSHE+5}. This also implies that the second term in \eqref{eq:mSHE+5} is $-\wt{\mathfrak{s}}\mathbf{Z}^{N}=-\mathfrak{s}\mathbf{Z}^{N}-\E_{0}\wt{\mathfrak{s}}\mathbf{Z}^{N}=-\mathfrak{s}\mathbf{Z}^{N}-\mathrm{R}_{2,3}\mathbf{Z}^{N}$. Lastly, the first term in \eqref{eq:mSHE+5} has the form {$N^{-1/2}\grad_{-2\mathfrak{l}_{\mathfrak{d}}}^{!}(\mathfrak{b}_{2;T,x}\mathbf{Z}^{N}_{T,x})$} for $|\mathfrak{b}_{2;}|\lesssim1$. Combining this paragraph with \eqref{eq:mSHE+5} gives the desired matching \eqref{eq:mSHE+last}.
\item We are left with analyzing the last term in \eqref{eq:mSHE+3}. For this first recall the gradient condition that we have assumed provides the current representation {$\mathfrak{d}_{T,x}\grad^{\mathbf{X}}_{1}\eta_{T,x}=\grad^{\mathbf{X}}_{1}\mathfrak{w}_{T,x}$} where $\mathfrak{w}$ is uniformly bounded. Moreover, we observe that $|\mathrm{T}^{+,N}|\lesssim N^{-1}$, which is a smaller estimate than what we had for $\mathrm{T}^{-,N}$ by a factor of $N^{-1/2}$. Thus, we may employ the exact same argument as the previous bullet point, precisely by replacing $\mathfrak{q}$ with $\mathfrak{w}$ and $-2\mathfrak{l}_{\mathfrak{d}}$ with $1$, to identify the last term in \eqref{eq:mSHE+3} to be of the form {$N^{-1/2}\grad^{!}_{1}(\mathfrak{b}_{2;T,x}\mathbf{Z}^{N}_{T,x})+\mathrm{O}(N^{-1/2})$}. We clarify that here, there is no matching $\wt{\mathfrak{s}}$-terms with $\mathrm{R}_{2}$-terms, because the $N^{-1/2}$ factor we gain from having a coefficient $\mathrm{T}^{+,N}$ instead of $\mathrm{T}^{-,N}$ renders all such terms order $N^{-1/2}$, thus absorbed by $\mathfrak{b}_{1;}$.
\end{itemize}
This completes the proof.
\end{proof}
%
%
%
\section{Proof of Theorem \ref{theorem:KPZ}}\label{section:proofKPZ}
At a high level, the proof of Theorem \ref{theorem:KPZ} is built on an analysis of the semi-discrete stochastic integral equation from Corollary \ref{corollary:mSHE+}. As with \cite{DT,Y}, our main goal will be to prove that only the first two terms therein contribute in the large-$N$ limit in a ``high probability" sense. The last two terms on the RHS of this equation are easily shown to vanish in the large-$N$ limit by \emph{deterministic} and \emph{analytic} estimates, at least if we assume that the Gartner transform and its space-time supremum are not totally ill-behaved; such assumption will ultimately be justified by virtue of the fact that the Gartner transform is supposed to resemble the solution of SHE on the compact torus, which itself is uniformly continuous in space-time. But the $\mathfrak{s}$-term in Corollary \ref{corollary:mSHE+} does not admit such an elementary analytic estimate, since it does not necessarily have a deterministically small prefactor. The probabilistic approach we take to study the heat operator with the $\mathfrak{s}$-functional is based on the feature that it vanishes at a level of ``hydrodynamic limits" since the global $\eta$-density for our initial data is roughly zero, and by construction the expectation of $\mathfrak{s}$ with respect to the product Bernoulli measure of this $\eta$-density is also zero. Equivalently, in the language of \cite{DT,Y} the $\mathfrak{s}$-term is ``weakly vanishing". We will make this ``hydrodynamic" argument precise in Lemma \ref{lemma:KPZ16}.

We are left with analyzing the order $N^{1/2}$ term in the stochastic equation of Corollary \ref{corollary:mSHE+}. Because $N^{1/2}$ certainly diverges in the large-$N$ limit, neither the previous analytic or hydrodynamic limit arguments will succeed. In fact, if we replace the particle-system-dependent term $\bar{\mathfrak{q}}$ with \emph{any} local $\mathfrak{f}$ that has ``zero hydrodynamic limit" like $\mathfrak{s}$ above, it is likely false that the corresponding heat operator term in Corollary \ref{corollary:mSHE+} acting on $N^{1/2}\mathfrak{f}\mathbf{Z}^{N}$ will vanish in the large-$N$ limit, based on the equilibrium calculations in \cite{GJ15}, for example. Therefore, we must take advantage of $\bar{\mathfrak{q}}$ being the local functional $\wt{\mathfrak{q}}$ after subtracting off its ``leading order" behavior beyond the hydrodynamic limit when averaged in space-time against the heat kernel and $\mathbf{Z}^{N}$. We will do this through a non-stationary first-order Boltzmann-Gibbs principle, which will require a combination of analytic and probabilistic ingredients. The analytic considerations required mainly amount to regularity estimates of $\mathbf{Z}^{N}$, which by calculus implies regularity of $\mathbf{h}^{N}$ and, by definition, controls local invariant measures that are parameterized by $\eta$-density. For this reason, first define the following stopping times, which uniformly control $\mathbf{Z}^{N}$ and its space-time regularity. In the construction below, we will require a strange integer condition that is ultimately unnecessary; it will just make presentation later in the paper clearer and more convenient.
\begin{definition}\label{definition:KPZ1}
\fsp Consider $\e_{\mathrm{ap}}>0$ arbitrarily small but bounded below uniformly and chosen so that $N^{\e_{\mathrm{ap}}}$ is an integer. We note this may force $\e_{\mathrm{ap}}$ to be $N$-dependent, but this is okay; we only need its uniform positivity and smallness. Define
\begin{align}
\mathfrak{t}_{\mathrm{ap}} \ \overset{\bullet}= \ \inf\left\{\mathrm{t}\in[0,1]: \ \|\mathbf{Z}^{N}\|_{\mathrm{t};\mathbb{T}_{N}}+\|(\mathbf{Z}^{N})^{-1}\|_{\mathrm{t};\mathbb{T}_{N}} \geq N^{\e_{\mathrm{ap}}}\right\}\wedge1,
\end{align}
where $\|\|_{\mathrm{t};\mathbb{K}}$ is the $\mathscr{L}^{\infty}([0,\mathrm{t}]\times\mathbb{K})$-norm. We now introduce space-time scales on which we want a priori regularity estimates:
\begin{itemize}
\item We first define $\mathbb{I}^{\mathbf{T},1}\overset{\bullet}=\{N^{-2+\mathfrak{j}\e_{\mathrm{ap}}}\}_{\mathfrak{j}\geq0}\cap[0,N^{-1}]$. Observe that $N^{-2+\mathfrak{j}\e_{\mathrm{ap}}}$ are positive integer multiples of $N^{-2}$.
\item We now define $\mathbb{I}^{\mathbf{T}}\overset{\bullet}=\ \{\mathfrak{k}N^{-2+\mathfrak{j}\e_{\mathrm{ap}}}\}$, in which $1\leq\mathfrak{k}\leq N^{\e_{\mathrm{ap}}}$ and $\mathfrak{j}\geq0$ ranges over all indices for which $N^{-2+\mathfrak{j}\e_{\mathrm{ap}}}\leq N^{-1}$.
\end{itemize}
We also define/assume $\e_{\mathrm{RN}}=999^{-999}\geq999^{999}\e_{\mathrm{ap}}$ and then define the length-scale $\mathfrak{l}_{N}\overset{\bullet}=N^{1/2+\e_{\mathrm{RN}}}$. We now define the two stopping times below in which we recall $\grad^{\mathbf{X}}_{\mathfrak{l}}\phi_{x}=\phi_{x+\mathfrak{l}}-\phi_{x}$ and $\grad^{\mathbf{T}}_{\mathrm{s}}\psi_{t,x}=\psi_{(1\wedge(t+\mathrm{s}))\vee0,x}-\psi_{t,x}$ for $(t,x)\in[0,1]\times\mathbb{T}_{N}$:
\begin{align}
\mathfrak{t}_{\mathrm{RN}}^{\mathbf{T}} \ &\overset{\bullet}= \ \inf\left\{\mathrm{t}\in[0,1]: \ {\sup}_{\mathrm{s}\in\mathbb{I}^{\mathbf{T}}}\left(\mathrm{s}^{-1/4}\|\grad_{-\mathrm{s}}^{\mathbf{T}}\mathbf{Z}^{N}\|_{\mathrm{t};\mathbb{T}_{N}}\right) \geq N^{\e_{\mathrm{ap}}}\left(1+\|\mathbf{Z}^{N}\|_{\mathrm{t};\mathbb{T}_{N}}^{2}\right)\right\}\wedge1 \\
\mathfrak{t}_{\mathrm{RN}}^{\mathbf{X}} \ &\overset{\bullet}= \ \inf\left\{\mathrm{t}\in[0,1]: \ {\sup}_{1\leq|\mathfrak{l}|\leq\mathfrak{l}_{N}}\left(N^{1/2}|\mathfrak{l}|^{-1/2}\|\grad_{\mathfrak{l}}^{\mathbf{X}}\mathbf{Z}^{N}\|_{\mathrm{t};\mathbb{T}_{N}}\right) \geq N^{\e_{\mathrm{ap}}}\left(1+\|\mathbf{Z}^{N}\|_{\mathrm{t};\mathbb{T}_{N}})^{2}\right)\right\}\wedge1.
\end{align}
We conclude by defining the stopping time $\mathfrak{t}_{\mathrm{st}}={\mathfrak{t}_{\mathrm{ap}}}\wedge\mathfrak{t}_{\mathrm{RN}}^{\mathbf{T}}\wedge\mathfrak{t}_{\mathrm{RN}}^{\mathbf{X}}$ that is contained in $[0,1]$ with probability 1 and whose purpose is to supply a priori space-time control on the Gartner transform. Let us clarify that the utility behind the two regularity stopping times $\mathfrak{t}_{\mathrm{RN}}^{\mathbf{T}}$ and $\mathfrak{t}_{\mathrm{RN}}^{\mathbf{X}}$ will be to yield a priori estimates that are necessary to perform a renormalization scheme during the proof of the Boltzmann-Gibbs principle, while the utility behind $\mathfrak{t}_{\mathrm{ap}}$ is to avoid having to simultaneously apply probabilistic and analytic estimates to study particle system data and control $\mathbf{Z}^{N}$, the latter being ignorable if we look before the stopping time $\mathfrak{t}_{\mathrm{ap}}$.
\end{definition}
\begin{remark}\label{remark:KPZ2}
\fsp The stopping time $\mathfrak{t}_{\mathrm{ap}}$ also gives a priori lower bounds on $\mathbf{Z}^{N}$. This will be important in the proof of the Boltzmann-Gibbs principle. In particular, we require regularity estimates of the \emph{height function}. However, since the height function solves an equation that becomes a singular SPDE in the large-$N$ limit, and because singular SPDE analysis becomes difficult to conduct at the level of the particle system, we instead deduce regularity of height functions in terms of regularity of the Gartner transform as the Gartner transform equation becomes a non-singular SPDE in the large-$N$ limit. Calculus then tells us that a priori upper and lower bounds for the Gartner transform suffice to deduce regularity of the height function.
\end{remark}
\begin{remark}\label{remark:KPZ3}
\fsp We expect the Gartner transform to look like the solution of SHE in the large-$N$ limit, which, roughly speaking, has Holder regularity with exponent $2^{-1}$ in space and with exponent $4^{-1}$ in time. Therefore, the conditions/inequalities defining the stopping times $\mathfrak{t}_{\mathrm{ap}}$ and $\mathfrak{t}_{\mathrm{RN}}^{\mathbf{T}}$ and $\mathfrak{t}_{\mathrm{RN}}^{\mathbf{X}}$ are actually quite lenient because of the $N^{\e_{\mathrm{ap}}}$ factors and the assumption that $\e_{\mathrm{ap}}>0$ is universal and thus uniformly bounded from below. In particular, we will eventually be able to show that these three stopping times are all equal to 1 with sufficiently high probability, so their a priori estimates ``self-propagate".
\end{remark}
\begin{remark}\label{remark:KPZ4}
\fsp The constant $999^{999}$ in $\e_{\mathrm{RN}}=999^{999}\e_{\mathrm{ap}}$ can be replaced by any sufficiently large but universal constant.
\end{remark}
To take advantage of stopping times in Definition \ref{definition:KPZ1}, we now introduce the following auxiliary processes, the first of which stops the Gartner transform at the minimum stopping time $\mathfrak{t}_{\mathrm{st}}$ and the second of which evolves according to the same type of SHE dynamic as the Gartner transform though ignoring space-time sets where the conditions defining $\mathfrak{t}_{\mathrm{st}}$ fail, thus making the second auxiliary process amenable to the analysis of this paper, including the proof of the Boltzmann-Gibbs principle.
\begin{definition}\label{definition:KPZ5}
\fsp Define $\mathbf{Y}^{N}_{T,x}=\mathbf{Z}^{N}_{T,x}\mathbf{1}(T\leq\mathfrak{t}_{\mathrm{st}})$, and define the process $\mathbf{U}^{N}$ on $\R_{\geq0}\times\mathbb{T}_{N}$ via the stochastic equation
\begin{align}
\mathbf{U}_{T,x}^{N} \ = \ \mathbf{H}_{T,x}^{N,\mathbf{X}}(\mathbf{Z}_{0,\bullet}^{N}) + \mathbf{H}_{T,x}^{N}(\mathbf{U}^{N}\d\xi^{N}) - \mathbf{H}_{T,x}^{N}(N^{\frac12}\bar{\mathfrak{q}}\mathbf{Y}^{N}) - \mathbf{H}_{T,x}^{N}(\mathfrak{s}\mathbf{U}^{N}) + N^{-\frac12}\mathbf{H}_{T,x}^{N}(\mathfrak{b}_{1;}\mathbf{U}^{N}) + N^{-\frac12}\mathbf{H}_{T,x}^{N}\left(\grad_{\star}^{!}\left(\mathfrak{b}_{2;}\mathbf{U}^{N}\right)\right),\nonumber
\end{align}
where $\grad_{\star}^{!}$ means what it does in Proposition \ref{prop:mSHE+}.
\end{definition}
\begin{remark}\label{remark:KPZ5b}
\fsp The product $\mathbf{U}^{N}\d\xi^{N}$ denotes compensated jumps of a martingale, where the jumps at $(T,x)$ are given by the jumps of $\d\xi^{N}$ at $(T,x)$ from Proposition \ref{prop:mSHE+} times the value $\mathbf{U}^{N}$ at $(T,x)$. In fact, whenever we write a product of a space-time function and $\d\xi^{N}$, we mean exactly this where $\mathbf{U}^{N}$ is replaced by said space-time function. We additionally observe that for any functions $\mathbf{F}_{1},\mathbf{F}_{2}$, we have the identity $\mathbf{F}_{1}\d\xi^{N}-\mathbf{F}_{2}\d\xi^{N}=(\mathbf{F}_{1}-\mathbf{F}_{2})\d\xi^{N}$, as $\mathbf{F}_{1}\d\xi^{N}$ and $\mathbf{F}_{2}\d\xi^{N}$ are coupled and always jump together.
\end{remark}
To justify studying the $\mathbf{U}^{N}$ process, let us observe that on the event $\mathfrak{t}_{\mathrm{st}}=1$ we have not changed the $\mathbf{Z}^{N}$ equation in Corollary \ref{corollary:mSHE+} and have simply defined $\mathbf{U}^{N}$ with the same stochastic equation. Because the stochastic equation is linear in the solution $\mathbf{U}^{N}$, we have uniqueness of solutions with same initial data by elementary considerations, and thus $\mathbf{Z}^{N}=\mathbf{U}^{N}$ on such an event. In general, we have this identification between $\mathbf{Z}^{N}$ and $\mathbf{U}^{N}$ until $\mathfrak{t}_{\mathrm{st}}$ regardless of its value.
\begin{lemma}\label{lemma:KPZ6}
\fsp Provided any $\mathrm{t}\in[0,1]$, we have the containment of events $\{\mathfrak{t}_{\mathrm{st}}=\mathrm{t}\}\subseteq \cap_{0\leq\mathrm{s}\leq\mathrm{t}}\cap_{x\in\mathbb{T}_{N}}\{\mathbf{Z}^{N}_{\mathrm{s},x}=\mathbf{U}^{N}_{\mathrm{s},x}\}$.
\end{lemma}
We reiterate that working with the $\mathbf{U}^{N}$ process will be convenient because of the a priori space-time regularity estimates built into the $\mathbf{Y}^{N}$ process therein, while Lemma \ref{lemma:KPZ6} guarantees us $\mathbf{Z}^{N}$ and $\mathbf{U}^{N}$ are equal on the event where $\mathfrak{t}_{\mathrm{st}}=1$ which, as noted in Remark \ref{remark:KPZ3}, we will show happens with sufficiently high probability. Then, \emph{after taking advantage} of the cutoff in the stopping time $\mathfrak{t}_{\mathrm{st}}$, we compare $\mathbf{U}^{N}$ to the following process that forgets the order $N^{1/2}$ term in the $\mathbf{Z}^{N}$ and $\mathbf{U}^{N}$ equations.
\begin{definition}\label{definition:KPZ7}
\fsp Define the process $\mathbf{Q}^{N}$ on $\R_{\geq0}\times\mathbb{T}_{N}$ via the following stochastic integral equation
\begin{align}
\mathbf{Q}_{T,x}^{N} \ = \ \mathbf{H}_{T,x}^{N,\mathbf{X}}(\mathbf{Z}_{0,\bullet}^{N}) + \mathbf{H}_{T,x}^{N}(\mathbf{Q}^{N}\d\xi^{N}) - \mathbf{H}_{T,x}^{N}(\mathfrak{s}\mathbf{Q}^{N}) + N^{-\frac12}\mathbf{H}_{T,x}^{N}(\mathfrak{b}_{1;}\mathbf{Q}^{N}) + N^{-\frac12}\mathbf{H}_{T,x}^{N}\left(\grad_{\star}^{!}\left(\mathfrak{b}_{2;}\mathbf{Q}^{N}\right)\right), \nonumber
\end{align}
where $\grad_{\star}^{!}$ means what it does in Proposition \ref{prop:mSHE+}.
\end{definition}
We will now introduce the three key ingredients in the proof of Theorem \ref{theorem:KPZ}. The first ingredient shows $\mathfrak{t}_{\mathrm{st}}=1$ with a notion of high probability we will introduce shortly. This first step allows us to deduce Theorem \ref{theorem:KPZ} from itself but replacing $\mathbf{Z}^{N}$ therein with $\mathbf{U}^{N}$ introduced in Definition \ref{definition:KPZ5}. The second step then compares $\mathbf{U}^{N}$ with the auxiliary process $\mathbf{Q}^{N}$ in Definition \ref{definition:KPZ7}. Proofs of these two ingredients require the Boltzmann-Gibbs principle and take up the majority of this paper. The third step is to then prove Theorem \ref{theorem:KPZ} but replacing $\mathbf{Z}^{N}$ with $\mathbf{Q}^{N}$. This last step is fairly standard, as noted at the beginning of this section.
\begin{definition}\label{definition:KPZ8}
\fsp Consider any generic event $\mathcal{E}$. In the following, we think of constants $\delta>0$ as arbitrarily small but universal, and we think of constants $\kappa\geq0$ as arbitrarily large but universal.
\begin{itemize}
\item We say $\mathcal{E}$ holds with \emph{high probability} if for any $\delta>0$, we have $\mathbf{P}(\mathcal{E}^{C})\leq\delta+C_{\delta}\mathrm{o}_{N}$, where $\mathrm{o}_{N}\to_{N\to\infty}0$ uniformly in $\delta$.
\item We say $\mathcal{E}$ holds with \emph{overwhelming probability} if for any $\kappa\geq0$, we have $\mathbf{P}(\mathcal{E}^{C})\lesssim_{\kappa}N^{-\kappa}$.
\end{itemize}
\end{definition}
\begin{remark}\label{remark:KPZ9}
\fsp Any event $\mathcal{E}$ that satisfies the probability estimate $\mathbf{P}(\mathcal{E}^{C})\lesssim N^{-\beta}$ for \emph{some}, not all, constant $\beta>0$ holds with high probability because we may take $\mathrm{o}_{N}=N^{-\beta}$. But, it does not necessarily hold with overwhelming probability.
\end{remark}
\begin{prop}\label{prop:KPZ10}
\fsp The event $\{\mathfrak{t}_{\mathrm{st}}=1\}$ holds with high probability.
\end{prop}
\begin{prop}\label{prop:KPZ11}
\fsp Define the difference process $\mathbf{D}^{N}=\mathbf{U}^{N}-\mathbf{Q}^{N}$ on $\R_{\geq0}\times\mathbb{T}_{N}$. There exists a universal constant $\beta>0$ such that the event $\{\|\mathbf{D}^{N}\|_{1;\mathbb{T}_{N}}\lesssim N^{-\beta}\}$ holds with high probability, where the implied constant is also universal.
\end{prop}
\begin{prop}\label{prop:KPZ12}
\fsp The rescaled process $\Gamma^{N}\mathbf{Q}^{N}$ is tight in the large-$N$ limit in the Skorokhod space $\mathscr{D}_{1}$. Moreover, every limit point in $\mathscr{D}_{1}$ of $\Gamma^{N}\mathbf{Q}^{N}$ is the solution of $\ \mathrm{SHE}(\bar{\mathfrak{d}})$ with initial data equal to the spatially rescaled initial data $\lim_{N\to\infty}\Gamma^{N,\mathbf{X}}\mathbf{Z}^{N}$.
\end{prop}
\begin{proof}[Proof of \emph{Theorem \ref{theorem:KPZ}}]
Proposition \ref{prop:KPZ10} shows the difference $\mathbf{Z}^{N}-\mathbf{U}^{N}$ converges to 0 in probability in the Skorokhod space $\mathscr{D}_{1}$; with high probability the difference is identically the zero process in $\mathscr{D}_{1}$. Proposition \ref{prop:KPZ11} shows the difference $\mathbf{D}^{N}=\mathbf{U}^{N}-\mathbf{Q}^{N}$ also converges to 0 in probability in $\mathscr{D}_{1}$ because Proposition \ref{prop:KPZ11} shows that $\mathbf{D}^{N}$ converges to 0 in probability with respect to the uniform metric on $\mathscr{D}_{1}$, and the uniform metric on $\mathscr{D}_{1}$ is stronger than the Skorokhod topology on $\mathscr{D}_{1}$. To justify the last claim, it is enough to take the identify function within the infimum on the RHS of (12.13) in \cite{Bil}; though \cite{Bil} studies just $\R$-valued processes, the same is true of processes valued in any separable Banach space, including $\mathscr{C}(\mathbb{T}^{1})$. Combining these two observations implies $\mathbf{Z}^{N}-\mathbf{Q}^{N}\to0$ in $\mathscr{D}_{1}$. Finally, observe that Proposition \ref{prop:KPZ12} implies $\Gamma^{N}\mathbf{Q}^{N}$ converges to what we propose $\Gamma^{N}\mathbf{Z}^{N}$ converges to in $\mathscr{D}_{1}$. Because $\Gamma^{N}$ is a rescaling operator that is continuous with respect to the Skorokhod topology on $\mathscr{D}_{1}$, as it only rescales in space, standard probability shows $\Gamma^{N}\mathbf{Z}^{N}\to \mathrm{SHE}(\bar{\mathfrak{d}})$ in $\mathscr{D}_{1}$ with initial data $\lim_{N\to\infty}\Gamma^{N,\mathbf{X}}\mathbf{Z}^{N}$.
\end{proof}
We have now established Theorem \ref{theorem:KPZ} by taking Proposition \ref{prop:KPZ10}, Proposition \ref{prop:KPZ11}, and Proposition \ref{prop:KPZ12} for granted. Again, the proofs for Proposition \ref{prop:KPZ10} and Proposition \ref{prop:KPZ11} will be the purpose for the rest of this paper after the current section, and for this we establish a version of the non-stationary first-order Boltzmann-Gibbs principle. On the other hand, proof for Proposition \ref{prop:KPZ12} is fairly straightforward, as we alluded to near the beginning of this section, provided the analysis in \cite{DT}, especially Lemma 2.5 therein and its proof. Also, Proposition \ref{prop:KPZ12} will be important in establishing Proposition \ref{prop:KPZ10} and Proposition \ref{prop:KPZ11}, because it will yield a priori stability/bounds for studying the $\mathbf{Z}^{N}$-SPDE. With this in mind, we prove Proposition \ref{prop:KPZ12} in the current section. We will then conclude this section by writing an outline and discussion of the proofs of Proposition \ref{prop:KPZ10} and Proposition \ref{prop:KPZ11} whose details we devote the rest of this paper to. We first make the following point to avoid circular logic. Proposition \ref{prop:KPZ12} consists of tightness of $\Gamma^{N}\mathbf{Q}^{N}$ and identification of limit points. Tightness is independent of Propositions \ref{prop:KPZ10} and \ref{prop:KPZ11}, but identification of limit points uses Propositions \ref{prop:KPZ10} and \ref{prop:KPZ11}. Propositions \ref{prop:KPZ10} and \ref{prop:KPZ11}, in turn, requires tightness of $\Gamma^{N}\mathbf{Q}^{N}$.
\subsection{Proof of Proposition \ref{prop:KPZ12}}
Following the proof of Theorem 1.1 in \cite{DT}, which is given in Section 3 of \cite{DT}, let us first show the tightness claim in Proposition \ref{prop:KPZ12} with moment estimates for the auxiliary process $\mathbf{Q}^{N}$; this is the analog of Proposition 3.2 and Corollary 3.3 in \cite{DT}. Afterwards, we will identify subsequential limit points of $\Gamma^{N}\mathbf{Q}^{N}$ in the Skorokhod space $\mathscr{D}_{1}$ as $\mathrm{SHE}(\bar{\mathfrak{d}})$. This is the analog of the analysis behind Section 3.2 in \cite{DT}, and it similarly amounts to proving that all the limit points of $\Gamma^{N}\mathbf{Q}^{N}$ satisfy a martingale problem associated to $\mathrm{SHE}(\bar{\mathfrak{d}})$. We emphasize there is no real difficulty in choosing nonzero $\bar{\mathfrak{d}}$.
\begin{lemma}\label{lemma:KPZ13}
\fsp The sequence $\Gamma^{N}\mathbf{Q}^{N}$ is tight with respect to the Skorokhod topology on $\mathscr{D}_{1}$.
\end{lemma}
\begin{proof}
Provided heat kernel estimates and martingale estimates in Appendix \ref{section:aux}, the moment estimates in Proposition 3.2 in \cite{DT} for the Gartner transform therein hold for $\mathbf{Q}^{N}$, as long as we remove the exponential weights therein and then replace the spatial gradients therein with spatial gradients on the torus $\mathbb{T}_{N}$. Thus, it suffices to follow the proof of Corollary 3.3 in \cite{DT}.
\end{proof}
We are now left with identifying limit points in $\mathscr{D}_{1}$ of the sequence $\Gamma^{N}\mathbf{Q}^{N}$. As we briefly alluded to above, we will require the following martingale problem formulation of $\mathrm{SHE}(\bar{\mathfrak{d}})$. Technically, the following martingale problem formulation of $\mathrm{SHE}(\bar{\mathfrak{d}})$ differs from the martingale problems for SPDEs introduced and employed in \cite{BG,DT} and related papers unless $\bar{\mathfrak{d}}=0$ because of the additional first-order linear transport operator. But this transport is lower-order and introduces only a linear drift.
\begin{definition}\label{definition:KPZ14}
\fsp Let us first choose any pair of spatial test functions $\psi_{1;\cdot},\psi_{2;\cdot}:\mathbb{T}\to\R$ and any pair of space-time test functions $\phi_{1;\cdot,\cdot},\phi_{2;\cdot,\cdot}:\R_{\geq0}\times\mathbb{T}\to\R$, where we recall that $\mathbb{T}=\R/\Z$ is the unit torus. We also define the following bilinear pairings.
\begin{itemize}
\item Define $\langle\psi_{1;\cdot},\psi_{2;\cdot}\rangle_{\mathbb{T}}=\int_{\mathbb{T}}\psi_{1;x}\psi_{2;x}\d x$ and for $\mathrm{t}\in[0,1]$, let $\langle\phi_{1;\cdot,\cdot},\phi_{2;\cdot,\cdot}\rangle_{\mathrm{t};\mathbb{T}}=\int_{0}^{\mathrm{t}}\langle\phi_{1;\mathrm{s},\cdot},\phi_{2;\mathrm{s},\cdot}\rangle_{\mathbb{T}}\d\mathrm{s}$.
\end{itemize}
Consider a possibly random continuous process $\mathbf{S}_{\cdot,\cdot}\in\mathscr{C}_{1}$. Let us say $\mathbf{S}_{\cdot,\cdot}$ solves the $\mathrm{SHE}(\bar{\mathfrak{d}})$ \emph{martingale problem} if:
\begin{itemize}
\item We have $\mathbf{S}_{0,\cdot}=\lim_{N\to\infty}\Gamma^{N,\mathbf{X}}\mathbf{Z}^{N}$ and the second moment bound $\|\mathbf{S}_{\mathrm{t},x}\|_{\omega;2}\lesssim1$ uniformly over all $\mathrm{t}\in[0,1]$ and $x\in\mathbb{T}$.
\item Let $\mathscr{L}^{\ast}$ be the formal adjoint of the continuum differential operator $\mathscr{L}=2^{-1}\Delta+\bar{\mathfrak{d}}\grad$ with respect to the Lebesgue measure on $\mathbb{T}$. For any smooth, time-constant test function $\phi\in\mathscr{C}^{\infty}(\mathbb{T})$, the following are local $\R$-valued martingales in $\mathrm{t}\in[0,1]$:
\begin{align}
\mathbf{m}_{\mathrm{t}}(\phi) \ \overset{\bullet}= \ \langle\phi,\mathbf{S}_{\mathrm{t},\cdot}\rangle_{\mathbb{T}}-\langle\phi,\mathbf{S}_{0,\cdot}\rangle_{\mathbb{T}}-\langle\mathscr{L}^{\ast}\phi,\mathbf{S}_{\cdot,\cdot}\rangle_{\mathrm{t};\mathbb{T}} \quad \mathrm{and} \quad \mathbf{m}_{\mathrm{t}}(\phi)^{2}-\langle\phi^{2},\mathbf{S}_{\cdot,\cdot}^{2}\rangle_{\mathrm{t};\mathbb{T}}.
\end{align}
\end{itemize}
\end{definition}
In \cite{BG,DT}, the key feature of the martingale problem for $\mathrm{SHE}=\mathrm{SHE}(0)$ is that any solution is equal to the mild solution as probability measures on the path-space $\mathscr{C}_{1}$. It turns out solutions of $\mathrm{SHE}(\bar{\mathfrak{d}})$ share a similar property, since $\mathrm{SHE}(\bar{\mathfrak{d}})$ is still \emph{linear}.
\begin{lemma}\label{lemma:KPZ15}
\fsp If $\mathbf{S}\in\mathscr{C}_{1}$ is a solution of the $\mathrm{SHE}(\bar{\mathfrak{d}})$ martingale problem, then $\mathbf{S}=\mathrm{SHE}(\bar{\mathfrak{d}})$ as probability measures on $\mathscr{C}_{1}$.
\end{lemma}
Let us now identify limit points in $\mathscr{D}_{1}$ of the sequence $\Gamma^{N}\mathbf{Q}^{N}$ as $\mathrm{SHE}(\bar{\mathfrak{d}})$ with the martingale problem in Definition \ref{definition:KPZ14} and the uniqueness result Lemma \ref{lemma:KPZ15}. To this end, we follow Section 3.2 of \cite{DT}. The first step is to compute the predictable bracket of the martingale differential $\mathbf{Q}^{N}\d\xi^{N}$. The environment-dependence is lower-order, so it is negligible in the large-$N$ limit. We ultimately get a predictable bracket equal to that in Proposition 3.4 in \cite{DT} for $m=1$, after replacing all Gartner transforms therein with $\mathbf{Q}^{N}$. The next step to identify $\Gamma^{N}\mathbf{Q}^{N}$ is then the following hydrodynamic limit, as was the situation in Section 3.2 in \cite{DT}. Thus, the proof of Proposition \ref{prop:KPZ12} amounts to proving the following parallel to Lemma 2.5 in \cite{DT}.
\begin{lemma}\label{lemma:KPZ16}
\fsp Consider any local function $\mathfrak{f}:\Omega\to\R$ whose support is a uniformly bounded neighborhood of $0\in\mathbb{T}_{N}$. Suppose $\E_{0}\mathfrak{f}=0$, in which $\E_{0}$ is the expectation with respect to the product Bernoulli measure on $\Omega$ whose one-dimensional marginals vanish in expectation. Let us define the space-time shift $\mathfrak{f}_{S,y}=\tau_{y}\mathfrak{f}(\eta_{S})$. For any $\phi\in\mathscr{C}^{\infty}(\mathbb{T})$ and $\mathrm{t}\in[0,1]$, we have
\begin{align}
\lim_{N\to\infty}\int_{0}^{\mathrm{t}}\wt{\sum}_{y\in\mathbb{T}_{N}}\phi_{y/N} \cdot\mathfrak{f}_{S,y}\mathbf{Q}_{S,y}^{N} \ \d S \ = \ 0.
\end{align}
\end{lemma}
\begin{proof}
Recall $\mathbf{Q}^{N}$ satisfies spatial regularity on macroscopic length-scales as we explained in the proof of Lemma \ref{lemma:KPZ13}. Thus, by following the proof of Lemma 2.5 in \cite{DT}, it suffices to replace $\mathfrak{f}_{S,y}$ in the proposed limit with its spatial average over a block of length $\delta N^{1/2}$ with $\delta>0$ small but taken to zero after taking the large-$N$ limit. At this point, following the proof of Lemma 2.5 from \cite{DT} suffices, because we also have the pointwise moment estimate to bound $\mathbf{Q}^{N}$ as argued in the proof of Lemma \ref{lemma:KPZ13}, and because $\E_{0}\mathfrak{f}=0$, the one-block and two-blocks schemes from Section 4 of \cite{DT} lets us replace the block average of $\mathfrak{f}$ by the block average of $\eta$; these steps are successful here as well by virtue of entropy production estimates in a finite volume of order $N$ that is even better than the entropy production in Lemma 4.1 in \cite{DT} that was used in the proof of Lemma 2.5 in \cite{DT}. Lastly, to estimate the block average of $\eta$ of length $\delta N^{1/2}$ with $\delta>0$ vanishing \emph{after} the large-$N$ limit, it suffices to note $\mathfrak{t}_{\mathrm{st}}=1$ with high probability by Proposition \ref{prop:KPZ10}, and $\mathfrak{t}_{\mathrm{st}}=1$ implies regularity estimates for $\mathbf{Z}^{N}$ that are used to show the vanishing of the $\eta$-block average at hand; see the proof of Lemma 2.5 in \cite{DT} for more details on this last point. This completes the proof.
\end{proof}
\subsection{Strategy}\label{subsection:s}
Recall $\bar{\mathfrak{q}}$ in Definition \ref{definition:mSHE+1}; it is the correction of a local statistic by its hydrodynamic limit and appropriate linear projection. In what follows, all we need from bold-face objects are that they are possibly random but have space-time regularity at worst matching SHE or KPZ, and all we need from the $\mathbf{H}^{N}$ operator is that it is integration in space-time against a reasonably smooth test function (though in this paper, we specify to the heat operator in Definition \ref{definition:heat}).

Proposition \ref{prop:KPZ10} and Proposition \ref{prop:KPZ11} effectively follow by showing the order $N^{1/2}$-term in the $\mathbf{U}^{N}$ equation is small. Indeed, this would imply the estimate in Proposition \ref{prop:KPZ11} by standard methods for linear equations, as the $\mathbf{U}^{N}$ equation in Definition \ref{definition:KPZ5} and the $\mathbf{Q}^{N}$ equation in Definition \ref{definition:KPZ7} differ only in this $N^{1/2}$ term. On the other hand, provided that $\mathbf{U}^{N}\approx\mathbf{Q}^{N}$ via Proposition \ref{prop:KPZ11}, space-time estimates for $\mathbf{U}^{N}$ are inherited from those for $\mathbf{Q}^{N}$, which we have already shown behaves like $\mathrm{SHE}(\bar{\mathfrak{d}})$ and therefore satisfies significantly improved versions of the estimates defining $\mathfrak{t}_{\mathrm{st}}$, namely replacing $\e_{\mathrm{ap}}$ therein with $\e_{\mathrm{ap}}/999$ for example. Thus $\mathbf{U}^{N}$ satisfies improved versions of the regularity estimates defining $\mathfrak{t}_{\mathrm{st}}$. Using Lemma \ref{lemma:KPZ6}, this implies that $\mathbf{Z}^{N}$ satisfies the same estimates before the stopping time $\mathfrak{t}_{\mathrm{st}}$, after which we may extend these estimates after the stopping $\mathfrak{t}_{\mathrm{st}}$ upon directly studying $\mathbf{Z}^{N}$ on \emph{very} short/sub-microscopic time-scales. This shows the estimates defining $\mathfrak{t}_{\mathrm{st}}$ are \emph{self-propagating}, and thus $\mathfrak{t}_{\mathrm{st}}=1$, with high probability as claimed in Proposition \ref{prop:KPZ10}.

We now discuss showing the order $N^{1/2}$-term in the $\mathbf{U}^{N}$ equation from Definition \ref{definition:KPZ5} is small, which we state as the following {heuristic} that is usually known as a Boltzmann-Gibbs principle; we will prove a stronger version of the following.
\begin{pprop}\label{pprop:s1}
\fsp We have the convergence-in-probability $\|\mathbf{H}^{N}(N^{1/2}\bar{\mathfrak{q}}\mathbf{Y}^{N})\|_{1;\mathbb{T}_{N}}\to0$ in the large-$N$ limit.
\end{pprop}
\subsubsection{Approach via Mesoscopic Equilibrium}
The strategy for {Heuristic} \ref{pprop:s1} is based on replacing $\bar{\mathfrak{q}}$ by its invariant measure expectation via ergodic theory on the mesoscopic length-scale $N^{1/2+\e_{\mathrm{RN}}}$. By invariant measure, we technically mean a canonical measure expectation of parameter $\sigma$ given by the $\eta$-density on a block of length of order $N^{1/2+\e_{\mathrm{RN}}}$; see Definition \ref{definition:ensembles} for what this means. The philosophy, coming from \cite{GPV}, of this approach is that the particle system evolves on extremely fast $N^{2}$ time-scales, and thus on \emph{mesoscopic}/\emph{local} scales, relaxation to invariant measure happens quickly.

To make this discussion a little more concrete, we will present evidence of the following statement. Again, we refer the reader to Definition \ref{definition:ensembles} for the canonical measure used below. We also refer the reader to {Lemma 1 and Lemma 3} in \cite{GJ15} for another set of results that are slightly weaker but philosophically analogous to the following.
\begin{plemma}\label{plemma:s2}
\fsp Let {$\mathsf{E}^{\mathrm{can}}_{1/2+\e_{\mathrm{RN}}}(\tau_{y}\eta_{S})$} be canonical measure expectation of $\bar{\mathfrak{q}}_{S,y}$ on a block of length-scale $\mathfrak{l}_{N}=N^{1/2+\e_{\mathrm{RN}}}$ whose $\sigma$-parameter is equal to the $\eta$-density in a length $\mathfrak{l}_{N}$ neighborhood of the support of $\bar{\mathfrak{q}}_{S,y}$. We have the convergence in probability $\|\mathbf{H}^{N}(N^{1/2}(\bar{\mathfrak{q}}_{S,y}-{\mathsf{E}^{\mathrm{can}}_{1/2+\e_{\mathrm{RN}}}(\tau_{y}\eta_{S})})\|_{1;\mathbb{T}_{N}}\to0$ in the large-$N$ limit.
\end{plemma}
To explain the benefit of {Heuristic} \ref{plemma:s2}, {Proposition 8} in \cite{GJ15} basically shows $|{\mathsf{E}^{\mathrm{can}}_{1/2+\e_{\mathrm{RN}}}(\tau_{y}\eta_{S})}|\lesssim\mathfrak{l}_{N}^{-1}=N^{-1/2-\e_{\mathrm{RN}}}$. This beats $N^{1/2}$, so it remains to prove {Heuristic} \ref{plemma:s2}. We clarify that this bound only holds if the $\eta$-density at scale $\mathfrak{l}_{N}$ is controlled by $\mathfrak{l}_{N}^{-1/2}$, which is basically equivalent to $\mathbf{h}^{N}$ and $\mathbf{Z}^{N}$ satisfying regularity estimates defining $\mathfrak{t}_{\mathrm{RN}}^{\mathbf{X}}$ in Definition \ref{definition:KPZ1}.
\subsubsection{Evidence for \emph{{Heuristic} \ref{plemma:s2}}}
The replacement that we proposed in {Heuristic}\ref{plemma:s2} will not be performed in one step. We need to replace $\bar{\mathfrak{q}}$ by its canonical measure expectations on progressively larger length-scales until we hit $\mathfrak{l}_{N}=N^{1/2+\e_{\mathrm{RN}}}$, with every length-scale being a small but universal power of $N$ larger than the previous scale. This ``renormalization" is similar to that of {Lemma 2} in \cite{GJ15}; {see also proofs of Theorems 1.1 and 1.2 in \cite{SX}}. {In what follows, $\log_{N}(\mathrm{a})=\frac{\log\mathrm{a}}{\log N}$ is the base $N$ logarithm.}
\begin{plemma}\label{plemma:s3}
\fsp For $\mathfrak{l}\in\Z_{\geq0}$ larger than the length of the support of $\bar{\mathfrak{q}}$, we let ${\mathsf{E}^{\mathrm{can}}_{\log_{N}\mathfrak{l}}(\tau_{y}\eta_{S})}$ be canonical measure expectation of $\bar{\mathfrak{q}}_{S,y}$ on a neighborhood of the support of $\bar{\mathfrak{q}}_{S,y}$ with length-scale $\mathfrak{l}$. Take $\e_{\mathrm{RN},1}>0$ sufficiently small but universal. We have the following uniformly, in probability, in all $\mathfrak{l}$-indices larger than the length of the support of $\bar{\mathfrak{q}}_{S,y}$ that also satisfy $N^{\e_{\mathrm{RN},1}}\mathfrak{l}\leq\mathfrak{l}_{N}$:
\begin{align}
\|\mathbf{H}^{N}(N^{\frac12}{\mathsf{R}_{\log_{N}\mathfrak{l}}(\tau_{y}\eta_{S})}\mathbf{Y}^{N}_{S,y})\|_{1;\mathbb{T}_{N}} \ \to_{N\to\infty} \ 0 \quad \mathrm{where} \quad {\mathsf{R}_{\log_{N}\mathfrak{l}}(\tau_{y}\eta_{S})=\mathsf{E}^{\mathrm{can}}_{\log_{N}\mathfrak{l}}(\tau_{y}\eta_{S})-\mathsf{E}^{\mathrm{can}}_{\e_{\mathrm{RN},1}+\log_{N}\mathfrak{l}}(\tau_{y}\eta_{S})}.
\end{align}
Let $\mathfrak{l}_{0}\in\Z_{\geq0}$ be any uniformly bounded length-scale that is larger than the length of the support of $\bar{\mathfrak{q}}$. We additionally have the claimed convergence in probability in \emph{{Heuristic} \ref{plemma:s2}} if we replace {$\mathsf{E}^{\mathrm{can}}_{1/2+\e_{\mathrm{RN}}}(\tau_{y}\eta_{S})$} therein by {$\mathsf{E}^{\mathrm{can}}_{\log_{N}\mathfrak{l}_{0}}(\tau_{y}\eta_{S})$}.
\end{plemma}
Let us focus on the proposed bound for {$\mathsf{R}$} terms in {Heuristic} \ref{plemma:s3} and discuss necessary adjustments for the difference between $\bar{\mathfrak{q}}$ and {$\mathsf{E}^{\mathrm{can}}_{\log_{N}\mathfrak{l}_{0}}$} afterwards; analyses of both will be similar to each other except for an important technical obstruction faced by the latter difference. The first key observation for the {$\mathsf{R}$} terms is their \emph{fluctuation property}; with respect to any invariant canonical measure on the support of {$\mathsf{R}$}, the function {$\mathsf{R}$} vanishes in expectation. Given that the support of {$\mathsf{R}$} is mesoscopic in scale and given the fastness to invariant measures on mesoscopic length-scales, the functional {$\mathsf{R}$} is rapidly fluctuating on mesoscopic time-scales. To take advantage of these fluctuations, we \emph{average} {$\mathsf{R}$} on \emph{mesoscopic} space-time scales, in contrast to macroscopic scales that are used in \cite{CYau}. Assuming we can replace {$\mathsf{R}$} by such time-average for now and additionally assuming that the law of the particle system around the support of {$\mathsf{R}$} is an invariant canonical measure, we would be able to control this mesoscopic space-time average of {$\mathsf{R}$} as if it were the space-time average of a noise by the Kipnis-Varadhan inequality; see Appendix 1.6 in \cite{KL}. To be precise, if {$\mathfrak{l}(\mathsf{R}_{\log_{N}\mathfrak{l}})$} is twice the support length of {$\mathsf{R}_{\log_{N}\mathfrak{l}}$},
\begin{align}
|\mathfrak{t}_{\mathrm{av}}^{-1}\int_{0}^{\mathfrak{t}_{\mathrm{av}}}\mathfrak{l}_{\mathrm{av}}^{-1}{\sum}_{|w|\leq\mathfrak{l}_{\mathrm{av}}}\tau_{w\cdot\mathfrak{l}({\mathsf{R}_{\log_{N}\mathfrak{l}}})}{\mathsf{R}_{\log_{N}\mathfrak{l}}(\tau_{y}\eta_{\mathfrak{r}})}\d\mathfrak{r}| \ \lesssim \ N^{-1}\mathfrak{t}_{\mathrm{av}}^{-\frac12}\mathfrak{l}_{\mathrm{av}}^{-\frac12}{\mathfrak{l}(\mathsf{R}_{\log_{N}\mathfrak{l}})|\mathsf{R}_{\log_{N}\mathfrak{l}}|}. \label{eq:s1}
\end{align}
Thus, the LHS exhibits Brownian behavior in space-time. The factor $N^{-1}$ on the RHS, which makes \eqref{eq:s1} extremely useful for obtaining {Heuristic} \ref{plemma:s3}, comes from the fact that the system evolves at a speed $N^{2}$. Therefore, convergence to invariant measure happens on time-scales of order $N^{-2}$, creating more fluctuation before time $\mathfrak{t}_{\mathrm{av}}$. In a similar spirit, the factor $\mathfrak{l}({\mathsf{R}_{\log_{N}\mathfrak{l}}})$ on the RHS of \eqref{eq:s1}, which actually makes \eqref{eq:s1} worse as we increase the length-scale of the support of ${\mathsf{R}_{\log_{N}\mathfrak{l}}}$ in {Heuristic} \ref{plemma:s3}, comes from the fact that we require the particle system to converge to invariant measure in a neighborhood of the support of ${\mathsf{R}_{\log_{N}\mathfrak{l}}}$ in order to exploit its fluctuations; this happens more slowly as the support of ${\mathsf{R}_{\log_{N}\mathfrak{l}}}$ increases. As for $|{\mathsf{R}_{\log_{N}\mathfrak{l}}}|$:
\begin{itemize}
\item {Heuristic} \ref{plemma:s2} replaces $\bar{\mathfrak{q}}$ by the {$\mathsf{E}^{\mathrm{can}}_{1/2+\e_{\mathrm{RN}}}$}-term that has deterministic bounds. This feature of {$\mathsf{E}^{\mathrm{can}}_{1/2+\e_{\mathrm{RN}}}$} is not exclusive to the length-scale $\mathfrak{l}_{N}$. Precisely, as length-scales of {$\mathsf{E}^{\mathrm{can}}$} terms in {Heuristic} \ref{plemma:s3} defining {$\mathsf{R}$} terms increases, such {$\mathsf{E}^{\mathrm{can}}$} functionals decrease in magnitude as noted after {Heuristic} \ref{plemma:s2}, and therefore so do {$\mathsf{R}$} functions. It turns out that this competition between $\mathfrak{l}({\mathsf{R}_{\log_{N}\mathfrak{l}}})$ and $|{\mathsf{R}_{\log_{N}\mathfrak{l}}}|$ on the RHS of \eqref{eq:s1} almost perfectly cancel, because {$\mathsf{E}^{\mathrm{can}}_{\log_{N}\mathfrak{l}}$} and ${\mathsf{R}_{\log_{N}\mathfrak{l}}}$ are controlled by the inverse of the length-scale. This is why the multiscale replacement in {Heuristic} \ref{plemma:s3} is feasible.
\end{itemize}
Provided the previous bullet point, we will want to take $\mathfrak{t}_{\mathrm{av}}\sim N^{-1}$ and $\mathfrak{l}_{\mathrm{av}}\gg1$. Actually, for our applications of estimates of the form \eqref{eq:s1} in this paper, we will take $\mathfrak{t}_{\mathrm{av}}$ slightly smaller than $N^{-1}$ and $\mathfrak{l}_{\mathrm{av}}$ a mesoscopic length-scale noticeably larger than simply $\mathfrak{l}_{\mathrm{av}}\gg1$, but this is entirely for technical reasons.
\subsubsection{Replacement by Space-Time Averages}
In the discussion of {Heuristic} \ref{plemma:s3} given after its statement, we omitted an important issue of replacing ${\mathsf{R}_{\log_{N}\mathfrak{l}}}$ terms with space-time averages. We first explain why we can replace by spatial averages.
\begin{itemize}
\item Recall in the paragraph following the above single bullet point that we want to replace ${\mathsf{R}_{\log_{N}\mathfrak{l}}}$ by its spatial average on length-scale $\mathfrak{l}_{\mathrm{av}}\gg1$. If we replace ${\mathsf{R}_{\log_{N}\mathfrak{l}}}$ in the heat operator in {Heuristic} \ref{plemma:s3} with its spatial average on length-scale $\mathfrak{l}_{\mathrm{av}}$, the error is controlled by the difference between ${\mathsf{R}_{\log_{N}\mathfrak{l}}}$ and its spatial translations, where the length-scale for the translations are at most $\mathfrak{l}_{\mathrm{av}}\mathfrak{l}({\mathsf{R}_{\log_{N}\mathfrak{l}}})$; see the LHS of \eqref{eq:s1}. These differences are spatial \emph{gradients} of ${\mathsf{R}_{\log_{N}\mathfrak{l}}}$; to bound these when multiplied by $\mathbf{Y}^{N}$ and plugged into the heat operator, we apply summation-by-parts. This lets us transfer the spatial gradients from ${\mathsf{R}_{\log_{N}\mathfrak{l}}}$ to $\mathbf{Y}^{N}$ and the $\mathbf{H}^{N}$ heat kernel, so it suffices to estimate these spatial gradients of smoother objects. The heat kernel is macroscopically smooth, so its spatial gradients carry a factor of $N^{-1}$. The $\mathbf{Y}^{N}$ term is constructed with a priori spatial regularity bounds; by Definition \ref{definition:KPZ1} and Definition \ref{definition:KPZ5}, such $\mathbf{Y}^{N}$ factor is basically \emph{macroscopically} spatially Holder-$\frac12$ continuous. Adding these two estimates for the spatial regularity of the $\mathbf{H}^{N}$ heat kernel and of $\mathbf{Y}^{N}$, the error in introducing the spatial average of ${\mathsf{R}_{\log_{N}\mathfrak{l}}}$ in {Heuristic} \ref{plemma:s3} is basically at most the following; recall from Definition \ref{definition:KPZ5} that $\mathbf{Y}^{N}$ is basically bounded and recall from earlier in this paragraph that the max length-scale of spatial gradients here is $\mathfrak{l}_{\mathrm{av}}\mathfrak{l}({\mathsf{R}_{\log_{N}\mathfrak{l}}})$:
\begin{align}
N^{-1}\mathfrak{l}_{\mathrm{av}}\mathfrak{l}({\mathsf{R}_{\log_{N}\mathfrak{l}}})N^{\frac12}|{\mathsf{R}_{\log_{N}\mathfrak{l}}}| + N^{\frac12}|{\mathsf{R}_{\log_{N}\mathfrak{l}}}|N^{-\frac12}\mathfrak{l}_{\mathrm{av}}^{\frac12}\mathfrak{l}({\mathsf{R}_{\log_{N}\mathfrak{l}}})^{\frac12}. \label{eq:s2}
\end{align}
As $\mathfrak{l}({\mathsf{R}_{\log_{N}\mathfrak{l}}})|{\mathsf{R}_{\log_{N}\mathfrak{l}}}|\lesssim1$ like in the proof-idea of {Heuristic} \ref{plemma:s3}, if $|{\mathsf{R}_{\log_{N}\mathfrak{l}}}|\gg1$, then $\mathfrak{l}_{\mathrm{av}}\gg_{\mathfrak{l}({\mathsf{R}_{\log_{N}\mathfrak{l}}})}1$ makes \eqref{eq:s2} small. 
\item If $|{\mathsf{R}_{\log_{N}\mathfrak{l}}}|\not\gg1$, the first term in \eqref{eq:s2} still vanishes in the large-$N$ limit if we pick $\mathfrak{l}_{\mathrm{av}}\ll N^{1/2}$, which we certainly will in this paper. As for the second term in \eqref{eq:s2}, we recall that term comes from blindly controlling the spatial gradients of $\mathbf{Y}^{N}$ via its spatial Holder regularity. However, we also know that $\mathbf{Y}^{N}$, whenever it is nonzero and equal to $\mathbf{Z}^{N}$, is explicit in terms of the particle system by definition. Therefore its spatial gradient on the length-scale $w\mathfrak{l}({\mathsf{R}_{\log_{N}\mathfrak{l}}})$ for $|w|\leq\mathfrak{l}_{\mathrm{av}}$, if nonzero, is $\mathbf{Y}^{N}$ itself times an explicit functional of the particle system whose support, it turns out after explicit calculation, is contained outside the support of ${\mathsf{R}_{\log_{N}\mathfrak{l}}}$ and of length at most $\mathfrak{l}_{\mathrm{av}}\mathfrak{l}({\mathsf{R}_{\log_{N}\mathfrak{l}}})\lesssim\mathfrak{l}_{\mathrm{av}}$ and thus not too large. We emphasize that this disjoint-support-condition we just mentioned is a consequence of the shifting $\mathfrak{q}$ in Definition \ref{definition:mSHE+1}, which is actually the exact purpose of that shift. Ultimately, the product between this functional and ${\mathsf{R}_{\log_{N}\mathfrak{l}}}$, which is lower-order because gradients of the explicit formula for $\mathbf{Y}^{N}$ introduce factors of $N^{-1/2}$, admits an inverse-length-scale estimate and satisfies a similar fluctuation property as ${\mathsf{R}_{\log_{N}\mathfrak{l}}}$ itself because of the disjoint support condition, so we can apply for it a simpler version of this analysis.
\item Again, we actually pick $\mathfrak{l}_{\mathrm{av}}\gg1$ more precisely. For ${\mathsf{R}_{\log_{N}\mathfrak{l}}}$ terms whose support lengths are asymptotically large but still below a certain $N$-dependent threshold, we take $\mathfrak{l}_{\mathrm{av}}\approx\mathfrak{l}({\mathsf{R}_{\log_{N}\mathfrak{l}}})$. For ${\mathsf{R}_{\log_{N}\mathfrak{l}}}$ whose supports have lengths above this threshold, we will be less strict with $\mathfrak{l}_{\mathrm{av}}$ and take advantage of the consequentially small $|{\mathsf{R}_{\log_{N}\mathfrak{l}}}|$ in \eqref{eq:s2}, letting it do the work.
\end{itemize}
We now discuss the problem of introducing a time-average of ${\mathsf{R}_{\log_{N}\mathfrak{l}}}$ \emph{after} introducing a spatial-average.
\begin{itemize}
\item Similar to the replacement by spatial-average for ${\mathsf{R}_{\log_{N}\mathfrak{l}}}$ terms from {Heuristic} \ref{plemma:s3}, replacements by time-averages for ${\mathsf{R}_{\log_{N}\mathfrak{l}}}$ terms therein contributes errors that are controlled by time-gradients of the $\mathbf{H}^{N}$ heat kernel and of $\mathbf{Y}^{N}$. The $\mathbf{H}^{N}$ heat kernel is smooth in time and the latter has time-regularity of Holder-$\frac14$, basically. Thus, similar to \eqref{eq:s2}, we deduce that the error in said replacement by time-average on time-scale $\mathfrak{t}_{\mathrm{av}}$ is controlled by the following, in which the $\mathfrak{l}_{\mathrm{av}}$-based factor comes from the fact that we have already spatially averaged ${\mathsf{R}_{\log_{N}\mathfrak{l}}}$ on length-scale $\mathfrak{l}_{\mathrm{av}}\gg1$ and thus gained an a priori estimate for ${\mathsf{R}_{\log_{N}\mathfrak{l}}}$/its scale-$\mathfrak{l}_{\mathrm{av}}$ spatial average because of its fluctuating behavior as in \eqref{eq:s1}: 
\begin{align}
\mathfrak{t}_{\mathrm{av}}N^{\frac12}\mathfrak{l}_{\mathrm{av}}^{-\frac12}|{\mathsf{R}_{\log_{N}\mathfrak{l}}}| + N^{\frac12}\mathfrak{l}_{\mathrm{av}}^{-\frac12}|{\mathsf{R}_{\log_{N}\mathfrak{l}}}|\mathfrak{t}_{\mathrm{av}}^{\frac14}. \label{eq:s3}
\end{align}
Recall we want $\mathfrak{t}_{\mathrm{av}}=N^{-1}$; the first term in \eqref{eq:s3} vanishes in the large-$N$ limit. The second term, however, clearly blows up. Instead, we take $\mathfrak{t}_{\mathrm{av}}=\mathfrak{t}_{\mathrm{av},1}=N^{-2}\mathfrak{l}_{\mathrm{av}}$ and pretend $\mathfrak{l}_{\mathrm{av}}=N^{\e}$ for $\e>0$ small but universal, which will ultimately be the case later in this paper. This choice of $\mathfrak{t}_{\mathrm{av}}$ makes it so both terms in \eqref{eq:s3} vanish in the large-$N$ limit.
\item We replaced the spatial average of ${\mathsf{R}_{\log_{N}\mathfrak{l}}}$ with its time-average on scale $\mathfrak{t}_{\mathrm{av},1}=N^{-2}\mathfrak{l}_{\mathrm{av}}$. Let us now replace \emph{this} time-average with its own time-average on a time-scale $\mathfrak{t}_{\mathrm{av},2}=N^{\rho}\mathfrak{t}_{\mathrm{av},1}$, where $\rho>0$ is small but universal. Similar to \eqref{eq:s3}, we establish the following rough estimate for the error in this time-average-replacement, but with a key distinction we explain below:
\begin{align}
\mathfrak{t}_{\mathrm{av},2}N^{\frac12}N^{-1}\mathfrak{t}_{\mathrm{av},1}^{-\frac12}\mathfrak{l}_{\mathrm{av}}^{-\frac12}|{\mathsf{R}_{\log_{N}\mathfrak{l}}}| + N^{\frac12}N^{-1}\mathfrak{t}_{\mathrm{av},1}^{-\frac12}\mathfrak{l}_{\mathrm{av}}^{-\frac12}|{\mathsf{R}_{\log_{N}\mathfrak{l}}}|\mathfrak{t}_{\mathrm{av},2}^{\frac14}. \label{eq:s4}
\end{align}
Besides replacing time-scales, the difference between \eqref{eq:s3} and \eqref{eq:s4} is $\mathfrak{t}_{\mathrm{av},1}$-based factors. These come from the fact that we have already time-averaged the spatial average of ${\mathsf{R}_{\log_{N}\mathfrak{l}}}$ on time-scale $\mathfrak{t}_{\mathrm{av},1}$, so we get an improved a priori estimate similar to \eqref{eq:s1}. If we choose $\mathfrak{t}_{\mathrm{av},2}\leq N^{-1}$, the first term in \eqref{eq:s4} certainly vanishes in the large-$N$ limit, as $\mathfrak{t}_{\mathrm{av},1}\gg N^{-2}$. On the other hand, a simple calculation implies that the second term in \eqref{eq:s4} vanishes in the large-$N$ limit if we choose $\rho=\e/999$, where we recall $\e$ is defined via $\mathfrak{l}_{\mathrm{av}}=N^{\e}$. Thus, we have succesfully replaced the time-average on time-scale $\mathfrak{t}_{\mathrm{av},1}$ of the spatial average of ${\mathsf{R}_{\log_{N}\mathfrak{l}}}$ with its time-average on the time-scale $\mathfrak{t}_{\mathrm{av},2}\gg\mathfrak{t}_{\mathrm{av},1}$. As the double time-average is basically an average on the larger time-scale, in this step we have basically replaced the time-scale $\mathfrak{t}_{\mathrm{av},1}$ by $\mathfrak{t}_{\mathrm{av},2}\gg\mathfrak{t}_{\mathrm{av},1}$.
\item We then iteratively boost the time-scale by $N^{\rho}$ until we hit the maximal time-scale. This strongly resembles the renormalization procedure discussed after \eqref{eq:s1} and in {Lemma 2} in \cite{GJ15} but in the time-direction and not the spatial-direction. In particular, it is key that each replacement of time-scale increases a priori estimates for ${\mathsf{R}_{\log_{N}\mathfrak{l}}}$.
\end{itemize}
\subsubsection{Non-Equilibrium Calculations}
The previous heuristics are justifiable if the model is at an invariant measure. In general, we will reduce estimates to invariant measure calculations by virtue of the local equilibrium method in \cite{GPV}, namely the one-block and two-blocks estimates, which basically suggest that statistics for our large-scale system are very close to some invariant measure at mesoscopic scales. In particular, we employ the following strategy that will later be made quantitatively precise.
\begin{itemize}
\item The local equilibrium method in \cite{GPV} is based on entropy-Dirichlet form duality and therefore highly robust under perturbations \cite{YauRE}, unlike the approach of Chang-Yau \cite{CYau} via the global invariant measure/eigenvalue problem. It implies the Dirichlet form of the system is very small on mesoscopic blocks. By the log-Sobolev inequality of \cite{Yau}, the same is true for relative entropy.
\item By the relative entropy inequality, we may try to reduce calculations to those at invariant measures. However, relative entropy estimates from the previous bullet point will not be good enough to perform any direct comparison to invariant measures for the purposes of proving {Heuristic} \ref{plemma:s3}; this is because local equilibrium reduction by the entropy inequality on larger scales needs sharper large-deviations bounds for terms we are trying to reduce to equilibrium. Thus, we see a competition between deterioration in local equilibrium reduction in replacing ${\mathsf{E}_{\log_{N}\mathfrak{l}}}$ in {Heuristic} \ref{plemma:s3} by itself on progressively larger scales, versus improving bounds for ${\mathsf{E}_{\log_{N}\mathfrak{l}}}$ on progressively larger scales. As before, this competition sufficiently cancels.
\end{itemize}
We conclude with the following outline for the paper in view of this strategy discussion/this entire section.
\begin{itemize}
\item In Section \ref{section:BGI}, we present the main technical ingredient of this paper, the non-stationary first-order Boltzmann-Gibbs principle. This is a quantitative version of {Heuristic} \ref{pprop:s1}. We state three ingredients for its proof, the first two of which give a quantitative version of {Heuristic} \ref{plemma:s3} and the last of which is the step used to prove {Heuristic} \ref{pprop:s1} \emph{assuming} {Heuristic} \ref{plemma:s2}, namely the inverse length-scale bound for ${\mathsf{E}_{\log_{N}\mathfrak{l}}}$ terms. We give only the relatively short proof of the last of these three ingredients in the next section; we defer technically involved proofs of the first two ingredients to the last part of the paper before the appendix.
\item In Section \ref{section:BGII}, we state and prove a second weaker version of the Boltzmann-Gibbs principle that controls gradients of the heat operator in {Heuristic} \ref{pprop:s1}. This will be used in order to prove the space-time regularity estimates defining the stopping time $\mathfrak{t}_{\mathrm{st}}$ in Definition \ref{definition:KPZ1} are \emph{self-propagating}. This is the goal of Section \ref{section:reg}, which we carry out by estimating space-time regularity of each term in the $\mathbf{U}^{N}$ equation individually using a moment calculation exactly like in the proof of Proposition 3.2 in \cite{DT}, for example, except we will require one application of the aforementioned second/weaker Boltzmann-Gibbs principle to control the gradient of the order $N^{1/2}$ heat operator term in Corollary \ref{corollary:mSHE+}.
\item In Section \ref{section:KPZ1011}, we combine the estimates in Sections \ref{section:BGI}, \ref{section:BGII}, and \ref{section:reg} with a priori $\mathbf{Q}^{N}$ estimates, which are standard to prove, to show Proposition \ref{prop:KPZ10} and Proposition \ref{prop:KPZ11}. 
\item {For the sake of clarity, we shortly reintroduce $\mathsf{E}^{\mathrm{can}}$ and $\mathsf{R}$ notation from this subsection (as well as a few additional and related constructions) more systematically.}
\end{itemize}
%
%
%
\section{Boltzmann-Gibbs Principle I -- Statement} \label{section:BGI}
The main result of this section is the Boltzmann-Gibbs principle. This allows us to access corrections to the $\mathfrak{q}$-term in Definition \ref{definition:mSHE+1}, or equivalently after spatial translation, the $\wt{\mathfrak{q}}$ functional therein, beyond its hydrodynamic limit.
\begin{theorem}\label{theorem:BGI}
\fsp In what follows, let $\E$ be expectation with respect to the law of the $\mathbf{h}_{T,\cdot}^{N}$ and $\eta_{t,\cdot}$ processes with stable initial data. There exists a universal constant $\beta_{\mathrm{BG}}>0$ independent of $\e_{\mathrm{RN}}>0$ so that with universal implied constant,
\begin{align}
\E\|\mathbf{H}_{T,x}^{N}(N^{1/2}\bar{\mathfrak{q}}\mathbf{Y}^{N})\|_{1;\mathbb{T}_{N}} \ \lesssim \ N^{-\beta_{\mathrm{BG}}}+N^{-\frac{99}{100}\e_{\mathrm{RN}}+10\e_{\mathrm{ap}}}. \label{eq:BGI}
\end{align}
\end{theorem}
\begin{remark}\label{remark:BGI2}
\fsp The Boltzmann-Gibbs principle, for example in \cite{BR,CYau}, is usually stated in a much weaker form, namely pointwise in space-time rather than {in} a uniform space-time norm {as} in Theorem \ref{theorem:BGI}. But such an estimate is not well-suited for \emph{norms}.
\end{remark}
\begin{remark}
\fsp The estimate \eqref{eq:BGI} holds if we change $\bar{\mathfrak{q}}$ by replacing $\mathfrak{q}$ in its definition (see Definition \ref{definition:mSHE+1}) with any local functional supported to the left of $0$, say $\mathfrak{y}$. By local, although we always use it in this paper to mean uniformly bounded support, we can actually allow for the support of this ``new" functional $\mathfrak{y}$ that replaces $\mathfrak{q}$ to grow with $N$; the RHS of \eqref{eq:BGI} for $\mathfrak{y}$ in place of $\mathfrak{q}$ would then have a factor that grows as the $100$-th power, for example, of the support length of $\mathfrak{y}$.
\end{remark}
The Boltzmann-Gibbs principle \emph{for sufficiently well-behaved stationary models} is generally accessible by \emph{one} application of the one-block estimate of \cite{GPV} and Sobolev inequalities, which hold generally exclusively for stationary models; see Chapter 11 in \cite{KL}. Like in \cite{CYau}, however, for non-stationary particle systems we require a multiscale idea, {and in this paper we will adopt the multiscale analysis in \cite{GJ15,SX}} that was actually originally implemented for stationary particle systems to prove a refinement of the Boltzmann-Gibbs principle, though our implementation is different than that in \cite{GJ15} due to the non-stationary nature of models considered herein. We set up such a multiscale analysis in the following constructions, which effectively outline a procedure of \emph{local} equilibrium on small mesoscopic blocks and a renormalization scheme that bootstraps equilibrium on smaller mesoscopic blocks to equilibrium on progressively larger mesoscopic blocks; see our discussion of {Heuristic} \ref{plemma:s3}. First, we must introduce key probability/invariant measures.
\begin{definition}\label{definition:ensembles}
\fsp Consider any subset $\mathbb{I}\subseteq\mathbb{T}_{N}$ and any $\sigma\in\R$. We define the canonical measure $\mu_{\sigma,\mathbb{I}}^{\mathrm{can}}$ to be the uniform measure on the set of $\eta\in\Omega_{\mathbb{I}}$ for which the $\eta$-average on $\mathbb{I}$ is equal to $\sigma$. Define the grand-canonical measure $\mu_{\sigma,\mathbb{I}}$ as the product Bernoulli measure on $\Omega_{\mathbb{I}}$ whose one-dimensional marginals have expectation equal to $\sigma$. These two probability measures are each defined precisely below, and we will also let $\mu_{\sigma}=\mu_{\sigma,\mathbb{T}_{N}}$ denote the grand-canonical ensemble of parameter $\sigma$ on the entire set $\mathbb{T}_{N}$:
\begin{align}
\mu_{\sigma,\mathbb{I}}^{\mathrm{can}} \ \overset{\bullet}= \ \mathrm{Unif}\left(\eta\in\Omega_{\mathbb{I}}: \ \wt{\sum}_{x\in\mathbb{I}}\eta_{x}=\sigma\right) \quad \mathrm{and} \quad \mu_{\sigma,\mathbb{I}} \ \overset{\bullet}= \ \bigotimes_{x\in\mathbb{I}}\left(\frac{1+\sigma}{2}\mathbf{1}_{\eta_{x}=1} + \frac{1-\sigma}{2}\mathbf{1}_{\eta_{x}=-1}\right)
\end{align}
For clarity, we mention that the canonical ensemble of parameter $\sigma$ on any subset $\mathbb{I}\subseteq\mathbb{T}_{N}$ is the measure obtained upon taking \emph{any} grand-canonical ensemble on $\mathbb{I}$ and conditioning on the support of the canonical measure/hyperplane with $\eta$-average on $\mathbb{I}$ equal to $\sigma$. Moreover, the projection/pushforward of this canonical ensemble onto any subset $\mathbb{I}'\subseteq\mathbb{I}$ is a convex combination of canonical measures on $\mathbb{I}'$; the coefficient in such a convex combination that corresponds to the canonical measure with parameter $\sigma'$ on $\mathbb{I}'$ is the probability of this $\sigma'$-hyperplane in $\Omega_{\mathbb{I}'}$ under the $\sigma$-canonical measure on $\mathbb{I}$. Lastly, when taking the expectation of any functional $\mathfrak{f}$ with respect to a grand-canonical measure, we make take this grand-canonical measure on any neighborhood of the support of $\mathfrak{f}$, as marginals are jointly independent under grand-canonical measures.
\end{definition}
\begin{definition}\label{definition:BGI2}
\fsp Below, we take $\e_{1},\e_{\mathrm{RN},1}>0$ arbitrarily small but universal and thus uniformly bounded from below.
\begin{itemize}
\item We establish two notations for the following empirical $\eta$-density at time $S\geq0$ in a neighborhood of $y\in\mathbb{T}_{N}$ of length $N^{\e_{1}}$. We use the $\sigma$-notation when we think of the following as a parameter for canonical and grand-canonical ensembles/measures in Definition \ref{definition:ensembles}, and we use the latter $\mathsf{A}$-notation when we think of it as an ``averaging operator" functional on $\Omega$:
\begin{align}
\sigma_{\e_{1},S,y} \ \overset{\bullet}= \ {\mathsf{A}^{\mathbf{X}}_{\e_{1},y}(\eta_{S})} \ \overset{\bullet}= \ \wt{\sum}_{0\leq w\leq N^{\e_{1}}}\eta_{S,y-w}.
\end{align}
\item Define the following conditional expectation of the $\bar{\mathfrak{q}}$ functional viewed as a function of $\sigma_{\e_{1},S,y}$ or $\eta_{S,\cdot}$ for $\cdot\in y-\llbracket0,N^{\e_{1}}\rrbracket$. This conditional expectation is expectation of {$\bar{\mathfrak{q}}_{S,y}$} with respect to the canonical measure of parameter $\sigma_{\e_{1},S,y}$ defined immediately above. We additionally define another expectation operator of $\bar{\mathfrak{q}}_{0,0}$ but now with respect to a \emph{grand-canonical} measure corresponding to the same $\eta$-density/profile $\sigma_{\e_{1},S,y}$ defined immediately above:
\begin{align}
{\mathsf{E}_{\e_{1}}^{\mathrm{can}}(\tau_{y}\eta_{S})} \ \overset{\bullet}= \ {\E_{0}}\left(\bar{\mathfrak{q}}_{S,y}\middle|{\mathsf{A}^{\mathbf{X}}_{\e_{1},y}(\eta_{S})}\right) \quad \mathrm{and} \quad {\mathsf{E}_{\e_{1}}^{\mathrm{gc}}(\tau_{y}\eta_{S})} \ \overset{\bullet}= \ {\E_{\sigma_{\e_{1},S,y}}}\bar{\mathfrak{q}}_{0,0}.
\end{align}
We emphasize that the support of $\bar{\mathfrak{q}}_{S,y}$ is contained strictly in $y-\llbracket0,N^{\e_{1}}\rrbracket$ for any $S\geq0$ and $y\in\mathbb{T}_{N}$, which we emphasize is the support of {$\mathsf{A}^{\mathbf{X}}_{\e_{1},y}(\eta_{S})$}. {More generally, given any functional $\mathfrak{f}:\Omega\to\R$ with support strictly contained in $y-\llbracket0,N^{\e_{1}}\rrbracket$, we let $\mathsf{E}^{\mathrm{can}}_{\e_{1}}(\tau_{y}\eta_{S};\mathfrak{f})$ be as above but replacing $\bar{\mathfrak{q}}_{S,y}$ by $\mathfrak{f}$.} We now define the difference between $\bar{\mathfrak{q}}$ and its $N^{\e_{1}}$-local expectation:
\begin{align}
{\mathsf{S}_{\e_{1}}(\tau_{y}\eta_{S})} \ \overset{\bullet}= \ \bar{\mathfrak{q}}_{S,y}-{\mathsf{E}_{\e_{1}}^{\mathrm{can}}(\tau_{y}\eta_{S})}.
\end{align}
\item Observe now that the previous constructions extend from our predetermined choice of $\e_{1}>0$ to any $\e_{1}\geq0$. With this, we conclude this construction with a renormalization/transfer-of-scales operator for any $\delta\geq0$:
\begin{align}
{\mathsf{R}_{\delta}(\tau_{y}\eta_{S}) \ \overset{\bullet}= \ \mathsf{E}_{\delta}^{\mathrm{can}}(\tau_{y}\eta_{S}) - \mathsf{E}^{\mathrm{can}}_{\delta+\e_{\mathrm{RN},1}}(\tau_{y}\eta_{S})}. \label{eq:defBGI24}
\end{align}
\item {We emphasize that the constructions in the above bullet points are functionals $\Omega\to\R$ evaluated at (shifts of) $\eta_{S}$. In particular, they make sense upon plugging in any $\eta$ instead of (shifts of) $\eta_{S}$.}
\end{itemize}
\end{definition}
We explain the proof of Theorem \ref{theorem:BGI}; even though we did so in the previous section, for clarity we present it with the above notation. The key is to replace $\bar{\mathfrak{q}}$ in \eqref{eq:BGI} by its $\mathsf{E}^{\mathrm{can}}$-expectation on the length-scale $N^{1/2+\beta'}$, in which $\beta'>0$ is universal. The motivation behind such replacement is the following observation. The functional $\bar{\mathfrak{q}}$ vanishes in $\E_{0}$ expectation, and because the global $\eta$ density is roughly 0, the fluctuations {$\E_{0}\bar{\mathfrak{q}}-\mathsf{E}^{\mathrm{can}}(\tau_{y}\eta_{S})$} at length-scale $\mathfrak{l}$ are at most order $\mathfrak{l}^{-1/2}$ by central limit theorem, for example. Taking $\mathfrak{l}=N^{1/2+\beta'}$ does not allow scale-$\mathfrak{l}$ expectation {$\mathsf{E}^{\mathrm{can}}(\tau_{y}\eta_{S})$} to beat the $N^{1/2}$ factor on the LHS of \eqref{eq:BGI}. But $\E_{0}\bar{\mathfrak{q}}=0$ requires only the correction $\E_{0}\wt{\mathfrak{q}}$ in Definition \ref{definition:mSHE+1}. The purpose of the additional linear correction, in a technical sense, is to actually cancel the leading order behavior of the scale-$\mathfrak{l}$ expectation of $\bar{\mathfrak{q}}$, so that, according to {Proposition 8} of \cite{GJ15}, the fluctuations at length-scale $\mathfrak{l}$ are order at most $\mathfrak{l}^{-1}$. Thus, our choice of $\mathfrak{l}$ beats $N^{1/2}$ because of the extra exponent $\beta'$. We note showing that the $\eta$-density is roughly 0 in the stationary case is easy; in the non-stationary case, we need regularity of $\mathbf{Y}^{N}$.

Let us now explain how the replacement of $\bar{\mathfrak{q}}$ in \eqref{eq:BGI} by its $\mathsf{E}^{\mathrm{can}}$-expectation on length $N^{1/2+\beta'}$ will be justified. As suggested by the constructions in Definition \ref{definition:BGI2}, we will first replace $\bar{\mathfrak{q}}$ with its $\mathsf{E}^{\mathrm{can}}$-expectation at the length-scale $N^{\e_{1}}$ with $\e_{1}>0$ from Definition \ref{definition:BGI2} sufficiently small though universal. The error in this first replacement step is the heat operator acting on $N^{1/2}\mathbf{Y}^{N}$ times the difference {$\mathsf{S}_{\e_{1}}(\tau_{y}\eta_{S})$} from Definition \ref{definition:BGI2}, which is a fluctuating factor with small support with size of order $N^{\e_{1}}$. We will then estimate this fluctuating factor using basically the methods of \cite{Y}; as noted in Section \ref{subsection:s}, this roughly amounts to averaging out \emph{in time} these fluctuations, applying the Kipnis-Varadhan inequality (see Appendix 1.6 in \cite{KL}) at stationarity, and then performing reduction to stationarity by a ``local equilibrium" estimate via the entropy inequality.

We now replaced $\bar{\mathfrak{q}}$ in \eqref{eq:BGI} with its $\mathsf{E}^{\mathrm{can}}$-expectation with respect to the small mesoscopic length-scale $N^{\e_{1}}$. The next step is to replace this $\mathsf{E}^{\mathrm{can}}$-expectation with another $\mathsf{E}^{\mathrm{can}}$-expectation but on the slightly larger mesoscopic length-scale $N^{\e_{1}+\e_{\mathrm{RN},1}}$ where $\e_{\mathrm{RN},1}$ in Definition \ref{definition:BGI2} is arbitrarily small but universal. As noted at the end of Section \ref{subsection:s}, we encounter additional obstructions when we try to replace by $\mathsf{E}^{\mathrm{can}}$-expectation on larger length-scales. Indeed, the entropy inequality breaks down when we try to reduce to equilibrium on larger subsets unless we have better a priori estimates for $\mathsf{E}^{\mathrm{can}}$ on larger-scales. This a priori control on $\mathsf{E}^{\mathrm{can}}$-expectations is explained in first paragraph after Definition \ref{definition:BGI2}, and it is enough extra benefit from the initial replacement to then perform a replacement by $\mathsf{E}^{\mathrm{can}}$-expectation on a slightly larger length-scale, so long as $\e_{\mathrm{RN},1}$ is sufficiently smaller than $\e_{1}$, so the jump in length-scales is not too large that the extra benefit in the previous scale-$N^{\e_{1}}$ replacement is not good enough. Ultimately, our analysis remains intact as we increase the length-scale. We then iterate until the desired length-scale $N^{1/2+\beta'}$.

We write three ingredients below for the proof of Theorem \ref{theorem:BGI}, each corresponding to one of the three paragraphs above. The first is initial replacement of $\bar{\mathfrak{q}}$ in \eqref{eq:BGI} with its $\mathsf{E}^{\mathrm{can}}$-expectation on scale $N^{\e_{1}}$ in the second paragraph above. The second is the multiscale ``renormalization" of length-scales from the third paragraph. The last is the inverse-length-scale bound on {$\mathsf{E}^{\mathrm{can}}(\tau_{y}\eta_{S})$}.
\begin{prop}\label{prop:BGI1}
\fsp Take $\e_{1}=1/14$. There exists a universal constant $\beta_{1}>0$, which is again uniformly bounded from below, such that the following holds, in which the $\|\|_{1;\mathbb{T}_{N}}$ norm is with respect to $(T,x)$-variables in the heat operator on the LHS:
\begin{align}
\E\|\mathbf{H}_{T,x}^{N}(N^{1/2}{\mathsf{S}_{\e_{1}}(\tau_{y}\eta_{S})}\mathbf{Y}^{N}_{S,y})\|_{1;\mathbb{T}_{N}} \ \lesssim \ N^{-\beta_{1}}. \label{eq:BGI1Prop}
\end{align}
\end{prop}
\begin{prop}\label{prop:BGI2}
\fsp Suppose $\e_{\mathrm{RN},1}>0$ is sufficiently small but universal depending only on $\e_{1}>0$. Define $\mathfrak{b}_{+}\in\Z_{\geq0}$ to be the last non-negative integer $\mathfrak{b}$ so that $\e_{1}+\mathfrak{b}\e_{\mathrm{RN},1}\leq\frac12+\e_{\mathrm{RN}}$, where $\e_{\mathrm{RN}}>0$ is the universal constant from \emph{Definition \ref{definition:KPZ1}}. There is a universal constant $\beta_{2}>0$, which is therefore uniformly bounded from below, such that the following expectation estimate holds, again in which the $\|\|_{1;\mathbb{T}_{N}}$ norm is with respect to $(T,x)$-variables in the heat operator on the LHS:
\begin{align}
\sup_{\mathfrak{b}=0,\ldots,\mathfrak{b}_{+}-1}\E\|\mathbf{H}_{T,x}^{N}(N^{1/2}{\mathsf{R}_{\e_{1}+\mathfrak{b}\e_{\mathrm{RN},1}}(\tau_{y}\eta_{S})}\mathbf{Y}^{N}_{S,y})\|_{1;\mathbb{T}_{N}} \ \lesssim \ N^{-\beta_{2}}. \label{eq:BGI2Prop}
\end{align}
We also have $\mathfrak{b}_{+}\lesssim_{\e_{1},\e_{\mathrm{RN},1},\e_{\mathrm{RN}}}1$, so the supremum on the LHS of \eqref{eq:BGI2Prop} may be replaced by a sum.
\end{prop}
\begin{prop}\label{prop:BGI3}
\fsp Suppose that $\e_{\mathrm{RN},1}\leq999^{-999}\e_{\mathrm{RN}}$, where $\e_{\mathrm{RN}}>0$ is from \emph{Definition \ref{definition:KPZ1}}. We have the following deterministic estimate, again in which the $\|\|_{1;\mathbb{T}_{N}}$ norm is with respect to $(T,x)$-variables in the heat operator on the LHS:
\begin{align}
\|\mathbf{H}_{T,x}^{N}(N^{1/2}{\mathsf{E}_{\e_{1}+\mathfrak{b}_{+}\e_{\mathrm{RN},1}}^{\mathrm{can}}(\tau_{y}\eta_{S})}\mathbf{Y}_{S,y}^{N})\|_{1;\mathbb{T}_{N}} \ \lesssim \ N^{-\frac{99}{100}\e_{\mathrm{RN}}+10\e_{\mathrm{ap}}}. \label{eq:BGI3I}
\end{align}
\end{prop}
\begin{remark}
Note that $\e_{1}+\mathfrak{b}_{+}\e_{\mathrm{RN},1}\leq\frac12+\e_{\mathrm{RN}}$, so we have a priori regularity estimates for $\mathbf{Y}^{N}$ on the length-scale $N^{\e_{1}+\mathfrak{b}_{+}\e_{\mathrm{RN},1}}$ defining the canonical measure expectation in \eqref{eq:BGI3I}; see Definitions \ref{definition:KPZ1} and \ref{definition:KPZ5} for why this is true.
\end{remark}
\begin{proof}[Proof of \emph{Theorem \ref{theorem:BGI}}]
We have the following tautological decomposition that uses linearity of the heat operator to replace $\bar{\mathfrak{q}}$ by its $\mathsf{E}^{\mathrm{can}}$ on length-scale $N^{\e_{1}+\mathfrak{b}_{+}\e_{\mathrm{RN},1}}$ and then collects the error $\mathsf{S}$:
\begin{align}
\mathbf{H}_{T,x}^{N}(N^{1/2}\bar{\mathfrak{q}}\mathbf{Y}^{N}) \ = \ \mathbf{H}_{T,x}^{N}(N^{1/2}{\mathsf{E}_{\e_{1}+\mathfrak{b}_{+}\e_{\mathrm{RN},1}}^{\mathrm{can}}(\tau_{y}\eta_{S})}\mathbf{Y}_{S,y}^{N}) + \mathbf{H}_{T,x}^{N}(N^{1/2}{\mathsf{S}_{\e_{1}+\mathfrak{b}_{+}\e_{\mathrm{RN},1}}(\tau_{y}\eta_{S})}\mathbf{Y}_{S,y}^{N}). \label{eq:BGI1}
\end{align}
We proceed with the following multiscale decomposition of the second term on the RHS of \eqref{eq:BGI1} that rewrites the difference $\mathsf{S}$ of $\bar{\mathfrak{q}}$ with $\mathsf{E}^{\mathrm{can}}$ on length-scale $N^{\e_{1}+\mathfrak{b}_{+}\e_{\mathrm{RN},1}}$ in terms of a telescoping sum of the successive differences of $\mathsf{E}^{\mathrm{can}}$ terms on progressively larger length-scales; again, the following is by definition and by linearity of the heat operator:
\begin{align}
\mathbf{H}_{T,x}^{N}(N^{1/2}{\mathsf{S}_{\e_{1}+\mathfrak{b}_{+}\e_{\mathrm{RN},1}}(\tau_{y}\eta_{S})}\mathbf{Y}_{S,y}^{N}) \ &= \ \mathbf{H}_{T,x}^{N}(N^{1/2}{\mathsf{S}_{\e_{1}}(\tau_{y}\eta_{S})}\mathbf{Y}_{S,y}^{N})+\sum_{\mathfrak{b}=0}^{\mathfrak{b}_{+}-1}\mathbf{H}_{T,x}^{N}(N^{1/2}{\mathsf{R}_{\e_{1}+\mathfrak{b}\e_{\mathrm{RN},1}}(\tau_{y}\eta_{S})}\mathbf{Y}^{N}_{S,y}). \label{eq:BGI2}
\end{align}
We plug \eqref{eq:BGI2} into the second term on the RHS of \eqref{eq:BGI1}. We then take $\|\|_{1;\mathbb{T}_{N}}$ norms of both sides of the resulting identity, employ the triangle inequality for $\|\|_{1;\mathbb{T}_{N}}$, take expectations, and apply Proposition \ref{prop:BGI1}, Proposition \ref{prop:BGI2}, and Proposition \ref{prop:BGI3}.
\end{proof}
We defer the proofs of Proposition \ref{prop:BGI1} and Proposition \ref{prop:BGI2} to the last non-appendix sections because of their complexity.
\subsection{Proof of Proposition \ref{prop:BGI3}}
The only preliminary ingredient we need for the current argument is the following estimate for which we employ crucially the a priori space-time regularity estimates in $\mathbf{Y}^{N}$. Its proof is relatively quick; it is an idea used in \cite{DT} in the proof of the hydrodynamic limit estimate of Lemma 2.5 therein where $\eta$-variables are realized as $\mathbf{h}^{N}$ gradients.
\begin{lemma}\label{lemma:BGI31}
\fsp Suppose the inequalities for $\e_{1}$ and $\e_{\mathrm{RN},1}$ and $\e_{\mathrm{RN}}$ and $\e_{\mathrm{ap}}$ in \emph{Definition \ref{definition:KPZ1}} and \emph{Proposition \ref{prop:BGI3}} hold. Then {we have the following deterministic estimates:}
\begin{align}
\||{\mathsf{A}^{\mathbf{X}}_{\e_{1}+\mathfrak{b}_{+}\e_{\mathrm{RN},1},x}(\eta_{T})}|^{2}|\mathbf{Y}_{T,x}^{N}|\|_{1;\mathbb{T}_{N}} \ \lesssim \ N^{\e_{\mathrm{ap}}}\||{\mathsf{A}^{\mathbf{X}}_{\e_{1}+\mathfrak{b}_{+}\e_{\mathrm{RN},1},x}(\eta_{T})}|^{2}\mathbf{1}(T\leq\mathfrak{t}_{\mathrm{st}})\|_{1;\mathbb{T}_{N}} \ \lesssim \ N^{-\frac12-\frac{99}{100}\e_{\mathrm{RN}}+10\e_{\mathrm{ap}}}. \label{eq:BGI31}
\end{align}
\end{lemma}
\begin{proof}
The first estimate in \eqref{eq:BGI31} is immediate by definition of $\mathbf{Y}^{N}$ in Definition \ref{definition:KPZ5}. Indeed, it suffices to look just at times $T\leq\mathfrak{t}_{\mathrm{st}}$ because afterwards, we have $\mathbf{Y}^{N}=0$. Similarly, until the stopping time $\mathfrak{t}_{\mathrm{st}}$ we have $\mathbf{Y}^{N}=\mathbf{Z}^{N}$, where $\mathbf{Z}^{N}$ is uniformly bounded by $N^{\e_{\mathrm{ap}}}$ times uniformly bounded factors. Thus we are left with proving the second bound in \eqref{eq:BGI31}. Note the following that relates the $\mathsf{A}^{\mathbf{X}}$ term to $\mathbf{h}^{N}$, whose proof follows by {$\eta_{T,x}=N^{1/2}(\mathbf{h}_{T,x}^{N}-\mathbf{h}_{T,x-1}^{N})$} and in which we set $\wt{\e}_{1}={\e_{1}+\mathfrak{b}_{+}\e_{\mathrm{RN},1}}$:
\begin{align}
{\mathsf{A}^{\mathbf{X}}_{\wt{\e}_{1},x}(\eta_{T})} \ = \ \wt{\sum}_{0\leq w\leq N^{\wt{\e}_{1}}}\eta_{T,x-w} \ = \ N^{\frac12}(1+N^{\wt{\e}_{1}})^{-1}\grad_{-N^{\wt{\e}_{1}}-1}^{\mathbf{X}}\log\mathbf{Z}_{T,x}^{N}. \label{eq:BGI311}
\end{align}
We refer to the proof of Lemma 2.5 in \cite{DT} for a similar identity in which $N^{\wt{\e}_{1}}$ is instead a small multiple of $N^{1/2}$. We now employ elementary calculus for the logarithm to establish the following estimate for the far RHS of \eqref{eq:BGI311}. Roughly speaking, because the derivative of the logarithm is bad at 0 and is otherwise uniformly smooth, the gradient on the far RHS of \eqref{eq:BGI311} may be controlled by the same gradient but of $\mathbf{Z}^{N}$, then times the space-time supremum of $(\mathbf{Z}^{N})^{-1}$. Extending \eqref{eq:BGI311} this way,
\begin{align}
|{\mathsf{A}^{\mathbf{X}}_{\wt{\e}_{1},x}(\eta_{T})}|^{2}\mathbf{1}(T\leq\mathfrak{t}_{\mathrm{st}}) \ \lesssim \ N(1+N^{\wt{\e}_{1}})^{-2}\|(\mathbf{Z}^{N})^{-1}\|_{\mathfrak{t}_{\mathrm{st}};\mathbb{T}_{N}}^{2}\||\grad^{\mathbf{X}}_{-N^{\wt{\e}_{1}}-1}\mathbf{Z}^{N}\|_{\mathfrak{t}_{\mathrm{st}};\mathbb{T}_{N}}^{2}. \label{eq:BGI312}
\end{align}
Observe that the space-time norms on the RHS of \eqref{eq:BGI312} are a space-time supremum until the stopping time $\mathfrak{t}_{\mathrm{st}}$. Until this stopping time, we have a uniform upper bound for the first norm on the RHS of \eqref{eq:BGI312} of $N^{2\e_{\mathrm{ap}}}$ by definition. Similarly, because we have assumed the inequality $N^{\wt{\e}_{1}}\leq\mathfrak{l}_{N}$ by construction in Proposition \ref{prop:BGI2}, where $\wt{\e}_{1}=\e_{1}+\mathfrak{b}_{+}\e_{\mathrm{RN},1}$ is from Proposition \ref{prop:BGI2} and $\mathfrak{l}_{N}\in\Z_{\geq0}$ is from Definition \ref{definition:KPZ1}, by definition of $\mathfrak{t}_{\mathrm{st}}$ in Definition \ref{definition:KPZ1} we get a priori spatial regularity estimates for $\mathbf{Z}^{N}$, which imply the second norm on the RHS of \eqref{eq:BGI312} is bounded above by {$N^{2\e_{\mathrm{ap}}}N^{-1}(1+N^{\wt{\e}_{1}})(1+\|\mathbf{Z}^{N}\|_{\mathfrak{t}_{\mathrm{st}};\mathbb{T}_{N}})^{4}$}, which may be thought of as $N^{2\e_{\mathrm{ap}}}$ times the square of the spatial Holder regularity estimate of exponent $\frac12$ for $\mathbf{Z}^{N}$. {By Definition \ref{definition:KPZ1}, we also know $\|\mathbf{Z}^{N}\|_{\mathfrak{t}_{\mathrm{st}};\mathbb{T}_{N}}\lesssim N^{\e_{\mathrm{ap}}}$.} Thus, we get via \eqref{eq:BGI312} and this paragraph that
\begin{align}
N^{\e_{\mathrm{ap}}}|{\mathsf{A}^{\mathbf{X}}_{\wt{\e}_{1},x}(\eta_{T})}|^{2}\mathbf{1}(T\leq\mathfrak{t}_{\mathrm{st}}) \ \lesssim \ N^{5\e_{\mathrm{ap}}}(1+N^{\wt{\e}_{1}})^{-1} \ \lesssim \ N^{-\e_{1}-\mathfrak{b}_{+}\e_{\mathrm{RN},1}+{9}\e_{\mathrm{ap}}}. \label{eq:BGI313}
\end{align}
Recall from Proposition \ref{prop:BGI2} that $\mathfrak{b}_{+}$ is the \emph{final} non-negative integer $\mathfrak{b}$ with $\e_{1}+\mathfrak{b}_{+}\e_{\mathrm{RN},1}\leq\frac12+\e_{\mathrm{RN}}$. As $\e_{\mathrm{RN},1}\leq999^{-999}\e_{\mathrm{RN}}$ by our assumption, we obtain the lower bound $\e_{1}+\mathfrak{b}_{+}\e_{\mathrm{RN},1}\geq\frac12+\frac{99}{100}\e_{\mathrm{RN}}$, for example, because if not, then we could increase $\mathfrak{b}_{+}$ by 1 while only adding $999^{-999}\e_{\mathrm{RN}}$, and this would not boost $\frac12+\frac{99}{100}\e_{\mathrm{RN}}$ past $\frac12+\e_{\mathrm{RN}}$. Combining \eqref{eq:BGI313} with this lower bound for $\e_{1}+\mathfrak{b}_{+}\e_{\mathrm{RN},1}$ finishes the proof of the lemma.
\end{proof}
We proceed with proof of Proposition \ref{prop:BGI3}. The first step we take is to replace the canonical measure expectation $\mathsf{E}^{\mathrm{can}}$ in the heat operator on the LHS of \eqref{eq:BGI3I} by a grand-canonical measure expectation $\mathsf{E}^{\mathrm{gc}}$ evaluated at the same $\eta$-density $\sigma_{\wt{\e}_{1},S,y}$ and the same functional $\bar{\mathfrak{q}}$, where we have again employed the notation $\wt{\e}_{1}={\e_{1}+\mathfrak{b}_{+}\e_{\mathrm{RN},1}}$ introduced in the proof of Lemma \ref{lemma:BGI31} just to ease notation. For this, we apply {Proposition 8} in \cite{GJ15} with the choice of function $f=\mathfrak{q}_{0,0}$ and with the choice of length-scale therein to be $\ell=N^{\wt{\e}_{1}}$:
\begin{align}
|{\mathsf{E}_{\wt{\e}_{1}}^{\mathrm{can}}(\tau_{y}\eta_{S})-\mathsf{E}_{\wt{\e}_{1}}^{\mathrm{gc}}(\tau_{y}\eta_{S})}| \ \lesssim \ N^{-\wt{\e}_{1}} \ \lesssim \ N^{-\frac12-\frac{99}{100}\e_{\mathrm{RN}}+10\e_{\mathrm{ap}}}. \label{eq:BGI32}
\end{align}
The last/second inequality in \eqref{eq:BGI32} follows by the same observation that we made in the final paragraph in the proof of Lemma \ref{lemma:BGI31}. If we multiply the LHS by $N^{1/2}\mathbf{Y}^{N}_{S,y}$ and put this in the heat operator, since $|\mathbf{Y}^{N}|\leq N^{\e_{\mathrm{ap}}}$, it is enough to show Proposition \ref{prop:BGI3} but with $\mathsf{E}^{\mathrm{gc}}$ in place of $\mathsf{E}^{\mathrm{can}}$, therefore completing the desired first step/replacement. To control $\mathsf{E}^{\mathrm{gc}}$, let us first recall from Definition \ref{definition:BGI2} that $\mathsf{E}^{\mathrm{gc}}$ is expectation of $\bar{\mathfrak{q}}_{0,0}$ with respect to a grand-canonical ensemble of parameter $\sigma_{\wt{\e}_{1},S,y}$. We will now Taylor expand this function of $\sigma_{\wt{\e}_{1},S,y}$ up to second order around the value $\sigma=0$ and obtain the following estimate:
\begin{align}
{\mathsf{E}_{\wt{\e}_{1}}^{\mathrm{gc}}(\tau_{y}\eta_{S})} \ = \ {\E_{0}}\bar{\mathfrak{q}}_{0,0}+\left(\partial_{\sigma}{\E_{\sigma}}\bar{\mathfrak{q}}_{0,0}\right)|_{\sigma=0}\sigma_{\wt{\e}_{1},S,y} + \mathrm{O}(\sigma_{\wt{\e}_{1},S,y})^{2}. \label{eq:BGI33}
\end{align}
The first term on the RHS of \eqref{eq:BGI33} is easily checked to be 0, as the linear term in $\bar{\mathfrak{q}}$ has expectation 0, and what is left is just $\wt{\mathfrak{q}}_{0,0}$ minus its expectation with respect to ${\E_{0}}$. The key idea is that the second term also vanishes because $\bar{\mathfrak{d}}\partial_{\sigma}{\E_{\sigma}}\eta=\bar{\mathfrak{d}}\partial_{\sigma}\sigma=\bar{\mathfrak{d}}$, and $\bar{\mathfrak{d}}$ is \emph{defined} to equal $\partial_{\sigma}{\E_{\sigma}}\wt{\mathfrak{q}}_{0,0}|_{\sigma=0}$, while the constant expectation of $\bar{\mathfrak{q}}_{0,0}$ certainly vanishes after $\partial_{\sigma}$ differentiation. Let us refer the reader to Definition \ref{definition:mSHE+1} for definitions of all functionals and factors just mentioned. Thus, by this paragraph and \eqref{eq:BGI33}, we are left with proving \eqref{eq:BGI3I} upon replacing $\mathsf{E}^{\mathrm{can}}$ with $\mathsf{E}^{\mathrm{gc}}$ and then replacing $\mathsf{E}^{\mathrm{gc}}$ with the big-Oh term on the RHS of \eqref{eq:BGI33}. That estimate follows by Lemma \ref{lemma:BGI31}, as $\sigma_{\wt{\e}_{1},S,y}={\mathsf{A}^{\mathbf{X}}_{\wt{\e}_{1},y}(\eta_{S})}$ by definition. This completes the proof. \qed
\begin{remark}\label{remark:regremark1}
\fsp If we were to apply our method to environment-dependence in reversible dynamics, it is fairly standard  \cite{BR,CYau,KL} that we would need to prove Theorem \ref{theorem:BGI} but with the spatial gradient of the heat operator on the LHS of \eqref{eq:BGI}. Since gradients of $\mathbf{H}^{N}$ introduce higher-degree short-time singularities of $\mathbf{H}^{N}$, we need to resolve more singular factors during the proof of \eqref{eq:BGI} with this extra gradient. There are ultimately several possible ways to resolve such singularities. For the purposes of computing scaling limits of fluctuations, however, for linear non-KPZ limits of interest in \cite{CYau}, for example, the simplest would be to smooth the short-time behavior of $\mathbf{H}^{N}$ by convolving against a time-1 heat kernel. This would remove the higher-order singularity while only changing this paper by revising the fluctuation scaling limit of main interest to hold only after smoothing, thus with respect to a weaker topology that is the topology used for fluctuation scaling limits in previous literature anyway; see \cite{BR,CYau,JM,KL}. But in the current paper, the singular on-diagonal factors in $\mathbf{H}^{N}$ actually pose no issue in proving convergence in Theorem \ref{theorem:KPZ} in quite a strong sense. This is a concrete example of ``analytic" strength of our method, compatible with PDE ideas to solve $\mathrm{SHE}$.
\end{remark}
%
%
%
\section{Boltzmann-Gibbs Principle II}\label{section:BGII}
The point of this section is a second version of the non-stationary first-order Boltzmann-Gibbs principle. To motivate it, we emphasize the proof of Theorem \ref{theorem:BGI} requires important a priori space-time regularity estimates on $\mathbf{Z}^{N}$ that were engineered into the definition of $\mathbf{Y}^{N}$ via the stopping time $\mathfrak{t}_{\mathrm{st}}$. We will need to establish such a priori space-time regularity estimates in order for $\mathbf{Y}^{N}$ to be a faithful proxy for $\mathbf{Z}^{N}$. It turns out that establishing the important time-regularity estimates is a rather straightforward set of moment estimates for the $\mathbf{Z}^{N}$ equation. However, for technical reasons, this is not true for establishing the required spatial regularity defining $\mathfrak{t}_{\mathrm{st}}$. Indeed, a direct moment bound on spatial regularity of the order $N^{1/2}$ term in the stochastic equation from Corollary \ref{corollary:mSHE+}, without analyzing $\bar{\mathfrak{q}}$ carefully, ends up being much worse than the required spatial regularity estimate in $\mathfrak{t}_{\mathrm{st}}$. In order to resolve such issue, we will need to estimate spatial gradients of the order $N^{1/2}$ term in the stochastic equation from Corollary \ref{corollary:mSHE+} by taking advantage of the fluctuating behavior of the $\bar{\mathfrak{q}}$ function {as we did in} the proof of Theorem \ref{theorem:BGI}. This leads to our second version of the Boltzmann-Gibbs principle in Theorem \ref{theorem:BGII}, which we present after introducing some notation.
\begin{definition}\label{definition:BGII1}
\fsp Consider any $\phi:\mathbb{T}_{N}\to\R$. Define the following \emph{normalized maximal gradient} on the length-scale $\mathfrak{l}_{+}\in\Z_{\geq0}$:
\begin{align}
\wt{\grad}_{\mathfrak{l}_{+}}^{\mathbf{X}}\phi_{x} \ &\overset{\bullet}= \ {\sup}_{1\leq|\mathfrak{l}|\leq\mathfrak{l}_{+}}|\mathfrak{l}|^{-1}|\grad_{\mathfrak{l}}^{\mathbf{X}}\phi_{x}|.
\end{align}
We extend the previous normalized maximal gradient to heat operators in the following fashion in which $\Phi:\R_{\geq0}\times\mathbb{T}_{N}\to\R$:
\begin{align}
|\wt{\grad}_{\mathfrak{l}_{+}}^{\mathbf{X}}|\mathbf{H}_{T,x}^{N,\mathbf{X}}(\Phi_{0,\cdot}) \ \overset{\bullet}= \ {\sum}_{y\in\mathbb{T}_{N}}|\wt{\grad}_{\mathfrak{l}_{+}}^{\mathbf{X}}\mathbf{H}_{0,T,x,y}^{N}||\Phi_{0,y}| \quad \mathrm{and} \quad |\wt{\grad}_{\mathfrak{l}_{+}}^{\mathbf{X}}|\mathbf{H}_{T,x}^{N}(\Phi) \ &\overset{\bullet}= \ \int_{0}^{T}{\sum}_{y\in\mathbb{T}_{N}}|\wt{\grad}_{\mathfrak{l}_{+}}^{\mathbf{X}}\mathbf{H}_{S,T,x,y}^{N}| |\Phi_{S,y}| \d S. \label{eq:BGII12def}
\end{align}
\end{definition}
\begin{remark}\label{remark:BGII2}
\fsp Intuitively, provided any function $\phi:\mathbb{T}_{N}\to\R$ that is ``smooth" on scale $\mathfrak{l}_{+}\in\Z_{\geq0}$, its normalized maximal gradient on this length-scale will be controlled, roughly speaking. The goal for Theorem \ref{theorem:BGII} will be to prove the homogenization estimate in Theorem \ref{theorem:BGI}, or actually a slightly weaker version, holds not just uniform in space-time but at the level of normalized maximal gradients with respect to the length-scale $\mathfrak{l}_{N}$ in Definition \ref{definition:KPZ1} on which we want to get spatial regularity of the Gartner transform. Let us also emphasize that the above extensions of the normalized maximal gradients to the spatial and space-time heat operators are emphatically not the normalized maximal gradients of the heat operators themselves when we view them as functions in their own right. This is because of the absolute value inside the sum and integral in \eqref{eq:BGII12def}.
\end{remark}
\begin{theorem}\label{theorem:BGII}
\fsp There exists a universal constant $\beta>0$, necessarily uniformly bounded below, that is independent of $\e_{\mathrm{RN}}>0$ from \emph{Definition \ref{definition:KPZ1}} such that for the length-scale $\mathfrak{l}_{N}$ in \emph{Definition \ref{definition:KPZ1}}, we have the expectation estimate
\begin{align}
\E\|\wt{\grad}_{\mathfrak{l}_{N}}^{\mathbf{X}}\mathbf{H}_{T,x}^{N}(N^{1/2}\bar{\mathfrak{q}}\mathbf{Y}^{N})\|_{1;\mathbb{T}_{N}} \ \lesssim \ N^{-\frac34-\beta}+N^{-\frac34-99\e_{\mathrm{RN}}}. \label{eq:BGII}
\end{align}
We clarify there are no absolute value bars around the normalized maximal gradient ``operator" on the LHS of \eqref{eq:BGII}.
\end{theorem}
The proof of Theorem \ref{theorem:BGII} is similar to that of Theorem \ref{theorem:BGI} in architecture; it is a mix of probabilistic homogenization estimates along with stochastic regularity estimates built into the $\mathbf{Y}^{N}$ process in the heat operator on the LHS of \eqref{eq:BGII}. However, before we discuss the proof, we briefly explain its utility; this will be explored in detail in Section \ref{section:reg}. Recall our motivation for Theorem \ref{theorem:BGII} is to show a priori spatial regularity in $\mathfrak{t}_{\mathrm{st}}$ ``propagates itself" with high probability, thus $\mathfrak{t}_{\mathrm{st}}=1$ with high probability. Take any $\mathfrak{l}\in\llbracket-\mathfrak{l}_{N},\mathfrak{l}_{N}\rrbracket$ with $\mathfrak{l}_{N}$ in Definition \ref{definition:KPZ1}/Theorem \ref{theorem:BGII}. Theorem \ref{theorem:BGII} gives, with high probability \emph{simultaneously} in $\mathfrak{l}$,
\begin{align}
\|\grad_{\mathfrak{l}}^{\mathbf{X}}\mathbf{H}_{T,x}^{N}(N^{1/2}\bar{\mathfrak{q}}\mathbf{Y}^{N})\|_{1;\mathbb{T}_{N}} \ \lesssim \ N^{-\frac34-\beta}|\mathfrak{l}|+N^{-\frac34-99\e_{\mathrm{RN}}}|\mathfrak{l}| \ \lesssim \ N^{-\frac34-\beta+\frac14+{\frac12\e_{\mathrm{RN}}}}|\mathfrak{l}|^{1/2} + N^{-\frac34-99\e_{\mathrm{RN}}+\frac14+{\frac12\e_{\mathrm{RN}}}}|\mathfrak{l}|^{1/2}. \label{eq:BGIIapp}
\end{align}
The last bound in the above display follows via bounding $|\mathfrak{l}|\lesssim|\mathfrak{l}_{N}|=N^{1/2+\e_{\mathrm{RN}}}$. Because $\beta>0$ in Theorem \ref{theorem:BGII} is uniformly bounded below and independent of $\e_{\mathrm{RN}}$ while $\e_{\mathrm{RN}}$ is arbitrarily small but universal, the far RHS of the previous display is at most $N^{-1/2}|\mathfrak{l}|^{1/2}$, proving the a priori spatial regularity of $\mathbf{Y}^{N}$ at least propagates the same level of spatial regularity of the order $N^{1/2}$ term in the stochastic equation for $\mathbf{U}^{N}$; we employ another argument in Section \ref{section:KPZ1011} via Lemma \ref{lemma:KPZ6} to transfer this to $\mathbf{Z}^{N}$.
\subsection{Proof of Theorem \ref{theorem:BGII}}
Take any $(T,x)\in[0,1]\times\mathbb{T}_{N}$ and define $T_{N}=T-N^{-1/2-999\e_{\mathrm{RN}}}$. The triangle inequality gives
\begin{align}
\|\wt{\grad}_{\mathfrak{l}_{N}}^{\mathbf{X}}\mathbf{H}_{T,x}^{N}(N^{\frac12}\bar{\mathfrak{q}}_{S,y}\mathbf{Y}^{N}_{S,y})\|_{1;\mathbb{T}_{N}} \ \leq \ \|\wt{\grad}_{\mathfrak{l}_{N}}^{\mathbf{X}}\mathbf{H}_{T,x}^{N}(N^{\frac12}\bar{\mathfrak{q}}_{S,y}\mathbf{Y}^{N}_{S,y}\mathbf{1}_{S\leq T_{N}})\|_{1;\mathbb{T}_{N}} + \|\wt{\grad}_{\mathfrak{l}_{N}}^{\mathbf{X}}\mathbf{H}_{T,x}^{N}(N^{\frac12}\bar{\mathfrak{q}}_{S,y}\mathbf{Y}^{N}_{S,y}\mathbf{1}_{S\geq T_{N}})\|_{1;\mathbb{T}_{N}}. \label{eq:BGII0}
\end{align}
The second term on the RHS of \eqref{eq:BGII0} is estimated deterministically. The first will be estimated basically via Theorem \ref{theorem:BGI}.
\begin{lemma}\label{lemma:BGII1}
\fsp We have the following estimate for the length-scale $\mathfrak{l}_{N}$ in \emph{Definition \ref{definition:KPZ1}}:
\begin{align}
\|\wt{\grad}_{\mathfrak{l}_{N}}^{\mathbf{X}}\mathbf{H}_{T,x}^{N}(N^{\frac12}\bar{\mathfrak{q}}_{S,y}\mathbf{Y}^{N}_{S,y}\mathbf{1}_{S\geq T_{N}})\|_{1;\mathbb{T}_{N}} \ \lesssim \ N^{-\frac34-\frac{999}{2}\e_{\mathrm{RN}}+\e_{\mathrm{ap}}}. \label{eq:BGII1main}
\end{align}
\end{lemma}
\begin{lemma}\label{lemma:BGII2}
\fsp There exists a universal constant $\beta>0$ such that for $\mathfrak{l}_{N}$ in \emph{Definition \ref{definition:KPZ1}}, we have
\begin{align}
\E\|\wt{\grad}_{\mathfrak{l}_{N}}^{\mathbf{X}}\mathbf{H}_{T,x}^{N}(N^{\frac12}\bar{\mathfrak{q}}_{S,y}\mathbf{Y}^{N}_{S,y}\mathbf{1}_{S\leq T_{N}})\|_{1;\mathbb{T}_{N}} \ \lesssim \ N^{-\frac34-\beta}. \label{eq:BGII2main}
\end{align}
\end{lemma}
Clearly, the triangle inequality \eqref{eq:BGII0} combined with Lemmas \ref{lemma:BGII1} and \ref{lemma:BGII2} implies Theorem \ref{theorem:BGII} as $\e_{\mathrm{ap}}\leq999^{-999}\e_{\mathrm{RN}}$. Lemma \ref{lemma:BGII1} will be a straightforward consequence of heat estimates in Proposition \ref{prop:heat}. Lemma \ref{lemma:BGII2} will be proved more delicately:
\begin{itemize}
\item We will replace $\bar{\mathfrak{q}}$ on the LHS of \eqref{eq:BGII2main} with $\mathsf{E}^{\mathrm{can}}_{\e_{1}+\mathfrak{b}_{+}\e_{\mathrm{RN},1}}$ from Proposition \ref{prop:BGI3}. However, we cannot directly cite Proposition \ref{prop:BGI1} and Proposition \ref{prop:BGI2} for this because of the normalized maximal gradient on the LHS of \eqref{eq:BGII2main} that is absent from Proposition \ref{prop:BGI1} and Proposition \ref{prop:BGI2}. This will require gymnastics with heat operators that we demonstrate when we give a precise proof.
\item Having made the previous replacement, we observe the proof of Proposition \ref{prop:BGI3} is done through Lemma \ref{lemma:BGI31}, which provides a deterministic estimate for $\mathsf{E}^{\mathrm{can}}_{\e_{1}+\mathfrak{b}_{+}\e_{\mathrm{RN},1}}$. Thus, we will have the gradient of the heat operator acting on a small function; this is small by the heat estimates in Proposition \ref{prop:heat}.
\end{itemize}
\subsubsection{Proof of \emph{Lemma \ref{lemma:BGII1}}}
Recall $\bar{\mathfrak{q}}\lesssim1$ and $|\mathbf{Y}^{N}|\leq N^{\e_{\mathrm{ap}}}$ by construction in Definitions \ref{definition:KPZ1} and \ref{definition:KPZ5}. Therefore, we have the following straightforward bound by controlling an integral/sum by replacing the integrand/summand with its absolute value:
\begin{align}
\|\wt{\grad}_{\mathfrak{l}_{N}}^{\mathbf{X}}\mathbf{H}_{T,x}^{N}(N^{\frac12}\bar{\mathfrak{q}}_{S,y}\mathbf{Y}^{N}_{S,y}\mathbf{1}_{S\geq T_{N}})\|_{1;\mathbb{T}_{N}} \ \lesssim \ N^{\frac12+\e_{\mathrm{ap}}}\||\wt{\grad}_{\mathfrak{l}_{N}}^{\mathbf{X}}|\mathbf{H}_{T,x}^{N}(\mathbf{1}_{S\geq T_{N}})\|_{1;\mathbb{T}_{N}}. \label{eq:BGII11}
\end{align}
It suffices to note the $\|\|_{1;\mathbb{T}_{N}}$-norm on the RHS of \eqref{eq:BGII11} is bounded by $N^{-1-1/4-999\e_{\mathrm{RN}}/2}$ since the heat operator is smooth on the macroscopic length-scale $N$, providing the factor of $N^{-1}$; see \eqref{eq:heatIV} in Proposition \ref{prop:heat}. We emphasize the short-time-integral in the heat operator coming from the indicator function of the length $N^{-1/2-999\e_{\mathrm{RN}}}$-interval given by $S\geq T_{N}$ above. \qed
\subsubsection{Proof of \emph{Lemma \ref{lemma:BGII2}}}
We will first employ the following triangle inequality, recalling notation from Proposition \ref{prop:BGI3}:
\begin{align}
\E\|\wt{\grad}_{\mathfrak{l}_{N}}^{\mathbf{X}}\mathbf{H}_{T,x}^{N}(N^{\frac12}\bar{\mathfrak{q}}_{S,y}\mathbf{Y}^{N}_{S,y}\mathbf{1}_{S\leq T_{N}})\|_{1;\mathbb{T}_{N}} \ &\leq \ \E\|\wt{\grad}_{\mathfrak{l}_{N}}^{\mathbf{X}}\mathbf{H}_{T,x}^{N}\left(N^{\frac12}(\bar{\mathfrak{q}}_{S,y}-{\mathsf{E}^{\mathrm{can}}_{\e_{1}+\mathfrak{b}_{+}\e_{\mathrm{RN},1}}(\tau_{y}\eta_{S}))}\mathbf{Y}^{N}_{S,y}\mathbf{1}_{S\leq T_{N}}\right)\|_{1;\mathbb{T}_{N}} \label{eq:BGII21a} \\
&+ \E\|\wt{\grad}_{\mathfrak{l}_{N}}^{\mathbf{X}}\mathbf{H}_{T,x}^{N}(N^{\frac12}{\mathsf{E}^{\mathrm{can}}_{\e_{1}+\mathfrak{b}_{+}\e_{\mathrm{RN},1}}(\tau_{y}\eta_{S})}\mathbf{Y}^{N}_{S,y}\mathbf{1}_{S\leq T_{N}})\|_{1;\mathbb{T}_{N}}. \label{eq:BGII21b}
\end{align}
Following the proof of Proposition \ref{prop:BGI3}, note the $\mathsf{E}^{\mathrm{can}}\mathbf{Y}^{N}$ term in \eqref{eq:BGII21b} is at most $N^{5\e_{\mathrm{ap}}}N^{-1/2-99\e_{\mathrm{RN}}/100}$ \emph{deterministically}. Thus, we get the following where we again use \eqref{eq:heatIV} in Proposition \ref{prop:heat} to get the last bound below as we did in the proof of Lemma \ref{lemma:BGII1}, while to get the first bound we also drop the time-set indicator function in the heat operator after replacing everything in the heat operator by its absolute value, including the heat kernel gradient, which is okay for the sake of an upper bound:
\begin{align}
\|\wt{\grad}_{\mathfrak{l}_{N}}^{\mathbf{X}}\mathbf{H}_{T,x}^{N}(N^{\frac12}{\mathsf{E}^{\mathrm{can}}_{\e_{1}+\mathfrak{b}_{+}\e_{\mathrm{RN},1}}(\tau_{y}\eta_{S})}\mathbf{Y}^{N}_{S,y}\mathbf{1}_{S\leq T_{N}})\|_{1;\mathbb{T}_{N}} \ \lesssim \ N^{-\frac{99}{100}\e_{\mathrm{RN}}+5\e_{\mathrm{ap}}}\||\wt{\grad}_{\mathfrak{l}_{N}}^{\mathbf{X}}|\mathbf{H}_{T,x}^{N}(1)\|_{1;\mathbb{T}_{N}} \ \lesssim \ N^{-1-\frac{99}{100}\e_{\mathrm{RN}}+5\e_{\mathrm{ap}}}. \nonumber
\end{align}
Because $\e_{\mathrm{ap}}\leq999^{-999}\e_{\mathrm{RN}}$, the above display shows the contribution of \eqref{eq:BGII21b} is certainly controlled by the RHS of the proposed estimate \eqref{eq:BGII2main}. Thus, it suffices to prove the same about the RHS of \eqref{eq:BGII21a}. To this end, we consider the following.
\begin{itemize}
\item For any $\phi:\R_{\geq0}\times\mathbb{T}_{N}\to\R$, we have the following identity by Proposition \ref{prop:heat} in which $\mathrm{t}_{(N)}=T-T_{N}=N^{-1/2-999\e_{\mathrm{RN}}}$; below, on the RHS, the outer spatial heat operator sums over $w\in\mathbb{T}_{N}$, and the inner space-time heat operator integrates/sums over space-time variables $(S,y)$:
\begin{align}
\mathbf{H}_{T,x}^{N}(\phi_{S,y}\mathbf{1}_{S\leq T_{N}}) \ = \ \mathbf{H}^{N,\mathbf{X}}_{\mathrm{t}_{(N)},x}\left(\mathbf{H}^{N}_{T_{N},w}(\phi_{S,y})\right). 
\end{align}
Taking gradients/normalized maximal gradients, from the above identity we establish the following estimate via the following reasoning. Let the normalized maximal gradient act on the outer spatial heat operator on the RHS of the previous identity. We control such normalized maximal gradient of the spatial heat operator by taking out the inner space-time heat operator it acts on while giving up its $\|\|_{1;\mathbb{T}_{N}}$-norm and replacing spatial gradients of $\mathbf{H}^{N}$ by their absolute values and sum over $\mathbb{T}_{N}$:
\begin{align}
\|\wt{\grad}_{\mathfrak{l}_{N}}^{\mathbf{X}}\mathbf{H}_{T,x}^{N}(\phi_{S,y}\mathbf{1}_{S\leq T_{N}})\|_{1;\mathbb{T}_{N}} \ \leq \ \||\wt{\grad}_{\mathfrak{l}_{N}}^{\mathbf{X}}|\mathbf{H}_{\mathrm{t}_{(N)},x}^{N,\mathbf{X}}(1)\|_{1;\mathbb{T}_{N}}\|\mathbf{H}^{N}(\phi_{S,y})\|_{1;\mathbb{T}_{N}}. \label{eq:BGII22}
\end{align}
\item The first factor within the RHS of \eqref{eq:BGII22} above is the 1-norm on $\mathbb{T}_{N}$ in the forwards spatial variable of the spatial gradient of the $\mathbf{H}^{N}$ heat kernel at time $\mathrm{t}_{(N)}$, maximized over $\mathbb{T}_{N}$ with respect to the backwards spatial variable. Via \eqref{eq:heatII} in Proposition \ref{prop:heat}, this is at most uniformly bounded factors times $N^{-1}\mathrm{t}_{(N)}^{-1/2}\lesssim N^{-3/4+999\e_{\mathrm{RN}}/2}$. Therefore, we get from this and \eqref{eq:BGII22}
\begin{align}
\|\wt{\grad}_{\mathfrak{l}_{N}}^{\mathbf{X}}\mathbf{H}_{T,x}^{N}(\phi_{S,y}\mathbf{1}_{S\leq T_{N}})\|_{1;\mathbb{T}_{N}} \ \lesssim \ N^{-\frac34+\frac{999}{2}\e_{\mathrm{RN}}}\|\mathbf{H}^{N}(\phi_{S,y})\|_{1;\mathbb{T}_{N}}. \label{eq:BGII23}
\end{align}
\end{itemize}
We use \eqref{eq:BGII23} for $\phi=N^{1/2}(\bar{\mathfrak{q}}-\mathsf{E}^{\mathrm{can}}_{\e_{1}+\mathfrak{b}_{+}\e_{\mathrm{RN},1}})\mathbf{Y}^{N}$. Following the multiscale decomposition \eqref{eq:BGI2} in the proof of Theorem \ref{theorem:BGI}, by Propositions \ref{prop:BGI1} and \ref{prop:BGI2}, we get the following that we explain shortly; below, $\beta>0$ is universal and independent of $\e_{\mathrm{RN}}$:
\begin{align}
\E\|\wt{\grad}_{\mathfrak{l}_{N}}^{\mathbf{X}}\mathbf{H}_{T,x}^{N}(\phi_{S,y}\mathbf{1}_{S\leq T_{N}})\|_{1;\mathbb{T}_{N}} \ \lesssim \ N^{-\frac34+\frac{999}{2}\e_{\mathrm{RN}}}\E\|\mathbf{H}^{N}(\phi_{S,y})\|_{1;\mathbb{T}_{N}} \ \lesssim \ N^{-\frac34-\beta+\frac{999}{2}\e_{\mathrm{RN}}}. \label{eq:BGII24}
\end{align}
The independence from $\e_{\mathrm{RN}}$/universal feature of the exponent $\beta$ on the RHS of \eqref{eq:BGII24} follows by the observation that to replace $\bar{\mathfrak{q}}$ with $\mathsf{E}^{\mathrm{can}}_{\e_{1}+\mathfrak{b}_{+}\e_{\mathrm{RN},1}}$ from the proof of Theorem \ref{theorem:BGI} using Proposition \ref{prop:BGI1} and Proposition \ref{prop:BGI2}, the estimates in Proposition \ref{prop:BGI1} and Proposition \ref{prop:BGI2} have upper bounds that are \emph{universal} negative powers of $N$ independent of $\e_{\mathrm{RN}}$. Taking $\e_{\mathrm{RN}}$ sufficiently small shows that the RHS of \eqref{eq:BGII21a} is bounded above by the far RHS of \eqref{eq:BGII24}, and thus controlled by the RHS of the proposed estimate \eqref{eq:BGII2main}, upon possibly adjusting the value of $\beta$ by a universal positive factor. \qed
\begin{remark}\label{remark:regremark2}
\fsp If we take a second-order spatial gradient in Theorem \ref{theorem:BGII} instead of first-order gradient, which would be relevant if we were to apply our method to derive Boltzmann-Gibbs principles to study environment-dependence in the reversible dynamics of the particle system, we would have to resolve a higher-degree short-time singularity of the heat kernel. Unlike Remark \ref{remark:regremark1}, however, we cannot just smooth since Theorem \ref{theorem:BGII} will be used later to control density fluctuations, which by definition leads us to the LHS of \eqref{eq:BGII} without smoothing. Instead, for non-KPZ fluctuations of interest in \cite{CYau}, we estimate Sobolev regularity of the density fluctuation $N\grad^{\mathbf{X}}_{1}\mathbf{h}^{N}$ by the proof of Theorem 2 in Chang-Yau \cite{CYau}. It amounts to estimating said regularity by a general energy estimate that becomes useful if we have an ``a priori" Boltzmann-Gibbs principle. Regularity gives the Boltzmann-Gibbs principle via our local method. Then we iterate via fixed-point methods, using this Boltzmann-Gibbs principle to get regularity and so forth. Though this approach is inapplicable here since we study singular KPZ fluctuations, we make this remark in case of potential interest and to emphasize how one may apply our methods to generalize \cite{CYau}, for example to non-trivial perturbations of environment-dependent exclusion processes as in \cite{JL,JM} or open boundary models, as we noted in the introduction.
\end{remark}
%
%
%
\section{Regularity Estimates}\label{section:reg}
The purpose of this section is to establish a ``self-propagating" aspect of the a priori \emph{regularity} estimates defining $\mathfrak{t}_{\mathrm{st}}$ and $\mathbf{Y}^{N}$; see Definitions \ref{definition:KPZ1} and \ref{definition:KPZ5}. The self-propagating feature of the time-regularity estimate follows from a fairly straightforward set of moment estimates; see (3.14) in Proposition 3.2 in \cite{DT}. The self-propagating feature of the spatial regularity estimates will require the second Boltzmann-Gibbs principle in Theorem \ref{theorem:BGII}, and this will produce a weaker but sufficient result. 
\begin{prop}\label{prop:reg}
\fsp Consider any arbitrarily small but universal constant $\vartheta>0$. Given any possibly random time $\mathfrak{t}_{\mathrm{r}}\in[0,1]$, let us define the following pair of events, in which we recall the notation of \emph{Definition \ref{definition:KPZ1}}:
\begin{align}
\mathcal{E}_{\vartheta}^{\mathbf{T}}(\mathfrak{t}_{\mathrm{r}};\mathbb{T}_{N}) \ &\overset{\bullet}= \ \left\{\sup_{\mathrm{s}\in\mathbb{I}^{\mathbf{T}}}\mathrm{s}^{-1/4}\|\grad_{-\mathrm{s}}^{\mathbf{T}}\mathbf{U}^{N}\|_{\mathfrak{t}_{\mathrm{r}};\mathbb{T}_{N}}\geq N^{\vartheta}(1+\|\mathbf{U}^{N}\|_{\mathfrak{t}_{\mathrm{r}};\mathbb{T}_{N}}^{2})\right\} \\
\mathcal{E}_{\vartheta}^{\mathbf{X}}(\mathfrak{t}_{\mathrm{r}};\mathbb{T}_{N}) \ &\overset{\bullet}= \ \left\{\sup_{1\leq|\mathfrak{l}|\leq\mathfrak{l}_{N}}N^{1/2}|\mathfrak{l}|^{-1/2}\|\grad_{\mathfrak{l}}^{\mathbf{X}}\mathbf{U}^{N}\|_{\mathfrak{t}_{\mathrm{r}};\mathbb{T}_{N}} \geq N^{\vartheta}{(1+\|\mathbf{U}^{N}\|_{\mathfrak{t}_{\mathrm{r}};\mathbb{T}_{N}})^{2}}\right\}.
\end{align}
There exists a universal constant $\beta_{\mathrm{r}}>0$, which is thus uniformly bounded from below, such that for any $\kappa>0$, we have
\begin{align}
\mathbf{P}\left(\mathcal{E}_{\vartheta}^{\mathbf{T}}(\mathfrak{t}_{\mathrm{r}};\mathbb{T}_{N})\right) \ \lesssim_{\vartheta,\kappa} \ N^{-\kappa} \quad \mathrm{and} \quad \mathbf{P}\left(\mathcal{E}_{\vartheta}^{\mathbf{X}}(\mathfrak{t}_{\mathrm{r}};\mathbb{T}_{N})\right) \ \lesssim_{\vartheta} \ N^{-\beta_{\mathrm{r}}}.
\end{align}
\end{prop}
\subsection{$\mathcal{E}_{\vartheta}^{\mathbf{T}}(\mathfrak{t}_{\mathrm{r}};\mathbb{T}_{N})$ estimate}
We first focus on getting the time-regularity estimate, namely the estimate for the $\mathcal{E}_{\vartheta}^{\mathbf{T}}(\mathfrak{t}_{\mathrm{r}};\mathbb{T}_{N})$ probability. Following the proof of time regularity estimates for the Gartner transform in Proposition 3.2 of \cite{DT}, we estimate time-regularity of $\mathbf{U}^{N}$ by its defining stochastic equation in Definition \ref{definition:KPZ5}. Specifically we control time-regularity of each term therein. We start with the following result that does this for all terms except initial data and $\d\xi^{N}$ terms, which we treat separately.
\begin{lemma}\label{lemma:regT1}
\fsp Take any $|\mathfrak{k}|\lesssim1$. We have the following deterministic estimate in which $\mathfrak{f}_{1},\mathfrak{f}_{2}:\Omega\to\R$ are uniformly bounded:
\begin{align}
\sup_{0<\mathrm{s}\leq N^{-1}}\mathrm{s}^{-\frac14}\|\grad_{-\mathrm{s}}^{\mathbf{T}}\mathbf{H}_{T,x}^{N}(N^{\frac12}\mathfrak{f}_{1}\mathbf{Y}^{N})\|_{\mathfrak{t}_{\mathrm{r}};\mathbb{T}_{N}} + \sup_{0<\mathrm{s}\leq N^{-1}}N^{-\frac12}\mathrm{s}^{-\frac14}\|\grad_{-\mathrm{s}}^{\mathbf{T}}\mathbf{H}_{T,x}^{N}(\grad_{-\mathfrak{k}}^{!}(\mathfrak{f}_{2}\mathbf{U}^{N}))\|_{\mathfrak{t}_{\mathrm{r}};\mathbb{T}_{N}} \ \lesssim \ 1+\|\mathbf{U}^{N}\|_{\mathfrak{t}_{\mathrm{r}};\mathbb{T}_{N}}^{2}. \label{eq:regT1}
\end{align}
\end{lemma}
\begin{proof}
We will apply time-regularity estimates on the heat operator $\mathbf{H}^{N}$ from Proposition \ref{prop:heat} to estimate the first supremum on the LHS of \eqref{eq:regT1}. We additionally apply a mixed space-time regularity estimate on $\mathbf{H}^{N}$ in Proposition \ref{prop:heat} to estimate the second supremum on the LHS of \eqref{eq:regT1}. The former time regularity estimate gives the following in which the equality follows trivially and the last estimate follows by recalling that we are restricting to time-scales $\mathrm{s}\leq N^{-1}$ and that $|\mathbf{Y}^{N}|\lesssim N^{\e_{\mathrm{ap}}}$ by construction:
\begin{align}
\|\grad_{-\mathrm{s}}^{\mathbf{T}}\mathbf{H}_{T,x}^{N}(N^{\frac12}\mathfrak{f}_{1}\mathbf{Y}^{N})\|_{\mathfrak{t}_{\mathrm{r}};\mathbb{T}_{N}} \ \lesssim_{\e_{\mathrm{ap}}} \ N^{\frac12+\e_{\mathrm{ap}}}\|\mathbf{Y}^{N}\|_{\mathfrak{t}_{\mathrm{r}};\mathbb{T}_{N}}\mathrm{s} \ = \ N^{\frac12+\e_{\mathrm{ap}}}\|\mathbf{Y}^{N}\|_{\mathfrak{t}_{\mathrm{r}};\mathbb{T}_{N}}\mathrm{s}^{\frac34}\mathrm{s}^{\frac14} \ \lesssim \ N^{-\frac14+2\e_{\mathrm{ap}}}\mathrm{s}^{\frac14}. \label{eq:regT11}
\end{align}
Similarly, the aforementioned mixed space-time regularity estimate for $\mathbf{H}^{N}$ in Proposition \ref{prop:heat} gives
\begin{align}
N^{-\frac12}\|\grad_{-\mathrm{s}}^{\mathbf{T}}\mathbf{H}_{T,x}^{N}(\grad_{-\mathfrak{k}}^{!}(\mathfrak{f}_{2}\mathbf{U}^{N}))\|_{\mathfrak{t}_{\mathrm{r}};\mathbb{T}_{N}} \ \lesssim_{\e_{\mathrm{ap}}} \ N^{-\frac12+\e_{\mathrm{ap}}}\|\mathbf{U}^{N}\|_{\mathfrak{t}_{\mathrm{r}};\mathbb{T}_{N}}\mathrm{s}^{\frac14}. \label{eq:regT12}
\end{align}
Combining the previous two estimates \eqref{eq:regT11} and \eqref{eq:regT12} while recalling $\e_{\mathrm{ap}}$ is arbitrarily small but still universal would provide the proposed estimate \eqref{eq:regT1} if we dropped the 1 term on the RHS of \eqref{eq:regT1}, and the squared norms therein were replaced by non-squared norms. But this would imply \eqref{eq:regT1} as written via the inequality $2|a|\leq1+a^{2}$, which holds for all $a\in\R$. 
\end{proof}
We will proceed by estimating time-regularity of the initial data term from the $\mathbf{U}^{N}$ equation in Definition \ref{definition:KPZ5}. At this point in this subsection, in contrast to Lemma \ref{lemma:regT1} our estimates will not be deterministic. In particular, the proof of the following estimate is based on establishing uniform upper bounds on moments of time-gradients \emph{for each point in space-time}. We will then glue the estimates to establish high probability time-regularity estimate \emph{simultaneously} over some very fine discretization of space-time. We then conclude with a much simpler estimate to control \emph{sub-microscopic} short-time regularity. This will allow us to bootstrap from a discrete set of times to a continuous set of times.
\begin{lemma}\label{lemma:regT2}
\fsp Consider any $\vartheta,\kappa>0$ arbitrarily small and large, respectively, but both universal. We have
\begin{align}
\mathbf{P}\left(\sup_{\mathrm{s}\in\mathbb{I}^{\mathbf{T}}}\mathrm{s}^{-1/4}\|\grad_{-\mathrm{s}}^{\mathbf{T}}\mathbf{H}_{T,x}^{N,\mathbf{X}}(\mathbf{Z}_{0,\cdot}^{N})\|_{\mathfrak{t}_{\mathrm{r}};\mathbb{T}_{N}} \geq N^{\vartheta}(1+\|\mathbf{U}^{N}\|_{\mathfrak{t}_{\mathrm{r}};\mathbb{T}_{N}}^{2})\right) \ \lesssim_{\vartheta,\kappa} \ N^{-\kappa}. \label{eq:regT2}
\end{align}
\end{lemma}
\begin{proof}
We proceed with steps briefly outlined prior to the statement of Lemma \ref{lemma:regT2}, namely a pointwise moment estimate, a union bound estimate, and the short-time regularity estimate, which we write in this order.
\begin{itemize}
\item Observe by Proposition \ref{prop:heat}, for example, the spatial operators $\mathbf{H}^{N,\mathbf{X}}$ satisfy the classical semigroup property for heat kernels and Markov processes. Thus, because we assume stable initial data for the Gartner transform, we will follow the proof of (3.14) of Proposition 3.2 in \cite{DT} to get the following for fixed $\mathrm{s}\in\mathbb{I}^{\mathbf{T}}$ and $T\geq0$ and $x\in\mathbb{T}_{N}$ with arbitrary $p\geq1$ and $\gamma>0$; below, the $\|\|_{\omega;2p}$-norm is with respect to the randomness in the particle system:
\begin{align}
\|\grad_{-\mathrm{s}}^{\mathbf{T}}\mathbf{H}_{T,x}^{N,\mathbf{X}}(\mathbf{Z}_{0,\cdot}^{N})\|_{\omega;2p} \ \lesssim_{p,\gamma} \ \mathrm{s}^{1/4-\gamma} \ \lesssim \ N^{2\gamma}\mathrm{s}^{1/4}. \label{eq:regT21}
\end{align}
Recall from the definition of stable initial data that we may take any $\gamma>0$ in the previous estimate \eqref{eq:regT21}. The last estimate in \eqref{eq:regT21} follows as $\mathrm{s}\in\mathbb{I}^{\mathbf{T}}$ implies $\mathrm{s}\geq N^{-2}$, and thus $\mathrm{s}^{-\gamma}\leq N^{2\gamma}$. Applying the Chebyshev inequality, from \eqref{eq:regT21} we establish the following probability estimate where, provided any $\vartheta,\kappa>0$, we take $\gamma>0$ sufficiently small with $p\geq1$ sufficiently large but both depending only on $\vartheta,\kappa>0$ so that $2p\gamma-2p\vartheta \leq -2\kappa$; we note that the following is uniform in space-time:
\begin{align}
\mathbf{P}\left(|\grad_{-\mathrm{s}}^{\mathbf{T}}\mathbf{H}_{T,x}^{N,\mathbf{X}}(\mathbf{Z}_{0,\cdot}^{N})| \geq N^{\vartheta}\mathrm{s}^{\frac14}\right) \ \lesssim_{p,\gamma} \ N^{-2p\vartheta}N^{2p\gamma} \ \leq \ N^{-2\kappa}. \label{eq:regT22}
\end{align}
We emphasize that the dependence on $\gamma>0$ and $p\geq1$ in \eqref{eq:regT22} is now dependence on $\vartheta,\kappa>0$.
\item Consider a time-discretization $\mathbb{I}^{\mathbf{T},\mathrm{d}}=\{\mathfrak{j}N^{-99}\}_{\mathfrak{j}=0}^{N^{99}}$ and let $\mathbb{I}^{\mathrm{d}}=\mathbb{I}^{\mathbf{T},\mathrm{d}}\times\mathbb{T}_{N}$ be the space-time for this time-discretization. We now employ a union bound along with the previous probability estimate \eqref{eq:regT22} to get
\begin{align}
\mathbf{P}\left(\sup_{\mathrm{s}\in\mathbb{I}^{\mathbf{T}}}\sup_{(T,x)\in\mathbb{I}^{\mathrm{d}}}|\grad_{-\mathrm{s}}^{\mathbf{T}}\mathbf{H}_{T,x}^{N,\mathbf{X}}(\mathbf{Z}_{0,\cdot}^{N})|\geq N^{\vartheta}\mathrm{s}^{\frac14}\right) \ &\leq \ \sum_{\mathrm{s}\in\mathbb{I}^{\mathbf{T}}}\sum_{(T,x)\in\mathbb{I}^{\mathrm{d}}}\mathbf{P}\left(|\grad_{-\mathrm{s}}^{\mathbf{T}}\mathbf{H}_{T,x}^{N,\mathbf{X}}(\mathbf{Z}_{0,\cdot}^{N})| \geq N^{\vartheta}\mathrm{s}^{\frac14}\right) \ \lesssim \ N^{-2\kappa+101}. \label{eq:regT23}
\end{align}
We emphasize that the final estimate on the far RHS of \eqref{eq:regT23} follows from applying \eqref{eq:regT22} to each probability in the summation in the middle of \eqref{eq:regT23} and then multiplying by the size of the product set $\mathbb{I}^{\mathbf{T}}\times\mathbb{I}^{\d}$; the size of $\mathbb{I}^{\mathbf{T}}$ from Definition \ref{definition:KPZ5} is uniformly bounded by $\kappa_{\e_{\mathrm{ap}}}N^{\e_{\mathrm{ap}}}$ because it is parameterized by one index set of size $N^{\e_{\mathrm{ap}}}$ and another index set that is in bijection with the set of exponents $\{-2+\mathfrak{j}\e_{\mathrm{ap}}\}_{\mathfrak{j}\geq0}\cap[-2,1]$. We also used the upper bound $|\mathbb{I}^{\d}|=|\mathbb{I}^{\mathbf{T},\d}||\mathbb{T}_{N}|\lesssim N^{99}N=N^{100}$.
\item Let us now bootstrap from the discretization estimate \eqref{eq:regT23} to the proposed estimate \eqref{eq:regT2} over the entire semi-discrete space-time $[0,\mathfrak{t}_{\mathrm{r}}]\times\mathbb{T}_{N}$. To this end, let us first observe that $\mathbf{H}^{N,\mathbf{X}}(\mathbf{Z}^{N})$ is in the kernel of the operator $\partial_{T}-\mathscr{L}_{N}$ because it is a linear combination of heat kernels, each of which vanish under this operator. Thus, given any $0\leq\mathfrak{t}_{1}\leq\mathfrak{t}_{2}$, we have
\begin{align}
\sup_{x\in\mathbb{T}_{N}}|\mathbf{H}^{N,\mathbf{X}}_{\mathfrak{t}_{2},x}(\mathbf{Z}_{0,\bullet}^{N})-\mathbf{H}^{N,\mathbf{X}}_{\mathfrak{t}_{1},x}(\mathbf{Z}_{0,\bullet}^{N})| \ \leq \ \int_{\mathfrak{t}_{1}}^{\mathfrak{t}_{2}}\sup_{x\in\mathbb{T}_{N}}|\mathscr{L}_{N}\mathbf{H}_{r,x}^{N,\mathbf{X}}(\mathbf{Z}_{0,\bullet}^{N})| \d r \ \leq \ |\mathfrak{t}_{2}-\mathfrak{t}_{1}|\sup_{\mathfrak{t}_{1}\leq r\leq\mathfrak{t}_{2}}\sup_{x\in\mathbb{T}_{N}}|\mathscr{L}_{N}\mathbf{H}_{r,x}^{N,\mathbf{X}}(\mathbf{Z}_{0,\bullet}^{N})|. \label{eq:regT24}
\end{align}
Observe $\mathscr{L}_{N}:\mathscr{L}^{\infty}(\mathbb{T}_{N})\to\mathscr{L}^{\infty}(\mathbb{T}_{N})$ has operator norm $\mathrm{O}(N^{2})$; see Proposition \ref{prop:mSHE+}. Combining this with \eqref{eq:regT24}, and the observation in Proposition \ref{prop:heat} that the spatial heat operator $\mathbf{H}^{N,\mathbf{X}}:\mathscr{L}^{\infty}(\mathbb{T}_{N})\to\mathscr{L}^{\infty}(\mathbb{T}_{N})$ has operator norm 1, provides
\begin{align}
\sup_{x\in\mathbb{T}_{N}}|\mathbf{H}^{N,\mathbf{X}}_{\mathfrak{t}_{2},x}(\mathbf{Z}_{0,\bullet}^{N})-\mathbf{H}^{N,\mathbf{X}}_{\mathfrak{t}_{1},x}(\mathbf{Z}_{0,\bullet}^{N})|  \ \lesssim \ N^{2}|\mathfrak{t}_{2}-\mathfrak{t}_{1}|\sup_{\mathfrak{t}_{1}\leq r\leq\mathfrak{t}_{2}}\sup_{x\in\mathbb{T}_{N}}|\mathbf{H}_{r,x}^{N,\mathbf{X}}(\mathbf{Z}_{0,\bullet}^{N})| \ \leq \  N^{2}|\mathfrak{t}_{2}-\mathfrak{t}_{1}|\|\mathbf{Z}^{N}\|_{0;\mathbb{T}_{N}}. \label{eq:regT25}
\end{align}
We will now establish the proposed estimate \eqref{eq:regT2}. First, we observe that, choosing $\kappa\geq300$ arbitrarily large but still universal, we may work on the complement of the event in the probability on the far LHS of \eqref{eq:regT23}; anything outside this event happens with probability at most $N^{-2\kappa+100}\leq N^{-3\kappa/2}$ times factors depending only on $\vartheta,\kappa$. Given any $\mathrm{t}\in[0,\mathfrak{t}_{\mathrm{r}}]$, let $\mathrm{t}^{\mathrm{d}}$ denote any element in $\mathbb{I}^{\mathbf{T},\mathrm{d}}$ which minimizes $|\mathrm{t}-\mathrm{t}^{\mathrm{d}}|$. Because the elements in $\mathbb{I}^{\mathbf{T},\d}$ are evenly spaced by $N^{-99}$, we automatically have $|\mathrm{t}^{\d}-\mathrm{t}|\leq N^{-99}$. We now transfer a time-gradient at $\mathrm{t}$ onto one at $\mathrm{t}^{\mathrm{d}}$ and collect the errors:
\begin{align}
\grad_{-\mathrm{s}}^{\mathbf{T}}\mathbf{H}_{\mathrm{t},x}^{N,\mathbf{X}}(\mathbf{Z}_{0,\cdot}^{N}) \ &= \ \grad_{-\mathrm{s}}^{\mathbf{T}}\mathbf{H}_{\mathrm{t}^{\d},x}^{N,\mathbf{X}}(\mathbf{Z}_{0,\cdot}^{N}) + \left(\mathbf{H}_{\mathrm{t}-\mathrm{s},x}^{N,\mathbf{X}}(\mathbf{Z}_{0,\cdot}^{N})-\mathbf{H}_{\mathrm{t}^{\d}-\mathrm{s},x}^{N,\mathbf{X}}(\mathbf{Z}_{0,\cdot}^{N})\right) + \left(\mathbf{H}^{N,\mathbf{X}}_{\mathrm{t}^{\d},x}(\mathbf{Z}_{0,\cdot}^{N})-\mathbf{H}^{N,\mathbf{X}}_{\mathrm{t},x}(\mathbf{Z}_{0,\cdot}^{N})\right). \label{eq:regT26}
\end{align}
We observe $(\mathrm{t}^{\d},x)\in\mathbb{I}^{\d}$ by construction. Because we work on the complement of the event in the probability on the far LHS of \eqref{eq:regT23}, the first term on the RHS of \eqref{eq:regT26}, after dividing by $\mathrm{s}^{1/4}$ and taking a supremum on $\mathbb{I}^{\d}$, is at most $N^{\vartheta}$. On the other hand, by \eqref{eq:regT25} for $\{\mathfrak{t}_{1},\mathfrak{t}_{2}\}=\{\mathrm{t}-\mathrm{s},\mathrm{t}^{\d}-\mathrm{s}\}$ and $\{\mathfrak{t}_{1},\mathfrak{t}_{2}\}=\{\mathrm{t},\mathrm{t}^{\d}\}$ we get the following upon recalling $\mathrm{s}\in\mathbb{I}^{\mathbf{T}}$ implies $\mathrm{s}^{-1/4}\lesssim N^{1/2}$. We explain the last estimate in the following display after; it is deterministic because of our conditioning:
\begin{align}
\mathfrak{s}^{-1/4}\|\mathbf{H}_{\mathrm{t}-\mathrm{s},x}^{N,\mathbf{X}}(\mathbf{Z}_{0,\cdot}^{N})-\mathbf{H}_{\mathrm{t}^{\d}-\mathrm{s},x}^{N,\mathbf{X}}(\mathbf{Z}_{0,\cdot}^{N})\|_{\mathfrak{t}_{\mathrm{r}};\mathbb{T}_{N}} \ \lesssim \ N^{1/2}N^{2}|\mathrm{t}-\mathrm{s}-\mathrm{t}^{\d}+\mathrm{s}|\|\mathbf{Z}^{N}\|_{0;\mathbb{T}_{N}} \ \lesssim \ N^{-96}\|\mathbf{U}^{N}\|_{\mathfrak{t}_{\mathrm{r}};\mathbb{T}_{N}}. \label{eq:regT27}
\end{align}
The last estimate in \eqref{eq:regT27} follows by recalling $|\mathrm{t}-\mathrm{s}-\mathrm{t}^{\d}+\mathrm{s}|=|\mathrm{t}-\mathrm{t}^{\d}| \leq N^{-99}$ and by realizing that $\mathbf{U}^{N}$ at time 0 is equal to $\mathbf{Z}^{N}$ at time 0 by construction in Definition \ref{definition:KPZ5}, and this lets us replace $\mathbf{Z}^{N}$ with $\mathbf{U}^{N}$ in the middle of \eqref{eq:regT27}; the final step then bounds $\|\|_{0;\mathbb{T}_{N}}$ by $\|\|_{\mathfrak{t}_{\mathrm{r}};\mathbb{T}_{N}}$. We establish the same estimate upon replacing what is inside the norm on the far LHS of \eqref{eq:regT27} by the third/last term on the RHS of \eqref{eq:regT26} via the same reasoning. Thus, \emph{on the complement of the event in the probability on the far LHS of} \eqref{eq:regT23}, the complement of the event in the probability in \eqref{eq:regT2} holds. As this complement event in \eqref{eq:regT23} fails with probability at most $N^{-\kappa}$ times $\vartheta,\kappa$-dependent factors, the same is true for the complement event in \eqref{eq:regT2} as well.
\end{itemize}
This completes the proof.
\end{proof}
We establish a similar time-regularity estimate for the $\d\xi^{N}$-term in the $\mathbf{U}^{N}$ equation from Definition \ref{definition:KPZ5}. The strategy is the same, but because of the martingale theory that needs to be employed to efficiently study this term, gymnastics are needed. We emphasize the quadratic nature of regularity estimates we are currently proving comes naturally via the proof of the next result.
\begin{lemma}\label{lemma:regT3}
\fsp Consider any $\vartheta,\kappa\in\R_{>0}$ arbitrarily small and large, respectively, but both universal. We have
\begin{align}
\mathbf{P}\left(\sup_{\mathrm{s}\in\mathbb{I}^{\mathbf{T}}}\mathrm{s}^{-1/4}\|\grad_{-\mathrm{s}}^{\mathbf{T}}\mathbf{H}_{T,x}^{N}(\mathbf{U}^{N}\d\xi^{N})\|_{\mathfrak{t}_{\mathrm{r}};\mathbb{T}_{N}} \geq N^{\vartheta}(1+\|\mathbf{U}^{N}\|_{\mathfrak{t}_{\mathrm{r}};\mathbb{T}_{N}}^{2})\right) \ \lesssim_{\vartheta,\kappa} \ N^{-\kappa}. \label{eq:regT3}
\end{align}
\end{lemma}
\begin{proof}
We employ a slightly adapted version of the strategy as in the proof of Lemma \ref{lemma:regT2}. In particular, the first step we will take, for $\kappa>0$ in the lemma large, is proving the following pointwise estimate for which $\mathrm{s}\in\mathbb{I}^{\mathbf{T}}$ and $(T,x)\in[0,1]\times\mathbb{T}_{N}$; we will defer the proof of the following estimate \eqref{eq:regT31} until later in this argument to avoid obscuring the strategy of this proof:
\begin{align}
\mathbf{P}\left(|\grad_{-\mathrm{s}}^{\mathbf{T}}\mathbf{H}_{T,x}^{N}(\mathbf{U}^{N}\d\xi^{N})| \geq N^{\vartheta}\mathrm{s}^{\frac14}(1+\|\mathbf{U}^{N}\|_{T;\mathbb{T}_{N}}^{2})\right) \ \lesssim_{\vartheta,\kappa} \ N^{-2\kappa}. \label{eq:regT31}
\end{align}
We proceed with a union bound over $(T,x)\in\mathbb{I}^{\d}$ with $\mathbb{I}^{\d}$ the discretization in the proof of Lemma \ref{lemma:regT2}; we deduce from \eqref{eq:regT31} the following whose proof follows that of \eqref{eq:regT23}, where the last estimate in \eqref{eq:regT32} below follows by choosing $\kappa>0$ large:
\begin{align}
\mathbf{P}\left(\sup_{\mathrm{s}\in\mathbb{I}^{\mathbf{T}}}\sup_{(T,x)\in\mathbb{I}^{\d}}\mathrm{s}^{-1/4}|\grad_{-\mathrm{s}}^{\mathbf{T}}\mathbf{H}_{T,x}^{N}(\mathbf{U}^{N}\d\xi^{N})|(1+\|\mathbf{U}^{N}\|_{T;\mathbb{T}_{N}}^{2})^{-1} \geq N^{\vartheta}\right) \ \lesssim_{\vartheta,\kappa} \ N^{-2\kappa+100} \ \leq \ N^{-3\kappa/2}. \label{eq:regT32}
\end{align}
Following the proof of Lemma \ref{lemma:regT2}, we obtain time-regularity of $\mathbf{H}^{N}(\mathbf{U}^{N}\d\xi^{N})$ for short order $N^{-99}$ times to bootstrap estimates in \eqref{eq:regT32} on $\mathbb{I}^{\d}$ to estimates over the entire semi-discrete space-time $[0,\mathfrak{t}_{\mathrm{r}}]\times\mathbb{T}_{N}$. However, dissimilar to the proof for Lemma \ref{lemma:regT2}, the quantity $\mathbf{H}^{N}(\mathbf{U}^{N}\d\xi^{N})$ of interest is a space-time heat operator, not a spatial heat operator. Therefore, the semi-discrete PDE that it satisfies is the same $\mathscr{L}_{N}$-heat equation satisfied by the $\mathbf{H}^{N}$ heat \emph{kernel} but with an additional martingale differential. This makes the short-time estimates for $\mathbf{H}^{N}(\mathbf{U}^{N}\d\xi^{N})$ more complicated, so we adopt another approach. Consider the $\mathbf{U}^{N}$ equation in Definition \ref{definition:KPZ5}. The short-time regularity for $\mathbf{H}^{N}(\mathbf{U}^{N}\d\xi^{N})$ is tautologically controlled by the short-time regularity for all other terms in that $\mathbf{U}^{N}$ equation. We have already addressed short-time regularity for all of these terms in the $\mathbf{U}^{N}$ equation, for example in Lemma \ref{lemma:regT1} and the proof for Lemma \ref{lemma:regT2}, except for $\mathbf{U}^{N}$ itself. Because $\mathbf{U}^{N}$ evolves in a large part through jumps in the particle system, we will not establish any deterministic short-time regularity estimates like we did for the other terms in the $\mathbf{U}^{N}$ equation, but we instead get high probability short-time regularity. Precisely, we get the following, for which we consider the space-time $\mathbb{I}^{\d,\mathfrak{t}_{\mathrm{r}}}=(\mathbb{J}^{\mathbf{T}}\cap[0,\mathfrak{t}_{\mathrm{r}}])\times\mathbb{T}_{N}$, where $\mathbb{J}^{\mathbf{T}}\subset[0,1]$ has size $|\mathbb{J}^{\mathbf{T}}|\leq N^{200}$; we eventually take, for example in \eqref{eq:regT34}, the set  $\mathbb{J}^{\mathbf{T}}=\{\mathrm{t}^{\d}-\mathrm{s}\}$ for $\mathrm{t}^{\d}\in\mathbb{I}^{\mathbf{T},\d}=\{\mathfrak{j}N^{-99}\}_{\mathfrak{j}=0}^{N^{99}}$ and $\mathrm{s}\in\mathbb{I}^{\mathbf{T}}$; see Definition \ref{definition:KPZ1}:
\begin{align}
\mathbf{P}\left(\sup_{(\mathrm{t}^{\d},x)\in\mathbb{I}^{\d,\mathfrak{t}_{\mathrm{r}}}}\sup_{|\mathfrak{r}|\leq N^{-99}}|\grad_{-\mathfrak{r}}^{\mathbf{T}}\mathbf{H}^{N}_{\mathrm{t}^{\d},x}(\mathbf{U}^{N}\d\xi^{N})| \gtrsim N^{-\frac12+\vartheta}(1+\|\mathbf{U}^{N}\|_{\mathfrak{t}_{\mathrm{r}};\mathbb{T}_{N}}^{2})\right) \ \lesssim_{\vartheta,\kappa} \ N^{-2\kappa}. \label{eq:regT33}
\end{align}
We again provide the proof of \eqref{eq:regT33} at the end of this argument to avoid obstructing the point. Let us restrict to the complement of the events inside the probabilities in \eqref{eq:regT32} and \eqref{eq:regT33}. Now, we will follow the proof of Lemma \ref{lemma:regT2} starting with \eqref{eq:regT26} and replacing $\mathbf{H}^{N,\mathbf{X}}(\mathbf{Z}^{N})$ by $\mathbf{H}^{N}(\mathbf{U}^{N}\d\xi^{N})$. The first term on the RHS of the resulting equation is controlled by restricting to the complement of the event in the probability in \eqref{eq:regT32}. To control the second and third terms on the RHS of the resulting equation, we use the following  obtained by restricting to the complement of the event in the probability in \eqref{eq:regT33}; indeed, with notation as in \eqref{eq:regT26}, in \eqref{eq:regT34} below we assume $|\mathrm{t}-\mathrm{t}^{\d}|=|\mathrm{t}-\mathrm{s}-(\mathrm{t}^{\d}-\mathrm{s})|\leq N^{-100}$, so we may control the LHS of \eqref{eq:regT34} if we restrict to the complement of the event in \eqref{eq:regT33} since the LHS of \eqref{eq:regT34} is a scale $\leq N^{-99}$ time-gradient of $\mathbf{H}^{N}(\mathbf{U}^{N}\d\xi^{N})$ evaluated at a point in the discretization $\mathbb{I}^{\d,\mathfrak{t}_{\mathfrak{r}}}$ with the choice of $\mathbb{J}^{\mathbf{T}}$ explained right before \eqref{eq:regT33}:
\begin{align}
\mathrm{s}^{-1/4}\|\mathbf{H}^{N}_{\mathrm{t}-\mathrm{s}}(\mathbf{U}^{N}\d\xi^{N})-\mathbf{H}^{N}_{\mathrm{t}^{\d}-\mathrm{s}}(\mathbf{U}^{N}\d\xi^{N})\|_{\mathfrak{t}_{\mathrm{r}};\mathbb{T}_{N}} \ \lesssim \ \mathrm{s}^{-1/4}N^{-\frac12+\vartheta}(1+\|\mathbf{U}^{N}\|_{\mathfrak{t}_{\mathrm{r}};\mathbb{T}_{N}}^{2}) \ \lesssim \ N^{\vartheta}(1+\|\mathbf{U}^{N}\|_{\mathfrak{t}_{\mathrm{r}};\mathbb{T}_{N}}^{2}). \label{eq:regT34}
\end{align}
The final estimate in \eqref{eq:regT34} follows by recalling $\mathrm{s}\in\mathbb{I}^{\mathbf{T}}$ implies $\mathrm{s}\geq N^{-2}$, and thus $\mathrm{s}^{-1/4}\lesssim N^{1/2}$. Thus, whenever the events from \eqref{eq:regT32} and \eqref{eq:regT33} themselves fail, we deduce that the event in the probability in \eqref{eq:regT3} fails as well. Because these events in \eqref{eq:regT32} and \eqref{eq:regT33} succeed with probability $\mathrm{O}_{\kappa,\vartheta}(N^{-3\kappa/2})$ each, like the end of the proof of Lemma \ref{lemma:regT2}, we deduce that the probability that the event in \eqref{eq:regT3} succeeds is at most $\mathrm{O}_{\kappa,\vartheta}(N^{-3\kappa/2})\lesssim_{\kappa,\vartheta}N^{-\kappa}$. This completes the proof modulo the proofs of the probability estimates \eqref{eq:regT31} and \eqref{eq:regT33}, which we provide below.
\begin{itemize}
\item We will first prove \eqref{eq:regT31}. To this end, let us first define $\mathscr{N}(\mathbf{U}^{N})=1+\|\mathbf{U}^{N}\|_{T;\mathbb{T}_{N}}^{2}$ in order to ease notation. We now employ the Chebyshev inequality with $p\geq2$ to be determined shortly to establish the following upper bound for the LHS of \eqref{eq:regT31}:
\begin{align}
\mathbf{P}\left(|\grad_{-\mathrm{s}}^{\mathbf{T}}\mathbf{H}_{T,x}^{N}(\mathbf{U}^{N}\d\xi^{N})| \geq N^{\vartheta}\mathrm{s}^{\frac14}(1+\|\mathbf{U}^{N}\|_{T;\mathbb{T}_{N}}^{2})\right) \ \lesssim \ N^{-2p\vartheta}\mathrm{s}^{-p/2}\E\left(\mathscr{N}(\mathbf{U}^{N})^{-2p}|\grad_{-\mathrm{s}}^{\mathbf{T}}\mathbf{H}_{T,x}^{N}(\mathbf{U}^{N}\d\xi^{N})| ^{2p}\right). \label{eq:regT35}
\end{align}
To motivate the next step, observe that for regularity of space-time heat operators in Lemma \ref{lemma:regT1}, we could pull out $\mathbf{U}^{N}$ and $\mathbf{Y}^{N}$ factors from the integral/heat operator upon inserting extra factors given by their space-time supremum norms; this is $\mathscr{L}^{1}/\mathscr{L}^{\infty}$ interpolation. However, to study the heat operator in the expectation on the RHS of \eqref{eq:regT35}, we require martingale inequalities. In particular, we need to take advantage of cancellations in $\mathbf{U}^{N}\d\xi^{N}$ that appear when integrating against the heat kernel. This prevents us from applying the $\mathscr{L}^{1}/\mathscr{L}^{\infty}$ interpolation. The alternative we take begins with a level set decomposition/bound:
\begin{align}
\E\left(\mathscr{N}(\mathbf{U}^{N})^{-2p}|\grad_{-\mathrm{s}}^{\mathbf{T}}\mathbf{H}_{T,x}^{N}(\mathbf{U}^{N}\d\xi^{N})| ^{2p}\right) \ &\lesssim_{p} \ {\sum}_{\mathfrak{l}=1}^{\infty}\mathfrak{l}^{-2p}\E\left(|\grad_{-\mathrm{s}}^{\mathbf{T}}\mathbf{H}_{T,x}^{N}(\mathbf{U}^{N}\d\xi^{N})| ^{2p}\mathbf{1}_{\mathscr{N}(\mathbf{U}^{N})\in[\mathfrak{l},\mathfrak{l}+1]}\right). \label{eq:regT36}
\end{align}
The estimate \eqref{eq:regT36} follows by considering level sets of $\mathscr{N}(\mathbf{U}^{N})$; on the $[\mathfrak{l},\mathfrak{l}+1]$ level set, we may employ the \emph{deterministic} bound $\mathscr{N}(\mathbf{U}^{N})^{-1}\lesssim\mathfrak{l}^{-1}$. Next, we move a factor of $\mathfrak{l}^{-p}$ in the expectation and get the following, which we explain after:
\begin{align}
{\sum}_{\mathfrak{l}=1}^{\infty}\mathfrak{l}^{-2p}\E\left(|\grad_{-\mathrm{s}}^{\mathbf{T}}\mathbf{H}_{T,x}^{N}(\mathbf{U}^{N}\d\xi^{N})| ^{2p}\mathbf{1}_{\mathscr{N}(\mathbf{U}^{N})\in[\mathfrak{l},\mathfrak{l}+1]}\right) \ &= \ {\sum}_{\mathfrak{l}=1}^{\infty}\mathfrak{l}^{-p}\E\left(|\grad_{-\mathrm{s}}^{\mathbf{T}}\mathbf{H}_{T,x}^{N}(\mathfrak{l}^{-1/2}\mathbf{U}^{N}\d\xi^{N})| ^{2p}\mathbf{1}_{\mathscr{N}(\mathbf{U}^{N})\in[\mathfrak{l},\mathfrak{l}+1]}\right) \nonumber \\
&\leq \ \left({\sum}_{\mathfrak{l}=1}^{\infty}\mathfrak{l}^{-p}\right)\sup_{\mathfrak{l}\geq1}\E\left(|\grad_{-\mathrm{s}}^{\mathbf{T}}\mathbf{H}_{T,x}^{N}(\mathfrak{l}^{-1/2}\mathbf{U}^{N}\d\xi^{N})| ^{2p}\mathbf{1}_{\mathscr{N}(\mathbf{U}^{N})\in[\mathfrak{l},\mathfrak{l}+1]}\right) \nonumber \\
&\lesssim \ \sup_{\mathfrak{l}\geq1}\E\left(|\grad_{-\mathrm{s}}^{\mathbf{T}}\mathbf{H}_{T,x}^{N}(\mathfrak{l}^{-1/2}\mathbf{U}^{N}\d\xi^{N})| ^{2p}\mathbf{1}_{\mathscr{N}(\mathbf{U}^{N})\in[\mathfrak{l},\mathfrak{l}+1]}\right). \label{eq:regT37}
\end{align}
The first identity in the previous display follows by moving $\mathfrak{l}^{-p}$ into the expectation, then into the $2p$-th power upon replacing it by $\mathfrak{l}^{-1/2}$, and then moving this \emph{deterministic} scalar through both the linear time-gradient and heat operator. The final estimate \eqref{eq:regT37} follows from an elementary bound on the summation in the line before. By definition of $\mathscr{N}(\mathbf{U}^{N})$, if $\mathscr{N}(\mathbf{U}^{N})\in[\mathfrak{l},\mathfrak{l}+1]$, the process $\mathfrak{l}^{-1/2}\mathbf{U}^{N}$ is uniformly bounded deterministically. Moreover, because $\mathbf{U}^{N}$ is adapted to the underlying filtration of the particle system, so is the deterministic multiple $\mathfrak{l}^{-1/2}\mathbf{U}^{N}$. In particular, we may replace $\mathfrak{l}^{-1/2}\mathbf{U}^{N}$ in \eqref{eq:regT37} with the adapted process $(\mathfrak{l}^{-1/2}\mathbf{U}^{N}\wedge 100)\vee(-100)$, drop the indicator function in \eqref{eq:regT37}, and then follow the proof of time-regularity (3.14) in \cite{DT}. For the last step, we will need to apply the time-regularity estimates from Proposition \ref{prop:heat} for the $\mathbf{H}^{N}$ heat kernel instead of those in \cite{DT} along with the martingale inequality in Lemma \ref{lemma:mg} that extends Lemma 3.1 in \cite{DT}, which is proved only for the Gartner transform, to uniformly bounded adapted processes. This ultimately gives, for $\varrho>0$ arbitrarily small but universal, the following in which we stress that $\mathfrak{l}^{-1/2}\mathbf{U}^{N}$ is uniformly bounded and adapted on the LHS below:
\begin{align}
\sup_{\mathfrak{l}\geq1}\E\left(|\grad_{-\mathrm{s}}^{\mathbf{T}}\mathbf{H}_{T,x}^{N}(\mathfrak{l}^{-1/2}\mathbf{U}^{N}\d\xi^{N})| ^{2p}\mathbf{1}_{\mathscr{N}(\mathbf{U}^{N})\in[\mathfrak{l},\mathfrak{l}+1]}\right) \ \lesssim_{p,\varrho} \ \mathrm{s}^{p/2-p\varrho} \ \lesssim \ N^{2p\varrho}\mathrm{s}^{p/2}. \label{eq:regT38}
\end{align}
We recall that times $\mathrm{s}\in\mathbb{I}^{\mathbf{T}}$ of interest satisfy $\mathrm{s}\geq N^{-2}$, which implies $\mathrm{s}^{-1}\leq N^{2}$ and thus provides the final estimate in \eqref{eq:regT38}. We now combine \eqref{eq:regT35}, \eqref{eq:regT36}, \eqref{eq:regT37}, and \eqref{eq:regT38} to deduce 
\begin{align}
\mathbf{P}\left(|\grad_{-\mathrm{s}}^{\mathbf{T}}\mathbf{H}_{T,x}^{N}(\mathbf{U}^{N}\d\xi^{N})| \geq N^{\vartheta}\mathrm{s}^{\frac14}(1+\|\mathbf{U}^{N}\|_{T;\mathbb{T}_{N}}^{2})\right) \ \lesssim_{p,\varrho} \ N^{-2p\vartheta}N^{2p\varrho}\mathrm{s}^{-p/2}\mathrm{s}^{p/2} \ = \ N^{-2p\vartheta+2p\varrho}. \label{eq:regT39}
\end{align}
Now, provided any $\vartheta,\kappa>0$, we choose $\varrho>0$ sufficiently small and $p\geq2$ sufficiently large, but both depending only on $\vartheta,\kappa$, so that the exponent on the far RHS of \eqref{eq:regT39} is less than or equal to $-\kappa$. We emphasize that the dependence on $p$ and $\varrho$ in \eqref{eq:regT39} becomes dependence on $\vartheta,\kappa$. This completes the proof of \eqref{eq:regT31}.
\item We move to the proof of \eqref{eq:regT33}. To this end, it suffices to replace $\mathbf{H}^{N}(\mathbf{U}^{N}\d\xi^{N})$ therein with each other term in the $\mathbf{U}^{N}$ equation from Definition \ref{definition:KPZ5}. Indeed, if we can prove that the short-time regularity for every other term in the $\mathbf{U}^{N}$ equation exceeds the lower bound in the event in the probability in \eqref{eq:regT33} with probability at most $N^{-2\kappa}$ times $\vartheta,\kappa$-dependent factors, then by using the triangle inequality and a union bound, we can deduce the same for $\mathbf{H}^{N}(\mathbf{U}^{N}\d\xi^{N})$, which is the proposed estimate \eqref{eq:regT33}, if we also adjust the implied constant by a factor of 100. This is the fact that if $a=b+c$, then $a\geq d$ implies $b\geq d/2$ or, not exclusively, $c\geq d/2$. To control short-time regularity of every other term besides $\mathbf{H}^{N}(\mathbf{U}^{N}\d\xi^{N})$ on the RHS of the $\mathbf{U}^{N}$ equation from Definition \ref{definition:KPZ5}, we apply Lemma \ref{lemma:regT1} and the third bullet point from the proof of Lemma \ref{lemma:regT2}. It remains to control short-time regularity for $\mathbf{U}^{N}$ itself. This is done in Lemma \ref{lemma:st2}, so we are done with proving \eqref{eq:regT33}.
\end{itemize}
This completes the proof.
\end{proof}
\begin{corollary}\label{corollary:regT}
\fsp Admit the setting of \emph{Proposition \ref{prop:reg}}. We have $\mathbf{P}(\mathcal{E}_{\vartheta}^{\mathbf{T}}(\mathfrak{t}_{\mathrm{r}};\mathbb{T}_{N})) \lesssim_{\vartheta,\kappa} N^{-\kappa}$.
\end{corollary}
\begin{proof}
Combine the $\mathbf{U}^{N}$ equation in Definition \ref{definition:KPZ5} with Lemma \ref{lemma:regT1}, Lemma \ref{lemma:regT2}, and Lemma \ref{lemma:regT3}.
\end{proof}
\subsection{$\mathcal{E}_{\vartheta}^{\mathbf{X}}(\mathfrak{t}_{\mathrm{r}};\mathbb{T}_{N})$ estimate}
The proof of the $\mathcal{E}_{\vartheta}^{\mathbf{X}}(\mathfrak{t}_{\mathrm{r}};\mathbb{T}_{N})$ estimate in Proposition \ref{prop:reg} will follow basically the same strategy but, as we mentioned at the beginning of the section, we require additional input of Theorem \ref{theorem:BGII} to control the spatial gradient of the order $N^{1/2}$ term in the $\mathbf{U}^{N}$ equation. In view of similarities with the proof of the $\mathcal{E}_{\vartheta}^{\mathbf{T}}(\mathfrak{t}_{\mathrm{r}};\mathbb{T}_{N})$ estimate, we start as follows.
\begin{lemma}\label{lemma:regX1}
\fsp Take any $|\mathfrak{k}|\lesssim1$ and $|\mathfrak{l}|\leq\mathfrak{l}_{N}$, with $\mathfrak{l}_{N}$ in \emph{Definition \ref{definition:KPZ1}}. If $\mathfrak{f}_{i}$ are uniformly bounded, then
\begin{align}
N^{\frac12}|\mathfrak{l}|^{-\frac12}\|\grad_{\mathfrak{l}}^{\mathbf{X}}\mathbf{H}_{T,x}^{N}(\mathfrak{f}_{1}\mathbf{Y}^{N})\|_{\mathfrak{t}_{\mathrm{r}};\mathbb{T}_{N}} + |\mathfrak{l}|^{-\frac12}\|\grad_{\mathfrak{l}}^{\mathbf{X}}\mathbf{H}_{T,x}^{N}(\grad_{-\mathfrak{k}}^{!}(\mathfrak{f}_{2}\mathbf{U}^{N}))\|_{\mathfrak{t}_{\mathrm{r}};\mathbb{T}_{N}} \ \lesssim \ 1+\|\mathbf{U}^{N}\|_{\mathfrak{t}_{\mathrm{r}};\mathbb{T}_{N}}^{2}. \label{eq:regX1}
\end{align}
\end{lemma}
\begin{proof}
We follow the proof of Lemma \ref{lemma:regT1} but instead of time-regularity estimates of the heat operator $\mathbf{H}^{N}$ in Proposition \ref{prop:heat}, we instead apply spatial-regularity estimates therein. As $\mathbf{H}^{N}$ is macroscopically smooth, this gives the estimate but for $N|\mathfrak{l}|^{-1}$ in place of $N^{1/2}|\mathfrak{l}|^{-1/2}$ in the first term on the LHS of \eqref{eq:regX1}. But this is stronger as $N|\mathfrak{l}|^{-1}\geq N^{1/2}|\mathfrak{l}|^{-1/2}$ for $|\mathfrak{l}|\leq\mathfrak{l}_{N}\leq N$.
\end{proof}
\begin{lemma}\label{lemma:regX2}
\fsp Consider any $\vartheta,\kappa>0$ arbitrarily small and large, respectively, but both universal. We have
\begin{align}
\mathbf{P}\left(\sup_{1\leq|\mathfrak{l}|\leq\mathfrak{l}_{N}}N^{1/2}|\mathfrak{l}|^{-1/2}\|\grad_{\mathfrak{l}}^{\mathbf{X}}\mathbf{H}_{T,x}^{N,\mathbf{X}}(\mathbf{Z}_{0,\cdot}^{N})\|_{\mathfrak{t}_{\mathrm{r}};\mathbb{T}_{N}} \geq N^{\vartheta}(1+\|\mathbf{U}^{N}\|_{\mathfrak{t}_{\mathrm{r}};\mathbb{T}_{N}}^{2})\right) \ \lesssim_{\vartheta,\kappa} \ N^{-\kappa}. \label{eq:regX2}
\end{align}
\end{lemma}
\begin{proof}
We follow the proof of Lemma \ref{lemma:regT2}. In particular, we prove a pointwise probability estimate analogous to \eqref{eq:regT22} and a union bound analogous to \eqref{eq:regT23}. We conclude with continuity/bootstrap analogous to the third bullet point in the proof of Lemma \ref{lemma:regT2}.
\begin{itemize}
\item Observe that the operator $\mathscr{L}_{N}$ commutes with any constant coefficient spatial gradient; this can be easily verified. Because the spatial heat operator $\mathbf{H}^{N,\mathbf{X}}$ is a matrix/operator exponential of a constant multiple of $\mathscr{L}_{N}$, spatial gradients commute with the spatial heat operator. With this and the proof of (3.13) in \cite{DT}, we establish the following for fixed $\mathfrak{l}\in\llbracket-\mathfrak{l}_{N},\mathfrak{l}_{N}\rrbracket$ and $T\geq0$ and $x\in\mathbb{T}_{N}$ with arbitrary $p\geq1$ and $\gamma>0$, in which we assume $\mathfrak{l}\neq0$ as this case is trivial:
\begin{align}
\|\grad^{\mathbf{X}}_{\mathfrak{l}}\mathbf{H}^{N,\mathbf{X}}(\mathbf{Z}_{0,\cdot}^{N})\|_{\omega;2p} \ = \ \|\mathbf{H}^{N,\mathbf{X}}(\grad^{\mathbf{X}}_{\mathfrak{l}}\mathbf{Z}_{0,\cdot}^{N})\|_{\omega;2p} \ \lesssim_{p,\gamma} \ N^{-1/2+\gamma}|\mathfrak{l}|^{1/2-\gamma} \ \leq \ N^{\gamma}N^{-1/2}|\mathfrak{l}|^{1/2}. \label{eq:regX21}
\end{align}
The last inequality \eqref{eq:regX21} follows from noting $|\mathfrak{l}|\neq0$ and it is an integer. When we follow the proof of (3.13) in \cite{DT}, we employ the heat kernel estimates in Proposition \ref{prop:heat} for $\mathbf{H}^{N,\mathbf{X}}$ rather than heat kernel estimates in \cite{DT}. The Chebyshev inequality then gives, for $p\geq1$, the following in which given $\vartheta,\kappa>0$, we choose $\gamma>0$ sufficiently small and $p$ sufficiently large, but both depending only on $\vartheta,\kappa$, such that $-2p\vartheta+2p\gamma\leq-2\kappa$:
\begin{align}
\mathbf{P}\left(|\grad^{\mathbf{X}}_{\mathfrak{l}}\mathbf{H}^{N,\mathbf{X}}(\mathbf{Z}_{0,\cdot}^{N})| \geq N^{-1/2+\vartheta}|\mathfrak{l}|^{1/2}\right) \ \lesssim_{p,\gamma} \ N^{p-2p\vartheta}|\mathfrak{l}|^{-p}N^{2p\gamma}N^{-p}|\mathfrak{l}|^{p} \ \leq \ N^{-2\kappa}. \label{eq:regX22}
\end{align}
Again, the dependence on $p,\gamma\in\R_{>0}$ in the previous estimate \eqref{eq:regX22} is now dependence on $\vartheta,\kappa$.
\item Consider the same discretization $\mathbb{I}^{\d}$ from the proof of Lemma \ref{lemma:regT2}. A union bound in the same fashion as that used to prove \eqref{eq:regT23}, when combined with \eqref{eq:regX22}, gives the following; recall $|\mathbb{I}^{\d}|\lesssim N^{100}$ as seen in the proof of Lemma \ref{lemma:regX2}, and $\mathfrak{l}_{N}\leq N$:
\begin{align}
\mathbf{P}\left(\sup_{1\leq|\mathfrak{l}|\leq\mathfrak{l}_{N}}\sup_{(T,x)\in\mathbb{I}^{\d}}N^{\frac12}|\mathfrak{l}|^{-\frac12}|\grad_{\mathfrak{l}}^{\mathbf{X}}\mathbf{H}^{N,\mathbf{X}}(\mathbf{Z}_{0,\cdot}^{N})\geq N^{\vartheta}\right) \ \lesssim_{\vartheta,\kappa} \ N^{-2\kappa}|\mathbb{I}||\llbracket-\mathfrak{l}_{N},\mathfrak{l}_{N}\rrbracket| \ = \ N^{-2\kappa+101}. \label{eq:regX23}
\end{align}
In what follows, we will take $\kappa$ sufficiently large so that $2\kappa-101\geq3\kappa/2$. 
\item We complete the proof via bootstrapping our estimate on $\mathbb{I}^{\d}$ to an estimate on the entire semi-discrete space-time $[0,\mathfrak{t}_{\mathrm{r}}]\times\mathbb{T}_{N}$. To this end, given any $\mathrm{t}\in[0,\mathfrak{t}_{\mathrm{r}}]$, we again let $\mathrm{t}^{\d}$ be any element in $\mathbb{I}^{\mathbf{T},\d}=\{\mathfrak{j}N^{-99}\}_{\mathfrak{j}\geq0}\cap[0,1]$ that minimizes $|\mathrm{t}-\mathrm{t}^{\d}|$. We now provide the following parallel to \eqref{eq:regT26} where $x\in\mathbb{T}_{N}$ is arbitrary:
\begin{align}
\grad_{\mathfrak{l}}^{\mathbf{X}}\mathbf{H}^{N,\mathbf{X}}_{\mathrm{t},x}(\mathbf{Z}_{0,\cdot}^{N}) \ = \ \grad_{\mathfrak{l}}^{\mathbf{X}}\mathbf{H}^{N,\mathbf{X}}_{\mathrm{t}^{\d},x}(\mathbf{Z}_{0,\cdot}^{N}) + \left(\mathbf{H}_{\mathrm{t},x+\mathfrak{l}}^{N,\mathbf{X}}(\mathbf{Z}_{0,\cdot}^{N})-\mathbf{H}_{\mathrm{t}^{\d},x+\mathfrak{l}}^{N,\mathbf{X}}(\mathbf{Z}_{0,\cdot}^{N})\right) + \left(\mathbf{H}_{\mathrm{t}^{\d},x}^{N,\mathbf{X}}(\mathbf{Z}_{0,\cdot}^{N})-\mathbf{H}_{\mathrm{t},x}^{N,\mathbf{X}}(\mathbf{Z}_{0,\cdot}^{N})\right). \label{eq:regX24}
\end{align}
Following the third bullet point in the proof of Lemma \ref{lemma:regT2}, we have an estimate for the first term on the RHS of \eqref{eq:regX24} outside an event of probability at most $N^{-3\kappa/2}$ times $\vartheta,\kappa$-dependent factors. We additionally have deterministic estimates for the second and third terms on the RHS of \eqref{eq:regX24} by short-time continuity; see \eqref{eq:regT25}. This gives the following analog of \eqref{eq:regT27}:
\begin{align}
N^{1/2}|\mathfrak{l}|^{-1/2}\|\mathbf{H}_{\mathrm{t},x+\mathfrak{l}}^{N,\mathbf{X}}(\mathbf{Z}_{0,\cdot}^{N})-\mathbf{H}_{\mathrm{t}^{\d},x+\mathfrak{l}}^{N,\mathbf{X}}(\mathbf{Z}_{0,\cdot}^{N})\|_{\mathfrak{t}_{\mathrm{r}};\mathbb{T}_{N}} \ \lesssim \ N^{5/2}|\mathrm{t}-\mathrm{t}^{\d}|\|\mathbf{Z}^{N}\|_{0;\mathbb{T}_{N}} \ \lesssim \ N^{-96}\|\mathbf{U}^{N}\|_{\mathfrak{t}_{\mathrm{r}};\mathbb{T}_{N}}. \label{eq:regX25}
\end{align}
The first estimate in \eqref{eq:regX25} follows by $|\mathfrak{l}|\geq1$ combined with \eqref{eq:regT25}. The second estimate in \eqref{eq:regX25} follows by $|\mathrm{t}-\mathrm{t}^{\d}|\leq N^{-99}$ and $\|\mathbf{Z}^{N}\|_{0;\mathbb{T}_{N}}\leq\|\mathbf{U}^{N}\|_{\mathfrak{t}_{\mathrm{r}};\mathbb{T}_{N}}$, both of which we used in the third bullet point in the proof of Lemma \ref{lemma:regT2}.
\end{itemize}
We now apply the reasoning of the last paragraph in the third bullet point in the proof of Lemma \ref{lemma:regT2} to finish the proof.
\end{proof}
\begin{lemma}\label{lemma:regX3}
\fsp Consider any $\vartheta,\kappa>0$ arbitrarily small and large, respectively, but both universal. We have
\begin{align}
\mathbf{P}\left(\sup_{1\leq|\mathfrak{l}|\leq\mathfrak{l}_{N}}N^{1/2}|\mathfrak{l}|^{-1/2}\|\grad_{\mathfrak{l}}^{\mathbf{X}}\mathbf{H}_{T,x}^{N}(\mathbf{U}^{N}\d\xi^{N})\|_{\mathfrak{t}_{\mathrm{r}};\mathbb{T}_{N}} \geq N^{\vartheta}(1+\|\mathbf{U}^{N}\|_{\mathfrak{t}_{\mathrm{r}};\mathbb{T}_{N}}^{2})\right) \ \lesssim_{\vartheta,\kappa} \ N^{-\kappa}. \label{eq:regX3}
\end{align}
\end{lemma}
\begin{proof}
We follow the proof of Lemma \ref{lemma:regT3} upon replacing $\mathrm{s}^{-1/4}$ factors by $N^{1/2}|\mathfrak{l}|^{-1/2}$ factors and replacing $\grad_{-\mathrm{s}}^{\mathbf{T}}\mathbf{H}^{N}(\mathbf{U}^{N}\d\xi^{N})$ terms by $\grad_{\mathfrak{l}}^{\mathbf{X}}\mathbf{H}^{N}(\mathbf{U}^{N}\d\xi^{N})$ terms. Precisely, we can first establish \eqref{eq:regT31} with these replacements upon using the same argument given in the proof of Lemma \ref{lemma:regT3}, except instead of following the proof of (3.14) in \cite{DT} we follow the proof of (3.13) in \cite{DT}. Taking a union bound over all length-scales $1\leq|\mathfrak{l}|\leq\mathfrak{l}_{N}$ and all space-time points in the discretization $\mathbb{I}^{\d}$ then gives \eqref{eq:regT32} with the same replacements. The short-time estimate \eqref{eq:regT33} without any replacements lets us bootstrap from an estimate on $\mathbb{I}^{\d}$ to one over the entire set $[0,\mathfrak{t}_{\mathrm{r}}]\times\mathbb{T}_{N}$ in the same fashion as the end of the proof of Lemma \ref{lemma:regT3}. 
\end{proof}
\begin{corollary}\label{corollary:regX}
\fsp Admit the setting of \emph{Proposition \ref{prop:reg}}. We have $\mathbf{P}(\mathcal{E}_{\vartheta}^{\mathbf{X}}(\mathfrak{t}_{\mathrm{r}};\mathbb{T}_{N})) \lesssim_{\vartheta} N^{-\beta_{\mathrm{r}}}$.
\end{corollary}
\begin{proof}
Like in the proof of Corollary \ref{corollary:regT}, first observe $\grad^{\mathbf{X}}\mathbf{U}^{N}$ is controlled by $\grad^{\mathbf{X}}$ of the terms on the RHS of the $\mathbf{U}^{N}$ equation in Definition \ref{definition:KPZ5}. Such $\grad^{\mathbf{X}}$ terms can be controlled by the proposed lower bound in the event $\mathcal{E}_{\vartheta}^{\mathbf{X}}(\mathfrak{t}_{\mathrm{r}};\mathbb{T}_{N})$ with the appropriate probability by applying Lemmas \ref{lemma:regX1}, \ref{lemma:regX2}, and \ref{lemma:regX3}, \emph{except} for the order $N^{1/2}$ term in the $\mathbf{U}^{N}$ equation. For this term, we employ \eqref{eq:BGIIapp}, which, by Theorem \ref{theorem:BGII}, holds with the desired probability of at least $1-\mathrm{O}_{\vartheta}(N^{-\beta_{\mathrm{r}}})$ for $\beta_{\mathrm{r}}>0$ universal.
\end{proof}
\subsection{Proof of Proposition \ref{prop:reg}}
It suffices to combine Corollary \ref{corollary:regT} and Corollary \ref{corollary:regX}. \qed
%
%
%
\section{Proof of Proposition \ref{prop:KPZ10} and Proposition \ref{prop:KPZ11}}\label{section:KPZ1011}
\subsection{Preliminary $\mathbf{Q}^{N}$ Estimates}
Recall that Proposition \ref{prop:KPZ11} proposes a comparison between $\mathbf{U}^{N}$ and $\mathbf{Q}^{N}$. For this, it will be important to ensure $\mathbf{Q}^{N}$ is ``reasonable"; because our proof of Proposition \ref{prop:KPZ10}, which amounts to establishing estimates for $\mathbf{U}^{N}$ and $\mathbf{Z}^{N}$, will use the comparison between $\mathbf{U}^{N}$ and $\mathbf{Q}^{N}$, we will actually \emph{need} to ensure that $\mathbf{Q}^{N}$ is ``reasonable". Before we start with the details of this subsection, let us recall the notions of high/overwhelming probability in Definition \ref{definition:KPZ8}.The first estimate we present is an \emph{upper} bound with respect to $\|\|_{1;\mathbb{T}_{N}}$ with overwhelming probability.
\begin{lemma}\label{lemma:final11}
\fsp Provided any $\vartheta>0$ uniformly bounded from below, we have $\|\mathbf{Q}^{N}\|_{1;\mathbb{T}_{N}}\leq N^{\vartheta}$ with overwhelming probability.
\end{lemma}
\begin{proof}
We start with the following inequality that we explain and justify afterwards. Roughly speaking, the following inequality controls a supremum over the semi-discrete space-time $[0,1]\times\mathbb{T}_{N}$ in terms of one over the discretization $\mathbb{I}^{\d}=\mathbb{I}^{\mathbf{T},\d}\times\mathbb{T}_{N}$ with $\mathbb{I}^{\mathbf{T},\d}=\{\mathfrak{j}N^{-99}\}_{0\leq \mathfrak{j}\leq N^{99}}$ and in terms of short-time estimates for $\mathbf{Q}^{N}$. We emphasize the following estimate is \emph{deterministic}:
\begin{align}
\|\mathbf{Q}^{N}\|_{1;\mathbb{T}_{N}} \ \leq \ \sup_{(T,x)\in\mathbb{I}^{\d}}|\mathbf{Q}_{T,x}^{N}| + \sup_{(T,x)\in\mathbb{I}^{\d}}\sup_{\mathrm{s}\in[0,N^{-99}]}|\grad_{\mathrm{s}}^{\mathbf{T}}\mathbf{Q}_{T,x}^{N}|. \label{eq:final110}
\end{align}
The estimate \eqref{eq:final110} is proved using reasoning similar to that used in the third bullet point from the proof of Lemma \ref{lemma:regT2}. Consider any $\mathrm{t}\in[0,1]$ and let $\mathrm{t}^{\d}\in\mathbb{I}^{\mathbf{T},\d}$ be any element in $\mathbb{I}^{\mathbf{T},\d}$ that minimizes $|\mathrm{t}-\mathrm{t}^{\d}|$. We will write $\mathbf{Q}^{N}$ evaluated at $\mathrm{t}$ as $\mathbf{Q}^{N}$ evaluated at $\mathrm{t}^{\d}$ plus the corresponding difference, the difference being a time-gradient on a time-scale $\mathrm{t}-\mathrm{t}^{\d}$ evaluated at $\mathrm{t}^{\d}\in\mathbb{I}^{\d}$. Therefore, we are left with estimating each term on the RHS of \eqref{eq:final110}. Observe that it suffices to estimate each by $2^{-1}N^{\vartheta}$ with overwhelming probability, as the intersection of two events, both of which hold with overwhelming probability, also holds with overwhelming probability itself, which is a consequence of the union bound. Moreover, according to Lemma \ref{lemma:st2}, the second term on the RHS of \eqref{eq:final110} is at most a small multiple of the first term on the RHS of \eqref{eq:final110} with overwhelming probability. Thus, it suffices to bound from above the first term on the RHS of \eqref{eq:final110} by $4^{-1}N^{\vartheta}$ with overwhelming probability. For this, we use moment bounds for $\mathbf{Q}^{N}$ resembling those for the Gartner transform from Proposition 3.2 in \cite{DT}.

Now recall from the proof of Lemma \ref{lemma:KPZ13} that for any $p\geq1$, the $2p$-moment of $\mathbf{Q}^{N}$ at any point in $[0,1]\times\mathbb{T}_{N}$ is bounded by a constant depending only on $p$. This follows via the observation that $\mathbf{Q}^{N}$ satisfies the moment estimate (3.12) in \cite{DT} if we remove the sub-exponential weights therein, which we made at the beginning of the proof of Lemma \ref{lemma:KPZ13}. Therefore, we have the following estimate by a union bound, the Cheybshev inequality, and this moment estimate for $\mathbf{Q}^{N}$:
\begin{align}
\mathbf{P}\left(\sup_{(T,x)\in\mathbb{I}^{\d}}|\mathbf{Q}_{T,x}^{N}|\geq4^{-1}N^{\vartheta}\right) \ \leq \ |\mathbb{I}^{\d}|\sup_{(T,x)\in\mathbb{I}^{\d}}\mathbf{P}\left(|\mathbf{Q}_{T,x}^{N}|\geq4^{-1}N^{\vartheta}\right) \ \lesssim_{p} \ |\mathbb{I}^{\d}|N^{-2p\vartheta} \ \lesssim \ N^{-2p\vartheta+100}. \label{eq:final111}
\end{align}
The final estimate in \eqref{eq:final111} follows from the straightforward observation $|\mathbb{I}^{\d}|=|\mathbb{I}^{\mathbf{T},\d}||\mathbb{T}_{N}|\lesssim N^{100}$. We can choose $p\geq1$ arbitrarily large but depending only on the fixed $\vartheta>0$ so that the complement of the event in the probability on the far LHS of \eqref{eq:final111} holds with overwhelming probability. Therefore, we have proved the lemma if we replace $\|\|_{1;\mathbb{T}_{N}}$ with the supremum over the discrete space-time set $\mathbb{I}^{\d}$, which completes the proof as noted after \eqref{eq:final110}.
\end{proof}
The second ingredient we present for this subsection is a \emph{lower} bound, or equivalently an upper bound for the inverse of $\mathbf{Q}^{N}$ with respect to the $\|\|_{1;\mathbb{T}_{N}}$ norm. We clarify that the following estimate holds with high probability, in contrast to the upper bound in Lemma \ref{lemma:final11} that holds with overwhelming probability. This is because the upcoming proof is slightly less quantitative. We also emphasize the importance of stable initial data for the upcoming lower bound estimate.
\begin{lemma}\label{lemma:final12}
\fsp Provided any $\vartheta>0$ uniformly bounded from below, we have $\|(\mathbf{Q}^{N})^{-1}\|_{1;\mathbb{T}_{N}}\leq N^{\vartheta}$ with high probability.
\end{lemma}
\begin{proof}
Observe that a space-time uniform upper bound of $N^{\vartheta}$ for the inverse of $\mathbf{Q}^{N}$ is equivalent to a space-time uniform lower bound of $N^{-\vartheta}$ for $\mathbf{Q}^{N}$ itself. Additionally, we observe that the initial data of $\mathbf{Q}^{N}$, which is the initial data of the Gartner transform $\mathbf{Z}^{N}$, is uniformly bounded above and below on $\mathbb{T}_{N}$, because it is the exponential of a function that is uniformly bounded above and below. Lemma \ref{lemma:final12} then follows from a standard analysis based on combining these observations, the comparison principle for the defining equation of $\mathbf{Q}^{N}$ in Definition \ref{definition:KPZ7}, and tightness estimates for the same defining equation, at least with continuous initial data, that we alluded to in Proposition \ref{prop:KPZ12}. Roughly speaking, because $\mathbf{Q}^{N}$ is initially uniformly bounded below, by the comparison principle for the $\mathbf{Q}^{N}$ equation it suffices to prove uniform lower bounds for constant initial data given by the infimum of the $\mathbf{Q}^{N}$ initial data. If the $\mathbf{Q}^{N}$ equation only had a spatial heat operator, constant data would be preserved and the result would follow. For some short but $N$-independent time $\mathfrak{t}^{+}$ and for any $N$-independent $\gamma>0$, perturbative analysis would provide a high probability $\mathfrak{t}^{+},\gamma$-dependent lower bound for $\mathbf{Q}^{N}$. Again, by using the comparison principle, it suffices to provide a uniform lower bound for time $1-\mathfrak{t}^{+}$ for the solution to the $\mathbf{Q}^{N}$ equation with initial data given by the $\mathfrak{t}^{+},\gamma$-dependent space-time infimum/lower bound. We then iterate this scheme, namely by providing high probability lower bounds for constant data for sufficiently short but $N$-independent times $\mathfrak{t}^{+}$, requiring only a $\mathfrak{t}^{+}$-dependent number of steps. We emphasize that this iteration does not break down because each perturbative step in this strategy amounts to estimates for the $\mathbf{Q}^{N}$ with constant initial data. By linearity of the $\mathbf{Q}^{N}$ equation, the value of this constant initial data does not matter in terms of how much smaller, in a proportional/multiplicative sense, the solution is after a short time $\mathfrak{t}^{+}$; namely, our analysis of the $\mathbf{Q}^{N}$ equation does not change in each step even if the constant initial data is different between steps. This completes the proof.
\end{proof}
\subsection{Proof of Proposition \ref{prop:KPZ11}}
The first step that we take is to write the difference $\mathbf{D}^{N}=\mathbf{U}^{N}-\mathbf{Q}^{N}$ explicitly in terms of the difference between the respective stochastic equations for $\mathbf{U}^{N}$ in Definition \ref{definition:KPZ5} and $\mathbf{Q}^{N}$ in Definition \ref{definition:KPZ7}. Because the heat operators are linear, and because the stochastic equations in Definitions \ref{definition:KPZ5} and \ref{definition:KPZ7} are both linear in their respective solutions $\mathbf{U}^{N}$ and $\mathbf{Q}^{N}$, it is straightforward to verify the following stochastic equation for $\mathbf{D}^{N}$, in which the spatial heat operators/initial data terms in the $\mathbf{U}^{N}$ and $\mathbf{Q}^{N}$ equations cancel, and in which we use notation in Definitions \ref{definition:KPZ5} and \ref{definition:KPZ7}:
\begin{align}
\mathbf{D}_{T,x}^{N} \ &= \ \mathbf{H}_{T,x}^{N}(\mathbf{D}^{N}\d\xi^{N}) - \mathbf{H}_{T,x}^{N}(N^{\frac12}\bar{\mathfrak{q}}\mathbf{Y}^{N}) - \mathbf{H}_{T,x}^{N}(\mathfrak{s}\mathbf{D}^{N}) + N^{-\frac12}\mathbf{H}_{T,x}^{N}(\mathfrak{b}_{1;}\mathbf{D}^{N}) + N^{-\frac12}\mathbf{H}_{T,x}^{N}\left(\grad_{\star}^{!}\left(\mathfrak{b}_{2;}\mathbf{D}^{N}\right)\right). \label{eq:KPZ121}
\end{align}
We emphasize that the order $N^{1/2}$ term in the $\mathbf{U}^{N}$ equation does not have a matching term in the $\mathbf{Q}^{N}$ equation. According to the Boltzmann-Gibbs principle in Theorem \ref{theorem:BGI}, we expect the second term from the RHS of \eqref{eq:KPZ121} to vanish in the large-$N$ limit. In this case, the $\mathbf{D}^{N}$ term solves a linear equation with zero initial data and vanishing small ``forcing", which suggests that $\mathbf{D}^{N}$ vanishes uniformly in the large-$N$ limit. To actually prove this, we will obtain bounds for $\mathbf{D}^{N}$ by employing the moment strategy for proofs of tightness in \cite{BG,DT}. We note, however, Theorem \ref{theorem:BGI} only provides an estimate with respect to first moment, whereas the aforementioned SPDE analysis of \cite{BG,DT} rely on bounds for quite high moments. This is technical, but it is also nontrivial to resolve. The first step we take to resolve it is introducing the following stopping time and ``cutoff"-type $\mathbf{C}^{N}$ process.
\begin{definition}\label{definition:KPZ111}
\fsp Recall the universal constant $\beta_{\mathrm{BG}}>0$ from {Theorem \ref{theorem:BGI}}; recall it is uniformly bounded from below. Define
\begin{align}
\mathfrak{t}_{\mathrm{BG}} \ \overset{\bullet}= \ \inf\left\{\mathrm{t}\in[0,1]: \ \|\mathbf{H}^{N}(N^{1/2}\bar{\mathfrak{q}}\mathbf{Y}^{N})\|_{\mathrm{t};\mathbb{T}_{N}} \geq N^{-\beta_{\mathrm{BG}}/999}\right\}\wedge1.
\end{align}
We additionally define the stopped process $\wt{\mathbf{Y}}^{N}=\mathbf{Y}^{N}\mathbf{1}(T\leq\mathfrak{t}_{\mathrm{BG}})$, and we also define $\mathbf{C}^{N}$ to be the solution to the following stochastic equation on $\R_{\geq0}\times\mathbb{T}_{N}$, whose solutions are unique by standard linear theory:
\begin{align}
\mathbf{C}_{T,x}^{N} \ = \ \mathbf{H}_{T,x}^{N}(\mathbf{C}^{N}\d\xi^{N}) - \mathbf{H}_{T,x}^{N}(N^{\frac12}\bar{\mathfrak{q}}\wt{\mathbf{Y}}^{N}) - \mathbf{H}_{T,x}^{N}(\mathfrak{s}\mathbf{C}^{N}) + N^{-\frac12}\mathbf{H}_{T,x}^{N}(\mathfrak{b}_{1;}\mathbf{C}^{N}) + N^{-\frac12}\mathbf{H}_{T,x}^{N}\left(\grad_{\star}^{!}\left(\mathfrak{b}_{2;}\mathbf{C}^{N}\right)\right), \nonumber
\end{align}
where $\grad_{\star}^{!}$ means what it does in Proposition \ref{prop:mSHE+}.
\end{definition}
The following first result in this subsection justifies analyzing $\mathbf{C}^{N}$ as a ``high probability proxy" for $\mathbf{D}^{N}$.
\begin{lemma}\label{lemma:KPZ111}
\fsp With high probability, we have $\mathfrak{t}_{\mathrm{BG}}=1$. Thus, with high probability, we have $\mathbf{C}^{N}=\mathbf{D}^{N}$ on $[0,1]\times\mathbb{T}_{N}$.
\end{lemma}
\begin{proof}
We emphasize that the second high probability claim in Lemma \ref{lemma:KPZ111} follows from the first high probability claim in Lemma \ref{lemma:KPZ111} by the same reason Lemma \ref{lemma:KPZ6} holds; the terms $\mathbf{D}^{N}$ and $\mathbf{C}^{N}$ solve the same stochastic equation, whose solutions are unique, until the time $\mathfrak{t}_{\mathrm{BG}}$, which according to the first claim is equal to $1$ with high probability. To prove the first high probability claim in Lemma \ref{lemma:KPZ111}, note $\mathfrak{t}_{\mathrm{BG}}\neq1$ implies $\mathfrak{t}_{\mathrm{BG}}<1$, as $\mathfrak{t}_{\mathrm{BG}}\leq1$ deterministically. We now claim $\mathfrak{t}_{\mathrm{BG}}<1$ implies
\begin{align}
\|\mathbf{H}^{N}(N^{1/2}\bar{\mathfrak{q}}\mathbf{Y}^{N})\|_{1;\mathbb{T}_{N}} \ \geq \ \|\mathbf{H}^{N}(N^{1/2}\bar{\mathfrak{q}}\mathbf{Y}^{N})\|_{\mathfrak{t}_{\mathrm{BG}};\mathbb{T}_{N}} \ \geq \ N^{-\beta_{\mathrm{BG}}/999}. \label{eq:KPZ1111}
\end{align}
The first inequality in \eqref{eq:KPZ1111} follows trivially because $\mathfrak{t}_{\mathrm{BG}}\leq1$ deterministically. To justify the second inequality in \eqref{eq:KPZ1111}, we assume the opposite, so the final inequality in \eqref{eq:KPZ1111} is reversed and strict. The heat operator is continuous in time with probability 1 as the product $N^{1/2}\bar{\mathfrak{q}}\mathbf{Y}^{N}$ is \emph{finite}, even if not uniformly bounded in $N$; we emphasize we are not claiming quantitative regularity of the heat operator that is controlled in the large-$N$ limit by any means. If the last inequality in \eqref{eq:KPZ1111} is reversed while strict, then for $\mathfrak{t}_{\mathrm{BG}}<1$ we may find $\mathrm{t}\in(\mathfrak{t}_{\mathrm{BG}},1]$ by continuity of the heat operator such that
\begin{align}
\sup_{\mathfrak{t}_{\mathrm{BG}}\leq\mathrm{s}\leq\mathrm{t}}\|\mathbf{H}^{N}(N^{1/2}\bar{\mathfrak{q}}\mathbf{Y}^{N})\|_{\mathrm{s};\mathbb{T}_{N}} \ < \ N^{-\beta_{\mathrm{BG}}/999}.
\end{align}
Because we have assumed the reverse of the last inequality in \eqref{eq:KPZ1111}, this implies $\mathfrak{t}_{\mathrm{BG}}$ must actually be at least $\mathrm{t}$ because until time $\mathrm{t}$, the $\mathbf{H}^{N}(N^{1/2}\bar{\mathfrak{q}}\mathbf{Y}^{N})$ term is strictly less than $N^{-\beta_{\mathrm{BG}}/999}$, and therefore by continuity in time of this $\mathbf{H}^{N}(N^{1/2}\bar{\mathfrak{q}}\mathbf{Y}^{N})$ term, we may ``wait" a positive amount after $\mathrm{t}$ to see $\mathbf{H}^{N}(N^{1/2}\bar{\mathfrak{q}}\mathbf{Y}^{N})$ exceed $N^{-\beta_{\mathrm{BG}}/999}$. The condition $\mathfrak{t}_{\mathrm{BG}}\geq\mathrm{t}$ we have just established contradicts the fact that $\mathrm{t}>\mathfrak{t}_{\mathrm{BG}}$ by construction. This provides the last inequality in \eqref{eq:KPZ1111}. 

We now recap that $\mathfrak{t}_{\mathrm{BG}}\neq1$ implies $\mathfrak{t}_{\mathrm{BG}}<1$, which in turn implies the inequalities \eqref{eq:KPZ1111}. Therefore, to prove $\mathfrak{t}_{\mathrm{BG}}=1$ occurs with high probability, we estimate the probability of observing the inequalities in \eqref{eq:KPZ1111}. This is at most $\mathrm{O}(N^{-\beta_{\mathrm{BG}}/99})$ by Theorem \ref{theorem:BGI}, which estimates the expectation of the far LHS of \eqref{eq:KPZ1111}, and the Markov inequality.
\end{proof}
We now introduce the following \emph{deterministic} estimate that explains the utility of introducing $\mathfrak{t}_{\mathrm{BG}}$ and $\wt{\mathbf{Y}}^{N}$ and $\wt{\mathfrak{D}}^{N}$.
\begin{lemma}\label{lemma:KPZ112}
\fsp We have the deterministic estimate $\|\mathbf{H}^{N}(N^{1/2}\bar{\mathfrak{q}}\wt{\mathbf{Y}}^{N})\|_{1;\mathbb{T}_{N}}\leq N^{-\beta_{\mathrm{BG}}/999}$.
\end{lemma}
\begin{proof}
Because $\mathfrak{t}_{\mathrm{BG}}\leq1$ deterministically, and because $\wt{\mathbf{Y}}^{N}$ vanishes after time $\mathfrak{t}_{\mathrm{BG}}$ by construction, we may employ Proposition \ref{prop:heat} to deduce the following consequence of $\mathscr{L}^{\infty}$-contractive property and the semigroup property of spatial heat operators; see the proof of Lemma \ref{lemma:BGII2}, namely \eqref{eq:BGII22} therein, for a ``gradient" version of the following:
\begin{align}
\|\mathbf{H}^{N}_{T,x}(N^{1/2}\bar{\mathfrak{q}}\wt{\mathbf{Y}}^{N})\|_{1;\mathbb{T}_{N}} \ = \ \|\mathbf{H}^{N,\mathbf{X}}_{(T-\mathfrak{t}_{\mathrm{BG}})\vee0,x}(\mathbf{H}_{\mathrm{t}_{\mathrm{BG}},\cdot}^{N}(N^{1/2}\bar{\mathfrak{q}}\wt{\mathbf{Y}}^{N}))\|_{1;\mathbb{T}_{N}} \ \leq \ \|\mathbf{H}^{N}(N^{1/2}\bar{\mathfrak{q}}\wt{\mathbf{Y}}^{N})\|_{\mathfrak{t}_{\mathrm{BG}};\mathbb{T}_{N}}. \label{eq:KPZ1121}
\end{align}
It now suffices to observe the RHS of \eqref{eq:KPZ1121} is bounded above by $N^{-\beta_{\mathrm{BG}}/999}$, because this proposed upper bound is true if we replace $\mathfrak{t}_{\mathrm{BG}}$ with any $\mathrm{t}<\mathfrak{t}_{\mathrm{BG}}$ and, like in the proof of Lemma \ref{lemma:KPZ111}, the heat operator $\mathbf{H}^{N}$ is continuous in space-time.
\end{proof}
The last ingredient we require for the proof for Proposition \ref{prop:KPZ11} is a pointwise moment estimate for $\mathbf{C}^{N}$ which is proved with stochastic analytic means like those used in the proof of Proposition 3.2 in \cite{DT} for the Gartner transform therein. Afterwards, we will ``glue" this pointwise estimate to a uniform estimate on $[0,1]\times\mathbb{T}_{N}$ via union bound and continuity.
\begin{lemma}\label{lemma:KPZ113}
\fsp Consider any $p\geq1$. We have the estimate $\|\mathbf{C}^{N}_{T,x}\|_{\omega;2p}\lesssim_{p}N^{-\beta_{\mathrm{BG}}/999}$ uniformly on $[0,1]\times\mathbb{T}_{N}$.
\end{lemma}
\begin{proof}
We estimate the $\|\|_{\omega;2p}^{2}$ squared norm for every term on the RHS of the $\mathbf{C}^{N}$ equation from Definition \ref{definition:KPZ111}. We first employ Lemma \ref{lemma:KPZ112} to establish the following estimate uniformly in $p\geq1$ and in uniformly in space-time; let us clarify the first bound below uses an elementary/general $\mathscr{L}^{p}\leq\mathscr{L}^{\infty}$ bound for random variables:
\begin{align}
\|\mathbf{H}^{N}_{T,x}(N^{1/2}\bar{\mathfrak{q}}\wt{\mathbf{Y}}^{N})\|_{\omega;2p}^{2} \ \leq \ \|\mathbf{H}^{N}_{T,x}(N^{1/2}\bar{\mathfrak{q}}\wt{\mathbf{Y}}^{N})\|_{\omega;\infty}^{2} \ \lesssim \ N^{-2\beta_{\mathrm{BG}}/999}. \label{eq:KPZ1131}
\end{align}
For the remaining terms in the $\mathbf{C}^{N}$ equation in Definition \ref{definition:KPZ111}, we will follow the proof of (3.12) in Proposition 3.2 in \cite{DT}. Similar to the proof of Lemma \ref{lemma:KPZ13}, all of the estimates used to prove (3.12) in Proposition 3.2 in \cite{DT} hold for the corresponding terms in the $\mathbf{C}^{N}$ equation from Definition \ref{definition:KPZ111} as we only need the heat kernel estimates in Proposition \ref{prop:heat} and, to control the $\mathbf{C}^{N}\d\xi^{N}$ term in the $\mathbf{C}^{N}$ equation, the martingale inequality in Lemma \ref{lemma:mg} that generalizes Lemma 3.1 in \cite{DT} beyond the Gartner transform. We ultimately deduce from \eqref{eq:KPZ1131} and this paragraph the following integral bound for $\|\|_{\omega;2p}^{2}$; recall that $\mathbf{O}_{S,T}=|T-S|$:
\begin{align}
\|\mathbf{C}_{T,x}^{N}\|_{\omega;2p}^{2} \ \lesssim_{p} \ N^{-2\beta_{\mathrm{BG}}/999} + \int_{0}^{T}\sup_{y\in\mathbb{T}_{N}}\|\mathbf{C}_{S,y}^{N}\|_{\omega;2p}^{2} \d S + \int_{0}^{T}\mathbf{O}_{S,T}^{-1/2}\sup_{y\in\mathbb{T}_{N}}\|\mathbf{C}_{S,y}^{N}\|_{\omega;2p}^{2} \d S. \label{eq:KPZ1132}
\end{align}
We emphasize the estimate \eqref{eq:KPZ1132} holds uniformly in space-time on the LHS. Thus we may extend \eqref{eq:KPZ1132} upon taking a supremum over $\mathbb{T}_{N}$ on the LHS therein. At this point, we may employ the Gronwall inequality to deduce, for times $T\leq1$:
\begin{align}
\sup_{x\in\mathbb{T}_{N}}\|\mathbf{C}_{T,x}^{N}\|_{\omega;2p}^{2} \ \lesssim_{p} \ N^{-2\beta_{\mathrm{BG}}/999}\exp\left(\sup_{0\leq\mathrm{t}\leq1}\int_{0}^{\mathrm{t}}\d S + \sup_{0\leq\mathrm{t}\leq1}\int_{0}^{\mathrm{t}}\mathbf{O}_{S,\mathrm{t}}^{-1/2}\d S\right) \ \lesssim \ N^{-2\beta_{\mathrm{BG}}/999}, \label{eq:KPZ1133}
\end{align}
with the last inequality above following by an elementary integral calculation inside the exponential in the middle of \eqref{eq:KPZ1133}.
\end{proof}
From Lemma \ref{lemma:KPZ113}, we get the following union bound estimate that controls $\mathbf{C}^{N}$ over a very fine discretization of space-time.
\begin{corollary}\label{corollary:KPZ114}
\fsp Define $\mathbb{I}^{\d}=\mathbb{I}^{\mathbf{T},\d}\times\mathbb{T}_{N}$, in which $\mathbb{I}^{\mathbf{T},\d}=\{\mathfrak{j}N^{-99}\}_{\mathfrak{j}=0}^{N^{99}}$. The following holds with overwhelming probability:
\begin{align}
\sup_{(\mathrm{t},x)\in\mathbb{I}^{\d}}|\mathbf{C}_{\mathrm{t},x}^{N}| \ \leq \ N^{-\beta_{\mathrm{BG}}/99999}. \label{eq:KPZ114}
\end{align}
\end{corollary}
\begin{proof}
Provided any $(\mathrm{t},x)\in\mathbb{I}^{\d}$, the Chebyshev inequality implies the probability estimate
\begin{align}
\mathbf{P}\left(|\mathbf{C}_{\mathrm{t},x}^{N}|\geq N^{-\beta_{\mathrm{BG}}/99999}\right) \ \leq \ N^{2p\beta_{\mathrm{BG}}/99999}\|\mathbf{C}_{\mathrm{t},x}^{N}\|_{\omega;2p}^{2p} \ \lesssim_{p} \ N^{-2p\beta_{\mathrm{BG}}/999+2p\beta_{\mathrm{BG}}/99999} \ \leq \ N^{-2p\beta_{\mathrm{BG}}/9999}. \label{eq:KPZ1141}
\end{align}
Therefore, a union bound implies that the probability the proposed estimate \eqref{eq:KPZ114} fails is bounded by $\kappa_{p}N^{-2p\beta_{\mathrm{BG}}/9999}$ times the number $|\mathbb{I}^{\d}|$ of points we take a union bound over, in which $\kappa_{p}\geq1$ depends only on $p\geq1$. As $|\mathbb{I}^{\d}|=|\mathbb{I}^{\mathbf{T},\d}||\mathbb{T}_{N}|\lesssim N^{100}$, we establish the probability the proposed estimate in \eqref{eq:KPZ114} fails is bounded by $\kappa_{p}N^{-2p\beta_{\mathrm{BG}}/9999+100}$. Taking $p\geq1$ arbitrarily large implies \eqref{eq:KPZ114} holds with overwhelming probability.
\end{proof}
\begin{proof}[Proof of \emph{Proposition \ref{prop:KPZ11}}]
Throughout, observe that the intersection of any uniformly bounded number of events that hold with high probability also holds with high probability, which can be easily shown with the union bound for the complements of these events. With this, by Lemma \ref{lemma:KPZ111}, it suffices to prove Proposition \ref{prop:KPZ11} but replacing $\mathbf{D}^{N}$ by $\mathbf{C}^{N}$. Next, we employ the estimate \eqref{eq:final110}, bootstrapping an estimate from a discretization to the continuum, but for $\mathbf{C}^{N}$ in place of $\mathbf{Q}^{N}$; observe the proof of \eqref{eq:final110} is blind to what $\mathbf{Q}^{N}$ actually is:
\begin{align}
\|\mathbf{C}^{N}\|_{1;\mathbb{T}_{N}} \ \leq \ \sup_{(T,x)\in\mathbb{I}^{\d}}|\mathbf{C}_{T,x}^{N}| + \sup_{(T,x)\in\mathbb{I}^{\d}}\sup_{\mathrm{s}\in[0,N^{-99}]}|\grad_{\mathrm{s}}^{\mathbf{T}}\mathbf{C}_{T,x}^{N}|. \label{eq:KPZ111}
\end{align}
It suffices to estimate each term from the RHS of \eqref{eq:KPZ111} by $N^{-\beta}$ times a universal constant with high probability. For the first term on the RHS of \eqref{eq:KPZ111}, we employ Corollary \ref{corollary:KPZ114}. For the second term, we employ Lemma \ref{lemma:st2}, which implies the second term on the RHS of \eqref{eq:KPZ111} is controlled by the first term on the RHS of \eqref{eq:KPZ111}. This finishes the proof of Proposition \ref{prop:KPZ11}.
\end{proof}
\subsection{Proof of Proposition \ref{prop:KPZ10}}
We use a continuity method that is frequently used in the study of PDE. Roughly speaking, we observe that for stable initial data, regularity estimates defining the stopping time $\mathfrak{t}_{\mathrm{st}}$ of current interest from Definition \ref{definition:KPZ1} are satisfied at time 0, at least with high probability. We then condition on \emph{path-space} events in which $\mathbf{U}^{N}$ admits sufficiently good upper bounds, which will be inherited by sufficiently good upper and lower bounds for $\mathbf{Q}^{N}$ and $\mathbf{D}^{N}=\mathbf{U}^{N}-\mathbf{Q}^{N}$. We also condition on \emph{path-space} events in which the space-time regularity of $\mathbf{U}^{N}$ is sufficiently good provided upper and lower bounds; this will follow from the probability estimates in Proposition \ref{prop:reg}. In particular, until time $\mathfrak{t}_{\mathrm{st}}$, we basically know the space-time estimates defining $\mathfrak{t}_{\mathrm{st}}$ with high probability except upon replacing $\mathbf{Z}^{N}$ in there with $\mathbf{U}^{N}$. However, Lemma \ref{lemma:KPZ6} implies that since we look at time before $\mathfrak{t}_{\mathrm{st}}$, we do not actually have to make such replacement. Now, if $\mathfrak{t}_{\mathrm{st}}\neq1$, in which case $\mathfrak{t}_{\mathrm{st}}<1$, we may apply the short-time estimates in Lemma \ref{lemma:st2} to push the space-time estimates in $\mathfrak{t}_{\mathrm{st}}$ for $\mathbf{Z}^{N}$ past $\mathfrak{t}_{\mathrm{st}}$ by a very small amount of time, thus contradicting the definition of $\mathfrak{t}_{\mathrm{st}}$ similar to our proof of Lemma \ref{lemma:KPZ111}. We clarify that the crux of the strategy is the observation that we can turn slightly suboptimal estimates for $\mathbf{Z}^{N}$ into slightly better suboptimal estimates, which are closer to ``the truth". This is because we need the a priori suboptimal estimates in $\mathfrak{t}_{\mathrm{st}}$ only to analyze the $\bar{\mathfrak{q}}$ term in the $\mathbf{Z}^{N}$ and $\mathbf{U}^{N}$ equations, and such term is vanishingly small anyway with respect to space-time regularity norms in $\mathfrak{t}_{\mathrm{st}}$. Therefore the space-time behavior of the $\bar{\mathfrak{q}}$ term in the $\mathbf{Z}^{N}$ and $\mathbf{U}^{N}$ equations is, with high probability, better than the space-time behavior of $\mathbf{Z}^{N}$ and $\mathbf{U}^{N}$ that we assume through the stopping time $\mathfrak{t}_{\mathrm{st}}$, while the other terms in the $\mathbf{Z}^{N}$ and $\mathbf{U}^{N}$ equations admit ``good" space-time estimates by standard moment bounds as in \cite{BG,DT}. To make this precise, we introduce another set of stopping times.
\begin{definition}\label{definition:KPZ101}
\fsp We define $\e_{\mathrm{ap},1}=999^{-999}\e_{\mathrm{ap}}\wedge999^{-999}\beta$, where $\e_{\mathrm{ap}}>0$ is from Definition \ref{definition:KPZ1} and $\beta>0$ is the universal constant in Proposition \ref{prop:KPZ11}. We now define the following pair of stopping times, the first of which provides uniform upper and lower bounds for $\mathbf{Q}^{N}$, and where the second stopping time below provides a uniform upper bound for $\mathbf{D}^{N}$:
\begin{align}
\mathfrak{t}_{\mathrm{st},1} \ &= \ \inf\left\{\mathrm{t}\in[0,1]: \ \|\mathbf{Q}^{N}\|_{\mathrm{t};\mathbb{T}_{N}}+\|(\mathbf{Q}^{N})^{-1}\|_{\mathrm{t};\mathbb{T}_{N}}\geq N^{\e_{\mathrm{ap},1}}\right\}\wedge1 \quad \mathrm{and} \quad \mathfrak{t}_{\mathrm{st},2} \ = \ \inf\left\{\mathrm{t}\in[0,1]: \ \|\mathbf{D}^{N}\|_{\mathrm{t};\mathbb{T}_{N}} \geq N^{-\beta/2}\right\}\wedge1. \nonumber
\end{align}
We proceed with the following time regularity stopping time, in which $\mathbb{I}^{\mathbf{T}}$ is the set of discrete mesoscopic time-scales that were defined in Definition \ref{definition:KPZ1}; we use the same exponent $\e_{\mathrm{ap},1}$ below as we did for $\mathfrak{t}_{\mathrm{st},1}$ and $\mathfrak{t}_{\mathrm{st},2}$:
\begin{align}
\mathfrak{t}_{\mathrm{st},3} \ &= \ \inf\left\{\mathrm{t}\in[0,1]: \ \sup_{\mathrm{s}\in\mathbb{I}^{\mathbf{T}}}\mathrm{s}^{-1/4}\|\grad_{-\mathrm{s}}^{\mathbf{T}}\mathbf{U}^{N}\|_{\mathrm{t};\mathbb{T}_{N}} \geq N^{\e_{\mathrm{ap},1}}(1+\|\mathbf{U}^{N}\|_{\mathrm{t};\mathbb{T}_{N}}^{2})\right\}\wedge1.
\end{align}
We additionally define the following spatial regularity stopping time, where $\mathfrak{l}_{N}$ is the maximal length-scale for spatial gradients that was used in the stopping time $\mathfrak{t}_{\mathrm{st}}$ in Definition \ref{definition:KPZ1}. We again use the exponent $\e_{\mathrm{ap},1}$ below as we did for $\mathfrak{t}_{\mathrm{st},1}$, $\mathfrak{t}_{\mathrm{st},2}$, and $\mathfrak{t}_{\mathrm{st},3}$:
\begin{align}
\mathfrak{t}_{\mathrm{st},4} \ &= \ \inf\left\{\mathrm{t}\in[0,1]: \ \sup_{1\leq|\mathfrak{l}|\leq\mathfrak{l}_{N}}N^{1/2}|\mathfrak{l}|^{-1/2}\|\grad_{\mathfrak{l}}^{\mathbf{X}}\mathbf{U}^{N}\|_{\mathrm{t};\mathbb{T}_{N}} \geq N^{\e_{\mathrm{ap},1}}(1+\|\mathbf{U}^{N}\|_{\mathrm{t};\mathbb{T}_{N}}^{2})\right\}\wedge1.
\end{align}
We proceed with defining the following a priori short-time estimate random time for $\mathbf{Z}^{N}$. We emphasize that the following time is \emph{not} a stopping time as it looks forward in the future and thus it is not adapted to the filtration of the interacting particle system. However, this will not be important as our analysis in this section is deterministic after we have established Lemma \ref{lemma:KPZ102} below:
\begin{align}
\mathfrak{t}_{\mathrm{st},5} \ &= \ \inf\left\{\mathrm{t}\in[0,1]: \ \sup_{\mathrm{s}\in[0,N^{-99}]}\sup_{0\leq\mathrm{t}_{0}\leq\mathrm{t}}\sup_{x\in\mathbb{T}_{N}}\|\mathbf{Z}^{N}\|_{\mathrm{t}_{0};\mathbb{T}_{N}}^{-1}|\grad_{\mathrm{s}}^{\mathbf{T}}\mathbf{Z}_{\mathrm{t}_{0},x}^{N}| \geq N^{-1/2+\e_{\mathrm{ap},1}}\right\}\wedge1.
\end{align}
We conclude by defining $\mathfrak{t}_{\mathrm{st},6}=\mathfrak{t}_{\mathrm{st},1}\wedge\mathfrak{t}_{\mathrm{st},2}\wedge\mathfrak{t}_{\mathrm{st},3}\wedge\mathfrak{t}_{\mathrm{st},4}\wedge\mathfrak{t}_{\mathrm{st},5}$.
\end{definition}
\begin{lemma}\label{lemma:KPZ102}
\fsp With high probability, we have $\mathfrak{t}_{\mathrm{st},6}=1$.
\end{lemma}
\begin{proof}
As remarked at the beginning of the proof for Proposition \ref{prop:KPZ11}, the intersection of a uniformly bounded number of events that hold with high probability also holds with high probability. Therefore, it suffices to prove that $\mathfrak{t}_{\mathrm{st},\mathfrak{j}}=1$ with high probability for any $\mathfrak{j}\in\{1,\ldots,6\}$. For $\mathfrak{j}=1$, we first observe that the $\|\|_{\mathrm{t};\mathbb{T}_{N}}$ norm is monotone non-decreasing in $\mathrm{t}$. Thus, because $\mathfrak{t}_{\mathrm{st},1}\leq1$, if $\mathfrak{t}_{\mathrm{st},1}\neq1$ then $\mathfrak{t}_{\mathrm{st},1}<1$, so the lower bound defining $\mathfrak{t}_{\mathrm{st},1}$ is actually realized for some $\mathrm{t}\in[0,1)$, and thus the upper bounds in Lemma \ref{lemma:final11} and Lemma \ref{lemma:final12} fail on this event. By the probability estimates in Lemmas \ref{lemma:final11} and \ref{lemma:final12}, such failure happens outside an event of high probability, so we have $\mathfrak{t}_{\mathrm{st},1}=1$ with high probability. A similar argument but when using Proposition \ref{prop:KPZ11} in place of Lemmas \ref{lemma:final11} and \ref{lemma:final12} shows $\mathfrak{t}_{\mathrm{st},2}=1$ with high probability as well. We proceed with showing $\mathfrak{t}_{\mathrm{st},3}=1$ with probability. Consider Proposition \ref{prop:reg} for the choice of stopping time $\mathfrak{t}_{\mathrm{r}}=\mathfrak{t}_{\mathrm{st},3}$. Let us assume that $\mathfrak{t}_{\mathrm{st},3}\neq1$, and thus like the previous argument we have $\mathfrak{t}_{\mathrm{st},3}<1$ for this event. Also similar to the previous argument, note if $\mathfrak{t}_{\mathrm{st},3}<1$, then the lower bound defining $\mathfrak{t}_{\mathrm{st},3}$ is realized at the time $\mathfrak{t}_{\mathrm{st},3}$. In particular, through the time-regularity estimate in Proposition \ref{prop:reg} we know that this only happens outside an event that happens with high probability, and thus $\mathfrak{t}_{\mathrm{st},3}=1$ with high probability. The same argument but using the spatial regularity estimate from Proposition \ref{prop:reg} for $\mathfrak{t}_{\mathrm{r}}=\mathfrak{t}_{\mathrm{st},4}$ implies that $\mathfrak{t}_{\mathrm{st},4}=1$ with high probability as well. We are left with proving $\mathfrak{t}_{\mathrm{st},5}=1$ with high probability. This follows immediately from Lemma \ref{lemma:st2}, so we are done.
\end{proof}
\begin{proof}[Proof of \emph{Proposition \ref{prop:KPZ10}}]
We first observe the following union bound inequality, which tells us that if $\mathfrak{t}_{\mathrm{st}}<1$, then $\mathfrak{t}_{\mathrm{st},6}<1$ or $\mathfrak{t}_{\mathrm{st},6}=1$ and $\mathfrak{t}_{\mathrm{st}}<1$, where $\mathfrak{t}_{\mathrm{st},6}$ is the last stochastic time defined in Definition \ref{definition:KPZ101}:
\begin{align}
\mathbf{P}\left(\mathfrak{t}_{\mathrm{st}}<1\right) \ \leq \ \mathbf{P}\left(\mathfrak{t}_{\mathrm{st},6}<1\right) + \mathbf{P}\left(\mathfrak{t}_{\mathrm{st},6}=1, \ \mathfrak{t}_{\mathrm{st}}<1\right). \label{eq:KPZ101}
\end{align}
We apply Lemma \ref{lemma:KPZ102} and deduce the first probability on the RHS of \eqref{eq:KPZ101} is at most $\gamma+\kappa_{\gamma}\mathrm{o}_{N}$ for any $\gamma>0$, where $\mathrm{o}_{N}$ vanishes in the large-$N$ limit uniformly in $\gamma>0$. Thus, it suffices to deduce the same estimate for the second term on the RHS of \eqref{eq:KPZ101}. Actually, we will prove the second probability on the RHS of \eqref{eq:KPZ101} is equal to 0. To this end, let us recall the definition of $\mathfrak{t}_{\mathrm{st}}$ from Definition \ref{definition:KPZ1} and, again using a union bound inequality, get the following upper bound for the second term on the RHS of \eqref{eq:KPZ101}, which follows by conditioning on which of $\mathfrak{t}_{\mathrm{ap}}$ and $\mathfrak{t}_{\mathrm{RN}}^{\mathbf{T}}$ and $\mathfrak{t}_{\mathrm{RN}}^{\mathbf{X}}$ in Definition \ref{definition:KPZ1} is smallest and equal to $\mathfrak{t}_{\mathrm{st}}$:
\begin{align}
\mathbf{P}\left(\mathfrak{t}_{\mathrm{st},6}=1, \ \mathfrak{t}_{\mathrm{st}}=\mathfrak{t}_{\mathrm{ap}}<1\right) + \mathbf{P}\left(\mathfrak{t}_{\mathrm{st},6}=1, \ \mathfrak{t}_{\mathrm{st}}=\mathfrak{t}_{\mathrm{RN}}^{\mathbf{T}}<1\right)+\mathbf{P}\left(\mathfrak{t}_{\mathrm{st},6}=1, \ \mathfrak{t}_{\mathrm{st}}=\mathfrak{t}_{\mathrm{RN}}^{\mathbf{X}}<1\right). \label{eq:KPZ102}
\end{align}
We are left with showing each term on the RHS the above is equal to 0; this would give the proof of Proposition \ref{prop:KPZ10}, again because \eqref{eq:KPZ102} is an upper bound for the second term on the RHS of \eqref{eq:KPZ101}. We will organize our computations for each probability in \eqref{eq:KPZ102} in one of three bullet points below. First, we assume $N$ is sufficiently large so that $N^{\e_{\mathrm{ap},1}}\geq99999$, for example.
\begin{itemize}
\item We treat the first term in \eqref{eq:KPZ102}. To this end, consider $0<\mathfrak{t}_{N}\leq N^{-100}$ so $\mathfrak{t}_{\mathrm{ap}}+\mathfrak{t}_{N}\leq1$. Because $\mathfrak{t}_{\mathrm{st},6}=1$ by assumption of the event we are working on, we know $\mathfrak{t}_{\mathrm{st},5}=1$ as well. By definition of $\mathfrak{t}_{\mathrm{st},5}$ in Definition \ref{definition:KPZ101}, we deduce the following short-time estimate, which relates the value of $\mathbf{Z}^{N}$ after time $\mathfrak{t}_{\mathrm{ap}}$ and until time $\mathfrak{t}_{\mathrm{ap}}+\mathfrak{t}_{N}$ to its values at time $\mathfrak{t}_{\mathrm{ap}}$; the following first inequality is proved by using the proof for \eqref{eq:final110}, which we recall is blind to what $\mathbf{Q}^{N}$ actually is, while the second inequality estimating short-time behavior of $\mathbf{Z}^{N}$ follows from the identity $\mathfrak{t}_{\mathrm{st},5}=1$ we have just noted:
\begin{align}
\|\mathbf{Z}^{N}\|_{\mathfrak{t}_{\mathrm{ap}}+\mathfrak{t}_{N};\mathbb{T}_{N}} \ \leq \ \|\mathbf{Z}^{N}\|_{\mathfrak{t}_{\mathrm{ap}};\mathbb{T}_{N}} + \sup_{\mathrm{s}\in[0,N^{-99}]}\|\grad_{\mathrm{s}}^{\mathbf{T}}\mathbf{Z}^{N}\|_{\mathfrak{t}_{\mathrm{ap}};\mathbb{T}_{N}} \ \leq \ \|\mathbf{Z}^{N}\|_{\mathfrak{t}_{\mathrm{ap}};\mathbb{T}_{N}} + N^{-\frac12+\e_{\mathrm{ap},1}}\|\mathbf{Z}^{N}\|_{\mathfrak{t}_{\mathrm{ap}};\mathbb{T}_{N}}. \label{eq:KPZ103}
\end{align}
Recall from Lemma \ref{lemma:KPZ6} that until time $\mathfrak{t}_{\mathrm{st}}=\mathfrak{t}_{\mathrm{ap}}$, we have the identification $\mathbf{Z}^{N}=\mathbf{U}^{N}=\mathbf{Q}^{N}+\mathbf{D}^{N}$, where we that recall $\mathbf{U}^{N}$ is defined in Definition \ref{definition:KPZ5} and $\mathbf{Q}^{N}$ is defined in Definition \ref{definition:KPZ7}, and $\mathbf{D}^{N}$ is defined in Proposition \ref{prop:KPZ11}. Since $\mathfrak{t}_{\mathrm{st},6}=1$, we also have $\mathfrak{t}_{\mathrm{st},1}=1$ and $\mathfrak{t}_{\mathrm{st},2}=1$ by assumption, and this allows us to extend \eqref{eq:KPZ103} as follows:
\begin{align}
\|\mathbf{Z}^{N}\|_{\mathfrak{t}_{\mathrm{ap}}+\mathfrak{t}_{N};\mathbb{T}_{N}} \ \lesssim \ \|\mathbf{Z}^{N}\|_{\mathfrak{t}_{\mathrm{ap}};\mathbb{T}_{N}} \ \leq \ \|\mathbf{Q}^{N}\|_{\mathfrak{t}_{\mathrm{ap}};\mathbb{T}_{N}}+\|\mathbf{D}^{N}\|_{\mathfrak{t}_{\mathrm{ap}};\mathbb{T}_{N}} \ \lesssim \ N^{\e_{\mathrm{ap},1}}. \label{eq:KPZ104}
\end{align}
We recall $\e_{\mathrm{ap},1}\leq999^{-999}\e_{\mathrm{ap}}$ with $\e_{\mathrm{ap}}>0$ in Definition \ref{definition:KPZ1}. We also recall $N$ is large enough so that even with the implied constants in \eqref{eq:KPZ104}, we deduce that the far LHS of \eqref{eq:KPZ104} is at most $N^{2\e_{\mathrm{ap,1}}}\leq N^{\e_{\mathrm{ap}}/2}$. Parallel to \eqref{eq:KPZ103}, we also get, by applying $a^{-1}-(a+b)^{-1}\leq ba^{-2}$ for $a,b\geq0$ and by recalling $\mathfrak{t}_{\mathrm{st},5}=1$ that controls $(\mathbf{Z}^{N})^{-1}\grad^{\mathbf{T}}\mathbf{Z}^{N}$ for short times, that
\begin{align}
\|(\mathbf{Z}^{N})^{-1}\|_{\mathfrak{t}_{\mathrm{ap}}+\mathfrak{t}_{N};\mathbb{T}_{N}} \ \leq \ \|(\mathbf{Z}^{N})^{-1}\|_{\mathfrak{t}_{\mathrm{ap}};\mathbb{T}_{N}} + \sup_{\mathrm{s}\in[0,N^{-99}]}\|(\mathbf{Z}^{N})^{-2}\grad_{\mathrm{s}}^{\mathbf{T}}\mathbf{Z}^{N}\|_{\mathfrak{t}_{\mathrm{ap}};\mathbb{T}_{N}} \ \lesssim \ \|(\mathbf{Z}^{N})^{-1}\|_{\mathfrak{t}_{\mathrm{ap}};\mathbb{T}_{N}}, \label{eq:KPZ105}
\end{align}
while parallel to \eqref{eq:KPZ104}, we extend \eqref{eq:KPZ105} to the following estimate in which we now invoke the lower bound for $\mathbf{Q}^{N}$ that comes from the constraint $\mathfrak{t}_{\mathrm{st},1}=1$ along with the upper bound for $\mathbf{D}^{N}$ that comes from our assumption $\mathfrak{t}_{\mathrm{st},2}=1$:
\begin{align}
\|(\mathbf{Z}^{N})^{-1}\|_{\mathfrak{t}_{\mathrm{ap}}+\mathfrak{t}_{N};\mathbb{T}_{N}} \ \lesssim \ \|(\mathbf{Z}^{N})^{-1}\|_{\mathfrak{t}_{\mathrm{ap}};\mathbb{T}_{N}} \ \lesssim \ \|(\mathbf{Q}^{N})^{-1}\|_{\mathfrak{t}_{\mathrm{ap}};\mathbb{T}_{N}} + \|(\mathbf{Q}^{N})^{-2}\mathbf{D}^{N}\|_{\mathfrak{t}_{\mathrm{ap}};\mathbb{T}_{N}} \ \lesssim \ N^{\e_{\mathrm{ap},1}}. \label{eq:KPZ106}
\end{align}
Indeed, the last estimate above follows from the assumption that $\e_{\mathrm{ap},1}\leq999^{-999}\beta$, and thus the $N^{-\e_{\mathrm{ap},1}}$ lower bound for $\mathbf{Q}^{N}$ that we get from $\mathfrak{t}_{\mathrm{st},1}=1$ is much larger than the $N^{-\beta/2}$ upper bound for $\mathbf{D}^{N}$ that we get from $\mathfrak{t}_{\mathrm{st},2}=1$, at least in the large-$N$ limit. We emphasize that $\e_{\mathrm{ap},1}\leq999^{-999}\e_{\mathrm{ap}}$ by construction as well, and thus the far LHS of \eqref{eq:KPZ106} is bounded above by $N^{\e_{\mathrm{ap}}/2}$ without any implied constants or extra factors. Thus, \eqref{eq:KPZ104} and \eqref{eq:KPZ106} imply the lower bound in the infimum defining $\mathfrak{t}_{\mathrm{ap}}$ fails for all times $\mathfrak{t}\in[0,1]$ before $\mathfrak{t}_{\mathrm{ap}}+\mathfrak{t}_{N}$. This contradicts the definition of $\mathfrak{t}_{\mathrm{ap}}$ if $\mathfrak{t}_{\mathrm{ap}}<1$, as these lower bounds necessarily fail at and/or immediately after $\mathfrak{t}_{\mathrm{ap}}<1$ by definition of $\mathfrak{t}_{\mathrm{ap}}$. This shows the first probability in \eqref{eq:KPZ102} is 0.
\item We move to the second probability in \eqref{eq:KPZ102}, which amounts to estimating time gradients of $\mathbf{Z}^{N}$. In particular, take $0<\mathfrak{t}_{N}\leq N^{-100}$ so $\mathfrak{t}_{\mathrm{RN}}^{\mathbf{T}}+\mathfrak{t}_{N}\leq1$ similar to the previous bullet point. Consider any $0\leq\mathfrak{t}\leq\mathfrak{t}_{\mathrm{RN}}^{\mathbf{T}}+\mathfrak{t}_{N}$. Define $0\leq\mathfrak{t}_{0}\leq\mathfrak{t}_{\mathrm{RN}}^{\mathbf{T}}$ to be the closest such time to $\mathfrak{t}$. Last, take any $\mathfrak{r}\in\mathbb{I}^{\mathbf{T}}$, with $\mathbb{I}^{\mathbf{T}}$ in Definition \ref{definition:KPZ1}. The time gradient of $\mathbf{Z}^{N}$ evaluated at time $\mathfrak{t}$ with respect to time-scale $-\mathfrak{r}$ is the time gradient of $\mathbf{Z}^{N}$ with respect to the same time-scale $-\mathfrak{r}$ but evaluated at time $\mathfrak{t}_{0}\leq\mathfrak{t}_{\mathrm{RN}}^{\mathbf{T}}$, if we include two error terms that result from replacing the times at which we evaluate $\mathbf{Z}^{N}$. To be precise, this first error term is given by the difference of $\mathbf{Z}^{N}$ at time $\mathfrak{t}+\mathfrak{r}$ with $\mathbf{Z}^{N}$ at time $\mathfrak{t}_{0}+\mathfrak{r}$, and the second error term is given by the difference of $\mathbf{Z}^{N}$ at time $\mathfrak{t}$ with $\mathbf{Z}^{N}$ at time $\mathfrak{t}_{0}$. We observe now that the difference between any of these two pairs of times at which we compare the values of $\mathbf{Z}^{N}$ is bounded by $N^{-100}$, because the distance of any time $\mathfrak{t}\leq\mathfrak{t}_{\mathrm{RN}}^{\mathbf{T}}+\mathfrak{t}_{N}$ to the set of times less than or equal to $\mathfrak{t}_{\mathrm{RN}}^{\mathbf{T}}$ is at most $\mathfrak{t}_{N}\leq N^{-100}$. The conclusion of the last three sentences is the following, uniform over allowable time-scales $\mathfrak{r}\in\mathbb{I}^{\mathbf{T}}$ and which is a time-gradient version of \eqref{eq:KPZ103}: 
\begin{align}
\|\grad_{-\mathfrak{r}}^{\mathbf{T}}\mathbf{Z}^{N}\|_{\mathfrak{t}_{\mathrm{RN}}^{\mathbf{T}}+\mathfrak{t}_{N};\mathbb{T}_{N}} \ \leq \ \|\grad_{-\mathfrak{r}}^{\mathbf{T}}\mathbf{Z}^{N}\|_{\mathfrak{t}_{\mathrm{RN}}^{\mathbf{T}};\mathbb{T}_{N}}+2\sup_{\mathrm{s}\in[0,N^{-99}]}\|\grad_{\mathrm{s}}^{\mathbf{T}}\mathbf{Z}^{N}\|_{\mathfrak{t}_{\mathrm{RN}}^{\mathbf{T}};\mathbb{T}_{N}}. \label{eq:KPZ107}
\end{align}
Because we assume $\mathfrak{t}_{\mathrm{RN}}^{\mathbf{T}}=\mathfrak{t}_{\mathrm{st}}$, the first term on the RHS of \eqref{eq:KPZ107} stays the same if we replace $\mathbf{Z}^{N}$ by $\mathbf{U}^{N}$, consequence of the pathwise identification in Lemma \ref{lemma:KPZ6}. Because $\mathfrak{t}_{\mathrm{st},6}=1$ by assumption of the event in the second probability in \eqref{eq:KPZ102} on which we are working, we have the identity $\mathfrak{t}_{\mathrm{st},3}=1$; see Definition \ref{definition:KPZ101}. The identity $\mathfrak{t}_{\mathrm{st},3}=1$ implies that the estimate in the infimum defining $\mathfrak{t}_{\mathrm{st},3}$ fails for $\mathfrak{t}=1$, which therefore controls the first term on the RHS of \eqref{eq:KPZ107} via an upper bound we specify shortly. On the other hand, because $\mathfrak{t}_{\mathrm{st},6}=1$ by assumption, we may similarly deduce that $\mathfrak{t}_{\mathrm{st},5}=1$ holds automatically. By definition of $\mathfrak{t}_{\mathrm{st},5}$, the identity $\mathfrak{t}_{\mathrm{st},5}=1$ similarly implies that the estimate in the infimum defining $\mathfrak{t}_{\mathrm{st},5}$ also fails if $\mathfrak{t}=1$. This bounds the second term on the RHS of \eqref{eq:KPZ107} by the short-time factor of $N^{-1/2+\e_{\mathrm{ap},1}}$ times the same norm but for $\mathbf{Z}^{N}$ instead of its scale-$\mathrm{s}$ time-gradient. Ultimately, from this paragraph and \eqref{eq:KPZ107}, we deduce the following for which we note $\mathfrak{t}_{\mathrm{RN}}^{\mathbf{T}}+\mathfrak{t}_{N}\leq1$, so all norms may be pushed to time 1 as we are only concerned with upper bounds. Let us clarify the second bound below follows by $\mathfrak{r}\in\mathbb{I}^{\mathbf{T}}$, which implies $\mathfrak{r}\geq N^{-2}$ and $\mathfrak{r}^{-1/4}\leq N^{1/2}$; also note $\mathbf{U}^{N}=\mathbf{Z}^{N}$ until time $\mathfrak{t}_{\mathrm{RN}}^{\mathbf{T}}$:
\begin{align}
\mathfrak{r}^{-\frac14}\|\grad_{-\mathfrak{r}}^{\mathbf{T}}\mathbf{Z}^{N}\|_{\mathfrak{t}_{\mathrm{RN}}^{\mathbf{T}}+\mathfrak{t}_{N};\mathbb{T}_{N}} \ &\leq \ \|\grad_{-\mathfrak{r}}^{\mathbf{T}}\mathbf{U}^{N}\|_{\mathfrak{t}_{\mathrm{RN}}^{\mathbf{T}};\mathbb{T}_{N}}+2\sup_{\mathrm{s}\in[0,N^{-99}]}\|\grad_{\mathrm{s}}^{\mathbf{T}}\mathbf{Z}^{N}\|_{\mathfrak{t}_{\mathrm{RN}}^{\mathbf{T}};\mathbb{T}_{N}} \\
&\leq \ N^{\e_{\mathrm{ap},1}}(1+\|\mathbf{U}^{N}\|_{1;\mathbb{T}_{N}}^{2}) + \mathfrak{r}^{-\frac14}N^{-\frac12+\e_{\mathrm{ap},1}}\|\mathbf{Z}^{N}\|_{\mathfrak{t}_{\mathrm{RN}}^{\mathbf{T}};\mathbb{T}_{N}} \\
&\leq \ N^{\e_{\mathrm{ap},1}}(1+\|\mathbf{U}^{N}\|_{1;\mathbb{T}_{N}}^{2}) + N^{\e_{\mathrm{ap},1}}\|\mathbf{U}^{N}\|_{\mathfrak{t}_{\mathrm{RN}}^{\mathbf{T}};\mathbb{T}_{N}}. \label{eq:KPZ108}
\end{align}
Because we now have $\mathfrak{t}_{\mathrm{RN}}^{\mathbf{T}}=\mathfrak{t}_{\mathrm{st}}$ on the event we currently work on, we may follow the second inequality in \eqref{eq:KPZ104} and estimate $\mathbf{U}^{N}$ by $\mathbf{Q}^{N}$ and $\mathbf{D}^{N}$. Similar to the end of \eqref{eq:KPZ104}, we remark that $\mathfrak{t}_{\mathrm{st},6}=1$ implies $\mathfrak{t}_{\mathrm{st},1}=1$ and $\mathfrak{t}_{\mathrm{st},2}=1$ automatically, which, again as in the end of \eqref{eq:KPZ104}, implies upper bounds for each of $\mathbf{Q}^{N}$ and $\mathbf{D}^{N}=\mathbf{U}^{N}-\mathbf{Q}^{N}$ given by $N^{\e_{\mathrm{ap},1}}$ each, for example, and thus an upper bound for $\mathbf{U}^{N}$ of the same order. In particular, via this paragraph and the estimate \eqref{eq:KPZ108}, we deduce the following estimate in which we again recall $\e_{\mathrm{ap},1}\leq999^{-999}\e_{\mathrm{ap}}$, so that the middle term below is at most $N^{\e_{\mathrm{ap}}/2}$ even with the implied constants/factors in the first estimate below:
\begin{align}
\mathfrak{r}^{-\frac14}\|\grad_{-\mathfrak{r}}^{\mathbf{T}}\mathbf{Z}^{N}\|_{\mathfrak{t}_{\mathrm{RN}}^{\mathbf{T}}+\mathfrak{t}_{N};\mathbb{T}_{N}} \ \lesssim \ N^{3\e_{\mathrm{ap},1}} \ \leq \ N^{\e_{\mathrm{ap}}/2}. \label{eq:KPZ109}
\end{align}
Because the last estimate in \eqref{eq:KPZ109} is uniform over admissible time-gradient time-scales $\mathfrak{r}\in\mathbb{I}^{\mathbf{T}}$, we observe that the estimate in the infimum defining $\mathfrak{t}_{\mathrm{RN}}^{\mathbf{T}}$ fails if $\mathfrak{t}=\mathfrak{t}_{\mathrm{RN}}^{\mathbf{T}}+\mathfrak{t}_{N}$. Because the LHS of said estimate in said infimum is monotone non-decreasing in $\mathfrak{t}\geq0$, we observe that it also fails for all times $\mathfrak{t}\leq\mathfrak{t}_{\mathrm{RN}}^{\mathbf{T}}+\mathfrak{t}_{N}$. Thus, by definition of $\mathfrak{t}_{\mathrm{RN}}^{\mathbf{T}}$, we have $\mathfrak{t}_{\mathrm{RN}}^{\mathbf{T}}>\mathfrak{t}_{\mathrm{RN}}^{\mathbf{T}}+\mathfrak{t}_{N}$ as long as $\mathfrak{t}_{\mathrm{RN}}^{\mathbf{T}}<1$, so that we can actually find $\mathfrak{t}_{N}>0$ satisfying $\mathfrak{t}_{\mathrm{RN}}^{\mathbf{T}}+\mathfrak{t}_{N}\leq1$. The previous inequality $\mathfrak{t}_{\mathrm{RN}}^{\mathbf{T}}>\mathfrak{t}_{\mathrm{RN}}^{\mathbf{T}}+\mathfrak{t}_{N}$ is a clear contradiction for $\mathfrak{t}_{N}>0$, so the second probability in \eqref{eq:KPZ102} must be that of an empty event and thus equal to zero.
\item We move to the last probability in \eqref{eq:KPZ102} for spatial gradients of $\mathbf{Z}^{N}$. We follow a strategy similar to the previous bullet point but replacing time-gradients by spatial gradients. In particular, let us first take $0<\mathfrak{t}_{N}\leq N^{-100}$ such that $\mathfrak{t}_{\mathrm{RN}}^{\mathbf{X}}+\mathfrak{t}_{N}\leq1$. We may replace any spatial gradient of $\mathbf{Z}^{N}$ evaluated at any time $\mathfrak{t}\leq\mathfrak{t}_{\mathrm{RN}}^{\mathbf{X}}+\mathfrak{t}_{N}$ with a spatial gradient of $\mathbf{Z}^{N}$ but evaluated at a time $\mathfrak{t}_{0}\leq\mathfrak{t}_{\mathrm{RN}}^{\mathbf{X}}$ satisfying $|\mathfrak{t}-\mathfrak{t}_{0}|\leq\mathfrak{t}_{N}\leq N^{-100}$, if we account for the resulting errors given by scale $\mathrm{s}\leq N^{-99}$ time-gradients of $\mathbf{Z}^{N}$, which come by replacing $\mathbf{Z}^{N}$ at times $\mathfrak{t}+\mathrm{s}$ and $\mathfrak{t}$ with $\mathbf{Z}^{N}$ at times $\mathfrak{t}_{0}+\mathrm{s}$ and $\mathfrak{t}_{0}$, respectively. Below, we have taken any arbitrary length-scale $1\leq|\mathfrak{l}|\leq\mathfrak{l}_{N}$ with $\mathfrak{l}_{N}=N^{1/2+\e_{\mathrm{RN}}}$ from Definition \ref{definition:KPZ1}:
\begin{align}
\|\grad_{\mathfrak{l}}^{\mathbf{X}}\mathbf{Z}^{N}\|_{\mathfrak{t}_{\mathrm{RN}}^{\mathbf{X}}+\mathfrak{t}_{N};\mathbb{T}_{N}} \ \leq \ \|\grad_{\mathfrak{l}}^{\mathbf{X}}\mathbf{Z}^{N}\|_{\mathfrak{t}_{\mathrm{RN}}^{\mathbf{X}};\mathbb{T}_{N}}+\sup_{\mathrm{s}\in[0,N^{-99}]}\|\grad_{\mathrm{s}}^{\mathbf{T}}\mathbf{Z}^{N}\|_{\mathfrak{t}_{\mathrm{RN}}^{\mathbf{X}};\mathbb{T}_{N}}. \label{eq:KPZ1010}
\end{align}
Let us multiply both sides of \eqref{eq:KPZ1010} by $N^{1/2}|\mathfrak{l}|^{-1/2}$. We argue as in the second bullet point. Because we have assumed $\mathfrak{t}_{\mathrm{RN}}^{\mathbf{X}}=\mathfrak{t}_{\mathrm{st}}$ on the event we are currently trying to prove has zero probability, the identification in Lemma \ref{lemma:KPZ6} lets us replace $\mathbf{Z}^{N}$ with $\mathbf{U}^{N}$ in the first term on the RHS of \eqref{eq:KPZ1010}. Because we have assumed $\mathfrak{t}_{\mathrm{st},6}=1$ on the current event as well, by Definition \ref{definition:KPZ101} we get $\mathfrak{t}_{\mathrm{st},4}=1$ automatically. This identity then implies the inequality in the infimum defining $\mathfrak{t}_{\mathrm{st},4}$ fails for $\mathfrak{t}=1$, thereby providing an upper bound for the first term on the RHS of \eqref{eq:KPZ1010}, which we specify shortly. On the other hand, we again note that $\mathfrak{t}_{\mathrm{st},6}=1$ implies that $\mathfrak{t}_{\mathrm{st},5}=1$ automatically, and this last identity provides an estimate for the second term on the RHS of \eqref{eq:KPZ1010}. Again, we note $\mathfrak{t}_{\mathrm{RN}}^{\mathbf{X}}+\mathfrak{t}_{N}\leq1$ by construction, so all norms may be pushed to time 1 since we are only concerned with upper bounds. We deduce the following parallel to \eqref{eq:KPZ108}, for which we note $|\mathfrak{l}|^{-1}\leq1$ trivially:
\begin{align}
N^{1/2}|\mathfrak{l}|^{-1/2}\|\grad_{\mathfrak{l}}^{\mathbf{X}}\mathbf{Z}^{N}\|_{\mathfrak{t}_{\mathrm{RN}}^{\mathbf{X}}+\mathfrak{t}_{N};\mathbb{T}_{N}} \ &\leq \ \|\grad_{\mathfrak{l}}^{\mathbf{X}}\mathbf{U}^{N}\|_{\mathfrak{t}_{\mathrm{RN}}^{\mathbf{X}};\mathbb{T}_{N}}+\sup_{\mathrm{s}\in[0,N^{-99}]}\|\grad_{\mathrm{s}}^{\mathbf{T}}\mathbf{Z}^{N}\|_{\mathfrak{t}_{\mathrm{RN}}^{\mathbf{X}};\mathbb{T}_{N}} \\
 &\leq \ N^{\e_{\mathrm{ap},1}}(1+\|\mathbf{U}^{N}\|_{1;\mathbb{T}_{N}}^{2}) + N^{1/2}|\mathfrak{l}|^{-1/2}N^{-\frac12+\e_{\mathrm{ap},1}}\|\mathbf{Z}^{N}\|_{\mathfrak{t}_{\mathrm{RN}}^{\mathbf{X}};\mathbb{T}_{N}} \label{eq:KPZ1011} \\
&\leq \ N^{\e_{\mathrm{ap},1}}(1+\|\mathbf{U}^{N}\|_{1;\mathbb{T}_{N}}^{2}) + N^{\e_{\mathrm{ap},1}}\|\mathbf{U}^{N}\|_{\mathfrak{t}_{\mathrm{RN}}^{\mathbf{X}};\mathbb{T}_{N}}. \label{eq:KPZ1012}
\end{align}
We now proceed with the argument in the second bullet point above starting with the paragraph immediately after \eqref{eq:KPZ108}. This provides an upper bound of $N^{\e_{\mathrm{ap}}/2}$ for the LHS of \eqref{eq:KPZ1011} uniformly in $1\leq|\mathfrak{l}|\leq\mathfrak{l}_{N}$, which, as we assumed $\mathfrak{t}_{\mathrm{RN}}^{\mathbf{X}}+\mathfrak{t}_{N}\leq1$ with $\mathfrak{t}_{N}>0$ given that $\mathfrak{t}_{\mathrm{RN}}^{\mathbf{X}}<1$, implies $\mathfrak{t}_{\mathrm{RN}}^{\mathbf{X}}>\mathfrak{t}_{\mathrm{RN}}^{\mathbf{X}}+\mathfrak{t}_{N}$, and this is a clear contradiction because $\mathfrak{t}_{N}$ is strictly positive.
\end{itemize}
We have shown each probability in \eqref{eq:KPZ102} is equal to zero. Combining this with \eqref{eq:KPZ101} and the paragraph following \eqref{eq:KPZ101} we used to control the first probability on the RHS of \eqref{eq:KPZ101}, this completes the proof of Proposition \ref{prop:KPZ10}.
\end{proof}
\section{Boltzmann-Gibbs Principle I -- Preliminary Estimates}\label{section:BGIPrelim}
We record general estimates for proofs of both Propositions \ref{prop:BGI1} and \ref{prop:BGI2} as their proofs will be similar in strategy. This includes a deterministic heat operator estimate that lets us replace space-time suprema of space-time heat operators by an integral whose expectation we can directly take. We estimate said expectation by a localization procedure for mesoscopic space-time averages of local functionals and then use a ``local equilibrium" estimate via one-block and two-blocks of \cite{GPV} and the log-Sobolev inequality of \cite{Yau}, ultimately reducing all of our calculations to standard equilibrium estimates that we will introduce. We conclude with a multiscale scheme to replace local functionals by space-time averages via step-by-step replacements.
\subsubsection{Heat Operator Estimate}
Our first estimate is deterministic. First, some convenient notation.
\begin{definition}\label{definition:hoe1}
\fsp For any possibly random function $\phi:\R_{\geq0}\times\mathbb{T}_{N}\to\R$ and any $\mathrm{t}\geq0$, define the space-time integral/sum
\begin{align}
\mathbf{I}_{\mathrm{t}}(\phi) \ = \ \mathbf{I}_{\mathrm{t}}(\phi_{\mathrm{s},y}) \ = \  \int_{0}^{\mathrm{t}}\wt{\sum}_{y\in\mathbb{T}_{N}}\phi_{\mathrm{s},y}\d\mathrm{s}.
\end{align}
\end{definition}
\begin{lemma}\label{lemma:hoe2}
\fsp Consider any possibly random function $\phi:\R_{\geq0}\times\mathbb{T}_{N}\to\R$. Provided any $\gamma>0$, we have the estimate
\begin{align}
\|\mathbf{H}^{N}(\phi\mathbf{Y}^{N})\|_{1;\mathbb{T}_{N}}^{3/2} \ \lesssim_{\gamma} \ N^{\gamma}\mathbf{I}_{1}(|\phi|^{3/2}|\mathbf{Y}^{N}|^{3/2}) \ \lesssim \ N^{\gamma+3\e_{\mathrm{ap}}/2}\mathbf{I}_{1}(|\phi|^{3/2}). \label{eq:hoe2}
\end{align}
\end{lemma}
\begin{proof}
Take $\mathrm{t}\in[0,1]$ and $x\in\mathbb{T}_{N}$. For any $\mathrm{s}\geq0$, we define $\mathrm{s}_{\sim}=\mathbf{O}_{\mathrm{s},\mathrm{t}}\vee N^{-2}$ with $\mathbf{O}_{\mathrm{s},\mathrm{t}}=|\mathrm{t}-\mathrm{s}|$. We have
\begin{align}
|\mathbf{H}^{N}_{\mathrm{t},x}(\phi\mathbf{Y}^{N})| \ = \ |\mathbf{H}^{N}_{\mathrm{t},x}(\mathrm{s}_{\sim}^{-1/3}\mathrm{s}_{\sim}^{1/3}\phi\mathbf{Y}^{N})| \ \leq \ \left(\mathbf{H}_{\mathrm{t},x}^{N}(\mathrm{s}_{\sim}^{-1})\right)^{1/3}\left(\mathbf{H}_{\mathrm{t},x}^{N}(\mathrm{s}_{\sim}^{1/2}|\phi|^{3/2}|\mathbf{Y}^{N}|^{3/2})\right)^{2/3}. \label{eq:hoe21}
\end{align}
The second estimate in \eqref{eq:hoe21} follows from first recalling the space-time heat operator $\mathbf{H}^{N}$ is integrating against the heat kernel in space-time. Thus, the second estimate in \eqref{eq:hoe21} is the immediate consequence of the Holder inequality, upon viewing integration as integrating against the heat kernel in space-time, with Holder conjugate exponents $3$ and $3/2$. To build off of \eqref{eq:hoe21}, let us treat the first factor from the far RHS of \eqref{eq:hoe21}. Note $\mathrm{s}_{\sim}^{-1}$ is independent of the spatial summation against the heat kernel, and because the heat kernel is a probability measure with respect to the forwards spatial variable, the first factor on the RHS of \eqref{eq:hoe21} turns into the integral of $\mathrm{s}_{\sim}^{-1}$ on the integration domain $[0,\mathrm{t}]\subseteq[0,1]$. Although $\mathbf{O}_{\mathrm{s},\mathrm{t}}^{-1}$ is not integrable near $\mathrm{t}$, because we have regularized $\mathbf{O}_{\mathrm{s},\mathrm{t}}$ with $\mathrm{s}_{\sim}$, the resulting integral is logarithmic in $N$ and therefore at most $C_{\gamma}N^{\gamma}$ where $\gamma>0$ is arbitrary. This gives
\begin{align}
\left(\mathbf{H}_{\mathrm{t},x}^{N}(\mathrm{s}_{\sim}^{-1})\right)^{1/3}\left(\mathbf{H}_{\mathrm{t},x}^{N}(\mathrm{s}_{\sim}^{1/2}|\phi|^{3/2}|\mathbf{Y}^{N}|^{3/2})\right)^{2/3} \ &= \ \left(\int_{0}^{\mathrm{t}}(|\mathrm{t}-\mathrm{s}|^{-1}\wedge N^{2})\d\mathrm{s}\right)^{1/3}\left(\mathbf{H}_{\mathrm{t},x}^{N}(\mathrm{s}_{\sim}^{1/2}|\phi|^{3/2}|\mathbf{Y}^{N}|^{3/2})\right)^{2/3} \\
&\lesssim_{\gamma} \ N^{\gamma}\left(\mathbf{H}_{\mathrm{t},x}^{N}(\mathrm{s}_{\sim}^{1/2}|\phi|^{3/2}|\mathbf{Y}^{N}|^{3/2})\right)^{2/3}. \label{eq:hoe22}
\end{align}
Thus, it remains to bound the heat operator $\mathbf{H}^{N}$ in \eqref{eq:hoe22} by $\mathbf{I}_{1}$. Recall this heat operator is integration in space-time against the heat kernel. By Proposition \ref{prop:heat}, the heat kernel is $\mathrm{O}(N^{-1}\mathrm{s}_{\sim}^{-1/2})$. The $\mathrm{s}_{\sim}^{-1/2}$ factor cancels the $\mathrm{s}_{\sim}^{1/2}$ factor in the heat operator in \eqref{eq:hoe22}. The $N^{-1}$ factor in this heat kernel estimate makes the sum over $\mathbb{T}_{N}$ into an average over $\mathbb{T}_{N}$, because $|\mathbb{T}_{N}|\lesssim N$, thus we are left with $\mathbf{I}_{\mathrm{t}}(\cdot)$ instead of $\mathbf{H}^{N}_{\mathrm{t},x}(\mathrm{s}_{\sim}^{1/2}\cdot)$. As $|\phi||\mathbf{Y}^{N}|\geq0$, we may extend $\mathbf{I}_{\mathrm{t}}(|\phi|^{3/2}|\mathbf{Y}^{N}|^{3/2})\leq\mathbf{I}_{1}(|\phi|^{3/2}|\mathbf{Y}^{N}|^{3/2})$, thus yielding the RHS of the proposed estimate from \eqref{eq:hoe22}. Because the RHS of the proposed estimate is independent of the original space-time variables $\mathrm{t}$ and $x$, it bounds the far LHS of \eqref{eq:hoe21} uniformly in these variables. This yields the first estimate in \eqref{eq:hoe2}. The second inequality follows by $|\mathbf{Y}^{N}|\lesssim N^{\e_{\mathrm{ap}}}$; see Definitions \ref{definition:KPZ1} and \ref{definition:KPZ5}.
\end{proof}
\subsubsection{Localization Map}
We eventually apply Lemma \ref{lemma:hoe2} with $\phi$ equal to the time-average of a local functional of the particle system. Although the functional in the time-average is local, its time-average itself is, in principle, completely non-local, because even on mesoscopic time-scales the values of $\eta$-variables far away from the support of the integrated local functional may affect the $\eta$-variables inside the support of the integrated local functional in finite time. This is just the fact that random walks can travel arbitrarily far in finite time. However, the probability of non-interacting random walks traveling much farther than their expected maximal displacement vanishes exponentially fast. We extend this to the $\eta$-variables, which are random walks that interact via exclusion. Before we give the main estimate of this localization, we introduce convenient notation for the rest of this paper.
\begin{definition}\label{definition:locmap1}
\fsp Provided any $\eta\in\Omega$ and any time $\mathrm{t}\geq0$ and any length-scale $\mathfrak{l}\in\Z_{\geq0}$, define a configuration $\mathrm{Loc}_{\mathrm{t},\mathfrak{l}}\eta\in\Omega$ by the following ``trivial extension" of the projection of $\eta$ onto $\mathbb{B}_{\mathrm{t},\mathfrak{l}}=\llbracket-\mathfrak{L}_{\mathrm{t},\mathfrak{l}},\mathfrak{L}_{\mathrm{t},\mathfrak{l}}\rrbracket$, in which $\mathfrak{L}_{\mathrm{t},\mathfrak{l}}=N^{1+\gamma_{0}}\mathrm{t}^{1/2}+N^{3/2+\gamma_{0}}\mathrm{t}+N^{\gamma_{0}}\mathfrak{l}$ where $\gamma_{0}>0$ is taken as a fixed universal constant satisfying $\gamma_{0}\leq999^{-999}\e_{\mathrm{ap}}\wedge999^{-999}\e_{\mathrm{RN}}\wedge999^{-999}\e_{1}\wedge999^{-999}\e_{\mathrm{RN},1}$:
\begin{align}
(\mathrm{Loc}_{\mathrm{t},\mathfrak{l}}\eta)_{x} \ = \ \mathbf{1}_{x\in\mathbb{B}_{\mathrm{t},\mathfrak{l}}}\eta_{x} + \mathbf{1}_{x\not\in\mathbb{B}_{\mathrm{t},\mathfrak{l}}}.
\end{align}
\end{definition}
\begin{remark}\label{remark:locmap1}
\fsp We briefly explain $\mathbb{B}_{\mathrm{t},\mathfrak{l}}$. Take a simple symmetric random walk of order $N^{2}$ speed plus a random order $N^{3/2}$ speed asymmetry. Suppose this random walk starts outside $\mathbb{B}_{\mathrm{t},\mathfrak{l}}$ and let it walk for time $\mathrm{t}$. The probability that this random walk hits the set $\llbracket-\mathfrak{l},\mathfrak{l}\rrbracket\subseteq\mathbb{B}_{\mathrm{t},\mathfrak{l}}$ is bounded by the probability the maximal process/displacement is at least $N^{1+\gamma_{0}}\mathrm{t}^{1/2}+N^{3/2+\gamma_{0}}\mathrm{t}$. Because of the extra $N^{\gamma_{0}}$ factor, this occurs with exponentially small probability courtesy of sub-Gaussian martingale inequalities applied to the simple symmetric random walk with large-deviations estimates for the Poisson number of asymmetric drift/jumps. Thus, by union bound, the probability that any of a polynomial number of such random walks hits $\llbracket-\mathfrak{l},\mathfrak{l}\rrbracket\subseteq\mathbb{B}_{\mathrm{t},\mathfrak{l}}$ is also exponentially small in $N$, as the sub-exponential bound beats the polynomial-in-$N$ number of random walks asymptotically.
\end{remark}
Roughly speaking, the primary technical goal of our analysis is to reduce estimates for local functionals to the same estimates but after pretending the model is at an invariant measure. The philosophy of local equilibrium from \cite{GPV}, which we make precise in Lemma \ref{lemma:le2} and Lemma \ref{lemma:le3}, will only succeed for \emph{local} functionals of the particle system. Thus, we want to ignore $\eta$-values outside the block $\mathbb{B}_{\mathrm{t},\mathfrak{l}}$ from Definition \ref{definition:locmap1} while affecting space-time averages of whatever functionals whose analysis we want to reduce to local equilibrium in an asymptotically negligible manner. This is the ultimate goal of Lemma \ref{lemma:locmap3} below, for example.

We proceed with additional notation for space-time averaging operators, which we will employ for local functionals.
\begin{definition}\label{definition:locmap2}
\fsp Provided any time-scale $\mathfrak{t}_{\mathrm{av}}\geq0$, any length-scale $\mathfrak{l}_{\mathrm{av}}\in\Z_{\geq0}$, and any functional $\mathfrak{f}:\Omega\to\R$, let us define the following where $\mathfrak{l}_{\mathfrak{f}}$ is the smallest non-negative integer for which $\mathfrak{f}$ and the shifts $\tau_{\mathfrak{l}_{\mathfrak{f}}}\mathfrak{f}$ and $\tau_{-\mathfrak{l}_{\mathfrak{f}}}\mathfrak{f}$ have mutually disjoint supports:
\begin{align}
\mathfrak{I}^{\mathbf{T}}_{\mathfrak{t}_{\mathrm{av}}}\mathfrak{I}^{\mathbf{X}}_{\mathfrak{l}_{\mathrm{av}}}(\mathfrak{f}_{S,y}) \ = \ \mathfrak{t}_{\mathrm{av}}^{-1}\int_{0}^{\mathfrak{t}_{\mathrm{av}}}\wt{\sum}_{w=1}^{\mathfrak{l}_{\mathrm{av}}}\tau_{-\mathfrak{l}_{\mathfrak{f}}w}\mathfrak{f}_{S+\mathfrak{r},y}\d\mathfrak{r}.
\end{align}
We adopt the convention that $\mathfrak{I}^{\mathbf{T}}_{0}$ and $\mathfrak{I}^{\mathbf{X}}_{0}$ and $\mathfrak{I}^{\mathbf{X}}_{1}$ are identity maps (there is morally no difference between $\mathfrak{I}^{\mathbf{X}}_{1}$ and the identity map except for a harmless spatial shift). We will drop any identity maps from the notation.
\end{definition}
Let $\mathscr{D}(\R_{\geq0},\Omega)$ be the space of sample particle system paths, on which the system induces a path-space probability measure.
\begin{lemma}\label{lemma:locmap3}
\fsp Consider a functional $\mathfrak{f}:\Omega\to\R$ whose support is contained in the block $\mathbb{B}_{\mathfrak{f}}=\llbracket-\mathfrak{l},\mathfrak{l}\rrbracket\subseteq\mathbb{T}_{N}$ with $\mathfrak{l}\in\Z_{\geq0}$. Provided any $\mathfrak{t}_{\mathrm{av}}\in[0,1]$ and $\mathfrak{l}_{\mathrm{av}}\in\Z_{\geq0}$, we have the following for any $\kappa>0$, where we use notation defined after:
\begin{align}
{\sup}_{\eta}|\E^{\mathrm{dyn}}_{\eta}\mathfrak{I}_{\mathfrak{t}_{\mathrm{av}}}^{\mathbf{T}}\mathfrak{I}_{\mathfrak{l}_{\mathrm{av}}}^{\mathbf{X}}(\mathfrak{f}_{0,0}) - \E^{\mathrm{dyn}}_{\mathrm{Loc}}\mathfrak{I}_{\mathfrak{t}_{\mathrm{av}}}^{\mathbf{T}}\mathfrak{I}_{\mathfrak{l}_{\mathrm{av}}}^{\mathbf{X}}(\mathfrak{f}_{0,0})| \ \lesssim_{\kappa,\gamma_{0}} \ N^{-\kappa}\|\mathfrak{f}\|_{\omega;\infty}. \label{eq:locmap3}
\end{align}
We first introduce the parameter $\mathfrak{l}_{\mathrm{tot}}=99\mathfrak{l}+99\mathfrak{l}\mathfrak{l}_{\mathrm{av}}$. Let us also recall the parameter $\gamma_{0}>0$ from \emph{Definition \ref{definition:locmap1}}. Moreover, the expectation $\E^{\mathrm{dyn}}_{\cdot}$ denotes the expectation with respect to the path-space measure on $\mathscr{D}(\R_{\geq0},\Omega)$ of the particle system with the initial configuration $\cdot\in\Omega$. We take $\cdot=\eta$ and $\cdot=\mathrm{Loc}=\mathrm{Loc}_{\mathfrak{t}_{\mathrm{av}},\mathfrak{l}_{\mathrm{tot}}}(\eta)$ in the previous estimate \emph{\eqref{eq:locmap3}}.
\end{lemma}
\begin{proof}
Note \eqref{eq:locmap3} compares expectations of the same space-time average of $\mathfrak{f}$ with respect to the same dynamics but with initial configurations that only disagree outside $\mathbb{B}_{\mathfrak{t}_{\mathrm{av}},\mathfrak{l}_{\mathrm{tot}}}$, so two $\eta$-processes with fixed initial configurations that are different outside $\mathbb{B}_{\mathfrak{t}_{\mathrm{av}},\mathfrak{l}_{\mathrm{tot}}}$; see Definition \ref{definition:locmap1}. Therefore, the LHS of \eqref{eq:locmap3} is bounded above by $\|\mathfrak{f}\|_{\omega;\infty}$ times the probability these two $\eta$-processes, with initial configurations $\eta$ and $\mathrm{Loc}_{\mathfrak{t}_{\mathrm{av}},\mathfrak{l}_{\mathrm{tot}}}(\eta)$, see different $\eta$-values in $\mathbb{B}_{\mathrm{tot}}$ at any time before or at $\mathfrak{t}_{\mathrm{av}}$, \emph{under some coupling of the two processes}. Indeed, the expectations on the LHS of \eqref{eq:locmap3} only differ on such event, as the time average of $\mathfrak{I}^{\mathbf{X}}_{\mathfrak{l}_{\mathrm{av}}}(\mathfrak{f})$ evaluated at $\mathfrak{t}_{\mathrm{av}}$ only depends on $\eta$ in $\mathbb{B}_{\mathrm{tot}}$ until $\mathfrak{t}_{\mathrm{av}}$. Thus, it suffices to bound the path-space probability that the two processes disagree in $\mathbb{B}_{\mathrm{tot}}=\llbracket-\mathfrak{l}_{\mathrm{tot}},\mathfrak{l}_{\mathrm{tot}}\rrbracket$, which contains the support of $\mathfrak{I}^{\mathbf{X}}_{\mathfrak{l}_{\mathrm{av}}}(\mathfrak{f})$, at any time before or at $\mathfrak{t}_{\mathrm{av}}$ by $\mathrm{O}_{\kappa,\gamma_{0}}(N^{-\kappa})$. 

It is left to couple the two $\eta$-processes with initial configurations $\eta$ and $\mathrm{Loc}_{\mathfrak{t}_{\mathrm{av}},\mathfrak{l}_{\mathrm{tot}}}(\eta)$. We will not use the basic coupling for exclusion processes, but we instead modify it slightly to be explained shortly. We refer to the process with initial configuration $\eta$ as Species 1, and that with initial configuration $\mathrm{Loc}_{\mathfrak{t}_{\mathrm{av}},\mathfrak{l}_{\mathrm{tot}}}(\eta)$ as Species 2.
\begin{itemize}
\item Define a \emph{discrepancy} between Species 1 and Species 2 as a point $x$ where $\eta_{x}=1$ in one species and $\eta_{x}=-1$ in the other.
\item For any jump under a Poisson clock coming from the symmetric part of the generator, we realize such a jump from one point in $\mathbb{T}_{N}$ to another as swapping $\eta$-values at those points; see the symmetric part $\mathsf{L}_{N,\mathrm{S}}$ in \eqref{eq:gen} of the generator $\mathsf{L}_{N}$. We couple the symmetric parts of the dynamics in Species 1 and Species 2 by coupling these ``spin-swap" \emph{bond} clocks; Species 1 and Species 2 always swap $\eta$-variables together under symmetric clocks. This coupling can never create new discrepancies, only transport them. Also, individual discrepancies evolve as \emph{free} and \emph{symmetric} random walks; under this coupling of symmetric dynamics, with speed $N^{2}/2$ a discrepancy will move as a simple symmetric random walk suppressed by nothing, including the exclusion condition in the particle system. This \emph{free} and \emph{symmetric} feature of the discrepancy walks would not be true if we instead employed the basic coupling for the symmetric dynamics, but it is directly verifiable for this bond coupling (according to symmetric bond clocks, a discrepancy always jumps without being blocked with equal speeds to the left and right, because Species 1 and 2 always swap $\eta$-variables together along the activated bond, and this moves the discrepancy along said bond).
\item To couple clocks of asymmetric dynamics, suppose there is a particle at $x$ in Species 1. If $x$ is vacant in Species 2, there is no coupling at $x$. If $x$ is occupied in Species 2 \emph{and} the speed of an asymmetric jump from $x$ is equal among both species (in both directions), we couple the jumps (so particles jump together in the same direction); this is the basic coupling. If the asymmetry speeds of jumps from $x$ are not equal among the two species in at least one direction, we do not couple the jumps and let them move independently. This difference in asymmetry speeds comes from the $\mathfrak{d}$-asymmetry. Thus, it can only happen if $\mathfrak{d}_{x}$ takes different values between the species. As $\mathfrak{d}$ has support length at most $2\mathfrak{l}_{\mathfrak{d}}$ (see Assumption \ref{ass:grad}), a difference between $\mathfrak{d}_{x}$-values in the two species can only happen in a $\mathrm{O}(\mathfrak{l}_{\mathfrak{d}})$-length neighborhood of an already present discrepancy between the two species. The basic coupling between particles in the two species whose asymmetry speeds are equal cannot create discrepancies; it only introduces a speed $\mathrm{O}(N^{3/2})$ random drift/killing. The ``non-coupling" of the asymmetry jumps of non-equal speeds, however, can create up to two discrepancies in a single clock ring. Because these discrepancies can be created potentially anywhere in a length-$\mathrm{O}(\mathfrak{l}_{\mathfrak{d}})$ neighborhood of a discrepancy, this introduces a ``branching" mechanism with uniformly bounded number of offspring plus $\mathrm{O}(\mathfrak{l}_{\mathfrak{d}})\lesssim1$ drift at speed $\mathrm{O}(N^{3/2})$ (actually, it is $\mathrm{O}(N)$, but $\mathrm{O}(N^{3/2})$ is enough).
\item To summarize, the dynamics of a discrepancy according to the previous bullet points is a branching symmetric simple random walk of speed $\mathrm{O}(N^{2})$ with a random uniformly bounded drift/killing of speed $\mathrm{O}(N^{3/2})$. Thus, it is a (nontrivially correlated) collection of a symmetric simple random walks of speed $\mathrm{O}(N^{2})$ with an additional random drift/killing with speed $\mathrm{O}(N^{3/2})$. Moreover, because the number of total discrepancies/walks is bounded by the total number of initial discrepancies, which is at most $|\mathbb{T}_{N}|\lesssim N$, plus the number of total ringings in two species until time $1$, which is Poisson of speed $\mathrm{O}(N^{10})$, by standard tail estimates for the Poisson distribution, we have $\mathrm{O}(N^{100})$-many discrepancies/walks outside of an event with exponentially low probability in $N$. (The number of discrepancies at any time is trivially at most $|\mathbb{T}_{N}|\lesssim N$, but it is not necessarily true that the number of discrepancy walks we must consider is at most $|\mathbb{T}_{N}|$, because a discrepancy can be killed to let another be born via branching, which implies the number of ``family members"/discrepancies we must consider can be arbitrarily large.)
\end{itemize}
According to the previous bullet points, we are left to bound the probability that $\mathrm{O}(N^{100})$-many discrepancy walks end up in the support of $\mathfrak{I}^{\mathbf{X}}_{\mathfrak{l}_{\mathrm{av}}}(\mathfrak{f})$ before time $\mathfrak{t}_{\mathrm{av}}$, where the law of these discrepancy walks is described in the second sentence in the final bullet point above. This means that one of these walks that starts outside $\mathbb{B}_{\mathfrak{t}_{\mathrm{av}},\mathfrak{l}_{\mathrm{tot}}}$ travels into $\mathbb{B}_{\mathrm{tot}}$. Per Remark \ref{remark:locmap1}, this probability is exponentially small in $N^{\gamma_{0}}$ and thus at most $\mathrm{O}_{\kappa,\gamma_{0}}(N^{-\kappa})$, so we are done.
\end{proof}
\subsubsection{Local Equilibrium}
In the current section, we take advantage of the estimate in Lemma \ref{lemma:locmap3} on the localization map therein. The first step that we will take is the following expectation estimate of the $\mathbf{I}_{1}$-term from Lemma \ref{lemma:hoe2}, in which we will take $\phi$ to be the space-time average from Definition \ref{definition:locmap2} for a generic choice of functional $\mathfrak{f}$. First, we introduce useful notation.
\begin{definition}\label{definition:le1}
\fsp Consider any initial probability measure $\mu_{0,N}$ on $\Omega$. Provided any $\mathrm{t}\geq0$, we define $\mu_{\mathrm{t},N}$ to be the probability measure on $\Omega$ obtained upon evolving $\mu_{0,N}$ under the forward Kolmogorov equation associated to the interacting particle system for time $\mathrm{t}$. Let us define $\mathfrak{P}_{\mathrm{t}}$ to be the Radon-Nikodym derivative of $\mu_{\mathrm{t},N}$ with respect to the grand-canonical measure $\mu_{0}$. We also define $\bar{\mathfrak{P}}_{1}$ as the average of $\mathfrak{P}_{\mathrm{t}}$ over space-time shifts, for which we define the action $\tau_{y}\mathfrak{P}_{\mathrm{t}}(\eta)=\mathfrak{P}_{\mathrm{t}}(\tau_{y}\eta)$ for any $y\in\mathbb{T}_{N}$:
\begin{align}
\bar{\mathfrak{P}}_{1} \ &= \ \int_{0}^{1}\wt{\sum}_{y\in\mathbb{T}_{N}}\tau_{-y}\mathfrak{P}_{\mathrm{t}}\d\mathrm{t}.
\end{align}
In the construction above, we can certainly replace the action of $\tau_{-y}$ by $\tau_{y}$ without changing $\bar{\mathfrak{P}}_{1}$. In general, we can replace $\tau_{-y}$ with any bijection on $\mathbb{T}_{N}$ evaluated at $y\in\mathbb{T}_{N}$; this follows immediately by changing variables in the summation.
\end{definition}
\begin{lemma}\label{lemma:le2}
\fsp Consider any $0\leq\mathfrak{t}_{\mathrm{av}}\leq1$ and any $\mathfrak{l}_{\mathrm{av}}\in\Z_{\geq0}$. Consider any functional $\mathfrak{f}:\Omega\to\R$ whose support is contained in the block $\mathbb{B}_{\mathfrak{f}}=\llbracket-\mathfrak{l},\mathfrak{l}\rrbracket\subseteq\mathbb{T}_{N}$. We again define $\mathfrak{l}_{\mathrm{tot}}=99\mathfrak{l}+99\mathfrak{l}\mathfrak{l}_{\mathrm{av}}$ as in \emph{Lemma \ref{lemma:locmap3}}. For any $\kappa>0$, we have the following in which we recall the $\E^{\mathrm{dyn}}$ expectations and $\mathrm{Loc}=\mathrm{Loc}_{\mathfrak{t}_{\mathrm{av}},\mathfrak{l}_{\mathrm{tot}}}(\eta)$ in \emph{Lemma \ref{lemma:locmap3}}, and $\bar{\mathfrak{P}}_{1}$ in \emph{Definition \ref{definition:le1}}:
\begin{align}
\E\mathbf{I}_{1}(|\mathfrak{I}_{\mathfrak{t}_{\mathrm{av}}}^{\mathbf{T}}\mathfrak{I}_{\mathfrak{l}_{\mathrm{av}}}^{\mathbf{X}}(\mathfrak{f}_{S,y})|^{3/2}) \ \lesssim_{\kappa} \ {\E_{0}}\bar{\mathfrak{P}}_{1}\left(\E_{\mathrm{Loc}}^{\mathrm{dyn}}|\mathfrak{I}_{\mathfrak{t}_{\mathrm{av}}}^{\mathbf{T}}\mathfrak{I}_{\mathfrak{l}_{\mathrm{av}}}^{\mathbf{X}}(\mathfrak{f}_{0,0})|^{3/2}\right) + N^{-\kappa}\|\mathfrak{f}\|_{\omega;\infty}. \label{eq:le2}
\end{align}
\end{lemma}
\begin{proof}
Let us start by computing the expectation on the far LHS of \eqref{eq:le2}. Because $\mathbf{I}_{1}$ is a deterministic and linear operator, we can move the expectation past the $\mathbf{I}_{1}$ operator; observe that what the expectation now hits is a functional of the \emph{path-space} $\mathscr{D}(\R_{\geq0},\Omega)$, namely of the $\eta$-process, starting at time $S$ until time $S+\mathfrak{t}_{\mathrm{av}}$ This is the same as sampling the time-$S$ configuration and using it as the time-zero/initial configuration for the process after ``resetting" time $S$ to be time 0. Therefore, we rewrite the expectation of this space-time average as the path-space expectation with a fixed initial configuration that is then sampled/taken expectation over with respect to the law of the particle system at time $S\geq0$. Precisely, we deduce the following with explanation given after; we note the following explanation additionally requires only re-centering $\mathfrak{f}_{S,y}$ and spatially shifting $\eta_{S}$ accordingly:
\begin{align}
\E\mathbf{I}_{1}(|\mathfrak{I}_{\mathfrak{t}_{\mathrm{av}}}^{\mathbf{T}}\mathfrak{I}_{\mathfrak{l}_{\mathrm{av}}}^{\mathbf{X}}(\mathfrak{f}_{S,y})|^{3/2}) \ = \ \mathbf{I}_{1}\left(\E\E_{\tau_{y}\eta_{S}}^{\mathrm{dyn}}|\mathfrak{I}_{\mathfrak{t}_{\mathrm{av}}}^{\mathbf{T}}\mathfrak{I}_{\mathfrak{l}_{\mathrm{av}}}^{\mathbf{X}}(\mathfrak{f}_{0,0})|^{3/2}\right). \label{eq:le21}
\end{align}
To establish \eqref{eq:le21}, when we rewrite the expectation of the path-space functional $\mathfrak{I}^{\mathbf{T}}\mathfrak{I}^{\mathbf{X}}(\mathfrak{f}_{S,y})$ as an expectation with respect to the path-space measure after time $S\geq0$, with initial configuration then taken expectation over with respect to the law of the particle at time $S\geq0$, we emphasize that the inner $\E^{\mathrm{dyn}}$ expectation should have an initial configuration $\eta_{S}$ instead of $\tau_{y}\eta_{S}$, and $\mathfrak{f}_{0,0}$ on the RHS should be $\mathfrak{f}_{0,y}$; although it is now evaluated at time $0$ and initial configuration $\eta_{S}$ due to our time-$S$ shift, it is still centered at $y\in\mathbb{T}_{N}$ and not at $0\in\mathbb{T}_{N}$. However, the path-space expectation $\E^{\mathrm{dyn}}$ is invariant under any spatial shift, because the particle system dynamic \emph{law} is invariant under spatial shifts, so we may shift the initial configuration via $\tau_{y}$ and study instead the space-time average of $\mathfrak{f}_{0,0}$ rather than $\mathfrak{f}_{0,y}$. We now implement the averaging procedure from the one-block step of \cite{GPV}. This starts by observing that the inner $\E^{\mathrm{dyn}}$ is a function of only $\tau_{y}\eta_{S}$, and the function itself at which we evaluate $\tau_{y}\eta_{S}$ is a dynamic path-space expectation, which is itself independent of $y\in\mathbb{T}_{N}$ and $S\geq0$. Now rewrite the RHS of \eqref{eq:le21} as follows by noting the expectation of $\tau_{y}\eta_{S}$ is that of $\tau_{y}\eta$ times the Radon-Nikodym derivative $\mathfrak{P}_{S}$ for the law of the particle system at time $S$ with respect to the grand-canonical product measure $\mu_{0}$, where $\eta$ is distributed according to said grand-canonical measure:
\begin{align}
\mathbf{I}_{1}\left(\E\E_{\tau_{y}\eta_{S}}^{\mathrm{dyn}}|\mathfrak{I}_{\mathfrak{t}_{\mathrm{av}}}^{\mathbf{T}}\mathfrak{I}_{\mathfrak{l}_{\mathrm{av}}}^{\mathbf{X}}(\mathfrak{f}_{0,0})|^{3/2}\right) \ = \ \mathbf{I}_{1}\left({\E_{0}}\mathfrak{P}_{S}\E_{\tau_{y}\eta}^{\mathrm{dyn}}|\mathfrak{I}_{\mathfrak{t}_{\mathrm{av}}}^{\mathbf{T}}\mathfrak{I}_{\mathfrak{l}_{\mathrm{av}}}^{\mathbf{X}}(\mathfrak{f}_{0,0})|^{3/2}\right). \label{eq:le22}
\end{align}
For the RHS of \eqref{eq:le22}, inside the outermost expectation we change variables $\eta\mapsto\tau_{-y}\eta$, and thus $\tau_{y}\eta\mapsto\eta$, per point $y\in\mathbb{T}_{N}$. The grand-canonical ensemble is invariant under these spatial shifts. This places the $\tau_{y}$ operator on the Radon-Nikodym derivative $\mathfrak{P}_{S}$ and leaves the resulting $\E^{\mathrm{dyn}}$ independent of the integration space-time variables in $\mathbf{I}_{1}$. With the Fubini theorem, this gives
\begin{align}
\mathbf{I}_{1}\left({\E_{0}}\mathfrak{P}_{S}\E_{\tau_{y}\eta}^{\mathrm{dyn}}|\mathfrak{I}_{\mathfrak{t}_{\mathrm{av}}}^{\mathbf{T}}\mathfrak{I}_{\mathfrak{l}_{\mathrm{av}}}^{\mathbf{X}}(\mathfrak{f}_{0,0})|^{3/2}\right) \ &= \ \mathbf{I}_{1}\left({\E_{0}}(\tau_{-y}\mathfrak{P}_{S})\E_{\eta}^{\mathrm{dyn}}|\mathfrak{I}_{\mathfrak{t}_{\mathrm{av}}}^{\mathbf{T}}\mathfrak{I}_{\mathfrak{l}_{\mathrm{av}}}^{\mathbf{X}}(\mathfrak{f}_{0,0})|^{3/2}\right) \\
&= \ {\E_{0}}\left(\mathbf{I}_{1}(\tau_{-y}\mathfrak{P}_{S})\cdot\E_{\eta}^{\mathrm{dyn}}|\mathfrak{I}_{\mathfrak{t}_{\mathrm{av}}}^{\mathbf{T}}\mathfrak{I}_{\mathfrak{l}_{\mathrm{av}}}^{\mathbf{X}}(\mathfrak{f}_{0,0})|^{3/2}\right). \label{eq:le23}
\end{align}
Note $\bar{\mathfrak{P}}_{1}=\mathbf{I}_{1}(\tau_{-y}\mathfrak{P}_{S})$; see Definition \ref{definition:le1}. Combining previous identities \eqref{eq:le21}, \eqref{eq:le22}, and \eqref{eq:le23} with this observation gives:
\begin{align}
\E\mathbf{I}_{1}(|\mathfrak{I}_{\mathfrak{t}_{\mathrm{av}}}^{\mathbf{T}}\mathfrak{I}_{\mathfrak{l}_{\mathrm{av}}}^{\mathbf{X}}(\mathfrak{f}_{S,y})|^{3/2}) \ = \ {\E_{0}}\bar{\mathfrak{P}}_{1}\E_{\eta}^{\mathrm{dyn}}|\mathfrak{I}_{\mathfrak{t}_{\mathrm{av}}}^{\mathbf{T}}\mathfrak{I}_{\mathfrak{l}_{\mathrm{av}}}^{\mathbf{X}}(\mathfrak{f}_{0,0})|^{3/2}. \label{eq:le24}
\end{align}
We are left with replacing the $\eta$-variable in $\E^{\mathrm{dyn}}$ from the far RHS of \eqref{eq:le24} with the localization map $\mathrm{Loc}_{\mathfrak{t}_{\mathrm{av}},\mathfrak{l}_{\mathrm{tot}}}$ in Definition \ref{definition:locmap1}. For this we employ Lemma \ref{lemma:locmap3}, which provides the additional $N^{-\kappa}\|\mathfrak{f}\|_{\omega;\infty}$ term in \eqref{eq:le2}.
\end{proof}
We will now take advantage of Lemma \ref{lemma:le2} by essentially removing the $\bar{\mathfrak{P}}_{1}$ density from the RHS of \eqref{eq:le2}, upon collecting additional error terms. The mechanism for this replacement is the relative entropy inequality, the log Sobolev inequality of \cite{Yau}, and an entropy production estimate, all of which are standard and whose uses will be specified and explained below. We state the following estimate in a general framework as both Propositions \ref{prop:BGI1} and \ref{prop:BGI2} require different modifications to the RHS of \eqref{eq:le2} before applying Lemma \ref{lemma:le3} below.
\begin{lemma}\label{lemma:le3}
\fsp Take any uniformly bounded functional $\mathfrak{h}:\Omega\to\R$ whose support is contained in a subset denoted by $\mathbb{B}$. Provided any $\kappa\geq0$ satisfying $\kappa\lesssim1+\|\mathfrak{h}\|_{\omega;\infty}^{-1}\lesssim\|\mathfrak{h}\|_{\omega;\infty}^{-1}$, we have the following (recall the canonical measures from \emph{Definition \ref{definition:ensembles}}):
\begin{align}
{\E_{0}}\bar{\mathfrak{P}}_{1}|\mathfrak{h}| \ \lesssim \ \kappa^{-1}N^{-2}|\mathbb{B}|^{3} + {\sup}_{\sigma\in\R}\E^{\mu_{\sigma,\mathbb{B}}^{\mathrm{can}}}|\mathfrak{h}|. \label{eq:le3}
\end{align}
\end{lemma}
\begin{proof}
First, observe we may replace $\bar{\mathfrak{P}}_{1}$ on the LHS of \eqref{eq:le3} with its projection/conditional expectation on $\mathbb{B}$, as the functional $\mathfrak{h}$ depends only on $\eta$-variables in $\mathbb{B}$. We will let $\Pi_{\mathbb{B}}\bar{\mathfrak{P}}_{1}$ denote this projection. Moreover, we may condition on the $\eta$-density on $\mathbb{B}$. If $\mathfrak{p}_{\sigma}$ is the probability of the support of $\mu_{\sigma,\mathbb{B}}^{\mathrm{can}}$ under the $\Pi_{\mathbb{B}}\bar{\mathfrak{P}}_{1}$  measure, we get the following where $\Sigma_{\sigma}\subseteq\Omega_{\mathbb{B}}$ is the support of $\mu_{\sigma,\mathbb{B}}^{\mathrm{can}}$ and in which the sum over all $\sigma\in\R$ on the RHS of \eqref{eq:le30} below is finite because only finitely many hyperplanes $\Sigma_{\sigma}\subseteq\Omega_{\mathbb{B}}$ in the finite set $\Omega_{\mathbb{B}}$ are non-empty; note the sum over $\sigma\in\R$ of the disjoint hyperplanes $\{\mathbf{1}_{\Sigma_{\sigma}}\}_{\sigma\in\R}$ is equal to 1:
\begin{align}
{\E_{0}}\bar{\mathfrak{P}}_{1}|\mathfrak{h}| \ = \ {\E_{0}}\Pi_{\mathbb{B}}\bar{\mathfrak{P}}_{1}|\mathfrak{h}| \ = \ {\sum}_{\sigma\in\R}\mathfrak{p}_{\sigma}{\E_{0}}\left(\left(\mathfrak{p}_{\sigma}^{-1}\Pi_{\mathbb{B}}\bar{\mathfrak{P}}_{1}\mathbf{1}_{\Sigma_{\sigma}}\right)|\mathfrak{h}|\right). \label{eq:le30}
\end{align}
We forget any $\sigma\in\R$ for which $\mathfrak{p}_{\sigma}=0$ on the far RHS of \eqref{eq:le30}, as these terms do not show up when we condition on all possible $\sigma$-values. We now observe that the $\sigma$-indexed expectation on the far RHS of \eqref{eq:le30} is expectation of $|\mathfrak{h}|$ times the Radon-Nikodym derivative of $\Pi_{\mathbb{B}}\bar{\mathfrak{P}}_{1}\mathbf{1}_{\Sigma_{\sigma}}\d\mu_{0}$ with respect to the \emph{canonical measure} $\d\mu_{\sigma,\mathbb{B}}^{\mathrm{can}}$. So, we may use the relative entropy inequality, which may be found in Appendix 1.8 of \cite{KL}, with a constant $\kappa>0$, in which $\mathfrak{D}^{\sigma}_{\mathrm{KL}}(\cdot)$ denotes relative entropy with respect to $\mu_{\sigma,\mathbb{B}}^{\mathrm{can}}$ onto $\mathbb{B}$, which also may be found/defined in Appendix 1.8 of \cite{KL}; for the second term on the RHS of \eqref{eq:le31} below, we estimate a sum over $\sigma\in\R$ against probabilities $\mathfrak{p}_{\sigma}$ in terms of a supremum over $\sigma\in\R$:
\begin{align}
{\sum}_{\sigma\in\R}\mathfrak{p}_{\sigma}{\E_{0}}\left(\left(\mathfrak{p}_{\sigma}^{-1}\Pi_{\mathbb{B}}\bar{\mathfrak{P}}_{1}\mathbf{1}_{\Sigma_{\sigma}}\right)|\mathfrak{h}|\right) \ \lesssim \ \kappa^{-1}{\sum}_{\sigma\in\R}\mathfrak{p}_{\sigma}\mathfrak{D}_{\mathrm{KL}}^{\sigma}(\mathfrak{p}_{\sigma}^{-1}\Pi_{\mathbb{B}}\bar{\mathfrak{P}}_{1}\mathbf{1}_{\Sigma_{\sigma}}) + \kappa^{-1}{\sup}_{\sigma\in\R}\log\E^{\mu_{\sigma,\mathbb{B}}^{\mathrm{can}}}\exp\left(\kappa|\mathfrak{h}|\right). \label{eq:le31}
\end{align}
We now study the RHS of \eqref{eq:le31}. Below, the first bullet point basically follows the standard probability calculations in the proof of Lemma 3.3 in \cite{CYau}, starting after (3.20) therein, and the usual one-block step in \cite{GPV}. The second bullet point is calculus.
\begin{itemize}
\item We first analyze the first term on the RHS of \eqref{eq:le31}. By the log Sobolev inequality with diffusive constant $\mathrm{O}(|\mathbb{B}|^{2})$ in Theorem A of \cite{Yau}, we bound the $\mathfrak{D}^{\sigma}_{\mathrm{KL}}$ term by $\mathrm{O}(|\mathbb{B}|^{2})$ times the \emph{Dirichlet form} of $\mathfrak{p}_{\sigma}^{-1}\Pi_{\mathbb{B}}\bar{\mathfrak{P}}_{1}\mathbf{1}_{\Sigma_{\sigma}}$. The resulting convex combination over $\sigma$ of these Dirichlet forms is, by standard probability, the Dirichlet form of $\Pi_{\mathbb{B}}\bar{\mathfrak{P}}_{1}$ with respect to the grand-canonical measure $\mu_{0}$ projected on $\mathbb{B}$. By standard entropy production as in Lemma 4.1 in \cite{DT}, without the need for boundary considerations, and Proposition 4.3 in \cite{DT}, this is then controlled by $N^{-2}|\mathbb{B}|$. This gives an upper bound for the first term on the RHS of \eqref{eq:le31} given by the first term on the RHS of the proposed estimate \eqref{eq:le3}.
\item Because $\kappa\lesssim\|\mathfrak{h}\|_{\omega;\infty}^{-1}$ by assumption, the argument $\kappa|\mathfrak{h}|$ in the exponential in \eqref{eq:le31} is uniformly bounded. Since the exponential function is uniformly Lipschitz on uniformly bounded sets, for $\wt{\kappa}>0$ universal and independent of $\kappa$,
\begin{align}
\log\E^{\mu_{\sigma,\mathbb{B}}^{\mathrm{can}}}\exp\left(\kappa|\mathfrak{h}|\right) \ \leq \ \log\E^{\mu_{\sigma,\mathbb{B}}^{\mathrm{can}}}\left(\exp(0) + \wt{\kappa}\kappa|\mathfrak{h}|\right) \ = \ \log\E^{\mu_{\sigma,\mathbb{B}}^{\mathrm{can}}}\left(1+\wt{\kappa}\kappa|\mathfrak{h}|\right) \ \leq \ \wt{\kappa}\kappa|\mathfrak{h}|.
\end{align}
Dividing by $\kappa$ estimates the second term on the RHS of \eqref{eq:le31} by the second term in the proposed estimate \eqref{eq:le3}.
\end{itemize}
This completes the proof.
\end{proof}
\subsubsection{Equilibrium Estimates}
We now record estimates on stationary particle systems that will be crucial to study expectations of space-time averages provided our reduction to local equilibrium in Lemma \ref{lemma:le3}. The first is a spatial average estimate, which exploits spatially fluctuating behavior of local functionals at the stationary measure. This will be used as a large-deviations-type estimate in future applications.
\begin{lemma}\label{lemma:ee}
\fsp Suppose $\{\mathfrak{f}_{\mathfrak{j}}\}_{\mathfrak{j}\geq0}$ are uniformly bounded, and their respective supports are contained inside $\{\mathbb{B}_{\mathfrak{j}}\}_{\mathfrak{j}\geq0}$. Suppose $\{\mathbb{B}_{\mathfrak{j}}\}_{\mathfrak{j}\geq0}$ are mutually disjoint, and that $\mathfrak{f}_{\mathfrak{j}}$ vanishes in expectation with respect to any canonical measure on its support for every $\mathfrak{j}$. We have the following for any $\gamma,\kappa>0$, where probability and expectation below are both with respect to any canonical measure on $\mathbb{B}_{1}\cup\ldots\cup\mathbb{B}_{\mathfrak{J}}$, and $\mathcal{E}_{\mathfrak{J}}$ is the event where the average of $\mathfrak{f}_{1},\ldots,\mathfrak{f}_{\mathfrak{J}}$ exceeds $N^{\gamma}|\mathfrak{J}|^{-1/2}\max_{\mathfrak{j}=1,\ldots,\mathfrak{J}}\|\mathfrak{f}_{\mathfrak{j}}\|_{\omega;\infty}$ in absolute value:
\begin{align}
\mathbf{P}\left(\mathcal{E}_{\mathfrak{J}}\right) \ \lesssim_{\gamma,\kappa} \ N^{-\kappa}. \label{eq:ee} 
\end{align}
\end{lemma}
\begin{proof}
We note $\mathfrak{f}_{\mathfrak{j}}$ are conditionally mean zero. Indeed, their supports are mutually disjoint, and each is mean zero with respect to every canonical measure; for any canonical measure on a set in $\mathbb{T}_{N}$, conditioning on one subset induces a convex combination of canonical measures on any other non-intersecting subset. Standard concentration inequalities like the Azuma martingale inequality, therefore give that the average of $\mathfrak{f}_{1},\ldots,\mathfrak{f}_{\mathfrak{J}}$ is sub-Gaussian with zero mean and variance of order $|\mathfrak{J}|^{-1}\max_{\mathfrak{j}=1,\ldots,\mathfrak{J}}\|\mathfrak{f}_{\mathfrak{j}}\|_{\omega;\infty}^{2}$, from which the proposed estimate follows by pretending that this average of $\mathfrak{f}_{1},\ldots,\mathfrak{f}_{\mathfrak{J}}$ functionals is Gaussian with zero mean and variance $|\mathfrak{J}|^{-1}\max_{\mathfrak{j}=1,\ldots,\mathfrak{J}}\|\mathfrak{f}_{\mathfrak{j}}\|_{\omega;\infty}^{2}$ along with standard Gaussian moment generating function control. This yields an exponentially small (in $N^{\gamma})$ estimate for $\mathbf{P}(\mathcal{E}_{\mathfrak{J}})$, which is exponentially small in $N^{\gamma}$ and thus $\mathrm{O}_{\gamma,\kappa}(N^{-\kappa})$ for any $\kappa>0$.
\end{proof}
We proceed with equilibrium estimates for space-time averages instead of just spatial averages. The primary advantage for this is the ability to take advantage of the ``more ergodic" time-averaging of statistics of the particle system; recall the time-scaling is $N^{2}$ whereas the spatial scaling is $N$. However, the following estimates only hold in a second moment at best, a priori, and thus quite far from the large deviations scale of Lemma \ref{lemma:ee}; see {Proposition 7 and Corollary 1} in \cite{GJ15} for more details.
\begin{lemma}\label{lemma:ee2}
\fsp Suppose that $\mathfrak{f}$ is a uniformly bounded functional, and its support is contained in $\mathbb{B}\subseteq\mathbb{T}_{N}$. We additionally assume that the expectation of $\mathfrak{f}$ with respect to any canonical measure on $\mathbb{B}$ is equal to zero. Provided any time-scale $\mathfrak{t}_{\mathrm{av}}\geq0$ and any length-scale $\mathfrak{l}_{\mathrm{av}}\in\Z_{\geq0}$ and any $\kappa>0$, we have the following estimate that we clarify/explain afterwards and for which we recall the notation of \emph{Definition \ref{definition:locmap1}}, \emph{Definition \ref{definition:locmap2}}, and \emph{Lemma \ref{lemma:locmap3}}:
\begin{align}
\sup_{\sigma\in\R}\left(\E^{\sigma}\E_{\mathrm{Loc}}^{\mathrm{dyn}}|\mathfrak{I}_{\mathfrak{t}_{\mathrm{av}}}^{\mathbf{T}}\mathfrak{I}_{\mathfrak{l}_{\mathrm{av}}}^{\mathbf{X}}(\mathfrak{f}_{0,0})|^{2}\right)^{1/2} \ \lesssim_{\kappa} \ N^{-1}\mathfrak{t}_{\mathrm{av}}^{-1/2}\mathfrak{l}_{\mathrm{av}}^{-1/2}|\mathbb{B}|\|\mathfrak{f}_{0,0}\|_{\omega;\infty}+N^{-\kappa}. \label{eq:ee2}
\end{align}
We have used the abbreviation $\mathrm{Loc}=\mathrm{Loc}_{\mathfrak{t}_{\mathrm{av}},\mathfrak{l}_{\mathrm{tot}}}\eta$, where $\mathfrak{l}_{\mathrm{tot}}=99|\mathbb{B}|+99|\mathbb{B}|\mathfrak{l}_{\mathrm{av}}$ is much larger than the support of $\mathfrak{I}^{\mathbf{X}}_{\mathfrak{l}_{\mathrm{av}}}(\mathfrak{f}_{0,0})$. Observe $\mathrm{Loc}$ is only a function of $\eta$-variables on the neighborhood $\mathbb{B}_{\mathfrak{t}_{\mathrm{av}},\mathfrak{l}_{\mathrm{tot}}}$; therefore, so is the inner expectation. The outer expectation on the LHS of \eqref{eq:ee2} is expectation over these $\eta$-variables in $\mathbb{B}_{\mathfrak{t}_{\mathrm{av}},\mathfrak{l}_{\mathrm{tot}}}$, sampled from canonical ensemble on $\mathbb{B}_{\mathfrak{t}_{\mathrm{av}},\mathfrak{l}_{\mathrm{tot}}}$ of $\eta$-density equal to $\sigma$. In particular, inside the supremum on the LHS of \eqref{eq:ee2} is the expectation of the square of the space-time average of $\mathfrak{f}_{0,0}$, where the initial configuration for the space-time average/particle system has $\eta$-variables in $\mathbb{B}_{\mathfrak{t}_{\mathrm{av}},\mathfrak{l}_{\mathrm{tot}}}$ sampled via the canonical ensemble of parameter $\sigma$ on $\mathbb{B}_{\mathfrak{t}_{\mathrm{av}},\mathfrak{l}_{\mathrm{tot}}}$ and has $\eta$-variables outside $\mathbb{B}_{\mathfrak{t}_{\mathrm{av}},\mathfrak{l}_{\mathrm{tot}}}$ deterministically equal to $1$.
\end{lemma}
\begin{proof}
Suppose that instead of the $\Omega$-valued process/particle system considered in this paper that the particle system in question in Lemma \ref{lemma:ee2} is actually valued in $\Omega_{\mathbb{B}_{2}}$ with $\mathbb{B}_{2}=\mathbb{B}_{\mathfrak{t}_{\mathrm{av}},\mathfrak{l}_{\mathrm{tot}}}$. In particular, suppose the particle system/particle random walks are $\mathbb{B}_{\mathfrak{t}_{\mathrm{av}},\mathfrak{l}_{\mathrm{tot}}}$-periodic, in which case the particle/$\eta$ configuration (on $\mathbb{B}_{\mathfrak{t}_{\mathrm{av}},\mathfrak{l}_{\mathrm{tot}}}$) in \eqref{eq:ee2} is distributed according to canonical measure on $\mathbb{B}_{\mathfrak{t}_{\mathrm{av}},\mathfrak{l}_{\mathrm{tot}}}$. Observe this $\mathbb{B}_{\mathfrak{t}_{\mathrm{av}},\mathfrak{l}_{\mathrm{tot}}}$-periodic system has canonical measures as invariant measures; this follows by the same reason that the $\mathbb{T}_{N}$-periodic system has canonical measures on $\mathbb{T}_{N}$ as invariant measures. Therefore, the Kipnis-Varadhan inequality in Appendix 1.6 of \cite{KL} implies that uniformly in $\sigma$, the double expectation on the LHS of \eqref{eq:ee2} is bounded above by $\mathrm{O}(\mathfrak{t}_{\mathrm{av}}^{-1})$ times a squared Sobolev norm of the spatial average $\mathfrak{I}^{\mathbf{X}}_{\mathfrak{l}_{\mathrm{av}}}(\mathfrak{f}_{0,0})$. {From Proposition 6} in \cite{GJ15}, said squared Sobolev norm of $\mathfrak{I}^{\mathbf{X}}_{\mathfrak{l}_{\mathrm{av}}}(\mathfrak{f}_{0,0})$ is $\mathrm{O}(N^{-2}\mathfrak{l}_{\mathrm{av}}^{-1}\|\mathfrak{f}\|_{\omega;\infty}^{2}|\mathbb{B}|^{2})$, where $|\mathbb{B}|$ is the support length of $\mathfrak{f}$. Thus, we have established the proposed estimate \eqref{eq:ee2} if we can replace the $\Omega$-valued/``original" particle system with the $\mathbb{B}_{\mathfrak{t}_{\mathrm{av}},\mathfrak{l}_{\mathrm{tot}}}$-periodic system, thereby forgetting $\eta$ outside $\mathbb{B}_{\mathfrak{t}_{\mathrm{av}},\mathfrak{l}_{\mathrm{tot}}}$.

We now make the aforementioned replacement and estimate the resulting error, which will provide the $N^{-\kappa}$-term on the RHS of \eqref{eq:ee2}. We will use a coupling argument similar to the proof of Lemma \ref{lemma:locmap3}. In what follows, we refer to the $\Omega$-valued/``original" particle system as Species 1, and we refer to the $\mathbb{B}_{\mathfrak{t}_{\mathrm{av}},\mathfrak{l}_{\mathrm{tot}}}$-periodic system appearing below as Species 2.
\begin{itemize}
\item As in the proof of Lemma \ref{lemma:locmap3}, the symmetric dynamic in Species 1 may be thought of as attaching Poisson clocks to \emph{bonds} in $\mathbb{T}_{N}$ connecting nearest neighbors, where the ringing of the Poisson clock associated to a given bond corresponds to swapping $\eta$-variables at the points attached to that bond. For Species 2, let us also construct the symmetric dynamic as attaching Poisson clocks to bonds in $\mathbb{B}_{\mathfrak{t}_{\mathrm{av}},\mathfrak{l}_{\mathrm{tot}}}$ that connect points that are distance 1 apart with respect to the geodesic/torus distance on $\mathbb{B}_{\mathfrak{t}_{\mathrm{av}},\mathfrak{l}_{\mathrm{tot}}}$; this includes the maximum and minimum of $\mathbb{B}_{\mathfrak{t}_{\mathrm{av}},\mathfrak{l}_{\mathrm{tot}}}$, for example. For those bonds that appear in both Species 1 and Species 2, we will use the same bond clocks, so that shared/common bonds \emph{always} swap $\eta$-spins together. For bonds which are shared between Species 1 and Species 2, we use the modified basic coupling for the respective asymmetric dynamics from the proof of Lemma \ref{lemma:locmap3} (to account for the $\mathfrak{d}$-asymmetry). All other bonds are then chosen arbitrarily/independently.
\item Observe that the error in the LHS of \eqref{eq:ee2} after replacing the $\Omega$-valued/``original" system with the $\mathbb{B}_{\mathfrak{t}_{\mathrm{av}},\mathfrak{l}_{\mathrm{tot}}}$-periodic system is $\mathrm{O}(\|\mathfrak{f}\|_{\omega;\infty}^{2})\lesssim1$ times the probability Species 1 and Species 2, under the coupling in the previous bullet point, have discrepancy inside the support of $\mathfrak{I}^{\mathbf{X}}_{\mathfrak{l}_{\mathrm{av}}}(\mathfrak{f}_{0,0})$, similar to the proof of Lemma \ref{lemma:locmap3}. Below, we identify a discrepancy in $\mathbb{B}_{\mathfrak{t}_{\mathrm{av}},\mathfrak{l}_{\mathrm{tot}}}$ with its entire ancestry, similar to the final bullet point in the proof of Lemma \ref{lemma:locmap3} when we considered a branching random walk as a collection of correlated random walks. In particular, even if the discrepancy was born from a branching, we identify it as a random walk that followed its ancestors until said branching, after which it becomes its own branching random walk.
\item Suppose that we observe a discrepancy in the support of $\mathfrak{I}^{\mathbf{X}}_{\mathfrak{l}_{\mathrm{av}}}(\mathfrak{f}_{0,0})$, and therefore in $\mathbb{B}_{\mathfrak{t}_{\mathrm{av}},\mathfrak{l}_{\mathrm{tot}}}$. This discrepancy must have been born at a point where the clocks are not all coupled between Species 1 and Species 2 (like in the proof of Lemma \ref{lemma:locmap3}, coupled clocks cannot create discrepancies). By construction, such points are initially within $\mathrm{O}(\mathfrak{l}_{\mathfrak{d}})$ of the boundary of $\mathbb{B}_{\mathfrak{t}_{\mathrm{av}},\mathfrak{l}_{\mathrm{tot}}}$. This discrepancy must have propagated into the support of $\mathfrak{I}^{\mathbf{X}}_{\mathfrak{l}_{\mathrm{av}}}(\mathfrak{f}_{0,0})$ by length-1 jumps from $\mathrm{O}(\mathfrak{l}_{\mathfrak{d}})$ of said boundary. Third, while said discrepancy \emph{in} $\mathbb{B}_{\mathfrak{t}_{\mathrm{av}},\mathfrak{l}_{\mathrm{tot}}}$ travels to the support of $\mathfrak{I}^{\mathbf{X}}_{\mathfrak{l}_{\mathrm{av}}}(\mathfrak{f}_{0,0})$, when it gets $\mathrm{O}(\mathfrak{l}_{\mathfrak{d}})$ away from the boundary of $\mathbb{B}_{\mathfrak{t}_{\mathrm{av}},\mathfrak{l}_{\mathrm{tot}}}$, it then travels according to the branching random walk that we described in the last bullet point in the proof of Lemma \ref{lemma:locmap3} because the different boundary conditions in the two species become irrelevant when we are in $\mathbb{B}_{\mathfrak{t}_{\mathrm{av}},\mathfrak{l}_{\mathrm{tot}}}$ and beyond $\mathrm{O}(\mathfrak{l}_{\mathfrak{d}})$ of its boundary. Therefore, we see said branching random walk travel at least the distance from within $\mathrm{O}(\mathfrak{l}_{\mathfrak{d}})$ of the boundary of $\mathbb{B}_{\mathfrak{t}_{\mathrm{av}},\mathfrak{l}_{\mathrm{tot}}}$ to the support of $\mathfrak{I}^{\mathbf{X}}_{\mathfrak{l}_{\mathrm{av}}}(\mathfrak{f}_{0,0})$, if we see a discrepancy in the support of $\mathfrak{I}^{\mathbf{X}}_{\mathfrak{l}_{\mathrm{av}}}(\mathfrak{f}_{0,0})$ at all. (It may be the case that one of these discrepancy random walks returns to within $\mathrm{O}(\mathfrak{l}_{\mathfrak{d}})$ of the boundary of $\mathbb{B}_{\mathfrak{t}_{\mathrm{av}},\mathfrak{l}_{\mathrm{tot}}}$, where it does not travel like the aforementioned branching random walk, but in this case, as it travels into the support of $\mathfrak{I}^{\mathbf{X}}_{\mathfrak{l}_{\mathrm{av}}}(\mathfrak{f}_{0,0})$ we just wait for it to get beyond $\mathrm{O}(\mathfrak{l}_{\mathfrak{d}})$ of said boundary again.) Thus, the probability that we see any discrepancy in the support of $\mathfrak{I}^{\mathbf{X}}_{\mathfrak{l}_{\mathrm{av}}}(\mathfrak{f}_{0,0})$ is controlled by random walk probabilities and a large deviations bound for the number of discrepancy walks as in the last bullet point in the proof of Lemma \ref{lemma:locmap3}.
\end{itemize}
This completes the proof.
\end{proof}
\subsubsection{Spatial Replacement}
We introduce a set of replacement estimates that allow us to introduce space-time averaging for a functional that is multiplied by $\mathbf{Y}^{N}$ and the heat kernel while estimating the error in doing so. 
\begin{definition}\label{definition:BGI10}
\fsp Consider any functional $\mathfrak{f}:\Omega\to\R$ and any pair of length-scales $\mathfrak{l},\mathfrak{l}'\in\Z_{\geq0}$. Define a transfer-of-length-scale operator $\mathfrak{D}^{\mathbf{X}}_{\mathfrak{l},\mathfrak{l}'}(\mathfrak{f})=\mathfrak{I}_{\mathfrak{l}}^{\mathbf{X}}(\mathfrak{f})-\mathfrak{I}_{\mathfrak{l}'}^{\mathbf{X}}(\mathfrak{f})$, where the $\mathfrak{I}^{\mathbf{X}}$ operator, with identity time-average $\mathfrak{I}^{\mathbf{T}}$ operator, is from Definition \ref{definition:locmap2}.
\end{definition}
\begin{lemma}\label{lemma:xr}
\fsp Consider any $\mathfrak{f}:\Omega\to\R$ whose support has length at most $\mathfrak{l}_{\mathfrak{f}}$ along with any length-scale $|\mathfrak{l}|\mathfrak{l}_{\mathfrak{f}}\leq\mathfrak{l}_{N}$ with $\mathfrak{l}_{N}$ from \emph{Definition \ref{definition:KPZ1}}. For any $\mathrm{t}\geq0$, we have the following in which we let $\bar{\mathfrak{l}}=|\mathfrak{l}|\mathfrak{l}_{\mathfrak{f}}$ in the statement and proof of this result:
\begin{align}
\E\|\mathbf{H}^{N}(\mathfrak{D}_{0,\mathfrak{l}}^{\mathbf{X}}(\mathfrak{I}^{\mathbf{T}}_{\mathrm{t}}(\mathfrak{f}_{S,y}))\mathbf{Y}_{S,y}^{N})\|_{1;\mathbb{T}_{N}} \ \lesssim \ N^{-\frac12+\e_{\mathrm{RN}}+\e_{\mathrm{ap}}}\|\mathfrak{f}\|_{\omega;\infty} + N^{-\frac12+5\e_{\mathrm{ap}}}\bar{\mathfrak{l}}^{1/2}\E\|\mathbf{H}^{N}(|\mathfrak{I}^{\mathbf{T}}_{\mathrm{t}}(\mathfrak{f}_{S,y})|)\|_{1;\mathbb{T}_{N}}. \label{eq:xr}
\end{align}
\end{lemma}
\begin{remark}\label{remark:xr}
The assumption $|\mathfrak{l}|\mathfrak{l}_{\mathfrak{f}}\leq\mathfrak{l}_{N}$ will be important because we need spatial regularity of $\mathbf{Y}^{N}$ on length-scale $|\mathfrak{l}|\mathfrak{l}_{\mathfrak{f}}$, and we only guarantee this if $|\mathfrak{l}|\mathfrak{l}_{\mathfrak{f}}\leq\mathfrak{l}_{N}$ by the constructions in Definition \ref{definition:KPZ1} and Definition \ref{definition:KPZ5}. We will actually soften moderately the assumption $|\mathfrak{l}|\mathfrak{l}_{\mathfrak{f}}\leq\mathfrak{l}_{N}$ in a later ``adapted" application of Lemma \ref{lemma:xr}, namely in the proof of Lemma \ref{lemma:BGI2II1}, with explanation. The first term on the RHS of \eqref{eq:xr} would not change if $\bar{\mathfrak{l}}=|\mathfrak{l}|\mathfrak{l}_{\mathfrak{f}}\approx\mathfrak{l}_{N}N^{\gamma}$ for $\gamma>0$ small; only the second one slightly would.
\end{remark}
\begin{proof}
The $\mathfrak{D}^{\mathbf{X}}$-term on the LHS of \eqref{eq:xr} may be realized as an average of spatial gradients of $\mathfrak{f}$ on length-scales that are at most $|\bar{\mathfrak{l}}|$. Indeed, the $\mathfrak{I}^{\mathbf{X}}(\mathfrak{f})$-term defining the $\mathfrak{D}^{\mathbf{X}}$-term on the LHS of \eqref{eq:xr} is an average of spatial translations of $\mathfrak{f}$ with length-scale at most $|\bar{\mathfrak{l}}|$, and the difference of each spatial translation with $\mathfrak{f}$ is a spatial gradient of $\mathfrak{f}$ of the same length-scale. Thus it suffices to prove \eqref{eq:xr} but replacing $\mathfrak{D}^{\mathbf{X}}_{0,\mathfrak{l}}$ on the LHS of \eqref{eq:xr} by $\grad^{\mathbf{X}}_{\mathfrak{l}'}$ for any $|\mathfrak{l}'|\leq\bar{\mathfrak{l}}=|\mathfrak{l}|\mathfrak{l}_{\mathfrak{f}}$. Letting $\mathfrak{l}'$ be such a length-scale, we start with the following discrete-type Leibniz rule; it may be checked directly:
\begin{align}
\|\mathbf{H}^{N}(\grad_{\mathfrak{l}'}^{\mathbf{X}}\mathfrak{I}^{\mathbf{T}}_{\mathrm{t}}(\mathfrak{f}_{S,y})\mathbf{Y}_{S,y}^{N})\|_{1;\mathbb{T}_{N}} \ &= \ \|\mathbf{H}^{N}(\grad_{\mathfrak{l}'}^{\mathbf{X}}(\mathfrak{I}^{\mathbf{T}}_{\mathrm{t}}(\mathfrak{f}_{S,y})\mathbf{Y}_{S,y-\mathfrak{l}'}^{N}))-\mathbf{H}^{N}(\mathfrak{I}_{\mathrm{t}}^{\mathbf{T}}(\mathfrak{f}_{S,y})\grad_{\mathfrak{l}'}^{\mathbf{X}}\mathbf{Y}_{S,y-\mathfrak{l}'}^{N})\|_{1;\mathbb{T}_{N}} \\
&\leq \ \|\mathbf{H}^{N}(\grad_{\mathfrak{l}'}^{\mathbf{X}}(\mathfrak{I}^{\mathbf{T}}_{\mathrm{t}}(\mathfrak{f}_{S,y})\mathbf{Y}_{S,y-\mathfrak{l}'}^{N}))\|_{1;\mathbb{T}_{N}} + \|\mathbf{H}^{N}(\mathfrak{I}^{\mathbf{T}}_{\mathrm{t}}(\mathfrak{f}_{S,y})\grad_{\mathfrak{l}'}^{\mathbf{X}}\mathbf{Y}_{S,y-\mathfrak{l}'}^{N})\|_{1;\mathbb{T}_{N}}. \label{eq:xr1}
\end{align}
The second line follows by the triangle inequality for $\|\|_{1;\mathbb{T}_{N}}$ and linearity of expectation. Note the additional spatial shift in $\mathbf{Y}^{N}$ follows from the discrete nature of the spatial gradients; if we considered instead an ``infinitesimal" length-scale, this shift would disappear as $\mathfrak{l}'\to0$ and we would recover the usual Leibniz rule. We will now estimate each of the terms in \eqref{eq:xr1}. For the first term, we may employ the heat operator gradient estimate in Proposition \ref{prop:heat} along with the estimate $|\mathbf{Y}^{N}|\lesssim N^{\e_{\mathrm{ap}}}$ that follows via Definitions \ref{definition:KPZ1} and \ref{definition:KPZ5}; this estimates the first term in \eqref{eq:xr1} by moving $\grad^{\mathbf{X}}$ onto the macroscopically smooth $\mathbf{H}^{N}$: 
\begin{align}
\|\mathbf{H}^{N}(\grad_{\mathfrak{l}'}^{\mathbf{X}}(\mathfrak{I}^{\mathbf{T}}_{\mathrm{t}}(\mathfrak{f}_{S,y})\mathbf{Y}_{S,y-\mathfrak{l}'}^{N}))\|_{1;\mathbb{T}_{N}} \ \leq \ \|\mathbf{Y}^{N}\|_{1;\mathbb{T}_{N}}\|\mathfrak{f}\|_{\omega;\infty}N^{-1}|\mathfrak{l}'| \ \leq \ N^{-\frac12+\e_{\mathrm{RN}}+\e_{\mathrm{ap}}}\|\mathfrak{f}\|_{\omega;\infty},
\end{align}
since any time-average is uniformly bounded by its input $|\mathfrak{I}^{\mathbf{T}}(\mathfrak{f})|\leq\|\mathfrak{f}\|_{\omega;\infty}$, since $|\mathbf{Y}^{N}|\leq N^{\e_{\mathrm{ap}}}$, and since $|\mathfrak{l}'|\leq|\bar{\mathfrak{l}}|\leq|\mathfrak{l}_{N}|=N^{1/2+\e_{\mathrm{RN}}}$; see Definitions \ref{definition:KPZ1} and \ref{definition:KPZ5}. Thus we are left to estimate the second term in \eqref{eq:xr1}. Observe $\mathbf{Y}^{N}=0$ or $\mathbf{Y}^{N}=\mathbf{Z}^{N}$. The former case is trivial, and the second case implies $\mathbf{Z}^{N}$ has a priori spatial regularity on length-scale $\mathfrak{l}'$, since $|\mathfrak{l}'|\leq|\bar{\mathfrak{l}}|\leq\mathfrak{l}_{N}$; see Definitions \ref{definition:KPZ1} and \ref{definition:KPZ5}. This spatial regularity controls the gradient in the second term in \eqref{eq:xr1} uniformly in space-time, so
\begin{align}
\|\mathbf{H}^{N}(\mathfrak{I}^{\mathbf{T}}_{\mathrm{t}}(\mathfrak{f}_{S,y})\grad_{\mathfrak{l}'}^{\mathbf{X}}\mathbf{Y}_{S,y-\mathfrak{l}'}^{N})\|_{1;\mathbb{T}_{N}} \ \leq \ \|\grad_{\mathfrak{l}'}^{\mathbf{X}}\mathbf{Y}^{N}\|_{1;\mathbb{T}_{N}}\|\mathbf{H}^{N}(|\mathfrak{I}^{\mathbf{T}}_{\mathrm{t}}(\mathfrak{f}_{S,y})|)\|_{1;\mathbb{T}_{N}} \ \leq \ N^{-\frac12+5\e_{\mathrm{ap}}}|\mathfrak{l}'|^{1/2}\|\mathbf{H}^{N}(|\mathfrak{I}^{\mathbf{T}}_{\mathrm{t}}(\mathfrak{f}_{S,y})|)\|_{1;\mathbb{T}_{N}}. \nonumber
\end{align}
As $|\mathfrak{l}'|\leq\bar{\mathfrak{l}}$, we ultimately deduce that the second term in \eqref{eq:xr1} is bounded by the second term in \eqref{eq:xr}, so we are done.
\end{proof}
\subsubsection{Multiscale Time-Replacement}
The last preliminary estimates we introduce will serve important for replacing functionals and their spatial averages with their respective time-averages. We emphasize that such replacement by mesoscopic time-average is difficult because of the poor time-regularity of the $\mathbf{Y}^{N}$ process against which we multiply the functionals/spatial averages that we want to replace with their respective time-averages. This ultimately leads us to a multiscale replacement, which, per standard multiscale analysis, forces us to simultaneously take advantage of progressively improving estimates for space-time averages on progressively larger time-scales; see \eqref{eq:ee2} and its dependence in the time-scale $\mathfrak{t}_{\mathrm{av}}$ therein. First, some convenient notation.
\begin{definition}\label{definition:mtr0}
\fsp Consider any $\mathfrak{f}:\Omega\to\R$ and any pair of time-scales $\mathfrak{t},\mathfrak{t}'\geq0$. We define the transfer-of-time-scale operator $\mathfrak{D}_{\mathfrak{t},\mathfrak{t}'}^{\mathbf{T}}(\mathfrak{f})=\mathfrak{I}_{\mathfrak{t}}^{\mathbf{T}}(\mathfrak{f})-\mathfrak{I}_{\mathfrak{t}'}^{\mathbf{T}}(\mathfrak{f})$, where $\mathfrak{I}^{\mathbf{T}}$ is defined in Definition \ref{definition:locmap2} by taking the identity spatial average/$\mathfrak{l}_{\mathrm{av}}=0$ therein.
\end{definition}
The first step we take is the introduction of a time-average with respect to \emph{some} time-scale, which we take as the microscopic time-scale $N^{-2}$ in the following preliminary estimate. We emphasize that the following estimate is established by an integration-by-parts-type calculation; in order to estimate the integrated time-gradient of a functional, we will move such time-gradient onto the other factors/integrands. We then estimate these time-gradients along with another pair of ultimately negligible ``short-time" boundary terms/integrals. We note that the proof of the following is a time-version of Lemma \ref{lemma:xr}, though it is somewhat more involved because time-gradients of $\mathbf{Y}^{N}$ may cross the time at which $\mathbf{Y}^{N}$ goes from being equal to $\mathbf{Z}^{N}$ to when it is zero.
\begin{lemma}\label{lemma:mtr1}
\fsp Consider any functional $\mathfrak{f}:\Omega\to\R$ and the time-scales $\mathfrak{t}_{-\infty}=0$ and $\mathfrak{t}_{0}=N^{-2}$. Provided any $\gamma>0$, we have
\begin{align}
\E\|\mathbf{H}^{N}(\mathfrak{D}_{\mathfrak{t}_{-\infty},\mathfrak{t}_{0}}^{\mathbf{T}}(\mathfrak{f}_{S,y})\mathbf{Y}_{S,y}^{N})\|_{1;\mathbb{T}_{N}} \ \lesssim_{\gamma} \ N^{-2+\gamma+\e_{\mathrm{ap}}}\|\mathfrak{f}\|_{\omega;\infty} + N^{-1/2+3\e_{\mathrm{ap}}}\E\|\mathbf{H}^{N}(|\mathfrak{f}_{S,y}|)\|_{1;\mathbb{T}_{N}}. \label{eq:mtr1}
\end{align}
\end{lemma}
\begin{proof}
Observe the $\mathfrak{D}$-operator on the LHS of \eqref{eq:mtr1} is a difference between a scale $N^{-2}$ time-average and $\mathfrak{f}$ itself, then multiplied by $\mathbf{Y}^{N}$ and integrated against the heat kernel defining $\mathbf{H}^{N}$. Thus, it is an average of time-gradients of $\mathfrak{f}$ with respect to time-scales that are between $\mathfrak{t}_{-\infty}=0$ and $\mathfrak{t}_{0}=N^{-2}$. We then estimate time-gradients with respect to these time-scales uniformly over said time-scales. In particular, it suffices to get, for any $\gamma>0$:
\begin{align}
{\sup}_{\mathrm{s}\in[0,N^{-2}]}\E\|\mathbf{H}^{N}(\grad_{\mathrm{s}}^{\mathbf{T}}\mathfrak{f}_{S,y}\mathbf{Y}_{S,y}^{N})\|_{1;\mathbb{T}_{N}} \ \lesssim_{\gamma} \ N^{-2+\gamma+\e_{\mathrm{ap}}}\|\mathfrak{f}\|_{\omega;\infty} + N^{-1/2+3\e_{\mathrm{ap}}}\E\|\mathbf{H}^{N}(|\mathfrak{f}_{S,y}|)\|_{1;\mathbb{T}_{N}}. \label{eq:mtr11}
\end{align}
We now write the following Leibniz-rule-type identity, which is a time-version of \eqref{eq:xr1}; because we are taking time-gradients with respect to positive time-scales, instead of differentiating on infinitesimal time-scales, the following identity includes additional time-shifts that would disappear if we took the time-scale $\mathrm{s}$ to zero. We emphasize, however, that the following identity may be easily checked just by expanding time-gradients on the RHS of the first line below and cancelling terms:
\begin{align}
\|\mathbf{H}^{N}(\grad_{\mathrm{s}}^{\mathbf{T}}\mathfrak{f}_{S,y}\mathbf{Y}_{S,y}^{N})\|_{1;\mathbb{T}_{N}} \ &= \ \|\mathbf{H}^{N}\left(\grad_{\mathrm{s}}^{\mathbf{T}}(\mathfrak{f}_{S,y}\mathbf{Y}_{S-\mathrm{s},y}^{N})\right) - \mathbf{H}^{N}\left(\mathfrak{f}_{S,y}\grad_{\mathrm{s}}^{\mathbf{T}}\mathbf{Y}_{S-\mathrm{s},y}^{N}\right)\|_{1;\mathbb{T}_{N}} \\
&\leq \ \|\mathbf{H}^{N}\left(\grad_{\mathrm{s}}^{\mathbf{T}}(\mathfrak{f}_{S,y}\mathbf{Y}_{S-\mathrm{s},y}^{N})\right)\|_{1;\mathbb{T}_{N}} + \|\mathbf{H}^{N}\left(\mathfrak{f}_{S,y}\grad_{-\mathrm{s}}^{\mathbf{T}}\mathbf{Y}_{S,y}^{N}\right)\|_{1;\mathbb{T}_{N}}. \label{eq:mtr12}
\end{align}
The second line \eqref{eq:mtr12} follows by the triangle inequality along with tautologically rewriting the $\mathbf{Y}^{N}$ gradient. Let us now estimate each term in \eqref{eq:mtr12} uniformly in the allowed time-scales $\mathrm{s}$. For the first term in \eqref{eq:mtr12}, we use Proposition \ref{prop:heat} and a priori upper bounds for $\mathbf{Y}^{N}$ from Definitions \ref{definition:KPZ1} and \ref{definition:KPZ5} to get the following \emph{deterministic} estimate for any $\gamma>0$, which moves $\grad^{\mathbf{T}}$ onto the macroscopically smooth (in time) heat operator $\mathbf{H}^{N}$:
\begin{align}
\|\mathbf{H}^{N}\left(\grad_{\mathrm{s}}^{\mathbf{T}}(\mathfrak{f}_{S,y}\mathbf{Y}_{S-\mathrm{s},y}^{N})\right)\|_{1;\mathbb{T}_{N}} \ \lesssim_{\gamma} \ N^{\gamma}\mathrm{s}\|\mathfrak{f}\mathbf{Y}^{N}\|_{1;\mathbb{T}_{N}} \ \lesssim \ N^{-2+\gamma+\e_{\mathrm{ap}}}\|\mathfrak{f}\|_{\omega;\infty}. \label{eq:mtr13}
\end{align}
Observe that the far RHS of \eqref{eq:mtr13} is the first term on the RHS of \eqref{eq:mtr1}, so are left to show the second term in \eqref{eq:mtr12} is controlled by the RHS of \eqref{eq:mtr1}. To this end, by construction in Definitions \ref{definition:KPZ1} and \ref{definition:KPZ5}, we have, for $\mathbb{I}=\mathfrak{t}_{\mathrm{st}}+[0,\mathrm{s}]$, that
\begin{align}
\|\mathbf{H}^{N}\left(\mathfrak{f}_{S,y}\grad_{-\mathrm{s}}^{\mathbf{T}}\mathbf{Y}_{S,y}^{N}\right)\|_{1;\mathbb{T}_{N}} \ &\leq \ \|\mathbf{H}^{N}\left(\mathfrak{f}_{S,y}(\grad_{-\mathrm{s}}^{\mathbf{T}}\mathbf{Y}_{S,y}^{N})\mathbf{1}_{S\not\in\mathbb{I}}\right)\|_{1;\mathbb{T}_{N}} + \|\mathbf{H}^{N}\left(\mathfrak{f}_{S,y}(\grad_{-\mathrm{s}}^{\mathbf{T}}\mathbf{Y}_{S,y}^{N}\mathbf{1}_{S\in\mathbb{I}})\right)\|_{1;\mathbb{T}_{N}} \\
&\leq \ \|\mathfrak{f}\|_{\omega;\infty}\|(\grad_{-\mathrm{s}}^{\mathbf{T}}\mathbf{Y}^{N}_{S,y})\mathbf{1}_{S\not\in\mathbb{I}}\|_{1;\mathbb{T}_{N}} + \|\mathbf{H}^{N}\left(\mathfrak{f}_{S,y}(\grad_{-\mathrm{s}}^{\mathbf{T}}\mathbf{Y}_{S,y}^{N})\mathbf{1}_{S\in\mathbb{I}}\right)\|_{1;\mathbb{T}_{N}} \\
&\leq \ \|\mathfrak{f}\|_{\omega;\infty}\|\grad_{-\mathrm{s}}^{\mathbf{T}}\mathbf{Z}^{N}\|_{\mathfrak{t}_{\mathrm{st}};\mathbb{T}_{N}} + \|\mathbf{H}^{N}\left(\mathfrak{f}_{S,y}(\grad_{-\mathrm{s}}^{\mathbf{T}}\mathbf{Y}_{S,y}^{N})\mathbf{1}_{S\in\mathbb{I}}\right)\|_{1;\mathbb{T}_{N}}, \label{eq:mtr13decompose}
\end{align}
where the last inequality follows by the observation $\mathrm{s}\geq0$ implies $\grad^{\mathbf{T}}_{-\mathrm{s}}\mathbf{Y}^{N}=\grad^{\mathbf{T}}_{-\mathrm{s}}\mathbf{Z}^{N}$ before time $\mathfrak{t}_{\mathrm{st}}$, and after time $\mathfrak{t}_{\mathrm{st}}+\mathrm{s}$, we have $\grad_{-\mathrm{s}}^{\mathbf{T}}\mathbf{Y}^{N}=0$ (see Definition \ref{definition:KPZ5}). We now take expectations in \eqref{eq:mtr13decompose}. For the first term, note $\|\mathfrak{f}\|_{\omega;\infty}$-factor is constant. By Lemma \ref{lemma:st2}, with overwhelming probability, we have $\|\grad_{-\mathrm{s}}^{\mathbf{T}}\mathbf{Z}^{N}\|_{\mathfrak{t}_{\mathrm{st}};\mathbb{T}_{N}}\lesssim_{\gamma} N^{-1/2+\gamma}\|\mathbf{Z}^{N}\|_{\mathfrak{t}_{\mathrm{st}};\mathbb{T}_{N}}$ with $\gamma>0$ arbitrary but fixed and for any $0\leq\mathrm{s}\leq N^{-2}$. On the complement of this event, by construction of $\mathfrak{t}_{\mathrm{st}}$ in Definition \ref{definition:KPZ1}, we know deterministically that $\|\grad_{-\mathrm{s}}^{\mathbf{T}}\mathbf{Z}^{N}\|_{\mathfrak{t}_{\mathrm{st}};\mathbb{T}_{N}}\lesssim N^{\e_{\mathrm{ap}}}$. Thus, again since $\|\mathbf{Z}^{N}\|_{\mathfrak{t}_{\mathrm{st}};\mathbb{T}_{N}}\lesssim N^{\e_{\mathrm{ap}}}$ by Definition \ref{definition:KPZ1}, for the first term in \eqref{eq:mtr13decompose},
\begin{align}
\E\|\mathfrak{f}\|_{\omega;\infty}\|\grad_{-\mathrm{s}}^{\mathbf{T}}\mathbf{Z}^{N}\|_{\mathfrak{t}_{\mathrm{st}};\mathbb{T}_{N}} \ \lesssim_{\gamma} \ N^{-1/2+\gamma}\|\mathfrak{f}\|_{\omega;\infty}\E\|\mathbf{Z}^{N}\|_{\mathfrak{t}_{\mathrm{st}};\mathbb{T}_{N}} + N^{-100+\e_{\mathrm{ap}}}\|\mathfrak{f}\|_{\omega;\infty} \ \lesssim \ N^{-1/2+\gamma+\e_{\mathrm{ap}}}\|\mathfrak{f}\|_{\omega;\infty},
\end{align}
so it remains to estimate the expectation of the second term in \eqref{eq:mtr13decompose}. For this, we recall that $|\mathbf{Y}^{N}|\lesssim N^{\e_{\mathrm{ap}}}$ by construction in Definitions \ref{definition:KPZ1} and \ref{definition:KPZ5}. So by \eqref{eq:heatopbasic} for $\mathbb{I}=\mathfrak{t}_{\mathrm{st}}=[0,\mathrm{s}]$, the second term in \eqref{eq:mtr13decompose} is controlled by $N^{\e_{\mathrm{ap}}}\|\mathfrak{f}\|_{\omega;\infty}$ times the length $|\mathbb{I}|=\mathrm{s}\leq N^{-2}$ since $\mathbf{H}^{N}$ integrates, in time over $\mathbb{I}$, the spatial contractions $\mathbf{H}^{N,\mathbf{X}}$:
\begin{align}
\|\mathbf{H}^{N}\left(\mathfrak{f}_{S,y}(\grad_{-\mathrm{s}}^{\mathbf{T}}\mathbf{Y}_{S,y}^{N})\mathbf{1}_{S\in\mathbb{I}}\right)\|_{1;\mathbb{T}_{N}} \ \lesssim \ |\mathbb{I}|\|\mathfrak{f}\|_{\omega;\infty}\|\mathbf{Y}^{N}\|_{1;\mathbb{T}_{N}} \ \lesssim \ N^{-2+\e_{\mathrm{ap}}}\|\mathfrak{f}\|_{\omega;\infty}.
\end{align}
Combining the previous two displays with \eqref{eq:mtr13decompose} along with \eqref{eq:mtr12} and \eqref{eq:mtr13} gives \eqref{eq:mtr11}, so we are done.
\end{proof}
The second step we take is the following multiscale estimate of this discussion. Its proof is basically that of Lemma \ref{lemma:mtr1}.
\begin{lemma}\label{lemma:mtr2}
\fsp Consider the set $\mathbb{I}^{\mathbf{T},1}$ of time-scales in \emph{Definition \ref{definition:KPZ1}}; set $\mathfrak{t}_{\mathfrak{j}}=N^{-2+\mathfrak{j}\e_{\mathrm{ap}}}\in\mathbb{I}^{\mathbf{T},1}$ for indices $\mathfrak{j}\geq0$. Provided any pair of adjacent time-scales $\mathfrak{t}_{\mathfrak{j}}$ and $\mathfrak{t}_{\mathfrak{j}+1}$ satisfying $\mathfrak{t}_{\mathfrak{j}},\mathfrak{t}_{\mathfrak{j}+1}\leq N^{-1}$, we have the following estimate provided any $\gamma>0$:
\begin{align}
\E\|\mathbf{H}^{N}(\mathfrak{D}_{\mathfrak{t}_{\mathfrak{j}},\mathfrak{t}_{\mathfrak{j}+1}}^{\mathbf{T}}(\mathfrak{f}_{S,y})\mathbf{Y}^{N}_{S,y})\|_{1;\mathbb{T}_{N}} \ \lesssim_{\gamma} \ N^{-1+\gamma+\e_{\mathrm{ap}}}\|\mathfrak{f}\|_{\omega;\infty} + N^{3\e_{\mathrm{ap}}}\mathfrak{t}_{\mathfrak{j}+1}^{1/4}\E\|\mathbf{H}^{N}(|\mathfrak{I}_{\mathfrak{t}_{\mathfrak{j}}}^{\mathbf{T}}(\mathfrak{f}_{S,y})|)\|_{1;\mathbb{T}_{N}}. \label{eq:mtr2a}
\end{align}
Provided any $\mathfrak{J}\in\Z_{\geq1}$ with $\mathfrak{t}_{\mathfrak{J}}\leq N^{-1}$, we additionally have the following the estimate again for any $\gamma>0$:
\begin{align}
\E\|\mathbf{H}^{N}(\mathfrak{D}_{\mathfrak{t}_{0},\mathfrak{t}_{\mathfrak{J}}}^{\mathbf{T}}(\mathfrak{f}_{S,y})\mathbf{Y}^{N}_{S,y})\|_{1;\mathbb{T}_{N}} \ \lesssim_{\gamma} \ N^{-1+\gamma+\e_{\mathrm{ap}}}|\mathfrak{J}|\|\mathfrak{f}\|_{\omega;\infty} + |\mathfrak{J}|\sup_{0\leq\mathfrak{j}<\mathfrak{J}}N^{3\e_{\mathrm{ap}}}\mathfrak{t}_{\mathfrak{j}+1}^{1/4}\E\|\mathbf{H}^{N}(|\mathfrak{I}_{\mathfrak{t}_{\mathfrak{j}}}^{\mathbf{T}}(\mathfrak{f}_{S,y})|)\|_{1;\mathbb{T}_{N}}. \label{eq:mtr2b}
\end{align}
\end{lemma}
\begin{proof}
The second estimate \eqref{eq:mtr2b} is an immediate consequence of \eqref{eq:mtr2a} courtesy of the following observations.
\begin{itemize}
\item Observe $\mathfrak{D}_{\mathfrak{t}_{0},\mathfrak{t}_{\mathfrak{J}}}^{\mathbf{T}}$ is a telescoping sum of $|\mathfrak{J}|$-many $\mathfrak{D}_{\mathfrak{t}_{\mathfrak{j}},\mathfrak{t}_{\mathfrak{j}+1}}^{\mathbf{T}}$ terms. 
\item Plugging the aforementioned telescoping sum into the LHS of \eqref{eq:mtr2b}, we apply the triangle inequality for $\|\|_{1;\mathbb{T}_{N}}$ and linearity of expectation, which implies the LHS of \eqref{eq:mtr2b} is at most $|\mathfrak{J}|$ times the supremum of the RHS of \eqref{eq:mtr2a} over $0\leq\mathfrak{j}<\mathfrak{J}$.
\end{itemize}
We will now prove \eqref{eq:mtr2a}. This starts by the following computation of the $\mathfrak{D}^{\mathbf{T}}$ difference operator on the LHS of \eqref{eq:mtr2a}. The next identity follows from observing that because $\mathfrak{t}_{\mathfrak{j+1}}$ is a positive integer multiple of $\mathfrak{t}_{\mathfrak{j}}$ by assumption/choice of $\e_{\mathrm{ap}}$ in Definition \ref{definition:KPZ1}, we may write the time-average on the time-scale $\mathfrak{t}_{\mathfrak{j}+1}$ as the average of $\mathfrak{t}_{\mathfrak{j}+1}\mathfrak{t}_{\mathfrak{j}}^{-1}$ many time-averages on the time-scale $\mathfrak{t}_{\mathfrak{j}}$, each of these time-averages carrying a time-shift given by an integer multiple of $\mathfrak{t}_{\mathfrak{j}}$. This resembles the fact that an average of 10 terms can be written as an average of 5 ``other" terms, where each of the 5 ``other" terms is an average of pairs of the 10 terms with neighboring indices, for example. Ultimately, we use this representation to rewrite $\mathfrak{D}^{\mathbf{T}}$ as an average of time-gradients of the time-average on the smaller time-scale $\mathfrak{t}_{\mathfrak{j}}$; the second step requires putting the first term in the middle of \eqref{eq:mtr21} into the average:
\begin{align}
\mathfrak{D}_{\mathfrak{t}_{\mathfrak{j}},\mathfrak{t}_{\mathfrak{j}+1}}^{\mathbf{T}}(\mathfrak{f}_{S,y}) \ = \ \mathfrak{I}_{\mathfrak{t}_{\mathfrak{j}}}^{\mathbf{T}}(\mathfrak{f}_{S,y}) - \wt{\sum}_{0\leq\mathfrak{k}\leq\mathfrak{t}_{\mathfrak{j}+1}\mathfrak{t}_{\mathfrak{j}}^{-1}-1}\mathfrak{I}_{\mathfrak{t}_{\mathfrak{j}}}^{\mathbf{T}}(\mathfrak{f}_{S+\mathfrak{k}\mathfrak{t}_{\mathfrak{j}},y}) \ = \ -\wt{\sum}_{0\leq\mathfrak{k}\leq\mathfrak{t}_{\mathfrak{j}+1}\mathfrak{t}_{\mathfrak{j}}^{-1}-1}\grad_{\mathfrak{k}\mathfrak{t}_{\mathfrak{j}}}^{\mathbf{T}}\mathfrak{I}_{\mathfrak{t}_{\mathfrak{j}}}^{\mathbf{T}}(\mathfrak{f}_{S,y}). \label{eq:mtr21}
\end{align}
By the triangle inequality, to establish \eqref{eq:mtr2a}, it suffices to prove the estimate
\begin{align}
\wt{\sum}_{0\leq\mathfrak{k}\leq\mathfrak{t}_{\mathfrak{j}+1}\mathfrak{t}_{\mathfrak{j}}^{-1}-1}\E\|\mathbf{H}^{N}(\grad_{\mathfrak{k}\mathfrak{t}_{\mathfrak{j}}}^{\mathbf{T}}\mathfrak{I}_{\mathfrak{t}_{\mathfrak{j}}}^{\mathbf{T}}(\mathfrak{f}_{S,y})\mathbf{Y}^{N}_{S,y})\|_{1;\mathbb{T}_{N}} \ \lesssim_{\gamma} \ N^{-1+\gamma+\e_{\mathrm{ap}}}\|\mathfrak{f}\|_{\omega;\infty} + N^{3\e_{\mathrm{ap}}}\mathfrak{t}_{\mathfrak{j}+1}^{1/4}\E\|\mathbf{H}^{N}(|\mathfrak{I}_{\mathfrak{t}_{\mathfrak{j}}}^{\mathbf{T}}(\mathfrak{f}_{S,y})|)\|_{1;\mathbb{T}_{N}}. \label{eq:mtr22}
\end{align}
We may certainly replace the averaged sum on the LHS of \eqref{eq:mtr22} with a supremum over the same index set and prove the resulting estimate. Consider any $\mathfrak{k}$ in the index set on the LHS of \eqref{eq:mtr22}. Similar to \eqref{eq:mtr12}, we write the time-gradient on the LHS of \eqref{eq:mtr22} as $\grad^{\mathbf{T}}\mathfrak{I}^{\mathbf{T}}(\mathfrak{f})\mathbf{Y}^{N}=\grad^{\mathbf{T}}(\mathfrak{I}^{\mathbf{T}}(\mathfrak{f})\mathbf{Y}^{N})-\mathfrak{I}^{\mathbf{T}}(\mathfrak{f})\grad^{\mathbf{T}}\mathbf{Y}^{N}$ with an additional time-shift in $\mathbf{Y}^{N}$ that ultimately turns into reversing the time-scale for the time-gradient of $\mathbf{Y}^{N}$. Again, similar to \eqref{eq:mtr12}, this and the triangle inequality give the estimate
\begin{align}
\|\mathbf{H}^{N}(\grad_{\mathfrak{k}\mathfrak{t}_{\mathfrak{j}}}^{\mathbf{T}}\mathfrak{I}_{\mathfrak{t}_{\mathfrak{j}}}^{\mathbf{T}}(\mathfrak{f}_{S,y})\mathbf{Y}^{N}_{S,y})\|_{1;\mathbb{T}_{N}} \ \leq \ \|\mathbf{H}^{N}\left(\grad_{\mathfrak{k}\mathfrak{t}_{\mathfrak{j}}}^{\mathbf{T}}(\mathfrak{I}_{\mathfrak{t}_{\mathfrak{j}}}^{\mathbf{T}}(\mathfrak{f}_{S,y})\mathbf{Y}^{N}_{S-\mathfrak{k}\mathfrak{t}_{\mathfrak{j}},y})\right)\|_{1;\mathbb{T}_{N}} + \|\mathbf{H}^{N}(\mathfrak{I}_{\mathfrak{t}_{\mathfrak{j}}}^{\mathbf{T}}(\mathfrak{f}_{S,y})\grad_{-\mathfrak{k}\mathfrak{t}_{\mathfrak{j}}}^{\mathbf{T}}\mathbf{Y}^{N}_{S,y})\|_{1;\mathbb{T}_{N}}. \label{eq:mtr23}
\end{align}
We are now left with estimating each term on the RHS of \eqref{eq:mtr23}. To this end, we will follow the calculation \eqref{eq:mtr13} from the proof of Lemma \ref{lemma:mtr1} and apply Proposition \ref{prop:heat} but with $\mathrm{s}$ in \eqref{eq:mtr13} replaced by $\mathfrak{k}\mathfrak{t}_{\mathfrak{j}}$. Because $\mathfrak{k}\leq\mathfrak{t}_{\mathfrak{j}+1}\mathfrak{t}_{\mathfrak{j}}^{-1}-1$, we may deduce the time-scale inequality $\mathfrak{k}\mathfrak{t}_{\mathfrak{j}}\leq\mathfrak{t}_{\mathfrak{j}+1}\leq N^{-1}$. Ultimately, we establish the following deterministic estimate provided any $\gamma>0$:
\begin{align}
\|\mathbf{H}^{N}\left(\grad_{\mathfrak{k}\mathfrak{t}_{\mathfrak{j}}}^{\mathbf{T}}(\mathfrak{I}_{\mathfrak{t}_{\mathfrak{j}}}^{\mathbf{T}}(\mathfrak{f}_{S,y})\mathbf{Y}^{N}_{S,y})\right)\|_{1;\mathbb{T}_{N}} \ \lesssim_{\gamma} \ N^{\gamma}\mathfrak{k}\mathfrak{t}_{\mathfrak{j}}\|\mathfrak{f}\mathbf{Y}^{N}\|_{1;\mathbb{T}_{N}} \ \leq \ N^{\gamma}\mathfrak{t}_{\mathfrak{j}+1}\|\mathfrak{f}\mathbf{Y}^{N}\|_{1;\mathbb{T}_{N}} \ \lesssim \ N^{-1+\gamma+\e_{\mathrm{ap}}}\|\mathfrak{f}\|_{\omega;\infty}. \label{eq:mtr24}
\end{align}
We emphasize the final inequality in \eqref{eq:mtr24}, similar to \eqref{eq:mtr13}, requires the a priori bound $|\mathbf{Y}^{N}|\leq N^{\e_{\mathrm{ap}}}$ that follows via definition of $\mathbf{Y}^{N}$ in Definition \ref{definition:KPZ5} and of $\mathfrak{t}_{\mathrm{st}}$ in Definition \ref{definition:KPZ1}. We now move to the second term on the RHS of \eqref{eq:mtr23}. For this, we follow the estimate for the second term in \eqref{eq:mtr12} given in the proof of \eqref{eq:mtr1}. In particular, we start with the calculation giving \eqref{eq:mtr13decompose} but with $\mathrm{s}$ therein and in the set $\mathbb{I}$ replaced by $\mathfrak{k}\mathfrak{t}_{\mathfrak{j}}$ and proceed verbatim. The only difference is that instead of using Lemma \ref{lemma:st2} to estimate $\grad^{\mathbf{T}}\mathbf{Y}^{N}$ before time $\mathfrak{t}_{\mathrm{st}}$, we instead use the a priori spatial regularity estimate defining $\mathfrak{t}_{\mathrm{st}}$ for $\mathbf{Y}^{N}$; see Definitions \ref{definition:KPZ1} and \ref{definition:KPZ5}. Also, the final display in the proof of Lemma \ref{lemma:mtr1} should have its far RHS replaced with $N^{-1+\e_{\mathrm{ap}}}$, because in this case $|\mathbb{I}|=\mathfrak{k}\mathfrak{t}_{\mathfrak{j}}\leq\mathfrak{t}_{\mathfrak{j}+1}\leq N^{-1}$ (while before we used $|\mathbb{I}|\lesssim N^{-2}$), but this does not change the validity of the lemma.
\end{proof}
%
%
%
\section{Boltzmann-Gibbs Principle I -- Proof of Proposition \ref{prop:BGI1}}\label{section:proofBGI}
This section consists of many technical gymnastics and applications of the preliminary ingredients in Section \ref{section:BGIPrelim}. To clarify the discussion, we will present the main ingredients needed for the proof of Proposition \ref{prop:BGI1} with explanations about their respective statements and proofs. We then combine these ingredients to deduce Proposition \ref{prop:BGI1} and afterwards provide the proofs for each.
\subsubsection{Spatial Average}
The first step we take is replacement of the fluctuation {$\mathsf{S}_{\e_{1}}(\tau_{y}\eta_{S})$} inside the heat operator on the LHS of \eqref{eq:BGI1Prop} by spatial average on length-scale $N^{1/6}$. The choice of this length-scale is motivated as follows. We first want to choose the length-scale long enough so that we may exploit cancellations in spatial-averages; see Lemmas \ref{lemma:ee} and \ref{lemma:ee2}. However, the larger the length-scale we pick, the larger the error in such a replacement. It turns out $N^{1/6}$ is an appropriate compromise.

In the following Lemma \ref{lemma:BGI11}, the importance of the fraction $6/25$ is an upper bound for the length-scale exponent $1/6$ plus the $\e_{1}$ exponent $1/14$. The following bound roughly follows by a summation-by-parts argument; it controls the difference between the $\mathsf{S}$-term and its spatial average, after multiplying by $\mathbf{Y}^{N}$ and $\mathbf{H}^{N}$ and integrating in space-time, by regularity of $\mathbf{Y}^{N}$ and $\mathbf{H}^{N}$. The only other input for Lemma \ref{lemma:BGI11} is an explicit formula for the spatial gradients of $\mathbf{Y}^{N}$ in terms of the particle system to estimate its spatial regularity explicitly; Lemma \ref{lemma:xr} is not enough. Recall the transfer-of-spatial scales in Definition \ref{definition:BGI10}.
\begin{lemma}\label{lemma:BGI11}
\fsp Define the length-scale $\mathfrak{l}_{1}=N^{1/6}$. We have the following estimate in which $\mathfrak{e}:\Omega\to\R$ is described afterwards:
\begin{align}
\|\mathbf{H}_{T,x}^{N}(N^{1/2}\mathfrak{D}_{0,\mathfrak{l}_{1}}^{\mathbf{X}}({\mathsf{S}_{\e_{1}}(\tau_{y}\eta_{S}))}\mathbf{Y}^{N}_{S,y}) - \mathbf{H}_{T,x}^{N}(N^{6/25}\mathfrak{e}_{S,y}\mathbf{Y}^{N}_{S,y})\|_{1;\mathbb{T}_{N}} \ \lesssim \ N^{-\frac12+\frac{6}{25}+\e_{\mathrm{ap}}} \ \leq \ N^{-\frac{6}{25}+\e_{\mathrm{ap}}}. \label{eq:BGI11}
\end{align}
%
\begin{itemize}
\item We have $\|\mathfrak{e}\|_{\omega;\infty}\lesssim1$. The support $\mathbb{B}_{\mathfrak{e}}$ of $\mathfrak{e}$ is contained in the ball of radius $N^{6/25}$ centered at $0\in\mathbb{T}_{N}$.
\item For any parameter $\sigma\in\R$, the functional $\mathfrak{e}$ vanishes in expectation with respect to the canonical measure $\mu_{\sigma,\mathbb{B}_{\mathfrak{e}}}^{\mathrm{can}}$ on its support.
\end{itemize}
\end{lemma}
\begin{remark}
As we noted before, the exponent $6/25$ on the LHS of \eqref{eq:BGI11} comes as an upper bound for $1/6+1/14$, which are the exponents in the length-scale $\mathfrak{l}_{1}$ and the support length of $\mathsf{S}$. In particular, the transfer-of-length-scales operator from the LHS of \eqref{eq:BGI11} is the difference of $\mathsf{S}$ and spatial translates \emph{with disjoint supports}, and thus has support length $N^{1/6}N^{1/14}\leq N^{6/25}$.
\end{remark}
We conclude the first step by introducing a cutoff on the spatial-average of $\mathsf{S}_{\e_{1}}$ on length-scale $N^{1/6}$. First, we will introduce notation for the cutoff of the spatial average that we motivate shortly and that will be used throughout the rest of this subsection. Its utility is providing a priori upper bounds that will be useful for applications of reduction to local equilibrium in Lemma \ref{lemma:le3}. Indeed, observe that in Lemma \ref{lemma:le3}, allowed choices for the constant $\kappa$ therein depend on a priori upper bounds on functionals that we want to use Lemma \ref{lemma:le3} for. With better a priori \emph{deterministic} bounds, we can pick a larger/better value of $\kappa$. For motivation, assuming $\mathfrak{f}$ vanishes in expectation with respect to any canonical ensemble on its support, averages of its spatial translates with disjoint support satisfies a CLT-type estimate in Lemma \ref{lemma:ee}. The cutoff defining $\bar{\mathfrak{I}}^{\mathbf{X}}$ below vanishes by default when $\mathfrak{I}^{\mathbf{X}}$ exceeds this CLT-type upper bound, which according to Lemma \ref{lemma:ee} occurs with exponentially small probability in $N$. Thus, we deduce that with respect to any canonical ensemble on the support of the average $\mathfrak{I}^{\mathbf{X}}$, the cutoff $\bar{\mathfrak{I}}^{\mathbf{X}}$ does nothing outside of an event of exponentially small probability. For general measures, we reduce locally to canonical measures via Lemma \ref{lemma:le3}.
\begin{definition}\label{definition:BGI11}
\fsp Provided any functional $\mathfrak{f}:\Omega\to\R$ and length-scale $\mathfrak{l}_{\mathrm{av}}\in\Z_{\geq0}$, we define, recalling Definition \ref{definition:locmap2},
\begin{align}
\bar{\mathfrak{I}}_{\mathfrak{l}_{\mathrm{av}}}^{\mathbf{X}}(\mathfrak{f}) \ = \ \mathfrak{I}_{\mathfrak{l}_{\mathrm{av}}}^{\mathbf{X}}(\mathfrak{f})\mathbf{1}(|\mathfrak{I}_{\mathfrak{l}_{\mathrm{av}}}^{\mathbf{X}}(\mathfrak{f})|\leq N^{\e_{\mathrm{ap}}}\mathfrak{l}_{\mathrm{av}}^{-1/2}\|\mathfrak{f}\|_{\omega;\infty}).
\end{align}
\end{definition}
\begin{lemma}\label{lemma:BGI12}
\fsp Define $\wt{\mathfrak{I}}^{\mathbf{X}}_{\mathfrak{l}}=\mathfrak{I}^{\mathbf{X}}_{\mathfrak{l}}-\bar{\mathfrak{I}}^{\mathbf{X}}_{\mathfrak{l}}$ for any $\mathfrak{l}\in\Z$. We have the following for which we recall $\mathfrak{l}_{1}=N^{1/6}$ in \emph{Lemma \ref{lemma:BGI11}}:
\begin{align}
\E\|\mathbf{H}^{N}\left(N^{1/2}|\wt{\mathfrak{I}}_{\mathfrak{l}_{1}}^{\mathbf{X}}({\mathsf{S}_{\e_{1}}(\tau_{y}\eta_{S})})|\mathbf{Y}^{N}_{S,y}\right)\|_{1;\mathbb{T}_{N}} \ \lesssim \ N^{-\frac23+2\e_{\mathrm{ap}}}. \label{eq:BGI12}
\end{align}
\end{lemma}
\subsubsection{Time Average}
We now replace the cutoff spatial average $\bar{\mathfrak{I}}^{\mathbf{X}}$ introduced in the previous Lemma \ref{lemma:BGI12} by a time-average on a mesoscopic time-scale that is roughly of order $N^{-1}$. We say ``roughly" since we need to pick a time-scale for time-averaging that lives in the set $\mathbb{I}^{\mathbf{T}}$ from Definition \ref{definition:KPZ1} in order to use Lemma \ref{lemma:mtr1} and Lemma \ref{lemma:mtr2}, as we only have time-regularity bounds on $\mathbf{Y}^{N}$, which will be important for the aforementioned time-average replacement, on the time-scales in $\mathbb{I}^{\mathbf{T}}$. For the statement of the following result, we first recall the transfer-of-time-scale operator in Definition \ref{definition:mtr0}. For the proof of the following result, we employ Lemma \ref{lemma:mtr1} and Lemma \ref{lemma:mtr2} along with the a priori estimates for the cutoff $\bar{\mathfrak{I}}^{\mathbf{X}}$, as Lemma \ref{lemma:mtr1} and Lemma \ref{lemma:mtr2} yield the error in the proposed time-average replacement below, while this error is controlled by estimates for $\bar{\mathfrak{I}}^{\mathbf{X}}$. We also provide an analog of Lemma \ref{lemma:BGI13} below but with $N^{1/2}\bar{\mathfrak{I}}^{\mathbf{X}}$ replaced by $N^{6/25}\mathfrak{e}$ and whose proof is almost that of Lemma \ref{lemma:BGI13} but with a few cosmetic changes. In particular, we will only provide the necessary adjustments when addressing $N^{6/25}\mathfrak{e}$.
\begin{lemma}\label{lemma:BGI13}
\fsp Consider $\mathfrak{j}_{1}\in\Z_{\geq0}$ such that $\mathfrak{t}_{\mathfrak{j}_{1}}\in\mathbb{I}^{\mathbf{T},1}$ is the largest time in $\mathbb{I}^{\mathbf{T},1}$ satisfying $\mathfrak{t}_{\mathfrak{j}_{1}}\leq N^{-10/9}$ and consider $\mathfrak{j}_{2}\in\Z_{\geq0}$ so that $\mathfrak{t}_{\mathfrak{j}_{2}}\in\mathbb{I}^{\mathbf{T},1}$ is the largest time in $\mathbb{I}^{\mathbf{T},1}$ satisfying $\mathfrak{t}_{\mathfrak{j}_{2}}\leq N^{-1}$. We have lower bounds $\mathfrak{t}_{\mathfrak{j}_{1}}\geq N^{-10/9-\e_{\mathrm{ap}}}$ and $\mathfrak{t}_{\mathfrak{j}_{2}}\geq N^{-1-\e_{\mathrm{ap}}}$, because $\mathfrak{t}_{\mathfrak{j}}$ increases by a factor of $N^{\e_{\mathrm{ap}}}$ in $\mathfrak{j}$. We additionally have the following pair of expectation estimates:
\begin{align}
\E\|\mathbf{H}^{N}\left(N^{1/2}\mathfrak{D}_{0,\mathfrak{t}_{\mathfrak{j}_{1}}}^{\mathbf{T}}(\bar{\mathfrak{I}}_{\mathfrak{l}_{1}}^{\mathbf{X}}({\mathsf{S}_{\e_{1}}(\tau_{y}\eta_{S})}))\mathbf{Y}^{N}_{S,y}\right)\|_{1;\mathbb{T}_{N}}+\E\|\mathbf{H}^{N}\left(N^{6/25}\mathfrak{D}_{0,\mathfrak{t}_{\mathfrak{j}_{2}}}^{\mathbf{T}}(\mathfrak{e}_{S,y})\mathbf{Y}^{N}_{S,y}\right)\|_{1;\mathbb{T}_{N}} \ \lesssim \ N^{-\frac{1}{99999}+10\e_{\mathrm{ap}}}. \label{eq:BGI13}
\end{align}
\end{lemma}
\subsubsection{Final Estimates}
We now take advantage of replacing functionals inside the heat operator by the respective time-averages on the mesoscopic time-scales $\mathfrak{t}_{\mathfrak{j}_{i}}$ from Lemma \ref{lemma:BGI13}. This starts with the following estimate whose proof is effectively given in the proof of Lemma \ref{lemma:BGI13} in terms of technical details, in particular by equilibrium considerations in Lemma \ref{lemma:ee2} and then a reduction to equilibrium by Lemma \ref{lemma:le3}. These are overviewed in the ``Strategy" subsection of Section \ref{section:proofKPZ}. We similarly establish an analog for the time-average of $N^{6/25}\mathfrak{e}$ on the time-scale $\mathfrak{t}_{\mathfrak{j}_{2}}$ instead of $\mathfrak{t}_{\mathfrak{j}_{1}}$. Again, the proof is basically given in that of Lemma \ref{lemma:BGI13}.
\begin{lemma}\label{lemma:BGI14}
\fsp Consider the time-scales $\mathfrak{t}_{\mathfrak{j}_{i}}\in\mathbb{I}^{\mathbf{T},1}$ from \emph{Lemma \ref{lemma:BGI13}}. We have the estimate
\begin{align}
\E\|\mathbf{H}^{N}\left(N^{1/2}\mathfrak{I}^{\mathbf{T}}_{\mathfrak{t}_{\mathfrak{j}_{1}}}\bar{\mathfrak{I}}^{\mathbf{X}}_{\mathfrak{l}_{1}}({\mathsf{S}_{\e_{1}}(\tau_{y}\eta_{S})})\mathbf{Y}^{N}_{S,y}\right)\|_{1;\mathbb{T}_{N}} + \E\|\mathbf{H}^{N}\left(N^{6/25}\mathfrak{I}^{\mathbf{T}}_{\mathfrak{t}_{\mathfrak{j}_{2}}}(\mathfrak{e}_{S,y})\mathbf{Y}^{N}_{S,y}\right)\|_{1;\mathbb{T}_{N}} \ \lesssim \ N^{-\frac{1}{99999}+10\e_{\mathrm{ap}}}. \label{eq:BGI14}
\end{align}
\end{lemma}
We may now deduce Proposition \ref{prop:BGI1} upon step-by-step replacements and the triangle inequality for $\|\|_{1;\mathbb{T}_{N}}$ and $\E$. In each of the replacements below, we inherit the notation of the lemma cited therein.
\begin{itemize}
\item By Lemma \ref{lemma:BGI11}, we may replace $N^{1/2}\mathsf{S}_{\e_{1}}$ by $N^{1/2}\mathfrak{I}_{\mathfrak{l}_{1}}^{\mathbf{X}}(\mathsf{S}_{\e_{1}})+N^{6/25}\mathfrak{e}$ while controlling the error in doing so. 
\item By Lemma \ref{lemma:BGI12}, we may further replace $N^{1/2}\mathfrak{I}_{\mathfrak{l}_{1}}^{\mathbf{X}}(\mathsf{S}_{\e_{1}})+N^{6/25}\mathfrak{e}$ by $N^{1/2}\bar{\mathfrak{I}}_{\mathfrak{l}_{1}}^{\mathbf{X}}(\mathsf{S}_{\e_{1}})+N^{6/25}\mathfrak{e}$. 
\item By Lemma \ref{lemma:BGI13}, we may then replace $N^{1/2}\bar{\mathfrak{I}}_{\mathfrak{l}_{1}}^{\mathbf{X}}(\mathsf{S}_{\e_{1}})+N^{6/25}\mathfrak{e}$ by $N^{1/2}\mathfrak{I}^{\mathbf{T}}_{\mathfrak{t}_{\mathfrak{j}_{1}}}\bar{\mathfrak{I}}_{\mathfrak{l}_{1}}^{\mathbf{X}}(\mathsf{S}_{\e_{1}})+N^{6/25}\mathfrak{I}^{\mathbf{T}}_{\mathfrak{t}_{\mathfrak{j}_{2}}}(\mathfrak{e})$.
\end{itemize}
It now suffices to apply Lemma \ref{lemma:BGI14} and the triangle inequality. \qed
\begin{proof}[Proof of \emph{Lemma \ref{lemma:BGI11}}]
We will make explicit the functional $\mathfrak{e}$ in this proof. The first step that we take is to observe the $\mathfrak{D}^{\mathbf{X}}$-term on the LHS of \eqref{eq:BGI11} is the average of spatial-gradients of $\mathsf{S}_{\e_{1}}$ with respect to length-scales between $N^{\e_{1}}$ and $\mathfrak{l}_{1}N^{\e_{1}}$, as the $\mathfrak{D}^{\mathbf{X}}$-term is the difference between $\mathsf{S}_{\e_{1}}$ itself and the average of all its spatial-translates $\tau_{-N^{\e_{1}}\mathfrak{k}}\mathsf{S}_{\e_{1}}$ for all $\mathfrak{k}=1$ to $\mathfrak{k}=\mathfrak{l}_{1}$. We then employ a discrete-type Leibniz rule similar to that used to establish \eqref{eq:xr1}. Ultimately, this gives
\begin{align}
\mathbf{H}_{T,x}^{N}(N^{1/2}\mathfrak{D}_{0,\mathfrak{l}_{1}}^{\mathbf{X}}({\mathsf{S}_{\e_{1}}(\tau_{y}\eta_{S})})\mathbf{Y}^{N}_{S,y}) \ &= \ -\wt{\sum}_{\mathfrak{k}=1}^{\mathfrak{l}_{1}}\mathbf{H}_{T,x}^{N}(N^{1/2}\grad_{-N^{\e_{1}}\mathfrak{k}}^{\mathbf{X}}{\mathsf{S}_{\e_{1}}(\tau_{y}\eta_{S})}\mathbf{Y}^{N}_{S,y}) \label{eq:BGI111}
\end{align}
by definition, and for $\mathsf{S}_{S,y}={\mathsf{S}_{\e_{1}}(\tau_{y}\eta_{S})}$, per $\mathfrak{k}$ on the RHS of \eqref{eq:BGI111}, we get, parallel to \eqref{eq:xr1},
\begin{align}
\mathbf{H}^{N}(N^{\frac12}\grad_{-N^{\e_{1}}\mathfrak{k}}^{\mathbf{X}}\mathsf{S}_{S,y}\mathbf{Y}^{N}_{S,y}) \ &= \ \mathbf{H}^{N}\left(N^{\frac12}\grad_{-N^{\e_{1}}\mathfrak{k}}^{\mathbf{X}}(\mathsf{S}_{S,y}\mathbf{Y}^{N}_{S,y+N^{\e_{1}}\mathfrak{k}})\right) - \mathbf{H}^{N}\left(N^{\frac12}\mathsf{S}_{S,y}\grad_{-N^{\e_{1}}\mathfrak{k}}^{\mathbf{X}}\mathbf{Y}^{N}_{S,y+N^{\e_{1}}\mathfrak{k}}\right). \label{eq:BGI112}
\end{align}
We eventually employ a spatial heat operator estimate in Proposition \ref{prop:heat} to analyze the first term on the RHS of \eqref{eq:BGI112} uniformly in $\mathfrak{k}$-variables on the RHS of \eqref{eq:BGI111}. First, we continue by expanding the second term on the RHS of \eqref{eq:BGI112}. To this end, we recall that either $\mathbf{Y}^{N}=0$ or $\mathbf{Y}^{N}=\mathbf{Z}^{N}$. We consider the latter case as the former case is trivial. By definition of the Gartner transform $\mathbf{Z}^{N}$ in terms of the $\eta$-variables, Taylor expansion implies the scale-$\mathfrak{k}$ spatial gradient of $\mathbf{Y}^{N}=\mathbf{Z}^{N}$ is equal to
\begin{align}
-\grad_{-N^{\e_{1}}\mathfrak{k}}^{\mathbf{X}}\mathbf{Y}^{N}_{S,y+N^{\e_{1}}\mathfrak{k}} =\grad_{N^{\e_{1}}\mathfrak{k}}^{\mathbf{X}}\mathbf{Y}^{N}_{S,y} = \grad_{N^{\e_{1}}\mathfrak{k}}^{\mathbf{X}}\mathbf{Z}^{N}_{S,y} = \left({\sum}_{\mathfrak{j}=1}^{\infty}\frac{(-1)^{\mathfrak{j}}}{\mathfrak{j}!}N^{-\frac{\mathfrak{j}}{2}+\mathfrak{j}\e_{1}}|\mathfrak{k}|^{\mathfrak{j}}\left(\wt{\sum}_{\mathfrak{l}=1}^{N^{\e_{1}}\mathfrak{k}}\eta_{S,y+\mathfrak{l}}\right)^{\mathfrak{j}}\right)\mathbf{Z}_{S,y}^{N}. \label{eq:BGI112b}
\end{align}
The infinite series in front of $\mathbf{Z}^{N}$ in \eqref{eq:BGI112b} is $\mathrm{O}(N^{-1/2+\e_{1}}|\mathfrak{k}|)$. Indeed, this infinite series converges absolutely provided $N^{\e_{1}}|\mathfrak{k}|\leq N^{\alpha}$ with $\alpha<1/2$, which is the case here for $\e_{1}=1/14$ and $|\mathfrak{k}|\leq\mathfrak{l}_{1}=N^{1/6}$ for $\alpha=6/25$. Let $\mathfrak{e}_{\mathfrak{k}}$ be the product of this infinite series factor in \eqref{eq:BGI112b} with {$N^{1/2}\mathsf{S}_{\e_{1}}(\tau_{y}\eta_{S})$}. We emphasize the following features of $\mathfrak{e}_{\mathfrak{k}}$.
\begin{itemize}
\item We have $|\mathfrak{e}_{\mathfrak{k}}|\lesssim N^{\e_{1}}|\mathfrak{k}|\leq N^{6/25}$, since, as we explained before, the infinite series in \eqref{eq:BGI112b} is $\mathrm{O}(N^{-1/2+\e_{1}}|\mathfrak{k}|)$ and $N^{1/2}|\mathsf{S}_{\e_{1}}|\lesssim N^{1/2}$ since $\bar{\mathfrak{q}}$ is uniformly bounded; see Definition \ref{definition:mSHE+1}.
\item The product $\mathfrak{e}_{\mathfrak{k}}$ has support contained in a neighborhood of radius $N^{6/25}$ centered at $0\in\mathbb{T}_{N}$. Indeed, the $N^{1/2}\mathsf{S}_{\e_{1}}$ factor has support contained in a radius $N^{\e_{1}}$ neighborhood of $0$ with $\e_{1}=14^{-1}$, and the infinite series in \eqref{eq:BGI112b} has support contained in $\llbracket0,N^{\e_{1}}\mathfrak{k}\rrbracket\subseteq\llbracket0,N^{6/25}\rrbracket$; this can be seen by looking at which $\eta$-variables appear in the far RHS of \eqref{eq:BGI112b}.
\item The product $\mathfrak{e}_{\mathfrak{k}}$ vanishes in expectation with respect to any canonical measure on it support. Indeed, this is the case for the $\mathsf{S}_{\e_{1}}$ factor as can be seen in Definition \ref{definition:BGI2}, while the support of $\mathsf{S}_{\e_{1}}$ is contained strictly to the left of $0\in\mathbb{T}_{N}$ and thus disjoint from the support of the infinite series in \eqref{eq:BGI112b}. Here, we crucially use the property that the projection of any canonical measure over one set onto any subset is a convex combination of canonical measures on the subset, which can be seen by observing that the canonical measure is always the uniform measure on its support. In particular, when we take the expectation of $\mathfrak{e}_{\mathfrak{k}}$ with respect to any canonical measure on its support, we may first take an expectation of the $\mathsf{S}_{\e_{1}}$ factor with respect to the projection of this canonical measure to the support of $\mathsf{S}_{\e_{1}}$, which equals a convex combination of canonical measures over the support of $\mathsf{S}_{\e_{1}}$, and deduce that the expectation of $\mathfrak{e}_{\mathfrak{k}}$ with respect to any canonical measure on its support vanishes.
\item Let $\mathfrak{e}=-N^{-6/25}\wt{\sum}_{\mathfrak{k}=1,\ldots,\mathfrak{l}_{1}}\mathfrak{e}_{\mathfrak{k}}$. The sign is not so important.
\end{itemize}
The support of $\mathfrak{e}$ satisfies the conditions of $\mathfrak{e}_{\mathfrak{k}}$ supports from the second bullet point above. Moreover, $\mathfrak{e}$ vanishes in expectation with respect to any canonical measure on its support because each $\mathfrak{e}_{\mathfrak{k}}$ that it averages together satisfies this condition, and projection of any canonical measure on the support of $\mathfrak{e}$ projects to a convex combination of canonical measures on the support of each $\mathfrak{e}_{\mathfrak{k}}$. Lastly, we have $|\mathfrak{e}|\lesssim N^{-6/25}\sup_{\mathfrak{k}}|\mathfrak{e}_{\mathfrak{k}}|\lesssim1$; see the first bullet point in the above list. Using everything after \eqref{eq:BGI111}, we obtain the following in which the $N^{6/25}$ factor on the LHS compensates introducing a factor of $N^{-6/25}$ for $\mathfrak{e}$:
\begin{align}
\mathbf{H}_{T,x}^{N}(N^{\frac12}\mathfrak{D}_{0,\mathfrak{l}_{1}}^{\mathbf{X}}({\mathsf{S}_{\e_{1}}(\tau_{y}\eta_{S})})\mathbf{Y}^{N}_{S,y}) - \mathbf{H}_{T,x}^{N}(N^{\frac{6}{25}}\mathfrak{e}_{S,y}\mathbf{Y}^{N}_{S,y}) \ = \ -\wt{\sum}_{\mathfrak{k}=1}^{\mathfrak{l}_{1}}\mathbf{H}_{T,x}^{N}\left(N^{\frac12}\grad_{-N^{\e_{1}}\mathfrak{k}}^{\mathbf{X}}({\mathsf{S}_{\e_{1}}(\tau_{y}\eta_{S})}\mathbf{Y}^{N}_{S,y+N^{\e_{1}}\mathfrak{k}})\right). \label{eq:BGI113}
\end{align}
It remains to take the $\|\|_{1;\mathbb{T}_{N}}$ of both sides of \eqref{eq:BGI113} and estimate the resulting RHS. By the triangle inequality, it suffices to control the $\|\|_{1;\mathbb{T}_{N}}$ of each $\mathfrak{k}$-indexed term on the RHS of \eqref{eq:BGI113} uniformly in the index $\mathfrak{k}$. For this, we apply the spatial gradient estimate in Proposition \ref{prop:heat}, which transfers the spatial gradient onto the heat kernel in $\mathbf{H}^{N}$ and then integrates the resulting time-integrable singularity. Ultimately, we get the following estimate uniformly in $\mathfrak{k}$-indices on the RHS of \eqref{eq:BGI113} with universal implied constant:
\begin{align}
\|\mathbf{H}_{T,x}^{N}\left(N^{1/2}\grad_{-N^{\e_{1}}\mathfrak{k}}^{\mathbf{X}}({\mathsf{S}_{\e_{1}}(\tau_{y}\eta_{S})}\mathbf{Y}^{N}_{S,y+N^{\e_{1}}\mathfrak{k}})\right)\|_{1;\mathbb{T}_{N}} \ \lesssim \ N^{1/2}N^{-1}N^{\e_{1}}|\mathfrak{k}|\|\mathbf{Y}^{N}\|_{1;\mathbb{T}_{N}} \ \leq \ N^{-\frac12+\frac{6}{25}+\e_{\mathrm{ap}}}. \label{eq:BGI114}
\end{align}
The final inequality in \eqref{eq:BGI114} follows by power-counting and $N^{\e_{1}}|\mathfrak{k}|\leq N^{6/25}$ and the a priori bound $|\mathbf{Y}^{N}|\leq N^{\e_{\mathrm{ap}}}$.
\end{proof}
\begin{proof}[Proof of \emph{Lemma \ref{lemma:BGI12}}]
Consider $\gamma=999^{-999}\e_{\mathrm{ap}}$. Via Lemma \ref{lemma:hoe2} for {$\phi_{S,y}=|\wt{\mathfrak{I}}^{\mathbf{X}}_{\mathfrak{l}_{1}}(\mathsf{S}_{\e_{1}}(\tau_{y}\eta_{S}))|$} and this choice of $\gamma$, we deduce
\begin{align}
\mathrm{LHS}(\eqref{eq:BGI12}) \ \lesssim_{\gamma} \ N^{\frac23\gamma+\e_{\mathrm{ap}}}\E\left(\mathbf{I}_{1}(N^{3/4}|\wt{\mathfrak{I}}_{\mathfrak{l}_{1}}^{\mathbf{X}}({\mathsf{S}_{\e_{1}}(\tau_{y}\eta_{S})})|^{3/2})\right)^{2/3} \ \leq \ N^{\frac23\gamma+\frac12+\e_{\mathrm{ap}}}\left(\E\mathbf{I}_{1}(|\wt{\mathfrak{I}}_{\mathfrak{l}_{1}}^{\mathbf{X}}({\mathsf{S}_{\e_{1}}(\tau_{y}\eta_{S})}|^{3/2})\right)^{2/3}, \label{eq:BGI121}
\end{align}
where the last inequality follows from applying the Holder inequality with respect to the $\E$-expectation for the Holder inequality exponent $3/2$. We now apply Lemma \ref{lemma:le2} to ``transfer" the space-time averaging on the RHS of \eqref{eq:BGI121} to the law of the particle system; in this application of Lemma \ref{lemma:le2}, we make the following choices for inputs/parameters:
\begin{itemize}
\item Pick $\mathfrak{t}_{\mathrm{av}},\mathfrak{l}_{\mathrm{av}}=0$ and $\mathfrak{f}_{S,y}=\wt{\mathfrak{I}}^{\mathbf{X}}_{\mathfrak{l}_{1}}({\mathsf{S}_{\e_{1}}(\tau_{y}\eta_{S})})=\mathrm{O}(1)$ with support in $\llbracket-\mathrm{O}(\mathfrak{l}_{1}N^{\e_{1}}),\mathrm{O}(\mathfrak{l}_{1}N^{\e_{1}})\rrbracket$ and $\e_{1}=\frac{1}{14}$.
\end{itemize}
In this case, the $\E^{\mathrm{dyn}}_{\mathrm{Loc}}$ expectation on the RHS of \eqref{eq:le2} does nothing since $\mathfrak{t}_{\mathrm{av}}=0$, so the path-space dependence of the space-time average from the RHS of \eqref{eq:BGI121} is only through its initial condition $\mathrm{Loc}(\eta)$ that is equal to $\eta$ itself as far as $\mathfrak{f}$ is concerned because the $\mathrm{Loc}$ map only cuts off $\eta$ outside the support of $\mathfrak{f}$ by construction in Definition \ref{definition:locmap1}/Lemma \ref{lemma:le2}. Thus, as $\mathfrak{f}$ is uniformly bounded, we deduce the following estimate from Lemma \ref{lemma:le2} with the aforementioned specialization:
\begin{align}
\E\mathbf{I}_{1}(|\wt{\mathfrak{I}}_{\mathfrak{l}_{1}}^{\mathbf{X}}({\mathsf{S}_{\e_{1}}(\tau_{y}\eta_{S})})|^{3/2}) \ \lesssim \ {\E_{0}}\bar{\mathfrak{P}}_{1}|\wt{\mathfrak{I}}_{\mathfrak{l}_{1}}^{\mathbf{X}}(\mathsf{S}_{\e_{1}}(\eta))|^{3/2} + N^{-100}. \label{eq:BGI122}
\end{align}
Let us now estimate the first term within the RHS of \eqref{eq:BGI122}. We will do this through Lemma \ref{lemma:le3} for $\mathfrak{h}=|\wt{\mathfrak{I}}^{\mathbf{X}}_{\mathfrak{l}_{1}}(\mathsf{S}_{\e_{1}}(\bar{\mathfrak{q}}))|^{3/2}$, whose support is contained in a block with length of order $N^{\e_{1}}\mathfrak{l}_{1}\lesssim N^{1/14+1/6}\leq N^{6/25}$. We also choose $\kappa=1$ in this application of Lemma \ref{lemma:le3}, so we deduce the following in which $\mathbb{B}$ denotes the support of our choice of $\mathfrak{h}=|\wt{\mathfrak{I}}^{\mathbf{X}}_{\mathfrak{l}_{1}}(\mathsf{S}_{\e_{1}}(\bar{\mathfrak{q}}))|^{3/2}$:
\begin{align}
{\E_{0}}\bar{\mathfrak{P}}_{1}|\wt{\mathfrak{I}}_{\mathfrak{l}_{1}}^{\mathbf{X}}(\mathsf{S}_{\e_{1}}(\eta))|^{3/2} \ \lesssim \ N^{-2+\frac{18}{25}} + {\sup}_{\sigma\in\R}\E^{\mu_{\sigma,\mathbb{B}}^{\mathrm{can}}}|\wt{\mathfrak{I}}_{\mathfrak{l}_{1}}^{\mathbf{X}}(\mathsf{S}_{\e_{1}}(\eta))|^{3/2}. \label{eq:BGI123}
\end{align}
Observe the term inside the expectation on the far RHS is equal to zero on the event where the indicator function defining $\wt{\mathfrak{I}}^{\mathbf{X}}$ is not zero. Thus, because $\bar{\mathfrak{q}}$ and its functionals are uniformly bounded, the expectation on the far RHS of \eqref{eq:BGI123} is at most uniformly bounded factors times the probability that the indicator function in Definition \ref{definition:BGI11} fails. We estimate this using Lemma \ref{lemma:ee} with the choice of functions $\mathfrak{f}_{\mathfrak{j}}=\tau_{-\mathfrak{j}N^{\e_{1}}}\mathsf{S}_{\e_{1}}(\eta)$ for $\mathfrak{j}\geq1$ and $\gamma=\e_{\mathrm{ap}}$, whose supports are mutually disjoint since $\mathsf{S}_{\e_{1}}$ has support length $N^{\e_{1}}$ by construction in Definition \ref{definition:BGI2}, and for $\mathfrak{J}=\mathfrak{l}_{1}$. Thus, we have 
\begin{align}
{\sup}_{\sigma\in\R}\E^{\mu_{\sigma,\mathbb{B}}^{\mathrm{can}}}|\wt{\mathfrak{I}}_{\mathfrak{l}_{1}}^{\mathbf{X}}(\mathsf{S}_{\e_{1}}(\eta))|^{3/2} \ \lesssim \ {\sup}_{\sigma\in\R}\E^{\mu_{\sigma,\mathbb{B}}^{\mathrm{can}}}\mathbf{1}(|\mathfrak{I}_{\mathfrak{l}_{1}}^{\mathbf{X}}(\mathsf{S}_{\e_{1}}(\eta))|\geq N^{\e_{\mathrm{ap}}}\mathfrak{l}_{1}^{-1/2}) \ \lesssim \ N^{-100}. \label{eq:BGI124}
\end{align}
We now combine \eqref{eq:BGI121}, \eqref{eq:BGI122}, \eqref{eq:BGI123}, and \eqref{eq:BGI124} along with elementary power-counting in $N$ to deduce the claim.
\end{proof}
\begin{proof}[Proof of \emph{Lemma \ref{lemma:BGI13}}]
We establish the proposed estimate for the first term on the LHS of \eqref{eq:BGI13}, so we formally set $\mathfrak{e}=0$ for now. Observe $\mathfrak{j}_{1}\lesssim_{\e_{\mathrm{ap}}}1$ for $\mathfrak{j}_{1}$ in the statement of Lemma \ref{lemma:BGI13}, as $\mathfrak{t}_{\mathfrak{j}}$ increases by a factor of $N^{\e_{\mathrm{ap}}}$ with each step in the index $\mathfrak{j}$. Also, we emphasize the important assumption $\mathfrak{t}_{\mathfrak{j}_{1}}\leq N^{-1}$. Lastly, we note that via the triangle inequality, it suffices to control the LHS of \eqref{eq:BGI13} both with the replacement $\mathfrak{t}_{\mathfrak{j}_{1}}$ by $\mathfrak{t}_{0}=N^{-2}$ and with the replacement $0$ in the LHS of \eqref{eq:BGI13} by $\mathfrak{t}_{0}=N^{-2}$, namely
\begin{align}
\E\|\mathbf{H}^{N}\left(N^{1/2}\mathfrak{D}_{0,\mathfrak{t}_{\mathfrak{j}_{1}}}^{\mathbf{T}}(\bar{\mathfrak{I}}_{\mathfrak{l}_{1}}^{\mathbf{X}}({\mathsf{S}_{\e_{1}}(\tau_{y}\eta_{S})}))\mathbf{Y}^{N}_{S,y}\right)\|_{1;\mathbb{T}_{N}} \ &\leq \ \E\|\mathbf{H}^{N}\left(N^{1/2}\mathfrak{D}_{0,N^{-2}}^{\mathbf{T}}(\bar{\mathfrak{I}}_{\mathfrak{l}_{1}}^{\mathbf{X}}({\mathsf{S}_{\e_{1}}(\tau_{y}\eta_{S})}))\mathbf{Y}^{N}_{S,y}\right)\|_{1;\mathbb{T}_{N}} \label{eq:BGI13zero} \\
&+ \E\|\mathbf{H}^{N}\left(N^{1/2}\mathfrak{D}_{N^{-2},\mathfrak{t}_{\mathfrak{j}_{1}}}^{\mathbf{T}}(\bar{\mathfrak{I}}_{\mathfrak{l}_{1}}^{\mathbf{X}}({\mathsf{S}_{\e_{1}}(\tau_{y}\eta_{S})}))\mathbf{Y}^{N}_{S,y}\right)\|_{1;\mathbb{T}_{N}}. \nonumber
\end{align}
Use Lemma \ref{lemma:mtr1} and \eqref{eq:mtr2b} in Lemma \ref{lemma:mtr2} with $\mathfrak{f}_{S,y}=N^{1/2}\bar{\mathfrak{I}}^{\mathbf{X}}_{\mathfrak{l}_{1}}{(\mathsf{S}_{\e_{1}}(\tau_{y}\eta_{S})})$ and $\gamma=\e_{\mathrm{ap}}$ and, for Lemma \ref{lemma:mtr2}, $\mathfrak{J}=\mathfrak{j}_{1}$. Lemma \ref{lemma:mtr1} estimates the first term on the RHS of \eqref{eq:BGI13zero} at the cost of $\mathrm{O}(\mathrm{RHS}(\eqref{eq:BGI13}))$ since our choice of $\mathfrak{f}$ admits an a priori cutoff:
\begin{align}
\E\|\mathbf{H}^{N}\left(N^{1/2}\mathfrak{D}_{0,N^{-2}}^{\mathbf{T}}(\bar{\mathfrak{I}}_{\mathfrak{l}_{1}}^{\mathbf{X}}({\mathsf{S}_{\e_{1}}(\tau_{y}\eta_{S})}))\mathbf{Y}^{N}_{S,y}\right)\|_{1;\mathbb{T}_{N}} \ &\lesssim \ N^{-2+\gamma+\e_{\mathrm{ap}}}\|\mathfrak{f}\|_{\omega;\infty} + N^{-1/2+3\e_{\mathrm{ap}}}\E\|\mathbf{H}^{N}(|\mathfrak{f}_{S,y}|)\|_{1;\mathbb{T}_{N}} \nonumber \\
&\lesssim \ N^{3\e_{\mathrm{ap}}}\|\bar{\mathfrak{I}}^{\mathbf{X}}_{\mathfrak{l}_{1}}(\mathsf{S}_{\e_{1}}(\eta))\|_{\omega;\infty} \ \lesssim \ N^{4\e_{\mathrm{ap}}}\mathfrak{l}_{1}^{-\frac12} \ = \ N^{-\frac{1}{12}+4\e_{\mathrm{ap}}}. \label{eq:BGI130}
\end{align}
We note \eqref{eq:mtr2b} in Lemma \ref{lemma:mtr2} controls the second term on the RHS of \eqref{eq:BGI13zero} with $\mathfrak{J}=\mathfrak{j}_{1}\lesssim1$ and $\mathfrak{f}=N^{1/2}\bar{\mathfrak{I}}^{\mathbf{X}}_{\mathfrak{l}_{1}}(\mathsf{S}_{\e_{1}}(\eta))$ and $\gamma=\e_{\mathrm{ap}}$. As this choice of $\mathfrak{f}$ satisfies $|\mathfrak{f}|\lesssim N^{1/2}$, this shows the second term on the RHS of \eqref{eq:BGI13zero} is
\begin{align}
&\lesssim_{\e_{\mathrm{ap}}} N^{-1+2\e_{\mathrm{ap}}}\|N^{1/2}\bar{\mathfrak{I}}_{\mathfrak{l}_{1}}^{\mathbf{X}}(\mathsf{S}_{\e_{1}}(\eta))|\|_{\omega;\infty} + \sup_{0\leq\mathfrak{j}<\mathfrak{j}_{1}}N^{3\e_{\mathrm{ap}}}\mathfrak{t}_{\mathfrak{j}+1}^{1/4}\E\|\mathbf{H}^{N}(N^{1/2}|\mathfrak{I}^{\mathbf{T}}_{\mathfrak{t}_{\mathfrak{j}}}\bar{\mathfrak{I}}^{\mathbf{X}}_{\mathfrak{l}_{1}}({\mathsf{S}_{\e_{1}}(\tau_{y}\eta_{S})})|)\|_{1;\mathbb{T}_{N}} \\
&\lesssim \ N^{-\frac12+2\e_{\mathrm{ap}}}+ \sup_{0\leq\mathfrak{j}<\mathfrak{j}_{1}}N^{3\e_{\mathrm{ap}}}\mathfrak{t}_{\mathfrak{j}+1}^{1/4}\E\|\mathbf{H}^{N}(N^{1/2}|\mathfrak{I}^{\mathbf{T}}_{\mathfrak{t}_{\mathfrak{j}}}\bar{\mathfrak{I}}^{\mathbf{X}}_{\mathfrak{l}_{1}}({\mathsf{S}_{\e_{1}}(\tau_{y}\eta_{S})})|)\|_{1;\mathbb{T}_{N}}.
\end{align}
We apply Lemma \ref{lemma:hoe2} with $\phi_{S,y}=N^{1/2}\mathfrak{I}^{\mathbf{T}}_{\mathfrak{t}_{\mathfrak{j}}}\bar{\mathfrak{I}}^{\mathbf{X}}_{\mathfrak{l}_{1}}({\mathsf{S}_{\e_{1}}(\tau_{y}\eta_{S})})$; for $\mathfrak{j}\leq\mathfrak{j}_{1}$, this gives
\begin{align}
N^{3\e_{\mathrm{ap}}}\mathfrak{t}_{\mathfrak{j}+1}^{\frac14}\E\|\mathbf{H}^{N}\left(N^{\frac12}|\mathfrak{I}_{\mathfrak{t}_{\mathfrak{j}}}^{\mathbf{T}}\bar{\mathfrak{I}}_{\mathfrak{l}_{1}}^{\mathbf{X}}({\mathsf{S}_{\e_{1}}(\tau_{y}\eta_{S})})|\right)\|_{1;\mathbb{T}_{N}} \ \lesssim \ \left(N^{8\e_{\mathrm{ap}}}\mathfrak{t}_{\mathfrak{j}+1}^{\frac38}\E\mathbf{I}_{1}(N^{\frac34}|\mathfrak{I}_{\mathfrak{t}_{\mathfrak{j}}}^{\mathbf{T}}\bar{\mathfrak{I}}_{\mathfrak{l}_{1}}^{\mathbf{X}}({\mathsf{S}_{\e_{1}}(\tau_{y}\eta_{S})})|^{\frac32})\right)^{\frac23}. \label{eq:BGI131}
\end{align}
It suffices to estimate the RHS of \eqref{eq:BGI131} uniformly in $\mathfrak{j}$ satisfying $\mathfrak{t}_{\mathfrak{j}}\leq N^{-1}$. To this end, we employ Lemma \ref{lemma:le2} to estimate the expectation of this individual integral by the expectation of a single functional against the space-time averaged law of the particle system. This provides the following for which we forget, for now, the $2/3$-power on the RHS of \eqref{eq:BGI131}, in which $\mathrm{Loc}=\mathrm{Loc}_{\mathfrak{t}_{\mathfrak{j}},\mathfrak{l}_{\mathrm{tot}}}$ of Definition \ref{definition:locmap1}/Lemma \ref{lemma:le2} is taken with $\gamma_{0}=\e_{\mathrm{ap}}$ and $\mathfrak{f}_{S,y}=\bar{\mathfrak{I}}^{\mathbf{X}}_{\mathfrak{l}_{1}}({\mathsf{S}_{\e_{1}}(\tau_{y}\eta_{S})})$ and $\mathfrak{l}_{\mathrm{av}}=1$, as our choice of $\mathfrak{f}$ already accounts for the spatial averaging, and $\mathfrak{l}=N^{\e_{1}}\mathfrak{l}_{1}\leq N^{6/25}$ equal to the support length of our choice of functional $\mathfrak{f}_{S,y}=\bar{\mathfrak{I}}^{\mathbf{X}}_{\mathfrak{l}_{1}}({\mathsf{S}_{\e_{1}}(\tau_{y}\eta_{S})})$:
\begin{align}
N^{8\e_{\mathrm{ap}}}\mathfrak{t}_{\mathfrak{j}+1}^{\frac38}\E\mathbf{I}_{1}(N^{\frac34}|\mathfrak{I}_{\mathfrak{t}_{\mathfrak{j}}}^{\mathbf{T}}\bar{\mathfrak{I}}_{\mathfrak{l}_{1}}^{\mathbf{X}}({\mathsf{S}_{\e_{1}}(\tau_{y}\eta_{S})})|^{\frac32}) \ \lesssim \ N^{\frac34+8\e_{\mathrm{ap}}}\mathfrak{t}_{\mathfrak{j}+1}^{\frac38}{\E_{0}}\bar{\mathfrak{P}}_{1}\E_{\mathrm{Loc}}^{\mathrm{dyn}}|\mathfrak{I}_{\mathfrak{t}_{\mathfrak{j}}}^{\mathbf{T}}\bar{\mathfrak{I}}_{\mathfrak{l}_{1}}^{\mathbf{X}}(\mathsf{S}_{\e_{1}})|^{\frac32} + N^{-100}. \label{eq:BGI132}
\end{align}
Plugging in the second term on the RHS of \eqref{eq:BGI132} into the RHS of \eqref{eq:BGI131}, its contribution is controlled by the RHS of the proposed estimate \eqref{eq:BGI13}, so it suffices to estimate the first term on the RHS of \eqref{eq:BGI132}. For this purpose, we will employ Lemma \ref{lemma:le3} with the following choices for inputs $\kappa$ and $\mathfrak{h}$; for the choices below, we recall $\mathfrak{l}_{1}=N^{1/6}$ from Lemma \ref{lemma:BGI11}.
\begin{itemize}
\item We will choose the constant $\kappa$ in the statement of Lemma \ref{lemma:le3} to be $\kappa=N^{-3\e_{\mathrm{ap}}/2}\mathfrak{l}_{1}^{3/4}=N^{1/8-3\e_{\mathrm{ap}}/2}$. 
\item Now choose $\mathfrak{h}$ in Lemma \ref{lemma:le3} to be the $\E^{\mathrm{dyn}}$ functional. Observe first that these two bullet points are ``compatible" for applying Lemma \ref{lemma:le3} because the $\E^{\mathrm{dyn}}$ functional is uniformly bounded by the time average it is taking expectation of. This time average is controlled uniformly by the $\|\|_{\omega;\infty}$-norm of the quantity it is averaging, which in this case is the $\bar{\mathfrak{I}}^{\mathbf{X}}$ functional. But this $\bar{\mathfrak{I}}^{\mathbf{X}}$ functional is at most $N^{-1/12+\e_{\mathrm{ap}}}$; see Definition \ref{definition:BGI11}. Taking the $-3/2$-power of this bound gives $\kappa$.
\item Observe the support of the $\mathfrak{h}=\E^{\mathrm{dyn}}$ functional is equal to the support of $\mathrm{Loc}$ from our application of Lemma \ref{lemma:le2} that yielded \eqref{eq:BGI132}, as the $\E^{\mathrm{dyn}}$ functional takes $\mathrm{Loc}$ as its initial configuration for the path-space expectation. The support of $\mathrm{Loc}$ is given in Definition \ref{definition:locmap1}/Lemma \ref{lemma:le2}, which we emphasize is taken with $\gamma_{0}=\e_{\mathrm{ap}}$ and $\mathfrak{t}=\mathfrak{t}_{\mathfrak{j}}$ and $\mathfrak{l}\lesssim N^{\e_{1}}\mathfrak{l}_{1}$ for $\e_{1}=1/14$ and $\mathfrak{l}_{1}=N^{1/6}$; indeed, according to Lemma \ref{lemma:le2} we take the parameter $|\mathfrak{l}\mathfrak{l}_{\mathrm{av}}|$ for the $\mathrm{Loc}$ support equal to $\mathrm{O}(1)$ times the support length of $\mathsf{S}$ that we are space-time averaging on the RHS of \eqref{eq:BGI132}, which is of order $N^{\e_{1}}$, times the length-scale of this spatial averaging, which is order $\mathfrak{l}_{1}$; we also add $\mathrm{O}(N^{\e_{1}})$, which is basically the support length of $\mathsf{S}_{\e_{1}}(\eta)$, but this is lower-order.
\end{itemize}
Lemma \ref{lemma:le3} with the aforementioned choices lets us control the first term on the RHS of \eqref{eq:BGI132} by two terms, one depending on the support $\mathbb{B}$ of $\mathfrak{h}$ and another being the supremum of canonical measure expectations. This first support-term, after multiplying by the prefactors before the expectation on the RHS of \eqref{eq:BGI132}, is ultimately negligible courtesy of the following calculation:
\begin{align}
N^{\frac34+8\e_{\mathrm{ap}}}\mathfrak{t}_{\mathfrak{j}+1}^{\frac38}\kappa^{-1}N^{-2}|\mathbb{B}|^{3} \ \lesssim \ N^{\frac34+8\e_{\mathrm{ap}}}\mathfrak{t}_{\mathfrak{j}+1}^{\frac38}\kappa^{-1}N^{-2}\left(N^{1+\e_{\mathrm{ap}}}\mathfrak{t}_{\mathfrak{j}}^{1/2}+N^{3/2+\e_{\mathrm{ap}}}\mathfrak{t}_{\mathfrak{j}}+N^{\e_{\mathrm{ap}}}N^{\e_{1}}\mathfrak{l}_{1}\right)^{3} \ \lesssim \ N^{-\frac{1}{999}+10\e_{\mathrm{ap}}}. \label{eq:BGI133}
\end{align}
The last inequality in \eqref{eq:BGI133} follows by $\mathfrak{l}_{1}=N^{1/6}$ and $\e_{1}=1/14$ and, from Definition \ref{definition:KPZ1}, that $\mathfrak{t}_{\mathfrak{j}}\leq N^{-1}$ and $\mathfrak{t}_{\mathfrak{j}+1}\leq N^{\e_{\mathrm{ap}}}\mathfrak{t}_{\mathfrak{j}}$. By plugging this in the RHS of \eqref{eq:BGI131} and taking its $2/3$-power, we deduce that its contribution is controlled by the RHS of the proposed estimate \eqref{eq:BGI13}. We are now left to estimate the supremum of canonical measure expectations on the RHS of the estimate we obtain when employing Lemma \ref{lemma:le3} with the previous list of choices for inputs. For clarity, let us record below the supremum we are left to estimate, insert into the RHS of \eqref{eq:BGI131}, and deduce is controlled by the RHS of the proposed estimate \eqref{eq:BGI13}, in which $\E^{\sigma}$ denotes expectation with respect to the canonical measure of parameter $\sigma$ on the support of $\E^{\mathrm{dyn}}$/of $\mathrm{Loc}$:
\begin{align}
\Phi \ \overset{\bullet}= \ {\sup}_{\sigma\in\R}N^{\frac34+8\e_{\mathrm{ap}}}\mathfrak{t}_{\mathfrak{j}+1}^{\frac38}\E^{\sigma}\E_{\mathrm{Loc}}^{\mathrm{dyn}}|\mathfrak{I}_{\mathfrak{t}_{\mathfrak{j}}}^{\mathbf{T}}\bar{\mathfrak{I}}_{\mathfrak{l}_{1}}^{\mathbf{X}}(\mathsf{S}_{\e_{1}})|^{\frac32}. \label{eq:BGI133b}
\end{align}
We take the same $\mathrm{Loc}$ as we did for our applications of Lemma \ref{lemma:le3} in the previous quantity $\Phi$. To estimate $\Phi$, we proceed with the following two-step estimate, which is basically applying Lemma \ref{lemma:ee2}, but first removing the cutoff for the spatial average on the RHS of \eqref{eq:BGI133b} that is absent in Lemma \ref{lemma:ee2}. Intuitively, this cutoff does nothing with very high probability by Lemma \ref{lemma:ee}.
\begin{itemize}
\item We first replace $\bar{\mathfrak{I}}^{\mathbf{X}}$ by $\mathfrak{I}^{\mathbf{X}}$. The cost in doing so is recorded in the following estimate:
\begin{align}
\E^{\sigma}\E_{\mathrm{Loc}}^{\mathrm{dyn}}|\mathfrak{I}_{\mathfrak{t}_{\mathfrak{j}}}^{\mathbf{T}}\bar{\mathfrak{I}}_{\mathfrak{l}_{1}}^{\mathbf{X}}(\mathsf{S}_{\e_{1}})|^{\frac32} \ \lesssim \ \E^{\sigma}\E_{\mathrm{Loc}}^{\mathrm{dyn}}|\mathfrak{I}_{\mathfrak{t}_{\mathfrak{j}}}^{\mathbf{T}}\mathfrak{I}_{\mathfrak{l}_{1}}^{\mathbf{X}}(\mathsf{S}_{\e_{1}})|^{\frac32} + \E^{\sigma}\E_{\mathrm{Loc}}^{\mathrm{dyn}}|\mathfrak{I}_{\mathfrak{t}_{\mathfrak{j}}}^{\mathbf{T}}(\mathfrak{I}_{\mathfrak{l}_{1}}^{\mathbf{X}}(\mathsf{S}_{\e_{1}})-\bar{\mathfrak{I}}_{\mathfrak{l}_{1}}^{\mathbf{X}}(\mathsf{S}_{\e_{1}}))|^{\frac32}. \label{eq:BGI134}
\end{align}
We will estimate the second term within the RHS of \eqref{eq:BGI134}. By thinking of the $\mathfrak{I}^{\mathbf{T}}$ time average as an expectation, we will first move the $3/2$-power and absolute value past the $\mathfrak{I}^{\mathbf{T}}$ average via the Holder inequality to get
\begin{align}
\E^{\sigma}\E_{\mathrm{Loc}}^{\mathrm{dyn}}|\mathfrak{I}_{\mathfrak{t}_{\mathfrak{j}}}^{\mathbf{T}}(\mathfrak{I}_{\mathfrak{l}_{1}}^{\mathbf{X}}(\mathsf{S}_{\e_{1}})-\bar{\mathfrak{I}}_{\mathfrak{l}_{1}}^{\mathbf{X}}(\mathsf{S}_{\e_{1}}))|^{\frac32} \ \leq \ \E^{\sigma}\E_{\mathrm{Loc}}^{\mathrm{dyn}}\mathfrak{I}_{\mathfrak{t}_{\mathfrak{j}}}^{\mathbf{T}}(|\mathfrak{I}_{\mathfrak{l}_{1}}^{\mathbf{X}}(\mathsf{S}_{\e_{1}})-\bar{\mathfrak{I}}_{\mathfrak{l}_{1}}^{\mathbf{X}}(\mathsf{S}_{\e_{1}})|^{\frac32}). \label{eq:BGI135}
\end{align}
Following the proof of Lemma \ref{lemma:ee2}, we first replace $\E^{\mathrm{dyn}}$ in \eqref{eq:BGI135} with an expectation with respect to the path-space measure corresponding to the particle system but with periodic boundary conditions on the support of $\mathrm{Loc}$ if we allow error of at most order $N^{-100}$, as $\bar{\mathfrak{q}}$ is uniformly bounded. We now move both expectations, after this replacement, on the RHS of \eqref{eq:BGI135} past the $\mathfrak{I}^{\mathbf{T}}$ time-average by the Fubini theorem. Also from the proof of Lemma \ref{lemma:ee2}, for this smaller periodic system on the support of $\mathrm{Loc}$, the $\sigma$-canonical measure defining the expectation $\E^{\sigma}$ on the RHS of \eqref{eq:BGI135} is an invariant measure. Therefore, it suffices to estimate the expectation of what is inside the $\mathfrak{I}^{\mathbf{T}}$ average on the RHS of \eqref{eq:BGI135} when we sample the $\eta$-variables in the support of $\mathrm{Loc}$ by the $\sigma$-canonical measure. As the support of $\mathfrak{I}^{\mathbf{X}}-\bar{\mathfrak{I}}^{\mathbf{X}}$ is contained in that of $\mathrm{Loc}$, and since projections of canonical measures onto smaller subsets are convex combinations of canonical measures, it suffices to estimate expectation of $|\mathfrak{I}^{\mathbf{X}}-\bar{\mathfrak{I}}^{\mathbf{X}}|$ with respect to any canonical measure. By the large-deviations estimate in Lemma \ref{lemma:ee}, as $|\mathfrak{I}^{\mathbf{X}}-\bar{\mathfrak{I}}^{\mathbf{X}}|$ is uniformly bounded, this expectation is at most the probability $\mathrm{O}(N^{-100})$ that the indicator function defining $\bar{\mathfrak{I}}^{\mathbf{X}}$ fails. Ultimately, from this paragraph and the bound $\mathfrak{t}_{\mathfrak{j}+1}\leq1$, we get the following, which then, after plugging into the RHS of \eqref{eq:BGI131} and taking its $2/3$-power, has contribution controlled by the RHS of the proposed \eqref{eq:BGI13}:
\begin{align}
N^{\frac34+8\e_{\mathrm{ap}}}\mathfrak{t}_{\mathfrak{j}+1}^{\frac38}\E^{\sigma}\E_{\mathrm{Loc}}^{\mathrm{dyn}}\mathfrak{I}_{\mathfrak{t}_{\mathfrak{j}}}^{\mathbf{T}}(|\mathfrak{I}_{\mathfrak{l}_{1}}^{\mathbf{X}}(\mathsf{S}_{\e_{1}})-\bar{\mathfrak{I}}_{\mathfrak{l}_{1}}^{\mathbf{X}}(\mathsf{S}_{\e_{1}})|^{\frac32}) \ \lesssim \ N^{\frac34+8\e_{\mathrm{ap}}}N^{-100} \ \lesssim \ N^{-99}. \label{eq:BGI136}
\end{align}
\item We now estimate the first term on the RHS of \eqref{eq:BGI134}. We first employ the Holder inequality to boost the $3/2$ exponent to $2$, so
\begin{align}
N^{\frac34+8\e_{\mathrm{ap}}}\mathfrak{t}_{\mathfrak{j}+1}^{\frac38}\E^{\sigma}\E_{\mathrm{Loc}}^{\mathrm{dyn}}|\mathfrak{I}_{\mathfrak{t}_{\mathfrak{j}}}^{\mathbf{T}}\mathfrak{I}_{\mathfrak{l}_{1}}^{\mathbf{X}}(\mathsf{S}_{\e_{1}})|^{\frac32} \ &\leq \left(N^{1+\frac{32}{3}\e_{\mathrm{ap}}}\mathfrak{t}_{\mathfrak{j}+1}^{\frac12}\E^{\sigma}\E_{\mathrm{Loc}}^{\mathrm{dyn}}|\mathfrak{I}_{\mathfrak{t}_{\mathfrak{j}}}^{\mathbf{T}}\mathfrak{I}_{\mathfrak{l}_{1}}^{\mathbf{X}}(\mathsf{S}_{\e_{1}})|^{2}\right)^{\frac34}. \label{eq:BGI137}
\end{align}
We now use Lemma \ref{lemma:ee2} to the RHS of \eqref{eq:BGI137} where $\mathfrak{f}$ in Lemma \ref{lemma:ee2} is taken to be $\mathsf{S}_{\e_{1}}(\bar{\mathfrak{q}})$ here, which satisfies the assumptions needed of $\mathfrak{f}$ in Lemma \ref{lemma:ee2} as noted in Definition \ref{definition:BGI2}. We clarify we also take $\mathfrak{t}_{\mathrm{av}}=\mathfrak{t}_{\mathfrak{j}}$ and $\mathfrak{l}_{\mathrm{av}}=\mathfrak{l}_{1}=N^{1/6}$. This ultimately provides the following estimate for which we recall $\mathfrak{t}_{\mathfrak{j}+1}=N^{\e_{\mathrm{ap}}}\mathfrak{t}_{\mathfrak{j}}$ and $\mathfrak{t}_{\mathfrak{j}}\geq N^{-2}$ and the support of $\mathsf{S}$ has length order $N^{\e_{1}}$ with $\e_{1}=1/14$, the first two of which follow by construction in Definition \ref{definition:KPZ1} and the last in Definition \ref{definition:BGI2}/Proposition \ref{prop:BGI1}:
\begin{align}
N^{1+\frac{32}{3}\e_{\mathrm{ap}}}\mathfrak{t}_{\mathfrak{j}+1}^{\frac12}\E^{\sigma}\E_{\mathrm{Loc}}^{\mathrm{dyn}}|\mathfrak{I}_{\mathfrak{t}_{\mathfrak{j}}}^{\mathbf{T}}\mathfrak{I}_{\mathfrak{l}_{1}}^{\mathbf{X}}(\mathsf{S}_{\e_{1}})|^{2} \ \lesssim \ N^{1+\frac{35}{3}\e_{\mathrm{ap}}}\mathfrak{t}_{\mathfrak{j}}^{\frac12}N^{-2}\mathfrak{t}_{\mathfrak{j}}^{-1}\mathfrak{l}_{1}^{-1}N^{\frac{2}{14}}+N^{-100} \ \lesssim \ N^{-\frac{1}{999}+\frac{35}{3}\e_{\mathrm{ap}}}. \label{eq:BGI138}
\end{align}
Plugging the above upper bound \eqref{eq:BGI138} in the RHS of \eqref{eq:BGI131} and taking its $2/3$-power proves the contribution of the first term in the $\Phi$-decomposition \eqref{eq:BGI134} is controlled by the RHS of the proposed estimate \eqref{eq:BGI13}.
\end{itemize}
The previous two bullet points estimate $\Phi$ in \eqref{eq:BGI133b}, so that its contribution, after plugging it into the RHS of \eqref{eq:BGI131} and taking its $2/3$-power, is appropriately controlled, so we are done.

It now suffices to estimate the second term on the LHS of \eqref{eq:BGI13}. To this end, it suffices to follow the argument we have given to estimate the first term on the LHS of \eqref{eq:BGI13} but with the following technical adjustments; we also explain intuitively why it works.
\begin{itemize}
\item In applying Lemma \ref{lemma:mtr1} and \eqref{eq:mtr2b} in Lemma \ref{lemma:mtr2}, we instead choose $\mathfrak{f}=N^{6/25}\mathfrak{e}$ from the second term on the LHS of \eqref{eq:BGI13}.
\item When applying Lemma \ref{lemma:hoe2} and Lemma \ref{lemma:le2}, we instead integrate/apply the heat operator against our choice $\mathfrak{f}=N^{6/25}\mathfrak{e}$ from the previous bullet point and choose $\mathrm{Loc}=\mathrm{Loc}_{\mathfrak{t}_{\mathfrak{j}},\mathfrak{l}_{\mathrm{tot}}}$ with $\mathfrak{l}_{\mathrm{tot}}\lesssim N^{6/25}$, because the support of $\mathfrak{e}$ has a length of order $N^{6/25}$, which follows by construction in Lemma \ref{lemma:BGI11}, and there is no added length-scale gain for $\mathfrak{f}=N^{6/25}\mathfrak{e}$ from spatial-averaging.
\item When applying Lemma \ref{lemma:le3}, we instead choose $\kappa=1$ and $\mathfrak{h}$ equal to $\E^{\mathrm{dyn}}$ of the time-average of $\mathfrak{e}$, which we recall has support with length of order $N^{6/25}$. These choices of $\kappa$ and $\mathfrak{h}$ are ``compatible" as $\mathfrak{e}$ is uniformly bounded according to Lemma \ref{lemma:BGI11}.
\item The strategy we used to bound the first term on the LHS of \eqref{eq:BGI13} but with these modifications successfully controls the second term on the LHS of \eqref{eq:BGI13} for the following reason. We have a smaller power of $N$ for this second term; this means our estimates should be $N^{-1/2+6/25}=N^{-13/50}$ better than those for the first term on the LHS of \eqref{eq:BGI13}. However, we also lose the spatial-averaging, which introduces factors basically of order $N^{-1/12}$, so our estimates are actually only $N^{-13/50+1/12}\leq N^{-53/300}$ better. Moreover, the support of the spatial-average $\mathfrak{I}^{\mathbf{X}}$ for the first term on the LHS of \eqref{eq:BGI13} has basically the same length as the support of $\mathfrak{e}$. Lastly, when applying Lemma \ref{lemma:ee2}, the length of the support of the functional we space-time average has now increased from order $N^{\e_{1}}=N^{1/14}$ to $N^{6/25}$. As the estimate in Lemma \ref{lemma:ee2} depends linearly on the support, our estimates are actually $N^{-53/300+6/25-1/14}\leq N^{-8/900}$ better. In particular, we get \emph{sharper} estimates for the second term on the LHS of \eqref{eq:BGI13} when we modify the analysis for the first term therein via the previous three bullet points.
\end{itemize}
This completes the proof, as we have estimated both terms on the LHS of the proposed estimate \eqref{eq:BGI13} by the RHS of \eqref{eq:BGI13}.
\end{proof}
\begin{proof}[Proof of \emph{Lemma \ref{lemma:BGI14}}]
We again forget the second term on the LHS of \eqref{eq:BGI14} for now and focus on the first term therein. The first step we take is to introduce \emph{additional} spatial-averaging for the space-time average on the LHS of \eqref{eq:BGI14}. Unlike Lemma \ref{lemma:BGI11}, however, we will not required an explicit formula for gradients of the $\mathbf{Y}^{N}$ process and instead employ the replacement estimate in Lemma \ref{lemma:xr}. We will pick $\mathfrak{l}=N^{1/6}$ in our forthcoming application of Lemma \ref{lemma:xr}. We also pick $\mathfrak{f}_{S,y}=N^{1/2}\bar{\mathfrak{I}}^{\mathbf{X}}_{\mathfrak{l}_{1}}({\mathsf{S}_{\e_{1}}(\tau_{y}\eta_{S})})$ on the LHS of \eqref{eq:BGI14}. Note the choice of $\mathfrak{f}$ depends only on $\eta$-variables in a block of length $N^{\e_{1}}\mathfrak{l}_{1}$, with $\e_{1}=1/14$ and $\mathfrak{l}_{1}=N^{1/6}$; see Definitions \ref{definition:BGI2} and \ref{definition:BGI11}. We ultimately deduce the first term on the LHS of \eqref{eq:BGI14} is bounded above by $\mathrm{O}(1)$ times
\begin{align}
\E\|\mathbf{H}^{N}(N^{\frac12}\mathfrak{I}^{\mathbf{T}}_{\mathfrak{t}_{\mathfrak{j}_{1}}}\mathfrak{I}^{\mathbf{X}}_{N^{1/6}}\bar{\mathfrak{I}}^{\mathbf{X}}_{\mathfrak{l}_{1}}({\mathsf{S}_{\e_{1}}(\tau_{y}\eta_{S})})\mathbf{Y}^{N}_{S,y})\|_{1;\mathbb{T}_{N}} + N^{5\e_{\mathrm{ap}}}\bar{\mathfrak{l}}^{\frac12}\E\|\mathbf{H}^{N}(|\mathfrak{I}^{\mathbf{T}}_{\mathfrak{t}_{\mathfrak{j}_{1}}}\bar{\mathfrak{I}}^{\mathbf{X}}_{\mathfrak{l}_{1}}({\mathsf{S}_{\e_{1}}(\tau_{y}\eta_{S})})|)\|_{1;\mathbb{T}_{N}} + N^{-\frac{1}{12}+\e_{\mathrm{RN}}+2\e_{\mathrm{ap}}}. \label{eq:BGI141}
\end{align}
The factor $\bar{\mathfrak{l}}$ in \eqref{eq:BGI141} via Lemma \ref{lemma:xr} is equal to the length-scale for spatial averaging $\mathfrak{l}_{1}=N^{1/6}$ times the length of the support of $\mathfrak{f}_{S,y}=N^{1/2}\bar{\mathfrak{I}}^{\mathbf{X}}_{\mathfrak{l}_{1}}({\mathsf{S}_{\e_{1}}(\tau_{y}\eta_{S})})$ (which is $\mathrm{O}(N^{\e_{1}}\mathfrak{l}_{1})\lesssim N^{6/25}$). Since $|\bar{\mathfrak{l}}|\leq N^{1/2}$, Lemma \ref{lemma:xr} applies. We now explain \eqref{eq:BGI141}.
\begin{itemize}
\item Lemma \ref{lemma:xr} for $\mathfrak{l}=N^{\frac16}$ and $\mathfrak{f}_{S,y}=N^{\frac12}\bar{\mathfrak{I}}^{\mathbf{X}}_{\mathfrak{l}_{1}}({\mathsf{S}_{\e_{1}}(\tau_{y}\eta_{S})})$ implies the difference between the first term on the LHS of \eqref{eq:BGI14} and the first term in \eqref{eq:BGI141} is controlled by the RHS of \eqref{eq:xr} with these choices. It suffices to note that these two terms on the RHS of \eqref{eq:xr} are controlled by the third and second terms in \eqref{eq:BGI141}, respectively, as $|\bar{\mathfrak{I}}^{\mathbf{X}}_{\mathfrak{l}_{1}}({\mathsf{S}_{\e_{1}}(\tau_{y}\eta_{S})})|\lesssim N^{\e_{\mathrm{ap}}}|\mathfrak{l}_{1}|^{-\frac12}=N^{-\frac{1}{12}+\e_{\mathrm{ap}}}$.
\end{itemize}
The last term in \eqref{eq:BGI141} is clearly controlled by the RHS of the proposed estimate \eqref{eq:BGI141}. It remains to control the first two terms in \eqref{eq:BGI141}, for which we employ the following two bullet points based on the proof of Lemma \ref{lemma:BGI13}.
\begin{itemize}
\item To analyze the second term in \eqref{eq:BGI141}, we directly follow the proof of Lemma \ref{lemma:BGI13} starting from \eqref{eq:BGI131} but dropping the prefactor $N^{3\e_{\mathrm{ap}}}\mathfrak{t}^{1/4}$ and choosing $\mathfrak{j}=\mathfrak{j}_{1}$ from Lemma \ref{lemma:BGI13}/Lemma \ref{lemma:BGI14}. In particular, we will make the same choices in our applications of results in Section \ref{section:BGIPrelim} and we ultimately deduce the second term in \eqref{eq:BGI141} is controlled by the RHS of the proposed estimate \eqref{eq:BGI14}. Intuitively we succeed because although we lose a factor of $\mathfrak{t}^{1/4}$, in the calculations starting at \eqref{eq:BGI131} in the proof of Lemma \ref{lemma:BGI13} we only use the bound $\mathfrak{t}_{\mathfrak{j}+1}\leq N^{-1}$, and we only lose a factor of $N^{-1/4}$. On the other hand, the prefactor is no longer $N^{1/2}$ but rather $\bar{\mathfrak{l}}^{1/2}$. Recalling $\bar{\mathfrak{l}}\leq N^{1/6}N^{\e_{1}}\mathfrak{l}_{1}\leq N^{1/6}N^{6/25} = N^{61/150}$, we also gain a factor $N^{-1/2+61/300}=N^{-89/300}$ that beats out the $N^{1/4}$ factor that we obtained in dropping $\mathfrak{t}^{1/4}$ from earlier in this bullet point.
\item We now analyze the first term in \eqref{eq:BGI141}. In this case, we will also follow the proof for Lemma \ref{lemma:BGI13} starting with \eqref{eq:BGI131}, although now we must address the additional $\mathfrak{I}^{\mathbf{X}}$ operator in the first term in \eqref{eq:BGI141}. We start with Lemma \ref{lemma:hoe2} for $\phi$ equal to the $\mathfrak{I}^{\mathbf{T}}\mathfrak{I}^{\mathbf{X}}\bar{\mathfrak{I}}^{\mathbf{X}}$ term in the first term in \eqref{eq:BGI141} to get the following parallel of \eqref{eq:BGI131}:
\begin{align}
\E\|\mathbf{H}^{N}\left(N^{\frac12}\mathfrak{I}^{\mathbf{T}}_{\mathfrak{t}_{\mathfrak{j}_{1}}}\mathfrak{I}^{\mathbf{X}}_{N^{1/6}}\bar{\mathfrak{I}}^{\mathbf{X}}_{\mathfrak{l}_{1}}({\mathsf{S}_{\e_{1}}(\tau_{y}\eta_{S})})\mathbf{Y}^{N}_{S,y}\right)\|_{1;\mathbb{T}_{N}} \ \lesssim \ \left(N^{\frac34+2\e_{\mathrm{ap}}}\E\mathbf{I}_{1}\left(|\mathfrak{I}^{\mathbf{T}}_{\mathfrak{t}_{\mathfrak{j}_{1}}}\mathfrak{I}^{\mathbf{X}}_{N^{1/6}}\bar{\mathfrak{I}}^{\mathbf{X}}_{\mathfrak{l}_{1}}({\mathsf{S}_{\e_{1}}(\tau_{y}\eta_{S})})|^{\frac32}\right)\right)^{\frac23}. \label{eq:BGI142}
\end{align}
We now apply Lemma \ref{lemma:le2} to obtain the following parallel of \eqref{eq:BGI132} in the proof of Lemma \ref{lemma:BGI13}; we make the same choices for inputs for Lemma \ref{lemma:le2}, except our choice for $\mathfrak{l}_{\mathrm{av}}$ is now equal to $N^{1/6}$ instead of 0. We basically establish \eqref{eq:BGI132} but with the additional $\mathfrak{I}^{\mathbf{X}}$ operator, which is present in \eqref{eq:BGI142}, and without any $\mathfrak{t}_{\mathfrak{j}+1}$-dependent prefactor, which is present in \eqref{eq:BGI132}, and for which the $\mathrm{Loc}$ term below is now defined with $\mathfrak{l}_{\mathrm{tot}}$ being that from the proof of Lemma \ref{lemma:BGI13} but times $N^{1/6}$, since $\mathfrak{l}_{\mathrm{tot}}$ takes into account the spatial-average-length-scale $\mathfrak{l}_{\mathrm{av}}=N^{1/6}$ coming from $\mathfrak{I}^{\mathbf{X}}_{N^{1/6}}$ in \eqref{eq:BGI142} (see Lemma \ref{lemma:le2} for $\mathrm{Loc}$ and $\mathfrak{l}_{\mathrm{tot}}$):
\begin{align}
N^{\frac34+2\e_{\mathrm{ap}}}\E\mathbf{I}_{1}\left(|\mathfrak{I}^{\mathbf{T}}_{\mathfrak{t}_{\mathfrak{j}_{1}}}\mathfrak{I}^{\mathbf{X}}_{N^{1/6}}\bar{\mathfrak{I}}^{\mathbf{X}}_{\mathfrak{l}_{1}}({\mathsf{S}_{\e_{1}}(\tau_{y}\eta_{S})})|^{\frac32}\right) \ &\lesssim \ N^{\frac34+2\e_{\mathrm{ap}}}{\E_{0}}\bar{\mathfrak{P}}_{1}\E_{\mathrm{Loc}}^{\mathrm{dyn}}|\mathfrak{I}^{\mathbf{T}}_{\mathfrak{t}_{\mathfrak{j}_{1}}}\mathfrak{I}^{\mathbf{X}}_{N^{1/6}}\bar{\mathfrak{I}}^{\mathbf{X}}_{\mathfrak{l}_{1}}(\mathsf{S}_{\e_{1}})|^{\frac32} + N^{-100}. \label{eq:BGI143}
\end{align}
The second term on the RHS of \eqref{eq:BGI143} is controlled by the RHS of the proposed estimate \eqref{eq:BGI14} after taking $2/3$-powers upon plugging its contribution into the RHS of \eqref{eq:BGI142}. We now apply Lemma \ref{lemma:le3} with the same choices as we made in the proof of Lemma \ref{lemma:BGI13}, which are explicitly declared prior to \eqref{eq:BGI133}, but $\mathfrak{l}_{\mathrm{av}}=N^{1/6}$. This bounds the first term on the RHS of \eqref{eq:BGI143} by the sum of a support term plus a supremum of canonical measure expectations of the $\E^{\mathrm{dyn}}$ term on the RHS of \eqref{eq:BGI143}. The first support term is estimated in the exact same fashion as \eqref{eq:BGI133}, except we do not have the helpful $\mathfrak{t}_{\mathfrak{j}+1}$-dependent factor, namely its $3/8$-power. However, this factor is actually not needed to prove the upper bound on the far RHS of \eqref{eq:BGI143}. Also, the support of $\E^{\mathrm{dyn}}$ is changed as $\mathfrak{l}_{\mathrm{tot}}$ has changed as noted before \eqref{eq:BGI143}, so our version of \eqref{eq:BGI133} must be adjusted via replacing $N^{\e_{1}}\mathfrak{l}_{1}$ therein by $\mathfrak{l}_{\mathrm{av}}N^{\e_{1}}\mathfrak{l}_{1}=N^{1/6}N^{\e_{1}}\mathfrak{l}_{1}$, though the upper bound in \eqref{eq:BGI133} still holds after this adjustment. Ultimately, the contribution of the support term/first term on the RHS of \eqref{eq:le3} for our choices of inputs in Lemma \ref{lemma:le3} is controlled by the RHS of the proposed estimate \eqref{eq:BGI14} after plugging into \eqref{eq:BGI142} and taking $2/3$-powers. We are left to bound canonical measure expectations; by Lemma \ref{lemma:le3} these are an analog of \eqref{eq:BGI133b}:
\begin{align}
\Phi \ \overset{\bullet}= \ {\sup}_{\sigma\in\R}N^{\frac34+2\e_{\mathrm{ap}}}\E^{\sigma}\E_{\mathrm{Loc}}^{\mathrm{dyn}}|\mathfrak{I}_{\mathfrak{t}_{\mathfrak{j}_{1}}}^{\mathbf{T}}\mathfrak{I}^{\mathbf{X}}_{N^{1/6}}\bar{\mathfrak{I}}_{\mathfrak{l}_{1}}^{\mathbf{X}}(\mathsf{S}_{\e_{1}})|^{\frac32}. \label{eq:BGI143b}
\end{align}
By following the first bullet point after \eqref{eq:BGI133b}, we can first remove the bar over $\bar{\mathfrak{I}}^{\mathbf{X}}$ in \eqref{eq:BGI143b}. Now we observe that the double $\mathfrak{I}^{\mathbf{X}}$ average on length-scales $N^{1/6}$ and $\mathfrak{l}_{1}$ is actually a single $\mathfrak{I}^{\mathbf{X}}$ average on the product of the length-scales $N^{1/6}\mathfrak{l}_{1}$. Thus,
\begin{align}
\Phi \ \lesssim \ {\sup}_{\sigma\in\R}N^{\frac34+\frac32\e_{\mathrm{ap}}}\E^{\sigma}\E_{\mathrm{Loc}}^{\mathrm{dyn}}|\mathfrak{I}_{\mathfrak{t}_{\mathfrak{j}_{1}}}^{\mathbf{T}}\mathfrak{I}^{\mathbf{X}}_{N^{1/6}\mathfrak{l}_{1}}(\mathsf{S}_{\e_{1}})|^{\frac32} + N^{-100} \ \overset{\bullet}= \ \Phi' + N^{-100}. \label{eq:BGI143c}
\end{align}
At this point, we will now directly follow the second bullet point containing the estimate \eqref{eq:BGI137} but now with a spatial-average length-scale equal to $N^{1/6}\mathfrak{l}_{1}$. Intuitively, in the estimate \eqref{eq:BGI138}, we lose the $\mathfrak{t}_{\mathfrak{j}+1}$-dependent factor, namely its square root, thus we gain the bad factor of $N^{5/9}$ because $\mathfrak{t}_{\mathfrak{j}_{1}}\geq N^{-10/9-\e_{\mathrm{ap}}}$ by Lemma \ref{lemma:BGI13}. On the other hand, the additional $N^{1/6}$ factor for the length-scale gives us an additional $N^{-1/6}$ factor in \eqref{eq:BGI138} because that estimate is ``inversely" linear in the length-scale:
\begin{align}
\E^{\sigma}\E_{\mathrm{Loc}}^{\mathrm{dyn}}|\mathfrak{I}_{\mathfrak{t}_{\mathfrak{j}_{1}}}^{\mathbf{T}}\mathfrak{I}^{\mathbf{X}}_{N^{1/6}\mathfrak{l}_{1}}(\mathsf{S}_{\e_{1}})|^{\frac32} \ \lesssim \ \left(N^{-2+6\e_{\mathrm{ap}}}\mathfrak{t}_{\mathfrak{j}_{1}}^{-1}N^{-\frac16}\mathfrak{l}_{1}^{-1}N^{\frac{2}{14}}+N^{-100}\right)^{3/4} \ \lesssim \ N^{-\frac34-\frac{1}{999}+8\e_{\mathrm{ap}}}.\label{eq:BGI144}
\end{align}
The last estimate in \eqref{eq:BGI144} follows by power-counting; recall $\mathfrak{l}_{1}=N^{1/6}$ in Lemma \ref{lemma:BGI11} and the $\mathfrak{t}_{\mathfrak{j}_{1}}$-lower bound prior to \eqref{eq:BGI144}.
\end{itemize}
We have estimated the first two terms in \eqref{eq:BGI141} by the above two bullet points, completing the proposed estimate for the first term on the LHS of \eqref{eq:BGI14}. It remains to estimate the second term on the LHS of \eqref{eq:BGI14}. For this, we will follow the analysis in the proof of Lemma \ref{lemma:BGI13} for the second term on the LHS of \eqref{eq:BGI13}. In particular, we lose an additional factor $\mathfrak{t}^{1/4}$, which at best provides a factor of $N^{-1/4}$, because $\mathfrak{t}\leq N^{-1}$ for this lemma. Let us now observe in the final bullet point in the proof of Lemma \ref{lemma:BGI13}, until the application of Lemma \ref{lemma:ee2} mentioned therein, the benefit we gain, over the explicitly written upper bounds in the proof of Lemma \ref{lemma:BGI13}, from having the smaller $N^{6/25}$ prefactor, as opposed to $N^{1/2}$, is a factor of $N^{-13/50}$. This certainly beats out the $N^{1/4}$ we have gained from forgetting $\mathfrak{t}^{1/4}$. As for the application of Lemma \ref{lemma:ee2} mentioned in the final bullet point in the proof of Lemma \ref{lemma:BGI13}, observe that we only use the bound $\mathfrak{t}^{-1/2}\leq N^{-1}$, whereas for the current lemma we have $\mathfrak{t}^{-1/2}\leq N^{-1/2+\e_{\mathrm{ap}}}$. Therefore, the $N^{1/4}$ we must include from forgetting $\mathfrak{t}^{1/4}$ that we noted before is compensated for by the additional $N^{-1/2+\e_{\mathrm{ap}}}$ factor we gain from the improved bound $\mathfrak{t}^{-1/2}\leq N^{-1} \to \mathfrak{t}^{-1/2}\leq N^{-1/2+\e_{\mathrm{ap}}}$. We conclude that the analysis for $N^{6/25}\mathfrak{e}$ near the end of the proof of Lemma \ref{lemma:BGI13} estimates the second term on the LHS of the proposed bound \eqref{eq:BGI14} by the RHS of \eqref{eq:BGI14}. We have now estimated both terms on the LHS of the proposed estimate \eqref{eq:BGI14}, so we are done.
\end{proof}
%
%
%
\section{Boltzmann-Gibbs Principle I -- Proof of Proposition \ref{prop:BGI2}, Case I}\label{section:proofBGI21}
The organization of this section is similar to that of Section \ref{section:proofBGI}. We will present the main ingredients that we need in the proof of Proposition \ref{prop:BGI2} in what we will define shortly as ``Case I" and then deduce ``Case I" from these ingredients. We then provide the proof of each of these ingredients that, similar to Section \ref{section:proofBGI}, consist of a replacement-by-spatial-average, a large-deviations-type cutoff for this spatial average, replacement-by-time-average via multiscale analysis, and a ``final estimate" for the time-average. We decompose the proof of Proposition \ref{prop:BGI2} into two cases, the first of which of interest here is the case where the index $\mathfrak{b}\in\Z_{\geq0}$ in the supremum on the LHS of \eqref{eq:BGI2Prop} is chosen so that $\e_{1}+\mathfrak{b}\e_{\mathrm{RN},1}\leq1/4$. In this case, our strategy follows basically that for the proof of Proposition \ref{prop:BGI1}, but it is technically easier since the $\mathsf{R}_{\delta}$ term we study in Proposition \ref{prop:BGI2} admits a priori estimates:
\begin{lemma}\label{lemma:BGI2I1}
\fsp Consider $\delta\geq0$ with $\delta+\e_{\mathrm{RN},1}\leq\frac12+\e_{\mathrm{RN}}$. We have the following estimate for which we recall the notation \emph{\eqref{eq:defBGI24}} and in which $\mathsf{R}^{\mathrm{cut}}$ is explained afterwards:
\begin{align}
|{\mathsf{R}_{\delta}(\tau_{y}\eta_{S})}\mathbf{Y}^{N}_{S,y} - {\mathsf{R}_{\delta}^{\mathrm{cut}}(\tau_{y}\eta_{S})}\mathbf{Y}^{N}_{S,y}| \ \lesssim \ N^{-100}.
\end{align}
We define ${\mathsf{R}_{\delta}^{\mathrm{cut}}(\tau_{y}\eta_{S})=\tau_{y}\mathsf{R}^{\mathrm{cut}}_{\delta}(\eta_{S})}$, where $\mathsf{R}_{\delta}^{\mathrm{cut}}(\eta)$ has support contained in that of $\mathsf{R}_{\delta}(\eta)$ in \emph{\eqref{eq:defBGI24}}. Moreover:
\begin{itemize}
\item We have the deterministic bound $|\mathsf{R}_{\delta}^{\mathrm{cut}}(\eta)|\lesssim N^{10\e_{\mathrm{ap}}}N^{-\delta}$, where $10$ is just a large constant to be treated loosely.
\item The term $\mathsf{R}_{\delta}^{\mathrm{cut}}(\eta)$ vanishes in expectation with respect to any canonical measure on its support; see \emph{Definition \ref{definition:ensembles}}.
\end{itemize}
\end{lemma}
It will be convenient for us to introduce the following notation that distinguishes the current ``Case I", namely restricting to $\mathfrak{b}\leq\mathfrak{b}_{\mathrm{mid}}$ in the supremum on the LHS of \eqref{eq:BGI2Prop}. We also introduce notation for the $\mathsf{R}^{\mathrm{cut}}$ functionals relevant to Proposition \ref{prop:BGI2}.
\begin{definition}\label{definition:BGI2I2}
\fsp Let us define $\mathfrak{b}_{\mathrm{mid}}\in\Z_{\geq0}$ as the largest non-negative integer for which $\e_{1}+\mathfrak{b}_{\mathrm{mid}}\e_{\mathrm{RN},1}\leq1/4$. In particular, observe that $\mathfrak{b}_{\mathrm{mid}}\leq\mathfrak{b}_{+}$, where $\mathfrak{b}_{+}$ is defined in Proposition \ref{prop:BGI2}. We additionally define $\mathsf{R}^{\mathfrak{b}}_{S,y}={\mathsf{R}_{\e_{1}+\mathfrak{b}\e_{\mathrm{RN},1}}^{\mathrm{cut}}(\tau_{y}\eta_{S})}$ that satisfies:
\begin{itemize}
\item The support length of $\mathsf{R}^{\mathfrak{b}}$ is order $N^{\e_{1}+\mathfrak{b}\e_{\mathrm{RN},1}+\e_{\mathrm{RN},1}}$ as it has the same support as $\mathsf{R}_{\e_{1}+\mathfrak{b}\e_{\mathrm{RN},1}}$ in \eqref{eq:defBGI24}.
\item The functional $\mathsf{R}^{\mathfrak{b}}$ satisfies the \emph{deterministic} estimate $|\mathsf{R}^{\mathfrak{b}}|\lesssim N^{10\e_{\mathrm{ap}}}N^{-\e_{1}-\mathfrak{b}\e_{\mathrm{RN},1}}$ by construction; see Lemma \ref{lemma:BGI2I1}.
\item Lastly, to ease notation, we will define and sometimes use $N^{\beta+\e_{\mathrm{RN},1}}\mathfrak{l}_{\beta,\mathfrak{b}}=N^{\e_{1}+\mathfrak{b}\e_{\mathrm{RN},1}+\e_{\mathrm{RN},1}}$ where $\beta=999^{-99}$.
\end{itemize}
We clarify the length-scale $\mathfrak{l}_{\beta,\mathfrak{b}}$ as basically the length of the support of $\mathsf{R}^{\mathfrak{b}-1}$, and thus basically of $\mathsf{R}^{\mathfrak{b}}$ up to ultimately negligible factors of $N^{\e_{\mathrm{RN},1}}$, but a factor of $N^{-\beta}$ smaller. Actually it will not be important to be so careful about $N^{\e_{\mathrm{ap}}}$ and $N^{\e_{\mathrm{RN},1}}$ factors, as $\beta$ is much larger than $\e_{\mathrm{ap}}$ and $\e_{\mathrm{RN},1}$ of Definition \ref{definition:KPZ1}, so $N^{-\beta}$ factors will beat all relevant powers of $N^{\e_{\mathrm{ap}}}$ and $N^{\e_{\mathrm{RN},1}}$.
\end{definition}
Outside a priori estimates for $\mathsf{R}^{\mathfrak{b}}$-terms in Lemma \ref{lemma:BGI2I1} and Definition \ref{definition:BGI2I2} that we did not have for the $\mathsf{S}$-terms in the proof of Proposition \ref{prop:BGI1}, we emphasize the proof of Proposition \ref{prop:BGI2} basically follows that of Proposition \ref{prop:BGI1} except for a few technical differences whose impact on the proof can be readily checked. In particular, many estimates have the same flavor with only minor differences in power-counting that ultimately amount to elementary arithmetic.
\subsubsection{Spatial Average}
The following result replaces $\mathsf{R}^{\mathfrak{b}}$ by spatial-averages on length-scales $\mathfrak{l}_{\beta,\mathfrak{b}}$ and provides an analog of Lemma \ref{lemma:BGI11} but for $\mathsf{R}^{\mathfrak{b}}$ instead of $\mathsf{S}$. We first emphasize a difference between the following result and Lemma \ref{lemma:BGI11}. In Lemma \ref{lemma:BGI11}, replacing $\mathsf{S}$ with a spatial average forces us to analyze explicitly the leading-order error term $N^{6/25}\mathfrak{e}$. For the following result, we instead employ the a priori estimate for $\mathsf{R}^{\mathfrak{b}}$ in Lemma \ref{lemma:BGI2I1} to avoid this issue. In particular, Lemma \ref{lemma:xr} is enough.
\begin{lemma}\label{lemma:BGI2I3}
\fsp Define the length-scale $\mathfrak{l}_{\beta,\mathfrak{b}}=N^{\e_{1}+\mathfrak{b}\e_{\mathrm{RN},1}-\beta}$, where $\beta=999^{-99}$. We have the following uniformly in $\mathfrak{b}\leq\mathfrak{b}_{\mathrm{mid}}$ for which we recall the notation for differences of spatial averages on different length-scales in \emph{Definition \ref{definition:BGI10}}:
\begin{align}
\E\|\mathbf{H}^{N}(N^{1/2}\mathfrak{D}_{0,\mathfrak{l}_{\beta,\mathfrak{b}}}^{\mathbf{X}}(\mathsf{R}^{\mathfrak{b}}_{S,y})\mathbf{Y}^{N}_{S,y})\|_{1;\mathbb{T}_{N}} \ \lesssim \ N^{-\frac12\beta+15\e_{\mathrm{ap}}+\e_{\mathrm{RN},1}}. \label{eq:BGI2I3}
\end{align}
\end{lemma}
Lemma \ref{lemma:BGI2I3} lets us replace $\mathsf{R}^{\mathfrak{b}}$ with its spatial average on length-scale $\mathfrak{l}_{\beta,\mathfrak{b}}$, which for clarity we recall is basically the support length of $\mathsf{R}^{\mathfrak{b}-1}$ times $N^{-\beta}$ for $\beta=999^{-999}$. Analogous to Lemma \ref{lemma:BGI12}, we now replace this spatial average by a cutoff that holds at a large deviations scale with respect to any canonical measure and therefore for general measures after space-time averaging courtesy of the local equilibrium reduction in Lemma \ref{lemma:le3}. We first provide a technical comment -- the local equilibrium reduction in Lemma \ref{lemma:le3} deteriorates as the support of the functional, in this case the length-$\mathfrak{l}_{\beta,\mathfrak{b}}$ average of $\mathsf{R}^{\mathfrak{b}}$, increases, thus as $\mathfrak{b}$ increases. However, it also improves as $\mathfrak{b}$ increases because, according to Lemma \ref{lemma:BGI2I1}, a priori estimates for $\mathsf{R}^{\mathfrak{b}}$ also improve as $\mathfrak{b}$ increases; this is ultimately enough to counter the aforementioned deterioration. We apply this observation throughout this section. Otherwise, the proof of the following ``cutoff replacement" follows the general strategy for that of Lemma \ref{lemma:BGI12}.
\begin{lemma}\label{lemma:BGI2I4}
\fsp Recall the operator $\wt{\mathfrak{I}}^{\mathbf{X}}=\mathfrak{I}^{\mathbf{X}}-\bar{\mathfrak{I}}^{\mathbf{X}}$ from \emph{Lemma \ref{lemma:BGI12}}. We have the following estimate uniformly in $\mathfrak{b}\leq\mathfrak{b}_{\mathrm{mid}}$, in which we recall the length-scale $\mathfrak{l}_{\beta,\mathfrak{b}}$ and $\beta=999^{-99}$ that were both used in the statement of \emph{Lemma \ref{lemma:BGI2I3}}:
\begin{align}
\E\|\mathbf{H}^{N}\left(N^{\frac12}|\wt{\mathfrak{I}}^{\mathbf{X}}_{\mathfrak{l}_{\beta,\mathfrak{b}}}(\mathsf{R}_{S,y}^{\mathfrak{b}})|\mathbf{Y}_{S,y}^{N}\right)\|_{1;\mathbb{T}_{N}} \ \lesssim \ N^{-\frac12+15\e_{\mathrm{ap}}+2\beta} \ \lesssim N^{-\frac13}. \label{eq:BGI2I4}
\end{align}
\end{lemma}
\subsubsection{Time Average}
Following Lemma \ref{lemma:BGI13}, we will now replace the cutoff spatial average of $\mathsf{R}^{\mathfrak{b}}$ introduced in Lemma \ref{lemma:BGI2I4} by a time-average on appropriate mesoscopic time-scale. However, we instead replace by time-average with respect to $\mathfrak{b}$-dependent time-scale that is shorter than the roughly $N^{-1}$ time-scale used in the proof of Proposition \ref{prop:BGI1}. This last difference is technical, as we will see when we estimate time-averages of $\mathsf{R}^{\mathfrak{b}}$ on this $\mathfrak{b}$-dependent time-scale. Before we state the following result, we first recall the transfer-of-time-scale operator in Definition \ref{definition:mtr0}. Let us also make another technical comment -- the Kipnis-Varadhan inequality for the equilibrium estimates in Lemma \ref{lemma:ee2} deteriorates as the support of the functional $\mathsf{R}^{\mathfrak{b}}$ we are time-averaging in Lemma \ref{lemma:ee2} increases, in this case as the index $\mathfrak{b}$ increases. We will counter such deterioration with the improving a priori bound on $\mathsf{R}^{\mathfrak{b}}$ in Lemma \ref{lemma:BGI2I1}. These competing factors basically cancel, so the proof of Lemma \ref{lemma:BGI13} holds almost verbatim.
\begin{lemma}\label{lemma:BGI2I5}
\fsp Provided $\mathfrak{b}\leq\mathfrak{b}_{\mathrm{mid}}$, consider $\mathfrak{j}_{+}\in\Z_{\geq0}$ such that $\mathfrak{t}_{\mathfrak{j}_{+}}\in\mathbb{I}^{\mathbf{T},1}$ is the largest time in $\mathbb{I}^{\mathbf{T},1}$ satisfying $\mathfrak{t}_{\mathfrak{j}_{+}}\leq N^{-1+\beta}\mathfrak{l}_{\beta,\mathfrak{b}}^{-1}$, where $\mathfrak{l}_{\beta,\mathfrak{b}}$ and $\beta$ are both defined in \emph{Lemma \ref{lemma:BGI2I3}}. As $\mathfrak{l}_{\beta,\mathfrak{b}}\geq N^{\e_{1}-\beta}$ with $\e_{1}=\frac{1}{14}\geq999\beta$, we have $\mathfrak{t}_{\mathfrak{j}_{+}}\leq N^{-1}$ and
\begin{align}
\sup_{\mathfrak{b}\leq\mathfrak{b}_{\mathrm{mid}}}\E\|\mathbf{H}^{N}\left(N^{1/2}\mathfrak{D}_{0,\mathfrak{t}_{\mathfrak{j}_{+}}}^{\mathbf{T}}(\bar{\mathfrak{I}}_{\mathfrak{l}_{\beta,\mathfrak{b}}}^{\mathbf{X}}(\mathsf{R}^{\mathfrak{b}}_{S,y}))\mathbf{Y}_{S,y}^{N}\right)\|_{1;\mathbb{T}_{N}} \ \lesssim \ N^{-\frac{1}{99999}\beta+100\e_{\mathrm{ap}}}.\label{eq:BGI2I5}
\end{align}
\end{lemma}
\subsubsection{Final Estimates}
We now estimate the time-average of $\mathsf{R}^{\mathfrak{b}}$ uniformly in $\mathfrak{b}\leq\mathfrak{b}_{\mathrm{mid}}$ on the time-scale $\mathfrak{t}_{\mathfrak{j}_{+}}$ ``reached" with multiscale replacement in Lemma \ref{lemma:BGI2I5}. This amounts to the analog below for Lemma \ref{lemma:BGI14}. We apply the same remarks about the simultaneous deterioration and improvement of the estimates implied by Lemma \ref{lemma:ee2} and Lemma \ref{lemma:BGI2I1} as the index $\mathfrak{b}$ increases. Otherwise, the proof of the following estimate is basically that of Lemma \ref{lemma:BGI14}.
\begin{lemma}\label{lemma:BGI2I6}
\fsp Consider any $\mathfrak{b}\leq\mathfrak{b}_{\mathrm{mid}}$ and the corresponding time-scale $\mathfrak{t}_{\mathfrak{j}_{+}}$ from \emph{Lemma \ref{lemma:BGI2I5}}. Uniformly in $\mathfrak{b}\leq\mathfrak{b}_{\mathrm{mid}}$, we have 
\begin{align}
\E\|\mathbf{H}^{N}\left(N^{1/2}\mathfrak{I}_{\mathfrak{t}_{\mathfrak{j}_{+}}}^{\mathbf{T}}\bar{\mathfrak{I}}_{\mathfrak{l}_{\beta,\mathfrak{b}}}^{\mathbf{X}}(\mathsf{R}_{S,y}^{\mathfrak{b}})\mathbf{Y}_{S,y}^{N}\right)\|_{1;\mathbb{T}_{N}} \ \lesssim \ N^{-\frac{1}{99999}\beta+100\e_{\mathrm{ap}}}. \label{eq:BGI2I6}
\end{align}
\end{lemma}
Let us now prove Proposition \ref{prop:BGI2} in Case I, where the index $\mathfrak{b}\in\Z_{\geq0}$ in the supremum on the LHS of \eqref{eq:BGI2Prop} satisfies $\mathfrak{b}\leq\mathfrak{b}_{\mathrm{mid}}$. We make the following replacements to the LHS of \eqref{eq:BGI2Prop} and cite results in this section that control errors in such replacements.
\begin{itemize}
\item Lemma \ref{lemma:BGI2I1} lets us replace $\mathsf{R}_{\e_{1}+\mathfrak{b}\e_{\mathrm{RN},1}}$ on the LHS of \eqref{eq:BGI2Prop} with $\mathsf{R}^{\mathfrak{b}}$ with clearly controllable error. Indeed, Lemma \ref{lemma:BGI2I1} is applicable for $\mathfrak{b}\leq\mathfrak{b}_{\mathrm{mid}}$ as by definition of $\mathfrak{b}_{\mathrm{mid}}$ in Definition \ref{definition:BGI2I2}, we have $\e_{1}+\mathfrak{b}\e_{\mathrm{RN},1}\leq1/4\leq1/2+\e_{\mathrm{RN}}$ if $\mathfrak{b}\leq\mathfrak{b}_{\mathrm{mid}}$.
\item We now replace $\mathsf{R}^{\mathfrak{b}}$ by $\mathfrak{I}^{\mathbf{X}}_{\mathfrak{l}_{\beta,\mathfrak{b}}}(\mathsf{R}^{\mathfrak{b}})$ with $\mathfrak{l}_{\beta,\mathfrak{b}}$ in Lemma \ref{lemma:BGI2I3}. Lemma \ref{lemma:BGI2I3} controls the error by a universal negative power of $N$ as $\beta=999^{-99}$ is much larger than $\e_{\mathrm{ap}},\e_{\mathrm{RN},1}$. 
\item We now replace $\mathfrak{I}^{\mathbf{X}}_{\mathfrak{l}_{\beta,\mathfrak{b}}}(\mathsf{R}^{\mathfrak{b}})$ by $\bar{\mathfrak{I}}^{\mathbf{X}}_{\mathfrak{l}_{\beta,\mathfrak{b}}}(\mathsf{R}^{\mathfrak{b}})$. The error is controlled by a universal negative power of $N$ by Lemma \ref{lemma:BGI2I4}.
\item Replace $\bar{\mathfrak{I}}^{\mathbf{X}}_{\mathfrak{l}_{\beta,\mathfrak{b}}}(\mathsf{R}^{\mathfrak{b}})$ by $\mathfrak{I}^{\mathbf{T}}_{\mathfrak{t}_{\mathfrak{j}_{+}}}\bar{\mathfrak{I}}^{\mathbf{X}}_{\mathfrak{l}_{\beta,\mathfrak{b}}}(\mathsf{R}^{\mathfrak{b}})$ with $\mathfrak{t}_{\mathfrak{j}_{+}}$ in Lemma \ref{lemma:BGI2I5}. The error, by Lemma \ref{lemma:BGI2I5}, is a universal negative power of $N$.
\item We now apply Lemma \ref{lemma:BGI2I6} to estimate the resulting heat operator acting on $N^{1/2}\mathfrak{I}^{\mathbf{T}}_{\mathfrak{t}_{\mathfrak{j}_{+}}}\bar{\mathfrak{I}}^{\mathbf{X}}_{\mathfrak{l}_{\beta,\mathfrak{b}}}(\mathsf{R}^{\mathfrak{b}})\mathbf{Y}^{N}$.
\end{itemize}
Combining the previous bullet points with the triangle inequality for $\|\|_{1;\mathbb{T}_{N}}$ and $\E$ completes the proof. \qed
\begin{proof}[Proof of \emph{Lemma \ref{lemma:BGI2I1}}]
We first extend the $\mathsf{E}^{\mathrm{can}}$-expectations in Definition \ref{definition:BGI2} to any functional $\mathfrak{f}$ instead of just $\bar{\mathfrak{q}}$. Let us observe the following quantity vanishes under expectation with respect to any canonical ensemble on its support. We clarify/emphasize the second term below is an expectation of the functional $\mathfrak{f}_{0,0}=\mathfrak{f}$ with respect to the canonical measure on $y-\llbracket1,N^{\delta+\e_{\mathrm{RN},1}}\rrbracket$ with parameter equal to the $\eta$-density $\sigma_{\delta+\e_{\mathrm{RN},1},S,y}$ on this set at time $S$; see Definition \ref{definition:BGI2}. We also clarify that we require $\mathfrak{f}$ to be supported in a uniformly bounded neighborhood to the left of $0\in\mathbb{T}_{N}$ like $\bar{\mathfrak{q}}$, and thus contained in $y-\llbracket1,N^{\delta+\e_{\mathrm{RN},1}}\rrbracket$. This implies the support of \eqref{eq:BGI2I10} is contained in that of {$\mathsf{R}_{\delta}(\tau_{y}\eta)$} of \eqref{eq:defBGI24}.
\begin{align}
\mathfrak{f}_{S,y} - {\mathsf{E}_{\delta+\e_{\mathrm{RN},1}}^{\mathrm{can}}(\tau_{y}\eta_{S};\mathfrak{f})}. \label{eq:BGI2I10}
\end{align}
Vanishing of \eqref{eq:BGI2I10} by canonical measure expectation follows by tower property of conditional expectation and that the projection of any canonical measure on any larger set onto any smaller subset is a convex combination of canonical measures; see the proof of {Lemma 2} in \cite{GJ15}. We eventually take $\mathsf{R}^{\mathrm{cut}}$ to be a quantity of the form \eqref{eq:BGI2I10} in which $\mathfrak{f}$ admits the deterministic upper bound required for $\mathsf{R}^{\mathrm{cut}}$. To this end, we first recall $\mathsf{R}$ in \eqref{eq:defBGI24}. Again by the projection property for canonical measures, now combined with the tower property for expectation, we get the following with notation explained afterwards:
\begin{align}
{\mathsf{R}_{\delta}(\tau_{y}\eta_{S}) \ = \ \mathsf{E}_{\delta}^{\mathrm{can}}(\tau_{y}\eta_{S}) - \mathsf{E}_{\delta+\e_{\mathrm{RN},1}}^{\mathrm{can}}(\tau_{y}\eta_{S};\mathsf{E}_{\delta,\sigma(\delta)}^{\mathrm{can}}(\bar{\mathfrak{q}})).} \label{eq:BGI2I11}
\end{align}
%
\begin{itemize}
\item {Let $\sigma(\delta)$ be a random $\eta$-density on $y+\mathbb{I}^{\delta}=y+\llbracket-N^{\delta+\e_{\mathrm{RN},1}},-1\rrbracket$. Its law is given by that of the $\eta$-density on $y+\mathbb{I}^{\delta}$ according to the measure defining the expectation {$\mathsf{E}^{\mathrm{can}}_{\delta+\e_{\mathrm{RN},1}}(\tau_{y}\eta_{S})$}. (Note $y+\mathbb{I}^{\delta}$ is the subset that the canonical measure in $\mathsf{E}^{\mathrm{can}}_{\delta}(\tau_{y}\eta_{S})$ is defined on.)}
\item Define $\mathfrak{f}_{1}=\mathsf{E}^{\mathrm{can}}_{\delta,\sigma(\delta)}(\bar{\mathfrak{q}})$ to be the canonical measure expectation of $\bar{\mathfrak{q}}$ with respect to the $\sigma(\delta)$-canonical measure on $y+\mathbb{I}^{\delta}$.
\item Observe that $\mathfrak{f}_{1}$ is a functional of the particle system, and it depends only on the random variable/$\eta$-density $\sigma(\delta)$ from the first bullet point. In particular, the second term/iterated $\mathsf{E}$-expectation is an expectation of $\mathfrak{f}_{1}$, where the randomness in the inside expectation is now through the random $\eta$-density $\sigma(\delta)$ that is sampled with respect to the canonical measure defining the outer expectation on the RHS of \eqref{eq:BGI2I11}. So, we get \eqref{eq:BGI2I11} by first conditioning $\bar{\mathfrak{q}}$ in the {$\mathsf{E}^{\mathrm{can}}_{\delta+\e_{\mathrm{RN},1}}(\tau_{y}\eta_{S})$}-expectation in \eqref{eq:defBGI24} on $\sigma(\delta)$; again, we emphasize canonical measures project to canonical measures, so taking said expectation conditioning on $\sigma(\delta)$ leads to canonical measure expectation of $\bar{\mathfrak{q}}$ with parameter $\sigma(\delta)$ on its defining set $y+\mathbb{I}^{\delta}$. This is what \eqref{eq:BGI2I11} says. We clarify that {$\mathsf{E}^{\mathrm{can}}_{\delta}(\tau_{y}\eta_{S})$} and $\mathsf{E}^{\mathrm{can}}_{\delta,\sigma(\delta)}(\bar{\mathfrak{q}})$ are the same function, but the former is evaluated at $\sigma_{\delta,S,y}$, i.e. the scale-$N^{\delta}$ density of the actual particle system at time $S$ and point $y$, and the latter is evaluated at the random $\sigma(\delta)$ sampled via {$\mathsf{E}^{\mathrm{can}}_{\delta+\e_{\mathrm{RN},1}}(\tau_{y}\eta_{S})$} as explained.
\end{itemize}
We now make the following observation, which implies that it suffices to provide a priori estimates for $\mathbb{I}^{\delta}$ expectations. 
\begin{itemize}
\item Suppose $|{\mathsf{E}^{\mathrm{can}}_{\delta}(\tau_{y}\eta_{S})}|+|\mathsf{E}^{\mathrm{can}}_{\delta,\sigma(\delta)}(\bar{\mathfrak{q}})|\leq N^{10\e_{\mathrm{ap}}}N^{-\delta}$. This is not necessarily true; we will show it is sufficiently close to true.
\item The previous bullet point would finish the proof, since \eqref{eq:BGI2I11} provides a representation of $\mathsf{R}_{\delta}(\eta)$ as \eqref{eq:BGI2I10} for $\mathfrak{f}= \mathsf{E}^{\mathrm{can}}_{\delta,\sigma(\delta)}(\bar{\mathfrak{q}})$, which would certainly satisfy be $\mathrm{O}(N^{10\e_{\mathrm{ap}}}N^{-\delta})$ if the previous bullet point were true.
\end{itemize}
In view of the previous bullet points, it suffices to make the following replacements in \eqref{eq:BGI2I11}, provided that $\mathbf{Y}^{N}\neq0$, for which we first establish convenient notation $\mathsf{C}_{\alpha}(a)=a\mathbf{1}(|a|\leq\alpha)$ for a cutoff operator/map where $a\in\R$ is any real number; note that the replacements below do not change the support of any term in \eqref{eq:BGI2I11}:
\begin{align}
{\mathsf{E}_{\delta}^{\mathrm{can}}(\tau_{y}\eta_{S})} \to \mathsf{C}_{N^{10\e_{\mathrm{ap}}}N^{-\delta}}({\mathsf{E}_{\delta}^{\mathrm{can}}(\tau_{y}\eta_{S})}) \quad \mathrm{and} \quad \mathsf{E}_{\delta,\sigma(\delta)}^{\mathrm{can}}(\bar{\mathfrak{q}}) \to \mathsf{C}_{N^{10\e_{\mathrm{ap}}}N^{-\delta}}(\mathsf{E}_{\delta,\sigma(\delta)}^{\mathrm{can}}(\bar{\mathfrak{q}})). \label{eq:BGI2I12}
\end{align}
Indeed, whenever $\mathbf{Y}^{N}=0$, then the proposed estimate in Lemma \ref{lemma:BGI2I1} is trivial. Moreover, defining $\mathsf{R}^{\mathrm{cut}}$ to be the RHS of \eqref{eq:BGI2I11} but with the replacements in \eqref{eq:BGI2I12}, the previous two bullet points would imply that $\mathsf{R}^{\mathrm{cut}}$ satisfies the proposed pair of properties claimed in the lemma. We now show the replacements \eqref{eq:BGI2I12} to the RHS of \eqref{eq:BGI2I11}, whenever $\mathbf{Y}^{N}\neq0$, only provide error that, after multiplication by $\mathbf{Y}^{N}$, is controlled by the RHS of the proposed estimate. 
\begin{itemize}
\item The first replacement in \eqref{eq:BGI2I12} gives no error whenever $\mathbf{Y}^{N}\neq0$. Indeed, Lemma \ref{lemma:BGI31} still holds with $\e_{1}+\mathfrak{b}_{+}\e_{\mathrm{RN},1}$ therein replaced with $\delta$ here, because all we require for Lemma \ref{lemma:BGI31} is a priori spatial regularity of $\mathbf{Y}^{N}$ on the length-scale $N^{\delta}\leq\mathfrak{l}_{N}$ if $\delta+\e_{\mathrm{RN},1}\leq\frac12+\e_{\mathrm{RN}}$, which holds by assumption here; see Definition \ref{definition:KPZ1} for $\mathfrak{l}_{N}$. Thus, we may assume the canonical measure parameter $\sigma_{\delta+\e_{\mathrm{RN},1},S,y}$ defining {$\mathsf{E}^{\mathrm{can}}_{\delta}(\tau_{y}\eta_{S})$} in \eqref{eq:BGI2I11} is at most $N^{5\e_{\mathrm{ap}}}N^{-\delta/2}$, from which we show the first replacement in \eqref{eq:BGI2I12} does nothing by following the proof of Proposition \ref{prop:BGI3}.
\item We move to the second replacement in \eqref{eq:BGI2I12} applied to \eqref{eq:BGI2I11}. Similar to the previous bullet point, for the outer expectation in the second term on the RHS of \eqref{eq:BGI2I11}, we know that its canonical measure parameter satisfies $\sigma_{\delta+2\e_{\mathrm{RN},1},S,y}\leq N^{4\e_{\mathrm{ap}}}N^{-\delta/2}$; combining the previous bullet point with this most recent observation gives
\begin{align}
{\mathsf{R}_{\delta}(\tau_{y}\eta_{S})}\mathbf{Y}_{S,y}^{N} \ &= \ \mathbf{1}(|\sigma_{\delta+\e_{\mathrm{RN},1},S,y}|\leq N^{5\e_{\mathrm{ap}}}N^{-\delta/2}){\mathsf{E}_{\delta}^{\mathrm{can}}(\tau_{y}\eta_{S})}\mathbf{Y}_{S,y}^{N} \label{eq:BGI2I1mida} \\
&- \mathbf{1}(|\sigma_{\delta+2\e_{\mathrm{RN},1},S,y}|\leq N^{4\e_{\mathrm{ap}}}N^{-\delta/2}){\mathsf{E}_{\delta+\e_{\mathrm{RN},1}}^{\mathrm{can}}(\tau_{y}\eta_{S};\mathsf{E}_{\delta,\sigma(\delta)}^{\mathrm{can}}(\bar{\mathfrak{q}}))}\mathbf{Y}_{S,y}^{N}. \label{eq:BGI2I1midb}
\end{align}
We clarify we have $\sigma_{\delta+2\e_{\mathrm{RN},1},S,y}\leq N^{4\e_{\mathrm{ap}}}N^{-\delta/2}$ instead of $\sigma_{\delta+2\e_{\mathrm{RN},1},S,y}\leq N^{5\e_{\mathrm{ap}}}N^{-\delta/2}$ because the extra $N^{\e_{\mathrm{ap}}}$ produced in the proof of Proposition \ref{prop:BGI3} and Lemma \ref{lemma:BGI31} comes from applying the bound $|\mathbf{Y}^{N}|\leq N^{\e_{\mathrm{ap}}}$ in the proof of Lemma \ref{lemma:BGI31} that we do not yet need because we are not bounding $\mathbf{Y}^{N}$ to get the previous display.
\item A digression. Fix $\sigma_{\delta+2\e_{\mathrm{RN},1},S,y}$, and consider the canonical measure on its support with this parameter, namely the measure defining {$\mathsf{E}^{\mathrm{can}}_{\delta+\e_{\mathrm{RN},1}}(\tau_{y}\eta_{S})$}. If we instead consider the corresponding grand-canonical measure, then $\sigma(\delta)$ would be an average of $|\mathbb{I}^{\delta}|=N^{\delta+\e_{\mathrm{RN},1}}$-many independent Bernoulli random variables with expectation $\sigma_{\delta+2\e_{\mathrm{RN},1},S,y}$. Concentration inequalities would then imply $|\sigma(\delta)-\sigma_{\delta+2\e_{\mathrm{RN},1},S,y}|\geq N^{\e_{\mathrm{ap}}}\sigma_{\delta+2\e_{\mathrm{RN},1},S,y}$ happens with exponentially small probability in $N$. This can be seen by viewing $\sigma(\delta)$ as a random walk indexed by $\mathbb{I}^{\delta}$ with drift $\sigma_{\delta+2\e_{\mathrm{RN},1},S,y}$. The only difference in this discussion if we look at the canonical ensemble instead of grand-canonical ensemble is that $\sigma(\delta)$ is the length-$|\mathbb{I}^{\delta}|$-increment of a random walk \emph{bridge} with drift $\sigma_{\delta+2\e_{\mathrm{RN},1},S,y}$, for which sub-exponential concentration inequalities are also readily available.
\item Given that $\sigma_{\delta+2\e_{\mathrm{RN},1},S,y}\leq N^{4\e_{\mathrm{ap}}}N^{-\delta/2}$, we note that the $\sigma(\delta)$ density on the smaller subset $\mathbb{I}^{\delta}$, inside the previous display, is bounded by $N^{5\e_{\mathrm{ap}}}N^{-\delta/2}$ with overwhelming probability; see Definition \ref{definition:KPZ8}. Indeed, by thinking of $\sigma(\delta)$ as an $\mathbb{I}^{\delta}$-indexed increment of a random walk bridge with drift $N^{4\e_{\mathrm{ap}}}$, standard sub-Gaussian concentration inequalities for random walk bridges shows that $\sigma(\delta)$ deviates from its normalized drift $N^{4\e_{\mathrm{ap}}}N^{-\delta/2}$ plus its Brownian-type fluctuation $N^{-\delta/2}$ by a factor of $N^{\e_{\mathrm{ap}}}$ with exponentially small probability in $N$, and therefore with overwhelming probability. 
\item We now have the following where $\mathcal{E}=\{|\sigma(\delta)|\leq N^{5\e_{\mathrm{ap}}}N^{-\delta/2}\}$ with complement $\mathcal{E}^{C}$; we explain these calculations after:
{
\begin{align}
\mathsf{E}_{\delta+\e_{\mathrm{RN},1}}^{\mathrm{can}}(\tau_{y}\eta_{S};\mathsf{E}_{\delta,\sigma(\delta)}^{\mathrm{can}}(\bar{\mathfrak{q}})) \ &= \ \mathsf{E}_{\delta+\e_{\mathrm{RN},1}}^{\mathrm{can}}(\tau_{y}\eta_{S};\mathbf{1}_{\mathcal{E}}\mathsf{E}_{\delta,\sigma(\delta)}^{\mathrm{can}}(\bar{\mathfrak{q}})) + \mathsf{E}_{\delta+\e_{\mathrm{RN},1}}^{\mathrm{can}}(\tau_{y}\eta_{S};\mathbf{1}_{\mathcal{E}^{C}}\mathsf{E}_{\delta,\sigma(\delta)}^{\mathrm{can}}(\bar{\mathfrak{q}})) \label{eq:BGI2I13a} \\
&= \ \mathsf{E}_{\delta+\e_{\mathrm{RN},1}}^{\mathrm{can}}(\tau_{y}\eta_{S};\mathbf{1}_{\mathcal{E}}\mathsf{C}_{N^{10\e_{\mathrm{ap}}}N^{-\delta}}(\mathsf{E}_{\delta,\sigma(\delta)}^{\mathrm{can}}(\bar{\mathfrak{q}}))) + \mathsf{E}_{\delta+\e_{\mathrm{RN},1}}^{\mathrm{can}}(\tau_{y}\eta_{S};\mathbf{1}_{\mathcal{E}^{C}}\mathsf{E}_{\delta,\sigma(\delta)}^{\mathrm{can}}(\bar{\mathfrak{q}})) \label{eq:BGI2I13b} \\
&= \ \mathsf{E}_{\delta+\e_{\mathrm{RN},1}}^{\mathrm{can}}(\tau_{y}\eta_{S};\mathsf{C}_{N^{10\e_{\mathrm{ap}}}N^{-\delta}}(\mathsf{E}_{\delta,\sigma(\delta)}^{\mathrm{can}}(\bar{\mathfrak{q}}))) + \mathsf{E}_{\delta+\e_{\mathrm{RN},1}}^{\mathrm{can}}(\tau_{y}\eta_{S};\mathrm{O}(\mathbf{1}_{\mathcal{E}^{C}})). \label{eq:BGI2I13c}
\end{align}
}The first line \eqref{eq:BGI2I13a} is trivial: $1=\mathbf{1}_{\mathcal{E}}+\mathbf{1}_{\mathcal{E}^{C}}$. The second line \eqref{eq:BGI2I13b} follows by the argument in the first bullet point in the current list. The third line \eqref{eq:BGI2I13c} follows again by writing $\mathbf{1}_{\mathcal{E}}=1-\mathbf{1}_{\mathcal{E}^{C}}$. If we plug the second term in \eqref{eq:BGI2I13c} into \eqref{eq:BGI2I1midb}, by the previous bullet point and that $|\bar{\mathfrak{q}}|\lesssim1$, we get $\mathrm{O}(N^{-200})$, which is controlled by the RHS of the proposed estimate.
\item As noted in the current bullet point list prior to \eqref{eq:BGI2I1mida} and \eqref{eq:BGI2I1midb}, we can now drop the indicator functions \eqref{eq:BGI2I1mida} and \eqref{eq:BGI2I1midb}, so the error in making the replacements \eqref{eq:BGI2I12} in \eqref{eq:BGI2I11} is also appropriately controlled by the RHS of the proposed estimate.
\end{itemize}
Defining $\mathsf{R}^{\mathrm{cut}}$ to be the RHS of \eqref{eq:BGI2I11} but with replacements \eqref{eq:BGI2I12} for the RHS of \eqref{eq:BGI2I11}, the previous bullet points provide the proposed estimate for $\mathbf{Y}^{N}\neq0$, whereas the estimate is trivial if $\mathbf{Y}^{N}=0$. Moreover, as noted prior to \eqref{eq:BGI2I12}, the $\mathsf{R}^{\mathrm{cut}}$ functional satisfies all the required properties in the statement of the lemma, so we are done.
\end{proof}
\begin{proof}[Proof of \emph{Lemma \ref{lemma:BGI2I3}}]
We apply Lemma \ref{lemma:xr} with the choices $\mathfrak{t}=0$ and $\mathfrak{f}=N^{1/2}\mathsf{R}^{\mathfrak{b}}$ and $\mathfrak{l}=\mathfrak{l}_{\beta,\mathfrak{b}}$. Along with a few other gymnastics and conditions that need to be checked, which we explain shortly, this ultimately provides the following estimate; we first clarify that the prefactor $N^{5\e_{\mathrm{ap}}}\mathfrak{l}_{\beta,\mathfrak{b}}^{1/2}N^{\frac12\e_{1}+\frac12\mathfrak{b}\e_{\mathrm{RN},1}+\frac12\e_{\mathrm{RN},1}}$ for the second term on the RHS of \eqref{eq:BGI2I31} comes from noting $\bar{\mathfrak{l}}$ in Lemma \ref{lemma:xr} is the product of $\mathfrak{l}=\mathfrak{l}_{\beta,\mathfrak{b}}$ and the support length of $\mathfrak{f}=N^{1/2}\mathsf{R}^{\mathfrak{b}}$, which is given in Definition \ref{definition:BGI2I2}:
\begin{align}
\E\|\mathbf{H}^{N}(N^{1/2}\mathfrak{D}_{0,\mathfrak{l}_{\beta,\mathfrak{b}}}^{\mathbf{X}}(\mathsf{R}^{\mathfrak{b}}_{S,y})\mathbf{Y}^{N}_{S,y})\|_{1;\mathbb{T}_{N}} \ \lesssim \ N^{\e_{\mathrm{RN}}+\e_{\mathrm{ap}}}\|\mathsf{R}^{\mathfrak{b}}\|_{\omega;\infty} + N^{5\e_{\mathrm{ap}}}\mathfrak{l}_{\beta,\mathfrak{b}}^{1/2}N^{\frac12\e_{1}+\frac12\mathfrak{b}\e_{\mathrm{RN},1}+\frac12\e_{\mathrm{RN},1}}\|\mathsf{R}^{\mathfrak{b}}\|_{\omega;\infty}. \label{eq:BGI2I31}
\end{align}
Lemma \ref{lemma:xr} with the previous choices yields \eqref{eq:BGI2I31} with no changes to the first term but with $\|\mathsf{R}^{\mathfrak{b}}\|_{\omega;\infty}$ in the second term on the RHS of \eqref{eq:BGI2I31} replaced by $\E\|\mathbf{H}^{N}(|\mathsf{R}^{\mathfrak{b}}|)\|_{1;\mathbb{T}_{N}}$ . But $\E\|\mathbf{H}^{N}(|\mathsf{R}^{\mathfrak{b}}|)\|_{1;\mathbb{T}_{N}}\leq\|\mathsf{R}^{\mathfrak{b}}\|_{\omega;\infty}$. Lastly, Lemma \ref{lemma:xr} with our choices of $\mathfrak{f}$ and $\mathfrak{l}$ may only be applied if the support length of $\mathfrak{f}=N^{1/2}\mathsf{R}^{\mathfrak{b}}$ times $\mathfrak{l}=\mathfrak{l}_{\beta,\mathfrak{b}}$ is at most $\mathfrak{l}_{N}=N^{1/2+\e_{\mathrm{RN}}}$ from Definition \ref{definition:KPZ1}. This follows since $\mathfrak{l}_{\beta,\mathfrak{b}}$ is at most $N^{-\beta}$ times the support length of $\mathsf{R}^{\mathfrak{b}}$, while the constraint $\mathfrak{b}\leq\mathfrak{b}_{\mathrm{mid}}$ guarantees the square of the support length of $\mathsf{R}^{\mathfrak{b}}$ is at most $N^{1/2+2\e_{\mathrm{RN},1}}\leq\mathfrak{l}_{N}$ by construction in Definition \ref{definition:BGI2I2}; for the last bound concerning $\mathfrak{l}_{N}$, we recall $\mathfrak{l}_{N}$ in Definition \ref{definition:KPZ1} and note that by Definition \ref{definition:BGI2I2}, we clearly have the inequality $100\e_{\mathrm{RN},1}\leq\e_{\mathrm{RN}}$. It now suffices to plug in $\mathfrak{l}_{\beta,\mathfrak{b}}$ and $\mathsf{R}^{\mathfrak{b}}$ bounds in Definition \ref{definition:BGI2I2}; note $|\mathsf{R}^{\mathfrak{b}}|\lesssim N^{-\e_{1}+10\e_{\mathrm{ap}}}=N^{-1/14+10\e_{\mathrm{ap}}}$, and $\beta\geq999\e_{\mathrm{RN}}+999\e_{\mathrm{RN},1}$.
\end{proof}
\begin{proof}[Proof of \emph{Lemma \ref{lemma:BGI2I4}}]
We will follow the proof of Lemma \ref{lemma:BGI12}. In particular, it suffices to copy and paste that argument except we formally replace $\mathsf{S}_{\e_{1}}$ with $\mathsf{R}^{\mathfrak{b}}_{S,y}$ and $\mathfrak{l}_{1}$ with $\mathfrak{l}_{\beta,\mathfrak{b}}$. This has the following effects on that argument and its proofs.
\begin{itemize}
\item The bounds \eqref{eq:BGI121} and \eqref{eq:BGI122} still hold as written after the aforementioned replacements. The estimate \eqref{eq:BGI123} also ``almost" holds because the application of Lemma \ref{lemma:le3} used to obtain it, even after the aforementioned replacement, is still valid since $\mathsf{R}^{\mathfrak{b}}$ is still uniformly bounded. However, the first term on the RHS of \eqref{eq:BGI123} must be adjusted as the support length of the ``new" functional $\mathsf{R}^{\mathfrak{b}}$ is no longer order $N^{6/25}$, whose cube is present in the first term on the RHS of \eqref{eq:BGI123}. Recalling from Definition \ref{definition:BGI2I2} that for $\mathfrak{b}\leq\mathfrak{b}_{\mathrm{mid}}$ the support length of $\mathsf{R}^{\mathfrak{b}}$ is of order at most $N^{\e_{1}+\mathfrak{b}\e_{\mathrm{RN},1}+\e_{\mathrm{RN},1}}\leq N^{1/3}$, after the replacement $\mathsf{S}\to\mathsf{R}^{\mathfrak{b}}$ the first term on the RHS of \eqref{eq:BGI123} has its $N^{18/25}$ factor replaced by $N$, which controls the cube of the support length of $\mathsf{R}^{\mathfrak{b}}$.
\item The estimate \eqref{eq:BGI124} still holds after the replacement of $\mathfrak{l}_{1}$ by $\mathfrak{l}_{\beta,\mathfrak{b}}$ and of $\mathsf{S}$ by $\mathsf{R}^{\mathfrak{b}}$, as $|\mathsf{R}^{\mathfrak{b}}|\lesssim1$ and $\mathsf{R}^{\mathfrak{b}}$ also vanishes in expectation with respect to any canonical measure on its support. In particular, Lemma \ref{lemma:ee} holds with $\mathfrak{f}_{\mathfrak{j}}$ equal to spatial shifts of $\mathsf{R}^{\mathfrak{b}}$ with mutually disjoint supports and with $\mathfrak{J}=\mathfrak{l}_{\beta,\mathfrak{b}}$.
\end{itemize}
We deduce the claim from directly following the proof of Lemma \ref{lemma:BGI12}, at least upon checking that the replacement of $N^{18/25}$ with $N$ on the RHS of \eqref{eq:BGI123} still makes the contribution of the first term on the RHS of \eqref{eq:BGI123}, after plugging into \eqref{eq:BGI121} and \eqref{eq:BGI122}, controlled by the RHS of the proposed estimate \eqref{eq:BGI2I4}. This follows by elementary power-counting, so we are done.
\end{proof}
\begin{proof}[Proof of \emph{Lemma \ref{lemma:BGI2I5}}]
We follow the proof of Lemma \ref{lemma:BGI13}. Observe $\mathfrak{t}_{\mathfrak{j}_{+}}\leq N^{-1}$. Indeed, $\mathfrak{l}_{\beta,\mathfrak{b}}$ is $N^{-\beta}$ times the support length of $\mathsf{R}^{\mathfrak{b}-1}$, and the support length of $\mathsf{R}^{\mathfrak{b}-1}$ is $\gtrsim N^{-10\e_{\mathrm{RN,1}}}N^{\e_{1}}$. As $\e_{\mathrm{RN},1}$ and $\beta$ are much smaller than $\e_{1}=1/14$ from the statement of Proposition \ref{prop:BGI1}, the factor of $N^{-\e_{1}}$ beats powers of $N^{\beta}$ and $N^{\e_{\mathrm{RN},1}}$, and by inspecting definition in the statement of Lemma \ref{lemma:BGI2I5}, we deduce the proposed time-scale upper bound. Also, throughout the following proof, we replace $\mathfrak{l}_{1}$ whenever we appeal to the proof of Lemma \ref{lemma:BGI13} by $\mathfrak{l}_{\beta,\mathfrak{b}}$ from Definition \ref{definition:BGI2I2}.

We directly follow the argument in the proof of Lemma \ref{lemma:BGI13} preceding \eqref{eq:BGI131} but now with cutoff-spatial-averages of $\mathsf{R}^{\mathfrak{b}}$ rather than $\mathsf{S}$. Because $|\mathsf{R}^{\mathfrak{b}}|\lesssim1$ and $|\bar{\mathfrak{I}}^{\mathbf{X}}_{\mathfrak{l}_{\beta,\mathfrak{b}}}(\mathsf{R}^{\mathfrak{b}})|\lesssim N^{-\alpha}$ for $\alpha\gtrsim1$, we obtain \eqref{eq:BGI130} but with $\mathfrak{f}=\mathsf{R}^{\mathfrak{b}}$ instead of $\mathfrak{f}=\mathsf{S}$ and with $\alpha>0$ universal instead of $1/12$ on the far RHS of \eqref{eq:BGI130}. Thus, it ultimately suffices to control from above the following quantity that is analogous to \eqref{eq:BGI131} uniformly in the indices $\mathfrak{j}<\mathfrak{j}_{+}$ of interest:
\begin{align}
N^{3\e_{\mathrm{ap}}}\mathfrak{t}_{\mathfrak{j}+1}^{\frac14}\E\|\mathbf{H}^{N}\left(N^{\frac12}|\mathfrak{I}_{\mathfrak{t}_{\mathfrak{j}}}^{\mathbf{T}}\bar{\mathfrak{I}}_{\mathfrak{l}_{\beta,\mathfrak{b}}}^{\mathbf{X}}(\mathsf{R}^{\mathfrak{b}}_{S,y})|\right)\|_{1;\mathbb{T}_{N}} \ \lesssim \ \left(N^{8\e_{\mathrm{ap}}}\mathfrak{t}_{\mathfrak{j}+1}^{\frac38}\E\mathbf{I}_{1}(N^{\frac34}|\mathfrak{I}_{\mathfrak{t}_{\mathfrak{j}}}^{\mathbf{T}}\bar{\mathfrak{I}}_{\mathfrak{l}_{\beta,\mathfrak{b}}}^{\mathbf{X}}(\mathsf{R}^{\mathfrak{b}}_{S,y})|^{\frac32})\right)^{\frac23}. \label{eq:BGI2I51}
\end{align}
The estimate in \eqref{eq:BGI2I51} follows from applying Lemma \ref{lemma:hoe2} as \eqref{eq:BGI131} did but now with $\mathsf{S}$ replaced by $\mathsf{R}^{\mathfrak{b}}$ and $\mathfrak{l}_{1}$ by $\mathfrak{l}_{\beta,\mathfrak{b}}$. Following the paragraph after \eqref{eq:BGI131} and prior to \eqref{eq:BGI132}, because of the first paragraph in this proof it suffices to estimate the RHS of \eqref{eq:BGI2I51} for all $N^{-2}\lesssim\mathfrak{t}_{\mathfrak{j}}\leq N^{-1}$, which by construction in Definition \ref{definition:KPZ1} means $\mathfrak{t}_{\mathfrak{j}+1}\leq N^{-1+\e_{\mathrm{ap}}}$. To this end, observe \eqref{eq:BGI132} holds with the replacement $\mathsf{S}\to\mathsf{R}^{\mathfrak{b}}$ except $\mathrm{Loc}$ therein is now with respect to length-scale $\mathfrak{l}_{\mathrm{tot}}$ from Definition \ref{definition:locmap1}/Lemma \ref{lemma:le2}, which is also taken with $\gamma_{0}=\e_{\mathrm{ap}}$, with the choice $\mathfrak{l}$ equal to the support length of $\bar{\mathfrak{I}}^{\mathbf{X}}_{\mathfrak{l}_{\beta,\mathfrak{b}}}(\mathsf{R}^{\mathfrak{b}})$, which is $\mathfrak{l}_{\beta,\mathfrak{b}}$ times the support length of $\mathsf{R}^{\mathfrak{b}}$ written in Definition \ref{definition:BGI2I2}, and with $\mathfrak{l}_{\mathrm{av}}=1$, because the spatial-averaging scale $\mathfrak{l}_{\beta,\mathfrak{b}}$ is already built into $\mathfrak{l}$. The effect of this distinction in $\mathfrak{l}_{\mathrm{tot}}$ will be given shortly. For clarity, let us record this estimate below, which we reference shortly:
\begin{align}
N^{8\e_{\mathrm{ap}}}\mathfrak{t}_{\mathfrak{j}+1}^{\frac38}\E\mathbf{I}_{1}(N^{\frac34}|\mathfrak{I}_{\mathfrak{t}_{\mathfrak{j}}}^{\mathbf{T}}\bar{\mathfrak{I}}_{\mathfrak{l}_{\beta,\mathfrak{b}}}^{\mathbf{X}}(\mathsf{R}^{\mathfrak{b}}_{S,y})|^{\frac32}) \ \lesssim \ N^{\frac34+8\e_{\mathrm{ap}}}\mathfrak{t}_{\mathfrak{j}+1}^{\frac38}{\E_{0}}\bar{\mathfrak{P}}_{1}\E_{\mathrm{Loc}}^{\mathrm{dyn}}|\mathfrak{I}_{\mathfrak{t}_{\mathfrak{j}}}^{\mathbf{T}}\bar{\mathfrak{I}}_{\mathfrak{l}_{\beta,\mathfrak{b}}}^{\mathbf{X}}(\mathsf{R}^{\mathfrak{b}}_{0,0})|^{\frac32} + N^{-100}. \label{eq:BGI2I52}
\end{align}
Similar to the second term on the RHS of \eqref{eq:BGI132}, the second term on the RHS of \eqref{eq:BGI2I52} has contribution ultimately controlled by the RHS of the proposed estimate. To study the first term on the RHS of \eqref{eq:BGI2I52}, we use Lemma \ref{lemma:le3} as with the first term on the RHS of \eqref{eq:BGI132} from the proof of Lemma \ref{lemma:BGI13}. We will make the same choices for inputs/ingredients for Lemma \ref{lemma:le3} as we made to analyze the first term on the RHS of \eqref{eq:BGI132}, except with the following adjustment that takes into consideration the different a priori estimates we have on the cutoff spatial average $\bar{\mathfrak{I}}^{\mathbf{X}}_{\mathfrak{l}_{\beta,\mathfrak{b}}}(\mathsf{R}^{\mathfrak{b}})$ as opposed to $\mathsf{S}_{\e_{1}}$.
\begin{itemize}
\item First, let us clarify we choose $\mathfrak{h}$ to be the $\E^{\mathrm{dyn}}$-term on the RHS of \eqref{eq:BGI2I52}, so with $\mathsf{R}^{\mathfrak{b}}$ and not $\mathsf{S}$.
\item We choose $\kappa=N^{-3\e_{\mathrm{ap}}/2}\mathfrak{l}_{\beta,\mathfrak{b}}^{3/4}N^{-15\e_{\mathrm{ap}}+3\e_{1}/2+3\mathfrak{b}\e_{\mathrm{RN},1}/2}$ for the $\kappa$ constant in the statement of Lemma \ref{lemma:le3}. As $\kappa|\bar{\mathfrak{I}}_{\mathfrak{l}_{\beta,\mathfrak{b}}}^{\mathbf{X}}(\mathsf{R}^{\mathfrak{b}})|\lesssim1$, this choice of $\kappa$ is compatible with our choice of $\mathfrak{h}$, so our application of Lemma \ref{lemma:le3} with these choices is legal.
\end{itemize}
Similar to the proof of Lemma \ref{lemma:BGI13} and bounds on the first term on the RHS of \eqref{eq:BGI132}, Lemma \ref{lemma:le3} bounds the first term on the RHS of \eqref{eq:BGI2I52} in terms of two quantities. The first of these two terms is the far LHS of \eqref{eq:BGI133}, which is ultimately negligible even with replacing $N^{\e_{1}}\mathfrak{l}_{1}$ in \eqref{eq:BGI133} with our new choice of $\mathfrak{l}_{\mathrm{tot}}$ adapted to the support length of $\mathsf{R}^{\mathfrak{b}}$ and the spatial-average-length-scale $\mathfrak{l}_{\beta,\mathfrak{b}}$. Recall from Lemma \ref{lemma:le2} that $\mathfrak{l}_{\mathrm{tot}}$ is bounded by the spatial-average-length-scale $\mathfrak{l}_{\beta,\mathfrak{b}}$ times the support length $N^{\e_{1}+\mathfrak{b}\e_{\mathrm{RN},1}+\e_{\mathrm{RN},1}}$, in Definition \ref{definition:BGI2I2}, of $\mathsf{R}^{\mathfrak{b}}$. With the new choice for $\kappa$ made in the bullet point list above, we deduce that the first upper-bound-term for the first term on the RHS of \eqref{eq:BGI2I52}/far LHS of \eqref{eq:BGI133}, but after replacing $N^{\e_{1}}\mathfrak{l}_{1}$ with $\mathfrak{l}_{\mathrm{tot}}$, is ultimately controlled by the RHS of the proposed estimate \eqref{eq:BGI2I5}. This can be verified with an elementary power-counting after plugging into the middle of \eqref{eq:BGI133} our choice of $\kappa$ in the bullet points above and replacing $N^{\e_{1}}\mathfrak{l}_{1}$ in \eqref{eq:BGI133} by our new $\mathfrak{l}_{\mathrm{tot}}$. In particular, if $|\mathbb{B}|$ denotes the support length of $\E^{\mathrm{dyn}}$, we have the following estimate that is analogous to \eqref{eq:BGI133}:
\begin{align}
N^{\frac34+8\e_{\mathrm{ap}}}\mathfrak{t}_{\mathfrak{j}+1}^{\frac38}\kappa^{-1}N^{-2}|\mathbb{B}|^{3} \ \lesssim \ N^{\frac34+8\e_{\mathrm{ap}}}\mathfrak{t}_{\mathfrak{j}+1}^{\frac38}\kappa^{-1}N^{-2}\left(N^{1+\e_{\mathrm{ap}}}\mathfrak{t}_{\mathfrak{j}}^{\frac12}+N^{\frac32+\e_{\mathrm{ap}}}\mathfrak{t}_{\mathfrak{j}}+N^{\e_{\mathrm{ap}}}\mathfrak{l}_{\beta,\mathfrak{b}}N^{\e_{1}+\mathfrak{b}\e_{\mathrm{RN},1}+\e_{\mathrm{RN},1}}\right)^{3}. \label{eq:BGI2I52b}
\end{align}
Recalling $\e_{1}+\mathfrak{b}\e_{\mathrm{RN},1}\leq1/4$ and $\mathfrak{l}_{\beta,\mathfrak{b}}\leq N^{1/4+\beta}$ if $\mathfrak{b}\leq\mathfrak{b}_{\mathrm{mid}}$ by construction in Definition \ref{definition:BGI2I2}, the contribution of the RHS of \eqref{eq:BGI2I52b}, after plugging into \eqref{eq:BGI2I52} and taking its $2/3$-power in \eqref{eq:BGI2I51}, is controlled by the RHS of the proposed estimate \eqref{eq:BGI2I5} as we also have $\mathfrak{t}_{\mathfrak{j}}\leq N^{-1}$ and $\mathfrak{t}_{\mathfrak{j}+1}\leq N^{-1+\e_{\mathrm{ap}}}$, the first noted in the first paragraph of this proof and the latter by Definition \ref{definition:KPZ1}.

We move to the second upper-bound-term for the first term on the RHS of \eqref{eq:BGI2I52} that results from our application of Lemma \ref{lemma:le3}. This is the $\Phi$-term in \eqref{eq:BGI133b} except $\mathsf{S}\to\mathsf{R}^{\mathfrak{b}}$ and $\mathfrak{l}_{1}\to\mathfrak{l}_{\beta,\mathfrak{b}}$. In particular, everything until/before \eqref{eq:BGI138} and after \eqref{eq:BGI133b} holds with the replacements $\mathsf{S}\to\mathsf{R}^{\mathfrak{b}}$ and $\mathfrak{l}_{1}\to\mathfrak{l}_{\beta,\mathfrak{b}}$; indeed, $\mathsf{R}^{\mathfrak{b}}$ is uniformly bounded and vanishes in expectation with respect to any canonical measure on its support, so Lemma \ref{lemma:ee} applies to $\mathsf{R}^{\mathfrak{b}}$ and averages of its spatial translates. However, the estimate \eqref{eq:BGI138} must be modified to account for the new/longer support length of $\mathsf{R}^{\mathfrak{b}}$ as well as the spatial-average-length-scale in the $\bar{\mathfrak{I}}^{\mathbf{X}}$-term on the RHS of \eqref{eq:BGI2I52} and the improved a priori deterministic estimates on $\mathsf{R}^{\mathfrak{b}}$. In particular, by Lemma \ref{lemma:ee2} but with the choice of $\mathfrak{f}=\mathsf{R}^{\mathfrak{b}}$, we have the following estimate similar to how \eqref{eq:BGI138} was derived; we justify the following estimate afterwards:
\begin{align}
N^{1+\frac{32}{3}\e_{\mathrm{ap}}}\mathfrak{t}_{\mathfrak{j}+1}^{\frac12}\E^{\sigma}\E_{\mathrm{Loc}}^{\mathrm{dyn}}|\mathfrak{I}_{\mathfrak{t}_{\mathfrak{j}}}^{\mathbf{T}}\mathfrak{I}_{\mathfrak{l}_{\beta,\mathfrak{b}}}^{\mathbf{X}}(\mathsf{R}^{\mathfrak{b}}_{0,0})|^{2} \ \lesssim \ N^{1+\frac{35}{3}\e_{\mathrm{ap}}}\mathfrak{t}_{\mathfrak{j}}^{\frac12}N^{-2}\mathfrak{t}_{\mathfrak{j}}^{-1}\mathfrak{l}_{\beta,\mathfrak{b}}^{-1}N^{2\e_{1}+2\mathfrak{b}\e_{\mathrm{RN},1}+2\e_{\mathrm{RN},1}}\|\mathsf{R}^{\mathfrak{b}}_{0,0}\|_{\omega;\infty}^{2}+N^{-100}. \label{eq:BGI2I53}
\end{align}
In contrast to \eqref{eq:BGI138}, the $N^{2/14}$-factor therein is replaced by the square of the support length of $\mathsf{R}^{\mathfrak{b}}$ that is order $N^{\e_{1}+\mathfrak{b}\e_{\mathrm{RN},1}+\e_{\mathrm{RN},1}}$ as written in Definition \ref{definition:BGI2I2}. Moreover, the spatial-average-length-scale $\mathfrak{l}_{1}$ in \eqref{eq:BGI138} is replaced by the length-scale $\mathfrak{l}_{\beta,\mathfrak{b}}$. Lastly, we included the $\|\|_{\omega;\infty}$-factor in Lemma \ref{lemma:ee2} in our estimate \eqref{eq:BGI2I53}, which we did not do in \eqref{eq:BGI138}. Recalling now the $\|\mathsf{R}^{\mathfrak{b}}\|_{\omega;\infty}$-estimate in Definition \ref{definition:BGI2I2} and $\mathfrak{t}_{\mathfrak{j}}\geq N^{-2}$ and $\mathfrak{l}_{\beta,\mathfrak{b}}\geq N^{\e_{1}-\beta}$ with $\e_{1}=1/14$ much larger than $\e_{\mathrm{ap}}$ and $\e_{\mathrm{RN},1}$ and $\beta$, an elementary power-counting calculation shows the RHS of \eqref{eq:BGI2I53} is $\mathrm{O}(N^{-\alpha})$ for $\alpha>0$ universal. Therefore, as with the end of the proof of Lemma \ref{lemma:BGI13} after \eqref{eq:BGI138}, we are done.
\end{proof}
\begin{proof}[Proof of \emph{Lemma \ref{lemma:BGI2I6}}]
Unlike the proof of Lemma \ref{lemma:BGI14}, we will not need to introduce additional spatial averaging, so the proof of the current Lemma \ref{lemma:BGI2I6} is much simpler. We start via the following version of \eqref{eq:BGI2I51}, which is just \eqref{eq:BGI2I51} but without prefactors and for the maximal time-scale $\mathfrak{t}_{\mathfrak{j}_{+}}$ and with an additional $\mathbf{Y}^{N}$-factor; we explain its quick proof/derivation afterwards:
\begin{align}
\E\|\mathbf{H}^{N}\left(N^{\frac12}\mathfrak{I}_{\mathfrak{t}_{\mathfrak{j}_{+}}}^{\mathbf{T}}\bar{\mathfrak{I}}_{\mathfrak{l}_{\beta,\mathfrak{b}}}^{\mathbf{X}}(\mathsf{R}^{\mathfrak{b}}_{S,y})\mathbf{Y}_{S,y}^{N}\right)\|_{1;\mathbb{T}_{N}} \ \lesssim \ \left(N^{8\e_{\mathrm{ap}}}\E\mathbf{I}_{1}(N^{\frac34}|\mathfrak{I}_{\mathfrak{t}_{\mathfrak{j}_{+}}}^{\mathbf{T}}\bar{\mathfrak{I}}_{\mathfrak{l}_{\beta,\mathfrak{b}}}^{\mathbf{X}}(\mathsf{R}^{\mathfrak{b}}_{S,y})|^{\frac32})\right)^{\frac23}. \label{eq:BGI2I61}
\end{align}
Indeed, to prove \eqref{eq:BGI2I61}, we recall $|\mathbf{Y}^{N}|\leq N^{\e_{\mathrm{ap}}}$ to forget $\mathbf{Y}^{N}$ on the LHS and apply Lemma \ref{lemma:hoe2} in the same way as we did to get \eqref{eq:BGI2I51} and deduce \eqref{eq:BGI2I61}. For the RHS of \eqref{eq:BGI2I61}, we have the following by Lemma \ref{lemma:le2} in the same way as we derived \eqref{eq:BGI2I52}, in which the $\mathrm{Loc}$ term is, like in \eqref{eq:BGI2I52}, also chosen in Definition \ref{definition:locmap1}/Lemma \ref{lemma:le2} with $\mathfrak{l}_{\mathrm{tot}}$ defined by $\mathfrak{l}_{\mathrm{av}}=1$ and $\mathfrak{l}$ equal to $\mathfrak{l}_{\beta,\mathfrak{b}}$ times the support length of $\mathsf{R}^{\mathfrak{b}}$, which we recall is explicitly written in Definition \ref{definition:BGI2I2}:
\begin{align}
N^{8\e_{\mathrm{ap}}}\E\mathbf{I}_{1}(N^{\frac34}|\mathfrak{I}_{\mathfrak{t}_{\mathfrak{j}_{+}}}^{\mathbf{T}}\bar{\mathfrak{I}}_{\mathfrak{l}_{\beta,\mathfrak{b}}}^{\mathbf{X}}(\mathsf{R}^{\mathfrak{b}}_{S,y})|^{\frac32}) \  \lesssim \ N^{\frac34+8\e_{\mathrm{ap}}}{\E_{0}}\bar{\mathfrak{P}}_{1}\E_{\mathrm{Loc}}^{\mathrm{dyn}}|\mathfrak{I}_{\mathfrak{t}_{\mathfrak{j}_{+}}}^{\mathbf{T}}\bar{\mathfrak{I}}_{\mathfrak{l}_{\beta,\mathfrak{b}}}^{\mathbf{X}}(\mathsf{R}^{\mathfrak{b}}_{0,0})|^{\frac32} + N^{-100}. \label{eq:BGI2I62}
\end{align}
Contribution of the second term on the RHS of \eqref{eq:BGI2I62}, after plugging it in the RHS of \eqref{eq:BGI2I61} and taking $2/3$-powers, is bounded by the RHS of the proposed estimate \eqref{eq:BGI2I6}. We now estimate the first term on the RHS of \eqref{eq:BGI2I62} via Lemma \ref{lemma:le3}. In particular, let us apply Lemma \ref{lemma:le3} in the same way as we did in the proof of Lemma \ref{lemma:BGI2I5}, namely with the same choices of $\mathfrak{h}$ and $\kappa$ therein. This estimates the first term on the RHS of \eqref{eq:BGI2I62} by two terms, just as in the proof of Lemma \ref{lemma:BGI2I5}. The first of these, namely the first term on the RHS of \eqref{eq:le3}, depends on the support of $\E^{\mathrm{dyn}}$. It is ultimately controlled via the following, where $\mathbb{B}$ is the support of $\E^{\mathrm{dyn}}$ for which $\mathfrak{l}_{\mathrm{tot}}$ is explained prior to \eqref{eq:BGI2I62}; we explain the estimates below after:
\begin{align}
N^{\frac34+8\e_{\mathrm{ap}}}\kappa^{-1}N^{-2}|\mathbb{B}|^{3} \ \lesssim \ N^{\frac34+8\e_{\mathrm{ap}}}\kappa^{-1}N^{-2}\left(N^{1+\e_{\mathrm{ap}}}\mathfrak{t}_{\mathfrak{j}_{+}}^{\frac12} + N^{\frac32+\e_{\mathrm{ap}}}\mathfrak{t}_{\mathfrak{j}_{+}} + N^{\e_{\mathrm{ap}}}\mathfrak{l}_{\mathrm{tot}}\right)^{3} \ \lesssim \ N^{-\frac{1}{999}+100\e_{\mathrm{ap}}}. \label{eq:BGI2I63}
\end{align}
Let us recall the support length $|\mathbb{B}|$ of $\E^{\mathrm{dyn}}$ is given in the statement of Lemma \ref{lemma:le2}, and this gives the first estimate in \eqref{eq:BGI2I63}. The second estimate in \eqref{eq:BGI2I63} follows by recalling choices below we made for terms in \eqref{eq:BGI2I63} and elementary power-counting. In the bullet points below, we refer back to Definition \ref{definition:BGI2I2} and Lemma \ref{lemma:BGI2I3} and the proof of Lemma \ref{lemma:BGI2I5} for notation/constructions.
\begin{itemize}
\item Recall from the statement of Lemma \ref{lemma:BGI2I5} that the time-scale in \eqref{eq:BGI2I63} satisfies the upper bound $\mathfrak{t}_{\mathfrak{j}_{+}}\leq N^{-1+\beta}\mathfrak{l}_{\beta,\mathfrak{b}}^{-1}$.
\item Recall from bullet points after \eqref{eq:BGI2I52} that $\kappa\gtrsim\mathfrak{l}_{\beta,\mathfrak{b}}^{3/4}N^{3\e_{1}/2+3\mathfrak{b}\e_{\mathrm{RN},1}/2}N^{10\e_{\mathrm{ap}}+20\beta}\gtrsim\mathfrak{l}_{\beta,\mathfrak{b}}^{9/4}$.
\item Third, note $\mathfrak{l}_{\beta,\mathfrak{b}}\geq N^{\e_{1}-\beta}$ with $\e_{1}=1/14$ much bigger than $\beta$ for all $\mathfrak{b}\geq0$, which follows by construction in Definition \ref{definition:BGI2I2}.
\item We clarify that $\mathfrak{l}_{\mathrm{tot}}$ is controlled by the product of the $\mathsf{R}^{\mathfrak{b}}$-support length $N^{\e_{1}+\mathfrak{b}\e_{\mathrm{RN},1}+\e_{\mathrm{RN},1}}$ and the spatial-averaging length-scale $\mathfrak{l}_{\beta,\mathfrak{b}}$, both of these from Definition \ref{definition:BGI2I2}, as we explained prior to \eqref{eq:BGI2I62}. Thus, $\mathfrak{l}_{\mathrm{tot}}\lesssim N^{10\beta}\mathfrak{l}_{\beta,\mathfrak{b}}\lesssim N^{1/4+\e_{\mathrm{RN},1}+10\beta}$.
\end{itemize}
The estimate \eqref{eq:BGI2I63} controls the first term in the bound for the first term on the RHS of \eqref{eq:BGI2I62} that arises from an application of Lemma \ref{lemma:le3}. Let us now estimate the second term in said bound/the RHS of \eqref{eq:le3}. Following the paragraph prior to \eqref{eq:BGI2I53}, this second term is a large negative power of $N$ plus the following with estimates below to be justified/explained afterwards:
\begin{align}
N^{1+\frac{32}{3}\e_{\mathrm{ap}}}\E^{\sigma}\E_{\mathrm{Loc}}^{\mathrm{dyn}}|\mathfrak{I}_{\mathfrak{t}_{\mathfrak{j}_{+}}}^{\mathbf{T}}\mathfrak{I}_{\mathfrak{l}_{\beta,\mathfrak{b}}}^{\mathbf{X}}(\mathsf{R}^{\mathfrak{b}}_{0,0})|^{2} \ \lesssim \ N^{1+\frac{32}{3}\e_{\mathrm{ap}}}N^{-2}\mathfrak{t}_{\mathfrak{j}_{+}}^{-1}\mathfrak{l}_{\beta,\mathfrak{b}}^{-1}N^{2\e_{1}+2\mathfrak{b}\e_{\mathrm{RN},1}+2\e_{\mathrm{RN},1}}\|\mathsf{R}^{\mathfrak{b}}_{0,0}\|_{\omega;\infty}^{2}+N^{-100}. \label{eq:BGI2I64}
\end{align}
We now make the following observations for factors in the first term on the RHS of \eqref{eq:BGI2I64}.
\begin{itemize}
\item Note $\mathfrak{t}_{\mathfrak{j}_{+}}\geq N^{-1+\beta-2\e_{\mathrm{ap}}}\mathfrak{l}_{\beta,\mathfrak{b}}^{-1}$, as $\mathfrak{t}_{\mathfrak{j}}$ increases by a factor of $N^{\e_{\mathrm{ap}}}$ in the index and $\mathfrak{t}_{\mathfrak{j}_{+}}$ is the last $\mathfrak{t}_{\mathfrak{j}}$ to satisfy $\mathfrak{t}_{\mathfrak{j}_{+}}\leq N^{-1+\beta}\mathfrak{l}_{\beta,\mathfrak{b}}^{-1}$. 
\item Note $N^{2\e_{1}+2\mathfrak{b}\e_{\mathrm{RN},1}+2\e_{\mathrm{RN},1}}\|\mathsf{R}^{\mathfrak{b}}_{0,0}\|_{\omega;\infty}^{2}\lesssim N^{2\e_{\mathrm{RN},1}}$; see Definition \ref{definition:BGI2I2}. This is the utility of bounds for $\mathsf{R}^{\mathfrak{b}}$ that improve in $\mathfrak{b}$. Also, we have $\mathfrak{l}_{\beta,\mathfrak{b}}\gtrsim N^{-\beta}N^{\e_{1}}$ for $\e_{1}=1/14$; again see Definition \ref{definition:BGI2I2}.
\end{itemize}
With this pair of observations, like the proof of Lemma \ref{lemma:BGI14}, we deduce the contribution of the first term on the RHS of \eqref{eq:BGI2I64} is controlled by the RHS of the proposed bound \eqref{eq:BGI2I6}. Combining this with \eqref{eq:BGI2I63} to estimate \eqref{eq:BGI2I61} completes the proof.
\end{proof}
%
%
%
\section{Boltzmann-Gibbs Principle I -- Proof of Proposition \ref{prop:BGI2}, Case II}\label{section:proofBGI22}
The strategy we take in this section is remarkably similar to the strategy of the previous section. In particular, we will employ Lemma \ref{lemma:BGI2I1} to replace ${\mathsf{R}_{\e_{1}+\mathfrak{b}\e_{\mathrm{RN},1}}(\tau_{y}\eta_{S})}$ by ${\mathsf{R}_{\e_{1}+\mathfrak{b}\e_{\mathrm{RN},1}}^{\mathrm{cut}}(\tau_{y}\eta_{S})}$ on the LHS of \eqref{eq:BGI2Prop} for all $\mathfrak{b}\in\llbracket\mathfrak{b}_{\mathrm{mid}}+1,\mathfrak{b}-1\rrbracket$. Recall $\mathfrak{b}_{\mathrm{mid}}$ is the index cutoff that distinguishes Case I and Case II of Proposition \ref{prop:BGI2}; see Definition \ref{definition:BGI2I2}. Afterwards:
\begin{itemize}
\item First, we define $\mathsf{R}^{\mathfrak{b}}_{S,y}={\mathsf{R}_{\e_{1}+\mathfrak{b}\e_{\mathrm{RN},1}}^{\mathrm{cut}}(\tau_{y}\eta_{S})}$ throughout this section, as in the previous section, to ease notation.
\item Second, we replace $\mathsf{R}^{\mathfrak{b}}$ with a spatial-average like with the previous section. However, we will average it here on spatial-scale $\mathfrak{l}_{\beta}=N^{\beta}$, not the length-scale $\mathfrak{l}_{\beta,\mathfrak{b}}$ that matches, up to the factor of $N^{-\beta}$, the length of the support of $\mathsf{R}^{\mathfrak{b}}$. We cannot average it on the length-scale $\mathfrak{l}_{\beta,\mathfrak{b}}$ in this section, as controlling the resulting spatial gradients would require spatial regularity estimates for $\mathbf{Y}^{N}$ on length-scales that are well beyond those which we have a priori $\mathbf{Y}^{N}$ estimates for. However, as the support of $\mathsf{R}^{\mathfrak{b}}$ is larger in Case II, the a priori estimates for $\mathsf{R}^{\mathfrak{b}}$ in Lemma \ref{lemma:BGI2I1} are better than they generally were in Case I; this helps. Actually, for this reason we ultimately will \emph{not} need to replace this spatial average of $\mathsf{R}^{\mathfrak{b}}$ with a cutoff as in Lemma \ref{lemma:BGI2I4}.
\item Third, we replace the spatial average of $\mathsf{R}^{\mathfrak{b}}$ with its time-average/the space-time average of $\mathsf{R}^{\mathfrak{b}}$ with respect to a time-scale that is roughly equal to $\mathfrak{t}_{\mathfrak{j}_{+}}=N^{-1-\beta/2}$ and in particular independent of $\mathfrak{b}$, although this last feature will not be important. Again, we say ``roughly" because we will need to use a time-scale contained in $\mathbb{I}^{\mathbf{T},1}$ for the technical reason that we only have a priori regularity estimates for $\mathbf{Y}^{N}$ on these time-scales. After this replacement we will estimate this last space-time average of $\mathsf{R}^{\mathfrak{b}}$ to complete the proof of Case II of Proposition \ref{prop:BGI2}, and thus the proof of Proposition \ref{prop:BGI2} when combined with the last section. 
\item The previous three steps, in terms of the technical estimates, are done with the same general tools introduced in Section \ref{section:BGIPrelim}. 
\end{itemize}
To ease the following reading we recall the following facts from after Definition \ref{definition:BGI2I2} about $\mathsf{R}^{\mathfrak{b}}$ that follow via Lemma \ref{lemma:BGI2I1}.
\begin{itemize}
\item The support of $\mathsf{R}^{\mathfrak{b}}$ has length of order $N^{\e_{1}+\mathfrak{b}\e_{\mathrm{RN},1}+\e_{\mathrm{RN},1}}$, and we have $\|\mathsf{R}^{\mathfrak{b}}\|_{\omega;\infty}\leq N^{10\e_{\mathrm{ap}}}N^{-\e_{1}-\mathfrak{b}\e_{\mathrm{RN},1}}$ by construction.
\item Lastly, as in Definition \ref{definition:BGI2I2}, we will define and sometimes use $N^{\beta+\e_{\mathrm{RN},1}}\mathfrak{l}_{\beta,\mathfrak{b}}=N^{\e_{1}+\mathfrak{b}\e_{\mathrm{RN},1}+\e_{\mathrm{RN},1}}$ where $\beta=999^{-99}$.
\end{itemize}
Similar to the previous two sections, we provide each of the previous ingredients listed above and use them to establish Case II of Proposition \ref{prop:BGI2}. We then provide the proof for each of the ingredients to complete this section. Only the proof of spatial-average replacement in Lemma \ref{lemma:BGI2II1} requires an additional idea, while other proofs will effectively be copied.
\subsubsection{Spatial Average}
We start with the aforementioned replacement of $\mathsf{R}^{\mathfrak{b}}$ with its spatial average on length-scale $N^{\beta}$. The proof of the following result is highly similar to that of Lemma \ref{lemma:BGI2I3}, so we refer to that argument with necessary adjustments, including one important detail, when we present the proof of Lemma \ref{lemma:BGI2II1} below.
\begin{lemma}\label{lemma:BGI2II1}
\fsp We define the length-scale $\mathfrak{l}_{\beta}=N^{\beta}$ for $\beta=999^{-99}$. Uniformly in $\mathfrak{b}\in\llbracket\mathfrak{b}_{\mathrm{mid}}+1,\mathfrak{b}_{+}-1\rrbracket$, we have the following estimate for which we recall the transfer-of-spatial-scale operator from \emph{Definition \ref{definition:BGI10}}:
\begin{align}
\E\|\mathbf{H}^{N}\left(N^{\frac12}\mathfrak{D}_{0,\mathfrak{l}_{\beta}}^{\mathbf{X}}(\mathsf{R}^{\mathfrak{b}}_{S,y})\mathbf{Y}^{N}_{S,y}\right)\|_{1;\mathbb{T}_{N}} \ \lesssim \  N^{-\frac{1}{9999}+\e_{\mathrm{RN}}+2\e_{\mathrm{RN},1}+12\e_{\mathrm{ap}}+2\beta}. \label{eq:BGI2II1}
\end{align}
\end{lemma}
\subsubsection{Time Average}
The following is replacement-by-time-average in the third bullet point. For its proof, we basically copy that of Lemma \ref{lemma:BGI2I5} with technical modifications. Recall $\mathbb{I}^{\mathbf{T},1}$ from Definition \ref{definition:KPZ1}.
\begin{lemma}\label{lemma:BGI2II2}
\fsp Let $\mathfrak{j}_{+}\in\Z_{\geq0}$ be the largest index for which $\mathfrak{t}_{\mathfrak{j}_{+}}\leq N^{-1-\beta/2}$ is the largest time in $\mathbb{I}^{\mathbf{T},1}$ satisfying this bound; here $\beta=999^{-99}$. For $\mathfrak{b}\in\llbracket\mathfrak{b}_{\mathrm{mid}}+1,\mathfrak{b}_{+}-1\rrbracket$, we have the following; recall the transfer-of-time-scale-operator in \emph{Definition \ref{definition:mtr0}}:
\begin{align}
\E\|\mathbf{H}^{N}\left(N^{\frac12}\mathfrak{D}_{0,\mathfrak{t}_{\mathfrak{j}_{+}}}^{\mathbf{T}}(\mathfrak{I}_{\mathfrak{l}_{\beta}}^{\mathbf{X}}(\mathsf{R}^{\mathfrak{b}}_{S,y}))\mathbf{Y}^{N}_{S,y}\right)\|_{1;\mathbb{T}_{N}} \ \lesssim \ N^{-\frac{1}{99999}\beta+100\e_{\mathrm{ap}}}.\label{eq:BGI2II2}
\end{align}
\end{lemma}
\subsubsection{Final Estimates}
Our last ingredient before we deduce Case II of Proposition \ref{prop:BGI2} is the following estimate on the space-time average $\mathfrak{I}^{\mathbf{T}}_{\mathfrak{t}_{\mathfrak{j}_{+}}}\mathfrak{I}^{\mathbf{X}}_{\mathfrak{l}_{\beta}}\mathsf{R}^{\mathfrak{b}}$ for $\mathfrak{l}_{\beta}$ in Lemma \ref{lemma:BGI2II1} and for $\mathfrak{t}_{\mathfrak{j}_{+}}$ in Lemma \ref{lemma:BGI2II2}. The following final ingredient serves as an analog of Lemma \ref{lemma:BGI14} and Lemma \ref{lemma:BGI2I6}. Indeed, similar to those two results, most of the work is done for the lemma immediately before.
\begin{lemma}\label{lemma:BGI2II3}
\fsp Take the time-scale $\mathfrak{t}_{\mathfrak{j}_{+}}$ from \emph{Lemma \ref{lemma:BGI2II2}}. Uniformly in $\mathfrak{b}\in\llbracket\mathfrak{b}_{\mathrm{mid}}+1,\mathfrak{b}_{+}-1\rrbracket$, we have the following estimate:
\begin{align}
\E\|\mathbf{H}^{N}\left(N^{\frac12}\mathfrak{I}_{\mathfrak{t}_{\mathfrak{j}_{+}}}^{\mathbf{T}}\mathfrak{I}_{\mathfrak{l}_{\beta}}^{\mathbf{X}}(\mathsf{R}^{\mathfrak{b}}_{S,y})\mathbf{Y}^{N}_{S,y}\right)\|_{1;\mathbb{T}_{N}} \ \lesssim \ N^{-\frac{1}{99999}\beta+100\e_{\mathrm{ap}}}.\label{eq:BGI2II3}
\end{align}
\end{lemma}
Case II of Proposition \ref{prop:BGI2}, namely Proposition \ref{prop:BGI2} but restricting to $\mathfrak{b}\in\llbracket\mathfrak{b}_{\mathrm{mid}}+1,\mathfrak{b}_{+}-1\rrbracket$, follows from Lemmas \ref{lemma:BGI2II1}, \ref{lemma:BGI2II2}, and \ref{lemma:BGI2II3} combined with the same replacement reasoning that we used in the proof of Case I of Proposition \ref{prop:BGI2} at the end of the previous section. Together with the previous section, this concludes the proof of Proposition \ref{prop:BGI2} entirely. \qed
\begin{proof}[Proof of \emph{Lemma \ref{lemma:BGI2II1}}]
Let us follow the proof of Lemma \ref{lemma:BGI2I3}, though our application of Lemma \ref{lemma:xr} will be somewhat illegal but remedied as we soon explain. Formally, let us apply Lemma \ref{lemma:xr} with $\mathfrak{f}=N^{1/2}\mathsf{R}^{\mathfrak{b}}$ and $\mathfrak{t}=0$ as in the proof of Lemma \ref{lemma:BGI2I3}, but now with $\mathfrak{l}=\mathfrak{l}_{\beta}=N^{\beta}$ where $\beta=999^{-99}$. We claim that this provides the following inequality that we justify afterwards:
\begin{align}
\E\|\mathbf{H}^{N}(N^{1/2}\mathfrak{D}_{0,\mathfrak{l}_{\beta}}^{\mathbf{X}}(\mathsf{R}^{\mathfrak{b}}_{S,y})\mathbf{Y}^{N}_{S,y})\|_{1;\mathbb{T}_{N}} \ \lesssim \ N^{\e_{\mathrm{RN}}+\e_{\mathrm{ap}}}\|\mathsf{R}^{\mathfrak{b}}\|_{\omega;\infty} + N^{5\e_{\mathrm{ap}}}\mathfrak{l}_{\beta}N^{\frac12\e_{1}+\frac12\mathfrak{b}\e_{\mathrm{RN},1}+\frac12\e_{\mathrm{RN},1}}\|\mathsf{R}^{\mathfrak{b}}\|_{\omega;\infty}. \label{eq:BGI2II11}
\end{align}
If we could apply Lemma \ref{lemma:xr} with the aforementioned choices then \eqref{eq:BGI2II11} would follow just as \eqref{eq:BGI2I31} did, except the extra square root of $\mathfrak{l}_{\beta}$ would not be necessary in the second term on the RHS of \eqref{eq:BGI2II11}. However, for $\mathfrak{b}<\mathfrak{b}_{+}$ it is not necessarily true that the support length of $\mathsf{R}^{\mathfrak{b}}$ times $\mathfrak{l}=\mathfrak{l}_{\beta}$ is bounded above by $\mathfrak{l}_{N}$ in Definition \ref{definition:KPZ1}; if $\mathfrak{b}=\mathfrak{b}_{+}-1$ then the support length of $\mathsf{R}^{\mathfrak{b}}$ is $\mathrm{O}(N^{\e_{1}+\mathfrak{b}_{+}\e_{\mathrm{RN},1}})$ as we noted in the bullet point list prior to Lemma \ref{lemma:BGI2II1}. It is certainly possible that the support length of $\mathsf{R}^{\mathfrak{b}}$ is very close to or basically equal to $\mathfrak{l}_{N}=N^{1/2+\e_{\mathrm{RN}}}$ by construction in the statement of Proposition \ref{prop:BGI2}, so after multiplying by $\mathfrak{l}=\mathfrak{l}_{\beta}=N^{\beta}$ the resulting product may exceed $\mathfrak{l}_{N}$. This is remedied by the following observations.
\begin{itemize}
\item The only reason why we require the $\bar{\mathfrak{l}}\leq\mathfrak{l}_{N}$ constraint in the proof of Lemma \ref{lemma:xr} is so that we have a priori spatial regularity estimates for $\mathbf{Y}^{N}$ on the length-scale $\bar{\mathfrak{l}}$, which is defined in the statement of Lemma \ref{lemma:xr}, by construction in Definition \ref{definition:KPZ5}.
\item However, even if $\bar{\mathfrak{l}}=\mathfrak{l}_{\beta}N^{\e_{1}+\mathfrak{b}\e_{\mathrm{RN},1}+\e_{\mathrm{RN},1}}$ exceeds $\mathfrak{l}_{N}$, it only does by a factor of order $\mathfrak{l}_{\beta}=N^{\beta}$ for $\mathfrak{b}<\mathfrak{b}_{+}$. Indeed, $\mathfrak{l}_{\beta}^{-1}\bar{\mathfrak{l}}$ is always bounded by $\mathfrak{l}_{N}=N^{1/2+\e_{\mathrm{RN}}}$ in Definition \ref{definition:KPZ1} for all $\mathfrak{b}<\mathfrak{b}_{+}$ by construction of $\mathfrak{b}_{+}$ in the statement of Proposition \ref{prop:BGI2}. Rewriting spatial gradients of $\mathbf{Y}^{N}$ on length-scales of order $\mathfrak{l}_{\beta}N^{\e_{1}+\mathfrak{b}\e_{\mathrm{RN},1}+\e_{\mathrm{RN},1}}$ as order-$\mathfrak{l}_{\beta}$-many spatial gradients on the length-scale $N^{\e_{1}+\mathfrak{b}\e_{\mathrm{RN},1}+\e_{\mathrm{RN},1}}$, we may control the spatial regularity of $\mathbf{Y}^{N}$ on length-scales of order $\mathfrak{l}_{\beta}N^{\e_{1}+\mathfrak{b}\e_{\mathrm{RN},1}+\e_{\mathrm{RN},1}}$ by $\mathfrak{l}_{\beta}=N^{\beta}$ times spatial regularity estimates for $\mathbf{Y}^{N}$ on length-scale $N^{\e_{1}+\mathfrak{b}\e_{\mathrm{RN},1}+\e_{\mathrm{RN},1}}$.
\item The above length-$\mathfrak{l}_{\beta}N^{\e_{1}+\mathfrak{b}\e_{\mathrm{RN},1}+\e_{\mathrm{RN},1}}$ spatial regularity bound on $\mathbf{Y}^{N}$ is $\mathfrak{l}_{\beta}^{1/2}$ worse than what the proof of Lemma \ref{lemma:xr} needs it to be since the proof of Lemma \ref{lemma:xr} uses Holder regularity with exponent basically $1/2$ for $\mathbf{Y}^{N}$, and our bound is linear in $\mathfrak{l}_{\beta}$ rather than square root. Because the spatial regularity of $\mathbf{Y}^{N}$ only is relevant for the second term on the RHS of \eqref{eq:xr}/\eqref{eq:BGI2II11}, this is why we get \eqref{eq:BGI2II11} with $\mathfrak{l}_{\beta}$ in the second term on the RHS and not its square root as Lemma \ref{lemma:xr} says; see Remark \ref{remark:xr}.
\end{itemize}
Given \eqref{eq:BGI2II11}, like the proof of Lemma \ref{lemma:BGI2I3}, it suffices to use $N^{\e_{1}/2+\mathfrak{b}\e_{\mathrm{RN},1}/2+\e_{\mathrm{RN},1}/2}|\mathsf{R}^{\mathfrak{b}}|\lesssim N^{10\e_{\mathrm{ap}}}N^{-\e_{1}/2-\mathfrak{b}\e_{\mathrm{RN},1}/2+\e_{\mathrm{RN},1}/2}\lesssim N^{-1/4+10\e_{\mathrm{ap}}+\e_{\mathrm{RN},1}/2}$ for $\mathfrak{b}>\mathfrak{b}_{\mathrm{mid}}$ and $\beta=999^{-99}$.
\end{proof}
\begin{proof}[Proof of \emph{Lemma \ref{lemma:BGI2II2}}]
We directly follow the proof of Lemma \ref{lemma:BGI2I5} verbatim, but we replace $\mathfrak{l}_{\beta,\mathfrak{b}}$ therein by $\mathfrak{l}_{\beta}=N^{\beta}$. For the sake of precision/clarity, the estimates \eqref{eq:BGI2I51} and \eqref{eq:BGI2I52} both hold with the previous length-scale replacement and for $\mathfrak{b}$-indices of interest in the current lemma, as do \eqref{eq:BGI2I52b} and \eqref{eq:BGI2I53}. Moreover, it is easy to check that in the latter two of these bounds, the upper bounds with the aforementioned length-scale replacement, after plugging into \eqref{eq:BGI2I51} and \eqref{eq:BGI2I52} and taking $2/3$-powers, are controlled by the RHS of the proposed estimate \eqref{eq:BGI2II2}. Indeed, given $\mathfrak{b}>\mathfrak{b}_{\mathrm{mid}}$, we have $\kappa=N^{-3\e_{\mathrm{ap}}/2}N^{-3\beta/4}\|\mathsf{R}^{\mathfrak{b}}\|_{\omega;\infty}^{3/2}\geq N^{-3\e_{\mathrm{ap}}/2-3\beta/4+3/8}$, which is enough to control \eqref{eq:BGI2I52b}; see Definition \ref{definition:BGI2I2}. For \eqref{eq:BGI2I53}, all that we need to estimate the RHS of \eqref{eq:BGI2I53} is $\mathfrak{l}_{\beta,\mathfrak{b}}\geq N^{\beta}$ since $\mathfrak{t}_{\mathfrak{j}}\geq N^{-2}$; by construction, we still have $\mathfrak{l}_{\beta}=N^{\beta}$, so replacing $\mathfrak{l}_{\beta,\mathfrak{b}}$ by $\mathfrak{l}_{\beta}$ is not an issue.
\end{proof}
\begin{proof}[Proof of \emph{Lemma \ref{lemma:BGI2II3}}]
We directly follow the proof of Lemma \ref{lemma:BGI2I6} verbatim except we replace $\mathfrak{l}_{\beta,\mathfrak{b}}$ therein with $\mathfrak{l}_{\beta}=N^{\beta}$ and we replace $\mathfrak{t}_{\mathfrak{j}_{+}}$ therein with $\mathfrak{t}_{\mathfrak{j}_{+}}$ defined in the statement of Lemma \ref{lemma:BGI2II2}/Lemma \ref{lemma:BGI2II3}. Similar to the proof of Lemma \ref{lemma:BGI2II2}, it is enough to verify that the estimates \eqref{eq:BGI2I61}, \eqref{eq:BGI2I62}, \eqref{eq:BGI2I63}, and \eqref{eq:BGI2I64} from the proof for Lemma \ref{lemma:BGI2I6} that we are following still hold with the aforementioned length-scale and time-scale replacements.  For \eqref{eq:BGI2I61} and \eqref{eq:BGI2I62}, this is because Lemma \ref{lemma:hoe2} and Lemma \ref{lemma:le2} do not care about the space-time scales in terms of applicability. For \eqref{eq:BGI2I63} and \eqref{eq:BGI2I64}, this is a consequence of power-counting in $N$.  For \eqref{eq:BGI2I63}, it is enough to note $\kappa\geq N^{-3\e_{\mathrm{ap}}-3\beta/4+\frac38}$ after our replacement $\mathfrak{l}_{\beta,\mathfrak{b}}\to\mathfrak{l}_{\beta}$ as noted in the proof of Lemma \ref{lemma:BGI2II2}. We clarify weakening $\mathfrak{t}_{\mathfrak{j}_{+}}$ from the proof of Lemma \ref{lemma:BGI2I6} to $\mathfrak{t}_{\mathfrak{j}_{+}}$ in the current lemma, which only weakens \eqref{eq:BGI2I64} by a factor of $N^{\beta/2}$, gets dominated by $\mathfrak{l}_{\beta}^{-1}\lesssim N^{-\beta}$ obtained by replacing $\mathfrak{l}_{\beta,\mathfrak{b}}$ by $\mathfrak{l}_{\beta}$ in \eqref{eq:BGI2I64}. 
\end{proof}

\appendix
\section{Auxiliary Estimates}\label{section:aux}
\subsection{Heat Estimates}
We start with {the following}, from which heat kernel estimates ultimately follow.
\begin{lemma}\label{lemma:heat}
\fsp {Let us define $\mathbf{H}^{N,\Z}$ to be the full-line heat kernel on $\Z$ satisfying the following conditions.
\begin{itemize}
\item Define $\Delta_{\Z}^{!!}=N^{2}\Delta_{\Z}$ and $\grad_{\Z,-1}^{!}=N\grad_{\Z,-1}$, with $\Delta_{\Z}$ the Laplacian on $\Z$ and $\grad_{\Z,-1}$ the negative-direction gradient on $\Z$.
\item Provided $0\leq S\leq T$ and $x,y\in\Z$, we have $\mathbf{H}_{S,S,x,y}^{N,\Z}=\mathbf{1}_{x=y}$ and $\partial_{T}\mathbf{H}_{S,T,x,y}^{N,\Z}=2^{-1}\Delta_{\Z}^{!!}\mathbf{H}_{S,T,x,y}^{N,\Z}+\bar{\mathfrak{d}}\grad_{\Z,-1}^{!}\mathbf{H}_{S,T,x,y}^{N,\Z}$.\end{itemize}
We have the following identity relating $\mathbf{H}^{N}$ and $\mathbf{H}^{N,\Z}$ and the following Chapman-Kolmogorov equation, in which $S\leq R\leq T$ and $x,y\in\mathbb{T}_{N}=\llbracket0,N-1\rrbracket$:}
\begin{align}
{\mathbf{H}_{S,T,x,y}^{N} \ = \ {\sum}_{\mathfrak{k}\in\Z}\mathbf{H}_{S,T,x,y+\mathfrak{k}|\mathbb{T}_{N}|}^{N,\Z} \quad \mathrm{and} \quad \mathbf{H}_{S,T,x,y}^{N} \ = \ {\sum}_{w\in\mathbb{T}_{N}}\mathbf{H}_{R,T,x,w}^{N}\mathbf{H}_{S,R,w,y}^{N}.} \label{eq:heatid}
\end{align}
\end{lemma}
\begin{proof}
{To show the first identity in \eqref{eq:heatid}, note both sides are equal to $\mathbf{1}_{x=y}$ if $S=T$. Indeed, if $x,y\in\mathbb{T}_{N}$, then $x=y+\mathfrak{k}|\mathbb{T}_{N}|$ can only happen for $\mathfrak{k}=0$. Next, we note that both sides vanish under $\partial_{T}-\mathscr{L}_{N}$ for $T>S$, where $\mathscr{L}_{N}$ acts on $x$. By uniqueness of solutions to linear ODEs, the first identity holds. To show the second identity, note both sides equal $\mathbf{H}^{N}_{S,R,x,y}$ at $T=R$. Then, note both sides vanish under $\partial_{T}-\mathscr{L}_{N}$ for $T>R$, where $\mathscr{L}_{N}$ acts on $x$. So, the second identity holds, again, by uniqueness.}
\end{proof}
\begin{definition}\label{definition:heat2}
\fsp Provided $\mathfrak{l},\mathfrak{l}'\in\Z$ and any function $\phi:\mathbb{T}_{N}\to\R$, we define the composition $\grad_{\mathfrak{l},\mathfrak{l}'}^{\mathbf{X}}\varphi=\grad_{\mathfrak{l}}^{\mathbf{X}}(\grad_{\mathfrak{l}'}^{\mathbf{X}}\phi)$.
\end{definition}
{The following result collects pointwise (and summed) estimates for the $\mathbf{H}^{N}$ heat kernel, which can be interpreted as those for a Gaussian heat kernel (or its periodic version) at times of order $N^{2}$. Proving them amounts to the following steps. First, to prove the pointwise and spatial regularity estimates listed below, it suffices to assume $\bar{\mathfrak{d}}=0$. Indeed, the $\mathbf{H}^{N}$ heat kernel is the density for a symmetric simple random walk plus constant speed drift. It is therefore the convolution of a Poisson density function (for the law of the position of the drift) with the $\mathbf{H}^{N}$ kernel for $\bar{\mathfrak{d}}=0$. Convolution with the Poisson density function is contractive in all pointwise and spatial regularity norms used below, so reduction to $\bar{\mathfrak{d}}=0$ follows. To prove bounds in the case of $\bar{\mathfrak{d}}=0$, it suffices to use the first identity in \eqref{eq:heatid} with bounds in Proposition A.1 and Corollary A.2 of \cite{DT}, which have sub-exponential decay in space, and their higher-order analogs, which are proven by the same method. To prove the time-regularity bounds below, it suffices to note that time gradients of $\mathbf{H}^{N}$ are time-integrals of its spatial gradients because of the PDE that $\mathbf{H}^{N}$ satisfies. Then, we can use spatial regularity estimates that we just explained. (In particular, even for mixed space-time gradients, we are always left with estimating iterated spatial gradients of $\mathbf{H}^{N}$.)}
\begin{prop}\label{prop:heat}
\fsp We first take $0\leq S\leq T\leq 1$. Provided any $\mathfrak{l},\mathfrak{l}'\in\Z$ and any $0\leq \nu\leq1$, we have the following estimates, in which spatial gradients act on $x\in\mathbb{T}_{N}$; recall $\mathbf{O}_{S,T}=|T-S|$:
\begin{align}
0\leq\mathbf{H}_{S,T,x,y}^{N} \quad \mathrm{and} \quad N^{\nu}\mathbf{O}_{S,T}^{\frac12\nu}\mathbf{H}_{S,T,x,y}^{N}+N^{2\nu}\mathbf{O}_{S,T}^{\nu}|\mathfrak{l}|^{-\nu}|\grad_{\mathfrak{l}}^{\mathbf{X}}\mathbf{H}_{S,T,x,y}^{N}|+N^{3\nu}\mathbf{O}_{S,T}^{\frac32\nu}|\mathfrak{l}\mathfrak{l}'|^{-\nu}|\grad_{\mathfrak{l},\mathfrak{l}'}^{\mathbf{X}}\mathbf{H}_{S,T,x,y}^{N}|\lesssim1. \label{eq:heatI}
\end{align}
We have the following summation estimates under the same assumptions made/with the same parameters prior to \eqref{eq:heatI}:
\begin{align}
{\sum}_{y\in\mathbb{T}_{N}}\mathbf{H}_{S,T,x,y}^{N} + N^{\nu}\mathbf{O}_{S,T}^{\frac12\nu}|\mathfrak{l}|^{-\nu}{\sum}_{y\in\mathbb{T}_{N}}|\grad_{\mathfrak{l}}^{\mathbf{X}}\mathbf{H}_{S,T,x,y}^{N}| + N^{2\nu}\mathbf{O}_{S,T}^{\nu}|\mathfrak{l}\mathfrak{l}'|^{-\nu}{\sum}_{y\in\mathbb{T}_{N}}|\grad_{\mathfrak{l},\mathfrak{l}'}^{\mathbf{X}}\mathbf{H}_{S,T,x,y}^{N}| \ \lesssim \ 1. \label{eq:heatII}
\end{align}
Additionally consider any time-scale $\mathrm{t}\geq0$. We have the following in which the time-gradient acts on $T\geq0$:
\begin{align}
N^{\nu}\mathbf{O}_{S,T}^{\frac32\nu}|\mathrm{t}|^{-\nu}|\grad_{\mathrm{t}}^{\mathbf{T}}\mathbf{H}_{S,T,x,y}^{N}| + N^{2\nu}\mathbf{O}_{S,T}^{2\nu}|\mathrm{t}|^{-\nu}|\mathfrak{l}|^{-\nu}|\grad_{\mathrm{t}}^{\mathbf{T}}\grad_{\mathfrak{l}}^{\mathbf{X}}\mathbf{H}_{S,T,x,y}^{N}| + \mathbf{O}_{S,T}^{\nu}|\mathrm{t}|^{-\nu}{\sum}_{y\in\mathbb{T}_{N}}|\grad_{\mathrm{t}}^{\mathbf{T}}\mathbf{H}_{S,T,x,y}^{N}| \ \lesssim \ 1. \label{eq:heatIII}
\end{align}
We now list heat operator estimates. For any $\phi:\R_{\geq0}\times\mathbb{T}_{N}\to\R$ and $\mathbb{I}\subseteq\R_{\geq0}$, we have space-time contraction estimates:
\begin{align}
\|\phi_{0,\bullet}\|_{0;\mathbb{T}_{N}}^{-1}\|\mathbf{H}^{N,\mathbf{X}}(\phi_{0,\bullet})\|_{1;\mathbb{T}_{N}}+(|\mathbb{I}|\wedge1)^{-1}\|\phi\|_{1;\mathbb{T}_{N}}^{-1}\|\mathbf{H}^{N}(\phi_{S,y}\mathbf{1}_{S\in\mathbb{I}})\|_{1;\mathbb{T}_{N}} \ \leq \ 1. \label{eq:heatopbasic}
\end{align}
Let us now recall notation of \emph{Definition \ref{definition:BGII1}}. Provided any $\mathfrak{r}\geq0$, we have the spatial-gradient estimates
\begin{align}
\|\phi\|_{1;\mathbb{T}_{N}}^{-1}\||\wt{\grad}_{\mathfrak{l}}^{\mathbf{X}}|\mathbf{H}^{N}(\phi)\|_{1;\mathbb{T}_{N}}+\mathfrak{r}^{-\frac12}\|\phi\|_{1;\mathbb{T}_{N}}^{-1}\||\wt{\grad}_{\mathfrak{l}}^{\mathbf{X}}|\mathbf{H}^{N}(\phi_{S,y}\mathbf{1}_{S\geq T-\mathfrak{r}})\|_{1;\mathbb{T}_{N}}+|\mathfrak{l}|^{-1}\|\phi\|_{1;\mathbb{T}_{N}}^{-1}\|\mathbf{H}^{N}(\grad_{\mathfrak{l}}^{\mathbf{X}}\phi)\|_{1;\mathbb{T}_{N}} \ \lesssim \ 1. \label{eq:heatIV}
\end{align}
We have the following time-regularity heat operator estimates if $\mathrm{t}\geq N^{-2}$; below we take $\gamma>0$ arbitrary:
\begin{align}
N^{-\gamma}|\mathrm{t}|^{-1}\|\phi\|_{1;\mathbb{T}_{N}}^{-1}\|\grad_{\mathrm{t}}^{\mathbf{T}}\mathbf{H}^{N}(\phi)\|_{1;\mathbb{T}_{N}}+N^{-\gamma}|\mathrm{t}|^{-1}\|\phi\|_{1;\mathbb{T}_{N}}^{-1}\|\mathbf{H}^{N}(\grad_{\mathrm{t}}^{\mathbf{T}}\phi)\|_{1;\mathbb{T}_{N}} \ \lesssim_{\gamma} \ 1. \label{eq:heatV}
\end{align}
The estimates in \emph{\eqref{eq:heatV}} also hold for $\mathrm{t}\in\R$ in general. Lastly, for any possibly random $\mathrm{t}_{0}\geq0$, we have the following two identities, the first by the Chapman-Kolmogorov equation in \eqref{eq:heatid} and the second by combining the first with the spatial contraction in \eqref{eq:heatopbasic}:
\begin{align}
\mathbf{H}^{N}_{T,x}(\phi_{S,y}\mathbf{1}_{S\leq(\mathrm{t}_{0}\wedge T)}) \ = \ \mathbf{H}^{N,\mathbf{X}}_{T-(\mathrm{t}_{0}\wedge T),x}\left(\mathbf{H}^{N}_{\mathrm{t}_{0}\wedge T,w}(\phi)\right) \quad \mathrm{and} \quad \|\mathbf{H}^{N}(\phi_{S,y}\mathbf{1}_{S\leq(\mathrm{t}_{0}\wedge T)})\|_{1;\mathbb{T}_{N}} \ \leq \ \|\mathbf{H}^{N}(\phi)\|_{\mathrm{t}_{0};\mathbb{T}_{N}}. \label{eq:heatVI}
\end{align}
\end{prop}
\subsection{Martingale Estimates}
We provide a generalization of the martingale inequality from Lemma 3.1 in \cite{DT}. The issue with Lemma 3.1 in \cite{DT} is that it only holds for the Gartner transform, as its explicit formula was important in the proof. On the other hand, the proof of Lemma 3.1 in \cite{DT} uses this explicit formula only to estimate the short-time behavior of the Gartner transform. Thus, because short-time behavior does not depend on explicit formulas, we have the following generalization to other processes such as $\mathbf{U}^{N}$ from Definition \ref{definition:KPZ5}, which is important to analyze the $\mathbf{U}^{N}\d\xi^{N}$ term in the $\mathbf{U}^{N}$ equation in Definition \ref{definition:KPZ5}. However, the following generalization of Lemma 3.1 of \cite{DT} is similar in proof and statement, so we refer to Lemma 3.1 in \cite{DT}.
\begin{lemma}\label{lemma:mg}
\fsp Consider any $\phi:\R_{\geq0}\times\mathbb{T}_{N}\to\R$ and the following local quadratic function of $\phi$ provided fixed times $0\leq\mathfrak{t}_{1}\leq\mathfrak{t}_{2}$; in the following, we additionally define $\lfloor t\rfloor_{N}$ as the largest element in $N^{-2}\Z_{\geq0}$ that is less than $t$:
\begin{align}
\wt{\phi}_{R,x,w}^{\mathfrak{t}_{1},\mathfrak{t}_{2}} \ &\overset{\bullet}= \ {\sup}_{\mathfrak{r}'\in[\mathfrak{t}_{1},\mathfrak{t}_{2}): \ \lfloor\mathfrak{r}'\rfloor_{N}=\lfloor R\rfloor_{N}}{\sup}_{|\mathfrak{j}|\leq1} |\phi_{\mathfrak{r}',x,w+\mathfrak{j}}\phi_{\mathfrak{r}',x,w}|.
\end{align}
Take $\mathbf{X}^{N}$ on $\R_{\geq0}\times\mathbb{T}_{N}$ satisfying the following for $\mathrm{V}_{i}:\R_{\geq0}\times\mathbb{T}_{N}\times\Omega\to\R$, and $\mathfrak{l}\in\Z$ fixed; recall $\mathscr{L}_{N}$ in \emph{Proposition \ref{prop:mSHE+}}:
\begin{align}
\d\mathbf{X}_{T,x}^{N} \ = \ \mathscr{L}_{N}\mathbf{X}_{T,x}^{N}\d T+\mathrm{V}_{3;T,x}\d T+\mathbf{X}_{T,x}^{N}\d\xi_{T,x}^{N} + \mathrm{V}_{1;T,x}\mathbf{X}_{T,x}^{N}\d T + \grad_{\mathfrak{l}}^{\mathbf{X}}(\mathrm{V}_{2;T,x}\mathbf{X}_{T,x}^{N})\d T.
\end{align}
Suppose $|\mathrm{V}_{i}|\leq N^{3/2}$ uniformly over $\R_{\geq0}\times\mathbb{T}_{N}\times\Omega_{\mathbb{T}_{N}}$. For any deterministic $p\geq1$ and $0\leq\mathfrak{t}_{1}\leq\mathfrak{t}_{2}$ and $\phi:\R_{\geq0}\times\mathbb{T}_{N}\to\R$,
\begin{align}
\|\int_{\mathfrak{t}_{1}}^{\mathfrak{t}_{2}}\sum_{w\in\mathbb{T}_{N}}\phi_{R,x,w}\mathbf{X}_{R,w}^{N}\d\xi_{R,w}^{N}\|_{\omega;2p}^{2} \ \lesssim_{p} \ \int_{\mathfrak{t}_{1}}^{\mathfrak{t}_{2}}N\left(\sup_{w\in\mathbb{T}_{N}}\|\mathbf{X}_{\lfloor R\rfloor_{N},w}^{N}\|_{\omega;2p}^{2}\right)\sum_{w\in\mathbb{T}_{N}}\wt{\phi}_{R,x,w}^{\mathfrak{t}_{1},\mathfrak{t}_{2}}\d R. \label{eq:mg}
\end{align}
Lastly \emph{\eqref{eq:mg}} holds for $\mathbf{X}^{N}=\mathbf{U}^{N}$ in \emph{Definition \ref{definition:KPZ5}} and $\mathbf{X}^{N}=\mathbf{Q}^{N}$ in \emph{Definition \ref{definition:KPZ7}} and $\mathbf{X}^{N}=\mathbf{C}^{N}$ in \emph{Definition \ref{definition:KPZ111}}.
\end{lemma}
Before we discuss the proof, we introduce a brief digression. Consider the fundamental solution $\mathbf{J}$ with variables in $\R_{\geq0}^{2}\times\mathbb{T}_{N}^{2}$ and $\mathbf{J}_{S,S,x,y}=\mathbf{1}_{x=y}$ to the following \emph{deterministic} parabolic equation on $\R_{\geq0}\times\mathbb{T}_{N}$ whose utility we explain afterwards:
\begin{align}
\partial_{T}\mathbf{J}_{S,T,x,y} \ = \ \mathscr{L}_{N,\mathbf{J}}\mathbf{J}_{S,T,x,y} \ = \ \mathscr{L}_{N}\mathbf{J}_{S,T,x,y} + 99N^{3/2}{\sum}_{|\mathfrak{j}|\leq|\mathfrak{l}|}\mathbf{J}_{S,T,x+\mathfrak{j},y} + 99N^{3/2}.
\end{align}
If we ``forget" the $\mathbf{X}^{N}\d\xi^{N}$ term in the $\mathbf{X}^{N}$ equation in Lemma \ref{lemma:mg}, the resulting equation is stochastic only via the $\mathrm{V}$ functions. Also, the resulting linear equation has fundamental solution controlled by $\mathbf{J}$, as the coefficients in said equation are bounded by $N^{3/2}$. This motivates the following PDE estimate that follows by standard estimates for $\mathscr{L}_{N}$ and the Gronwall inequality.
\begin{lemma}\label{lemma:mg2}
\fsp We have $\mathbf{J}\geq0$ and, defining $\mathbf{J}_{x,y}^{\mathrm{s}}=\sup_{0\leq\mathrm{t}\leq N^{-2}}\mathbf{J}_{\mathrm{s},\mathrm{s}+\mathrm{t},x,y}$, we have the deterministic estimate
\begin{align}
\sup_{\mathrm{s}\geq0}\sup_{x\in\mathbb{T}_{N}}{\sum}_{y\in\mathbb{T}_{N}}\mathbf{J}^{\mathrm{s}}_{x,y} \ \lesssim \ 1. \label{eq:mg2I}
\end{align}
\end{lemma}
\begin{proof}[Proof of \emph{Lemma \ref{lemma:mg}}]
The estimate \eqref{eq:mg} is basically that of Lemma 3.1 in \cite{DT}, except we take spatial suprema of moments on the RHS. In particular, in view of the paragraph preceding Lemma \ref{lemma:mg} it suffices to show, for $\lfloor t\rfloor_{N}$ defined in Lemma \ref{lemma:mg},
\begin{align}
\sup_{x\in\mathbb{T}_{N}}\|\mathbf{X}_{t,x}^{N}\|_{\omega;2p}^{2} \ \lesssim_{p} \ \sup_{x\in\mathbb{T}_{N}}\|\mathbf{X}_{\lfloor t\rfloor_{N},x}^{N}\|_{\omega;2p}^{2}.
\end{align}
We note $\mathbf{X}_{t}^{N}$ can be controlled by $\mathbf{J}_{\lfloor t\rfloor_{N},t}$ spatially integrated against $\mathbf{X}_{\lfloor t\rfloor_{N}}^{N}$ times an exponential of a Poisson clock counter that was introduced in the proof of Lemma 3.1 of \cite{DT}. We then follow the proof of Lemma 3.1 of \cite{DT} upon estimating the spatial integral of $\mathbf{J}$ against $\|\mathbf{X}^{N}_{\lfloor t\rfloor_{N}}\|_{\omega;2p}^{2}$ by the spatial supremum of the latter $\mathbf{X}^{N}$-moment, Lemma \ref{lemma:mg2}, and exponential estimates for the Poisson distribution in the proof of Lemma 3.1 of \cite{DT}.
\end{proof}
\subsection{Short Time Estimates}
We provide a general short-time bound, not with respect to moments like the short-time estimates used in Lemma \ref{lemma:mg} but space-time supremum norms. Lemma \ref{lemma:st2} follows by deterministic control on $\mathscr{L}_{N,\mathrm{V}}$ below and noting that jumps in $\mathbf{X}^{N}$, which are order $N^{-1/2}\mathbf{X}^{N}$, have polynomial-in-$N$ speed that cannot ring too much in very short times.
\begin{lemma}\label{lemma:st2}
\fsp Consider any process $\mathbf{X}^{N}$ on $\R_{\geq0}\times\mathbb{T}_{N}$ satisfying the following stochastic equation, where $\mathrm{V}_{i}$ are functionals on $\R_{\geq0}\times\mathbb{T}_{N}\times\Omega_{\mathbb{T}_{N}}$, and the operator $\mathscr{L}_{N,\mathrm{V}}$ is defined via the second equation below for $\mathfrak{l}\in\Z$:
\begin{align}
\d\mathbf{X}_{T,x}^{N} \ = \ \mathscr{L}_{N,\mathrm{V}}\mathbf{X}_{T,x}^{N}\d T + \mathbf{X}_{T,x}^{N}\d\xi_{T,x}^{N} \ = \ \mathscr{L}_{N}\mathbf{X}_{T,x}^{N}+\mathbf{X}_{T,x}^{N}\d\xi_{T,x}^{N} + \mathrm{V}_{1;T,x}\mathbf{X}_{T,x}^{N}\d T + \grad_{\mathfrak{l}}^{\mathbf{X}}(\mathrm{V}_{2;T,x}\mathbf{X}_{T,x}^{N})\d T.
\end{align}
Suppose $\mathrm{V}_{1}$ and $\mathrm{V}_{2}$ satisfy the estimates $|\mathrm{V}_{1}|+|\mathrm{V}_{2}|\lesssim N^{\frac32}$ uniformly in all variables, and suppose $|\mathfrak{l}|\lesssim1$. If $\mathbf{X}^{N}\not\equiv0$, we have the following estimate with overwhelming probability (see \emph{Definition \ref{definition:KPZ8}}) in which $\e_{\mathrm{ap},1}>0$ is a small universal constant:
\begin{align}
\sup_{|\mathrm{s}|\leq N^{-2}}\sup_{0\leq\mathrm{t}\leq1}\sup_{x\in\mathbb{T}_{N}} \|\mathbf{X}^{N}\|_{\mathrm{t};\mathbb{T}_{N}}^{-1}|\grad_{\mathrm{s}}^{\mathbf{T}}\mathbf{X}_{\mathrm{t},x}^{N}| \ \leq \ N^{-1/2+\e_{\mathrm{ap},1}}. \label{eq:st2}
\end{align}
\end{lemma}


\end{document}